\documentclass[11pt]{article}

\usepackage{amsmath,amsthm,latexsym,amsfonts,amssymb,epsfig,graphicx,setspace, color, hyperref, dsfont, indentfirst, tikz, array, tabularx}

\usetikzlibrary{positioning, patterns, shadings, 3d, math}

\usepackage[square,numbers]{natbib} 
\hypersetup{
    colorlinks=true,
    linkcolor=blue,
    filecolor=magenta,      
    urlcolor=cyan,
    pdftitle={Overleaf Example},
    pdfpagemode=FullScreen,
    }

\newtheorem{thm}{Theorem}[section]

\newtheorem{proposition}[thm]{Proposition}
\newtheorem{lemma}[thm]{Lemma}
\newtheorem{corollary}[thm]{Corollary}
\newtheorem{property}[thm]{Property}

\theoremstyle{definition}
\newtheorem{definition}[thm]{Definition}

\newtheorem{remark}[thm]{Remark}

\newtheorem{notation}[thm]{Notation}

\numberwithin{equation}{section}

\newcommand{\n}{-\mathds{N}^*}

\newcommand{\h}{\hspace{5mm}}

\title{Solenoids of Split Sequences}
\author{Sarasi Jayasekara}
\date{}

\begin{document}

\begin{center}
    \begin{Large}
        \textbf{Solenoids of Split Sequences}\\
    \end{Large}
    \vspace{1mm}
    By Sarasi Jayasekara
\end{center}



\begin{abstract}

Solenoids induced by split sequences are introduced, as the inverse limit object of a sequence of fold maps. The topology of a solenoid is explored, and it is established that solenoids have naturally arising singular foliated structures. Our main goal is to answer the question: ``When is a solenoid minimal, both in a topological sense, and a measure theoretic sense?"
To aid this, we introduce the notions of leaves, partial leaves and transversals of a solenoid and explore their properties. A combinatorial criterion for topological minimality of a solenoid, is introduced. 
The primary tool we construct to study dynamics of solenoids is contained in the following theorem:
When a given solenoid $X$ doesn't contain finite partial leaves, the space of transverse measures of $X$, denoted $TM(X)$, is equal to the inverse limit of a certain sequence of linear maps on convex cones. We use this machinery to show that $TM(X)$ is a finite dimensional cone, and then to provide a combinatorial criterion called ``Semi-Normality" that allows us to recognize a wide class of uniquely ergodic solenoids.
\end{abstract} 


\tableofcontents




\pagenumbering{arabic} \pagestyle{myheadings} \markboth{}{}

\vspace{1cm}

\begin{center}
    \section{Introduction}
\end{center}
\label{introduction}

\vspace{4mm}

\h Our main inspirations lie in studies of $transverse \ structures$, such as metrics or measures, invariant under specific $dynamical \ systems$, such as iterated interval exchange maps, foliated surfaces, and toral solenoids. 
While studying a particular class of dynamical systems,
a fundamental question of interest is the following: 

\vspace{4mm}

``Given a dynamical system $\mathcal{S}$, and an invariant structure type (such as metrics, measures or transverse measures), when does the given system $\mathcal{S}$, admit only one invariant structure of that type (up to scaling)?"

\vspace{4mm}
We say that a dynamical system is ``Uniquely Ergodic" if it admits only one invariant measure up to scaling. If a system is uniquely ergodic, it follows that this system cannot be broken apart into smaller sub-systems. 
In that same vein, for a given dynamical system $\mathcal{S}$, it's natural to study the space of all invariant structures (of a given type) on  $\mathcal{S}$, since the limitations $\mathcal{S}$ puts on the structures it admits, can expose a great deal about the underlying system itself.

\vspace{4mm}

In 1978, W. Veech conducted a study of Interval Exchange Transformations (IETs) where he laid out a combinatorial criterion under which an IET is uniquely ergodic \cite{Veech1978}. In 1982, he \cite{413dc63a-6607-304b-a389-80bc0b0e5018} and H. Masur \cite{bb0e1df0-cb17-3718-8ef6-9df709f6df42} independently showed that almost all $irreducible$ IET's are uniquely ergodic. In addition to this, Masur showed that almost all $minimal$ foliations (on a compact surface) are uniquely ergodic.
In 1985, S. Kerckhoff \cite{Kerckhoff_1985} generalized Veech and Masur's techniques, and presented a tool that he called a $simplicial \ system$ (which is a sequence of linear maps between simplicies), and used them to encode information about invariant structures on a given underlying dynamical system. 

\vspace{4mm}

In this paper, we build on an idea by L. Mosher, and 
introduce ``Solenoids of Split Sequences" as a generalization of Toral Solenoids, then study their topology and dynamics, in a similar vein to Veech's explorations of IET's.

\vspace{4mm}

A $toral \ solenoid$ is the inverse limit of a sequence $S_0 \xleftarrow[]{f_{-1}} S_{-1} \xleftarrow[]{f_{-2}} S_{-2} \xleftarrow[]{f_{-3}} ...$ where, 
\begin{enumerate}
    \item for each $j \in \n$, $S_j$ is a circle (with a metric on it), and,
    \item for each $j \in -\mathds{N}$, $f_j$ is an immersion that's a local isometry. 
\end{enumerate}

\vspace{4mm}

A $solenoid \ induced \ by \ a \ split \ sequence$ (formally introduced in Definition \ref{def: Solenoids and Proper Solenoids}) is the inverse limit of a sequence $G_0 \xleftarrow[]{f_{-1}} G_{-1} \xleftarrow[]{f_{-2}} G_{-2} \xleftarrow[]{f_{-3}} ...$ where, 
\begin{enumerate}
    \item for each $j \in \n$, $G_j$ is a finite core graph [Def. \ref{Core Graphs}], and,
    \item for each $j \in -\mathds{N}$, $f_j$ is a fold [Def. \ref{def: fold}]. 
\end{enumerate}

\vspace{4mm}

\begin{figure}[htp]
    \centering
    \includegraphics[width=10cm]{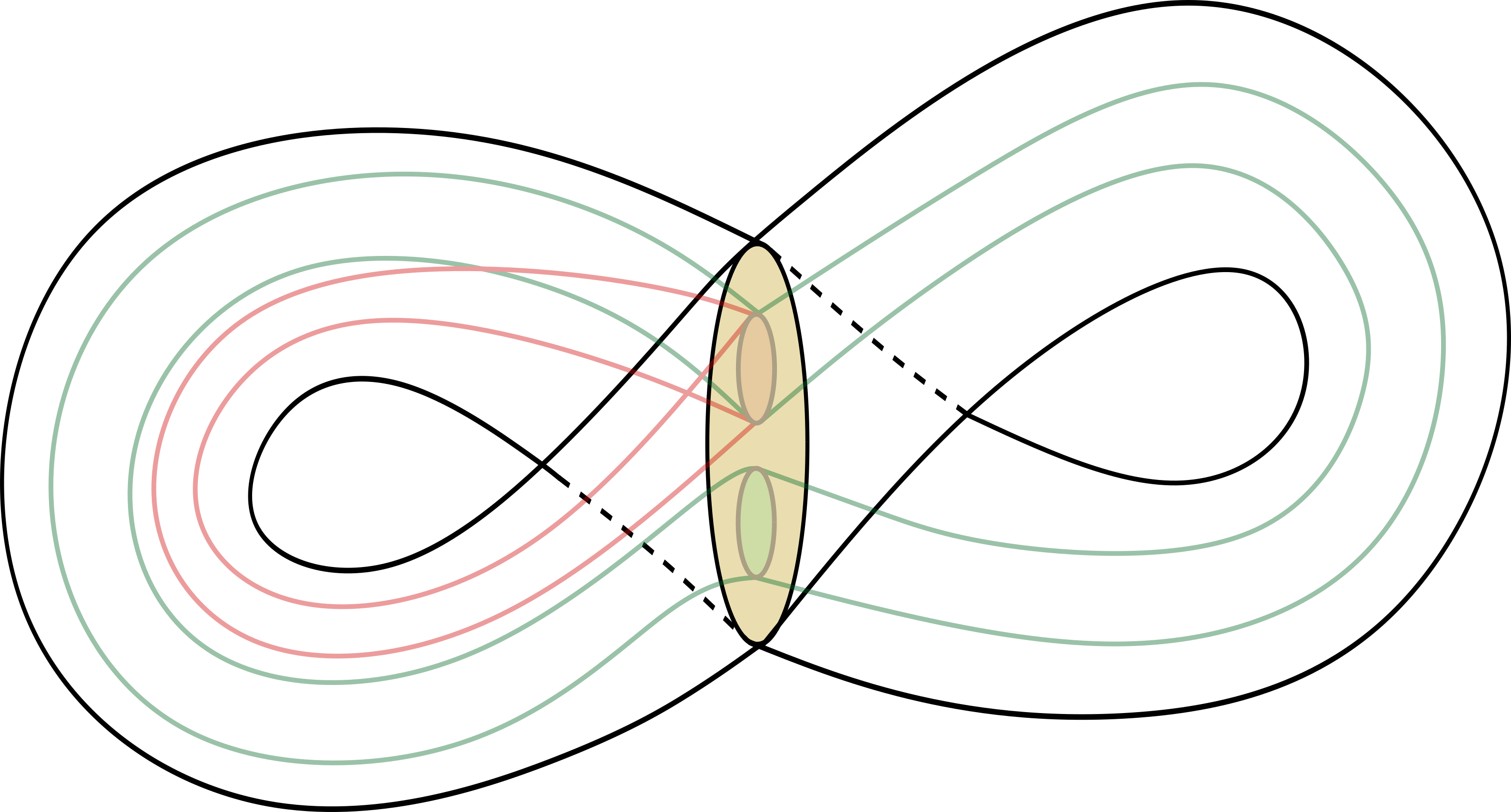}
    \caption{A Partial Visualization of a Solenoid}
    \label{fig:A Solenoid}
\end{figure}

\vspace{4mm}

The exploration of a solenoid's topology and dynamics, carried out in this paper,
is heavily inspired by studies of singular foliated surfaces \cite{Thurston2012}. As one moves from studying foliations on a torus, to studying singular foliations on other compact surfaces, the study becomes more elaborate, which will be mimicked by our generalization of toral solenoids to solenoids induced by split sequences. 

\vspace{4mm}

Our interest in studying the space of transverse measures of a solenoid, is further motivated by the eventuality of realizing spaces of solenoids inside $spaces$ $of$ $currents$ $on$ $free$ $groups$ (introduced by R. Martin in \cite{Reiner1995}, and further developed by I. Kapovich and co-authors in \cite{e27ba6e99705458aa1570a8f9450a550}, \cite{Kapovich_2009}, etc) and potentially using this perspective to investigate dynamics of $outer$ $automorphisms$ $of$ $free$ $groups$ in a similar vein to how M. Bestvina and co-authors utilized $unfolding$ $paths$ in \cite{Bestvina_2024}. While, in this text we provide a criterion for a solenoid being uniquely ergodic, we leave adopting this perspective into the context of $currents$, to future works.

\vspace{4mm}

\subsection*{Statements of Results}
\label{sec: Statements of Results}

\vspace{4mm}



In Chapter \ref{Ch:Solenoid}, we establish the fundamental concepts necessary to explore the topology and dynamics of a solenoid. 

\vspace{4mm}

In Chapter \ref{Ch: Solenoids as Foliated Spaces},
we draw inspiration from singular measured foliations on surfaces, and build a singular foliated atlas for a given solenoid. We show that, unlike in the surface case, this singular foliated structure naturally arises from the topology of the given solenoid.

\vspace{4mm}

To model the local neighborhoods of a solenoid, we use the following types of spaces. 
An ``Abstract Tunnel" $T$ is a product space $C \times I$ where $C$ is a totally disconnected set, and $I$ an open interval. For each $t \in I$, we call $C \times \{t\}$ a ``Cross Section of $T$", and for each $c \in C$, we call $\{c\} \times I$ a ``Pre-Leaf Segment of $T$".

\vspace{4mm}

To visualize the solenoid, we use two specific types of abstract tunnels, called ``Turn Tunnels, and Extended Turn Tunnels" [\ref{Turn Tunnels and Extended Turn Tunnels}].
Turn Tunnels are open in $X$ [\ref{turn tunnels are open}], and specific unions of extended turn tunnels, called ``Star Tunnel Components" [\ref{ST(O)}], are also open in $X$ [\ref{sing stnd nbhds vs star tunnel nbhds}].
Furthermore, each cross section of the aforementioned types of tunnels, is either compact or have a one point compactification in $X$ [\ref{ExTun-CS are compact}, \ref{H has a point compactification}].

\vspace{4mm}

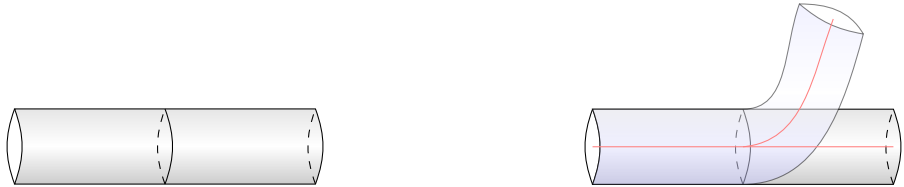
\begin{figure}[htbp]
    \centering
    \begin{minipage}{0.45\textwidth}
        \centering
        \begin{tikzpicture}
            \draw[top color=gray!20, bottom color=black!20, middle color=white!20] (0 , 0) -- (4 , 0) to[out=70, in=-70] (4 , 1) -- (0,1) to[out=-70, in=70] cycle;
            \draw (0 , 0) to[out=110, in=-110] (0 , 1); 
            \draw[dashed] (2 , 0) to[out=110, in=-110] (2 , 1); 
            \draw (2 , 0) to[out=70, in=-70] (2 , 1); 
            \draw[dashed] (4 , 0) to[out=110, in=-110] (4 , 1); 
        \end{tikzpicture}
    \end{minipage}
    \hfill 
    \begin{minipage}{0.45\textwidth}
        \centering
        \begin{tikzpicture}
            \draw[top color=gray!20, bottom color=black!20, middle color=white!20] (0 , 0) -- (4 , 0) to[out=70, in=-70] (4 , 1) -- (0,1) to[out=-70, in=70] cycle;
            \draw (0 , 0) to[out=110, in=-110] (0 , 1); 
            \draw[dashed] (2 , 0) to[out=110, in=-110] (2 , 1); 
            \draw (2 , 0) to[out=70, in=-70] (2 , 1); 
            \draw[dashed] (4 , 0) to[out=110, in=-110] (4 , 1); 

            \draw[top color=blue!10, bottom color=blue!20, middle color=white!20, opacity=0.5] (0 , 0) -- (2,0) to[out=0 , in=-105] (3.6 , 2) to[out=170 , in=-40] (2.75 , 2.4) to[out=-110 , in=0] (2,1) -- (0,1) to[out=-70, in=70] cycle;
            \draw[opacity=0.5] (3.6 , 2) to[out=120 , in=0] (2.75 , 2.4); 

            \draw[red!50] (0,0.5) -- (4,0.5);
            \draw[red!50] (2,0.5) to[out=5 , in=-110]  (3.2, 2.2);

            \draw[draw=none] (0,-1) rectangle (2,-1.4);
        \end{tikzpicture}
    \end{minipage}
    \caption{A Tunnel Neighborhood and a $3-$Pronged Star Tunnel Neighborhood}
    \label{fig: A Tunnel Neighborhood and a Star Tunnel Neighborhood}
\end{figure}

\vspace{4mm}

Let $n \in \mathds{N}$. An ``$n-$Pronged Star Set" is a quotient space obtained by taking $n$ many half closed intervals and gluing together the unique boundary point of each of the intervals. A ``Turn" of a star set $V$ is an open interval $I \subseteq V$ such that there is no open interval $\Tilde{I}$ such that $I \subset \Tilde{I} \subseteq V$. An ``$n-$Pronged Star Tunnel Set" [\ref{star tunnel sets}] is a quotient space obtained by taking an $n-$pronged star set $V$, and selecting for each turn $I$ of $V$, a totally disconnected set $C_I$, then identifying each turn $I$ of $V$ with one leaf segment in the abstract tunnel $C_I \times I$ [See Figure \ref{fig: A Tunnel Neighborhood and a Star Tunnel Neighborhood}]. The specific star tunnel sets we focus on in a solenoid $X$ ($star$ $tunnel$ $components$), will be open in $X$, and each of the underlying abstract tunnels in their constructions will have compact cross sections.

\vspace{4mm}

Before exploring a solenoid's topology, we impose the requirement that the underlying solenoid be ``Proper" [\ref{Proper Split Sequences}], which means that the corresponding split sequence,
\begin{equation*}
    \zeta: G_0 \xleftarrow[]{f_{-1}} G_{-1} \xleftarrow[]{f_{-2}} G_{-2} \xleftarrow[]{f_{-3}} ...
\end{equation*} 
satisfies the following: the subset in $G_0$ given by $\{ (f_{-1} \circ ... \circ f_j) (p) \in G_0 : j \in -\mathds{N},$ and $p$ is a natural vertex of $G_j \}$ is a finite set. Properness is meant to be a weaker criterion compared to ``Strong Properness" [\ref{Strongly Proper Split Sequences}], where we would require that all fold maps involved in $\zeta$ be graph maps. We will use properness as a baseline hypothesis while exploring the topology of solenoids (Ch \ref{Ch: Solenoids as Foliated Spaces} - Ch \ref{ch: Minimality and Mingling}), but when investigating dynamics of solenoids (Ch \ref{Ch: TM(X)} and beyond), we'll narrow our focus to strongly proper solenoids.

\vspace{4mm}

\begin{proposition}[\ref{STC is fin a open cover}, \ref{intersecting elements of STC(X)}]
    Each proper solenoid $X$ has a finite open cover $\mathcal{STC}(X)$, called the ``Star Tunnel Cover of $X$", such that each element $T \in \mathcal{STC}(X)$ is either a tunnel neighborhood [\ref{Tunnel Sets}] or a star tunnel neighborhood [\ref{star tunnel sets}] of $X$. Furthermore, for each pair of distinct elements $T, T' \in \mathcal{STC}(X)$ that intersect non-trivially, $T \cap T'$ is a finite disjoint union of tunnels in $X$. 
\end{proposition}


\vspace{4mm}

Given a proper solenoid $X$, a ``Singularity of $X$" is a point $x \in X$ where each neighborhood of $x$ in $X$ contains a star tunnel neighborhood. We will denote by ``$Sing(X)$" the set of singularities in $X$. 

\vspace{4mm}

\begin{lemma}[\ref{finitely many singularities}]
    Each proper solenoid $X$ has finitely many singularities. 
\end{lemma}

\vspace{4mm}

\begin{proposition}[\ref{how turn atlas charts intersect}]
    Let $X$ be a proper solenoid. Then $X - Sing(X)$ has a finite open cover $\mathcal{TC}(X)$, called the ``Tunnel Cover of $X$", such that each element $T \in \mathcal{TC}(X)$ is a tunnel neighborhood of $X$. Furthermore, for each pair of distinct elements $T, T' \in \mathcal{TC}(X)$ that intersect non-trivially, $T \cap T'$ is either a tunnel in $X$ or the disjoint union of two tunnels in $X$.
\end{proposition}

\vspace{4mm}

In Chapter \ref{Ch: Leaves and Transversals}, 
for a given solenoid $X$, we define the notions of ``Leaves and Partial Leaves" [\ref{Leaves, Partial Leaves}], ``Transversals" [\ref{Transversals}], and ``Transverse Measures" [\ref{Transverse Measures}], then explore their properties. A precursor to defining transversals of a solenoid $X$, is a special case called ``Turn Transversals" [\ref{Turn Transversals}] which are cross sections of a specific kind of abstract tunnel (called a turn tunnel) we find in $X$. In certain instances, focusing on turn transversals shall prove to be a convenient way of observing useful facts about all transversals. 

\vspace{4mm}

\begin{property}
    Let $X$ be a proper solenoid. Then leaves and transversals of $X$ possess the following properties.
    \begin{description}
    \item[\ref{leaves vs path components}.] Given $L \subseteq X$, $L$ is a leaf of $X$ (resp. a partial leaf of $X$) $\Longleftrightarrow L$ is a path component of $X$ (resp. of $X - Sing(X)$).
    \item[\ref{leaves are immersed graphs}.] Each leaf of $X$ (resp. partial leaf of $X$) is the image of a bijective immersion of a $1-$manifold or a branched $1-$manifold (resp. a $1-$manifold) in $X$.
    \item[\ref{Transversals are totally disconnected}.] Each transversal of $X$ is totally disconnected.
    \item[\ref{H has a point compactification}.] Each turn transversal of $X$ is either compact, or has a one point compactification in $X$.
    \item[\ref{Topology of a Turn Transversal}.] The topology of each turn transversal of $X$, is generated by the collection of turn transversals contained in it.
    \end{description}
\end{property}


Since leaves of solenoids are its path components, unlike in the case of foliations on surfaces, the singular foliations on solenoids arise naturally as a byproduct of their topology.

\vspace{4mm}

In Chapter \ref{ch: Minimality and Mingling}, we explore the concept of topological minimality in the context of solenoids, and provide a criterion for a given proper solenoid being minimal.

\vspace{4mm}

Given a proper solenoid $X$, we say that $X$ is minimal [\ref{Def: Minimal Solenoids}] if each partial leaf of $X$ is dense in $X$.
Given a split sequence $\zeta : G_0 \xleftarrow[]{f_{-1}} G_{-1} \xleftarrow[]{f_{-2}} G_{-2} \xleftarrow[]{f_{-3}} ...$,
we say that $\zeta$ is ``Fully Mingling" [\ref{Full Mingling}] if, there exists
a strictly decreasing sequence $\{J_k\}_{k \in \mathds{N}} \subset \n$ of non-positive integers such that, for each $k \in \mathds{N}$, the transition matrix [\ref{transition matrix}] of the fold composition $f_{J_{2k-1}} \circ ... \circ f_{J_{2k}}$ is positive. 

\vspace{4mm}

\begin{proposition}
    [The Mingling Lemma]
    If $\zeta$ is fully mingling, then the solenoid $X(\zeta)$ induced by $\zeta$ is minimal [\ref{The Mingling Lemma}].
\end{proposition}

\vspace{4mm}

In Chapter \ref{Ch: TM(X)}, we devise a mechanism called a ``Sequence of Pseudo-Weight Cones" that encode each (non-atomic) transverse measure on the underlying solenoid, using countably many parameters. This mechanism is meant to facilitate the usage of techniques W. Veech employed in \cite{Veech1978} (for his investigations of the dynamics of IETs) in the context of solenoids. 

\vspace{4mm}

Given a strongly proper split sequence $\zeta : G_0 \xleftarrow[]{f_{-1}} G_{-1} \xleftarrow[]{f_{-2}} G_{-2} \xleftarrow[]{f_{-3}} ...$, and $j \in \n$, ``The Level $j$ Pseudo-Weight Space of $\zeta$", denoted ``$R_j (\zeta)$", is the real vector space spanned by the set of natural edges in $G_j$. Then, ``The Sequence of Pseudo-Weight Cones of $\zeta$", is the sequence,
\begin{equation*}
    \Lambda_0 (\zeta) \xleftarrow[]{T_{-1} |_{\Lambda_{-1} (\zeta)}} \Lambda_{-1} (\zeta)  \xleftarrow[]{T_{-2} |_{\Lambda_{-2} (\zeta)}} \Lambda_{-2} (\zeta)  \xleftarrow[]{T_{-3} |_{\Lambda_{-3} (\zeta)}} ...
\end{equation*}
where,
\begin{enumerate}
    \item for each $j \in \n$, ``$\Lambda_j (\zeta) $", called the ``Level $j$ Pseudo-Weight Cone of $\zeta$", is the cone of non-negative vectors in $R_j (\zeta) $, and,
    \item for each $j \in -\mathds{N}$, ``$T_j$", called the ``Level $j$ Weight Map of $\zeta$", is the linear map (from $R_{j} (\zeta)$ to $R_{j-1} (\zeta)$) defined by the transition matrix of $f_j$.
\end{enumerate}

\vspace{4mm}

The aforementioned mechanism is similar to the independently developed techniques that appear in \cite{namazi2014ergodicdecompositionsfoldingunfolding} and \cite{BedHilLus2020}, that respectively investigate particular kinds of laminations, while we adopt a singular foliation perspective. Our proofs are grounded in the geometric setting of fold paths, (introduced in \cite{Sta1968} and further utilized in \cite{8dd38a0749a24ffa864bc39b814e7fbc}, \cite{handel2006axesouterspace}, \cite{Bestvina_2024}), and dependent on the topology of solenoids. Our applications are targeted towards providing a criterion for uniquely ergodic solenoids (which is included in this text), and proving that the said criterion is generic among all minimal solenoids (which shall be developed in future works).

\vspace{4mm}

To state the following result, we will assume that the underlying split sequence is ``Expanding" [\ref{def: Expanding Solenoids}] which is analogous to the criterion of a $graph \ tower$ being expanding in \cite{BedHilLus2020}. An expanding solenoid does not contain any finite partial leaves. It should be noted that, if $\zeta$ is a fully mingling split sequence (which assures that the induced solenoid $X(\zeta)$ is minimal), then $\zeta$ is expanding [\ref{FM implies Expanding}].



\vspace{4mm}

\begin{thm}[\ref{TM(X) is an inverse limit}]
    \label{thm1}
    Let $\zeta: G_0 \xleftarrow[]{f_{-1}} G_{-1} \xleftarrow[]{f_{-2}} G_{-2} \xleftarrow[]{f_{-3}} ...$ be a strongly proper expanding split sequence, and consider the solenoid $X=X(\zeta)$ induced by $\zeta$. 
    The space of transverse measures of $X$, $TM(X)$ can be expressed as the inverse limit of a sequence of linear maps between convex cones,
    \begin{equation*}
        \Lambda_0 (\zeta) \xleftarrow[]{T_{-1} |_{\Lambda_{-1} (\zeta)}} \Lambda_{-1} (\zeta) \xleftarrow[]{T_{-2} |_{\Lambda_{-2} (\zeta)}} \Lambda_{-2} (\zeta) \xleftarrow[]{T_{-3} |_{\Lambda_{-3} (\zeta)}} ...
    \end{equation*}
    such that,
    \begin{enumerate}
        \item[i.] for each $j \in \n$, $\Lambda_j (\zeta)$ is the non-negative cone in $\mathds{R}^{d_j}$ where $d_j$ is equal to the number of natural edges in $G_j$, and,
        \item[ii.] for each $j \in -\mathds{N}$, $T_j$ is the linear map (from $R_{j} (\zeta)$ to $R_{j-1} (\zeta)$) given by the transition matrix of the fold $f_j$. 
    \end{enumerate}
\end{thm}

\vspace{4mm}

A detailed proof of this theorem is contained in Chapter \ref{Ch: TM(X)}, which relies on everything else presented in this text up to that point. We consider this to be our main result. 
One immediate consequence of the above result is that the space of transverse measures of a (strongly proper expanding) solenoid is finite dimensional. 

\vspace{4mm}

\begin{corollary}[\ref{coarse ub for dim TM(X)}]
    Let $n \in \mathds{N}$ and let $\zeta: G_0 \xleftarrow[]{f_{-1}} G_{-1} \xleftarrow[]{f_{-2}} G_{-2} \xleftarrow[]{f_{-3}} ...$ be a strongly proper expanding split sequence. We say that $\zeta$ is `of \textit{rank n}' if for each $j \in \n$, the fundamental group of $G_j$ is the free group of rank $n$.
    Now suppose that $\zeta$ is of rank $n$, and let $X := X(\zeta)$. Then, 
    \begin{equation*}
        dim(TM(X)) \leq 3(n-1).
    \end{equation*}
\end{corollary}

\vspace{4mm}

In Chapter \ref{Ch: A Criterion for Unique Ergodcity}, drawing inspiration from W. Veech's work on Interval Exchange Transformations \cite{Veech1978}, and providing an application of Theorem \ref{thm1}, we present a combinatoral criterion for unique ergodicity of solenoids of split sequences.

\vspace{4mm}

We say that a proper split sequence $\zeta$ is ``Semi-Normal" (and that the solenoid $X(\zeta)$ induced by $\zeta$, is ``Semi-Normal") if, there exist $d \in \mathds{N}$, a positive $d \times d$ matrix $M$, and
a strictly decreasing sequence $\{J_k\}_{k \in \mathds{N}}$ of non-positive integers such that, for each $k \in \mathds{N}$,
    \begin{enumerate}
        \item the number of natural edges in $G_{J_k}$ is equal to $d$, and,
        \item $M$ is the transition matrix of $f_{J_{2k-1}} \circ ... \circ f_{J_{2k}}$.
    \end{enumerate}

\vspace{4mm}

\begin{thm}[\ref{main thm}]
    Let $X$ be a strongly proper expanding solenoid. If $X$ is 
    semi-normal, then $X$ is uniquely ergodic. 
\end{thm}

\vspace{4mm}

In addition to being uniquely ergodic, semi-normal solenoids are also fully mingling, assuring that they are topologically minimal as well.

\vspace{4mm}

In the studies conducted by W. Veech \cite{Veech1978}, \cite{413dc63a-6607-304b-a389-80bc0b0e5018} and S. Kerckhoff \cite{Kerckhoff_1985}, on their choice of dynamical systems (IET's and measured foliations), a criterion analogous to semi-normality was used to identify a class of uniquely ergodic  systems, and then it was shown that, with respect to a natural measure on the space of systems (the space of IET's and the space of measured foliations), the subspace of `semi-normal' dynamical systems has full measure (i.e. almost every system is semi-normal, and therefore uniquely ergodic). It still remains unknown whether almost every solenoid is semi-normal. Part of the challenge that exists in the context of solenoids, is conceptualizing the question: What shall be regarded as a `natural measure' on a space of solenoids?

\vspace{8mm}

\textbf{Acknowledgments.} The author would like to thank Lee Mosher for his ideas, inspiration, insightful conversations, support, and most of all, for having a contagious amount of enthusiasm for the subject matter. 

\vspace{4mm}





\vspace{4mm}

\begin{center}
    \section{Solenoids Induced by Split Sequences} \label{Ch:Solenoid}
\end{center}

\vspace{4mm}

\h In this chapter, we will define the ``Solenoid Induced by a Split Sequence", and explore its topological attributes. 

\vspace{4mm}

We will start by defining ``Folds", ``Splits", ``Split Sequences", ``Solenoids" and ``Proper Solenoids" [Section \ref{sec: Preliminary Definitions}]. We will limit our discussion to Proper Solenoids as it seems the most natural class of solenoids for our context. In Section \ref{sec: The Topology of a Proper Solenoid}, we will lay out a convenient basis for the topology of a proper solenoid $X$, called the ``Standard Basis of $X$", and use it to investigate some key features of $X$.

\vspace{4mm}

The following is a more detailed summary of Section \ref{sec: The Topology of a Proper Solenoid}.

\vspace{4mm}

Let $X$ be a proper solenoid. In Subsection \ref{sec: The Standard Basis for a Proper Solenoid}, we will explore the topology of a standard basis element $O$ of $X$ in terms of its path connected components called ``Plaques" [Definition \ref{def: plaque}]. More specifically, We will define ``Star Sets" [Definition \ref{Star Set}] and show (in a later part of the section) that each plaque of $O$ is a star set [Lemmas \ref{shapse of plaques 1}, and \ref{shapse of plaques 2}]. In Subsection \ref{sec: Fibers}, we will introduce ``Fibers" [Definition \ref{def: fibers}] of $X$ (that can be considered, in a loose sense, as being transverse to plaques) and show that each fiber of $X$ is compact [Property \ref{Fibers are compact}] and totally disconnected [Property \ref{Fibers are totally disconnected}]. Plaques and fibers of $X$ can be treated as precursors to the concepts of ``Leaves" and "Transversals" of $X$ that shall be explored in chapter \ref{Ch: Leaves and Transversals}.

\vspace{4mm}

In Subsection \ref{sec: Plaques, Leaf Interior Points and Singularities}, we will establish that points in a proper solenoid can be classified into two types called ``Leaf Interior Points" and ``Singularities" [Definition \ref{Singularity}]. To identify where singularities occur and then to understand the nature of those singularities, we will develop a tool called ``Star Chains" [Definition \ref{star chain}]. Then we will show that a proper solenoid has only finitely many singularities [Proposition \ref{finitely many singularities}]. 

\vspace{4mm}

Finally, in Section \ref{sec: Stabilizing Split Sequences}, we will observe that, given a proper split sequence $\zeta$, the process of studying its induced solenoid $X = X(\zeta)$ becomes easier, if $\zeta$ satisfies a collection of criteria that we shall call the ``Stabilizing Hypotheses" [Remark \ref{The Stabilizing Hypotheses}]. Then we will show that each proper solenoid can be induced by a proper stabilized split sequence [Proposition \ref{Every proper split sequence can be stabilized}]. And so it shall be justified, that from Chapter \ref{Ch: Solenoids as Foliated Spaces} onward, we will consider a proper stabilized split sequence and its induced solenoid, as our chosen objects of study. 

\vspace{4mm}

\vspace{4mm}

\subsection{Preliminary Definitions}\label{sec: Preliminary Definitions}

\subsubsection{Folds and Splits}


The concept of folds was first introduced in \cite{Sta1968} by John Stallings and used in various contexts since then. We will start with a constructive definition for a fold, since it will be the most fundamental building block of the discussion that follows.

\vspace{4mm}

\hypertarget{natural vertex}{Given} a directed graph $G$, a vertex $v$ of $G$ is called a ``Natural Vertex" if the valence of $v$ is greater than $2$. A directed edge of $G$ is called a ``Natural Edge" if both its initial and terminal vertices are natural vertices.

\vspace{4mm}

\begin{definition}[Fold]
    \label{def: fold}
    \hypertarget{fold}{Given a directed metric graph $G$, a ``Fold of $G$"} is a quotient map $f$ from $G$ to another metric graph $H$ defined as follows. We start with choosing two particular directed \hyperlink{natural vertex}{natural edges} $E_1$ and $E_2$ of $G$ that start at the same vertex $v$, a positive number $L \leq min_{i=1,2} \{ length(E_i)\}$, and two isometric embeddings $l_i : [0,L] \longrightarrow E_i $ that take $0$ to the common starting vertex $v$ such that,
    \begin{enumerate}
        \item $ l_1 ((0,L]) \bigcap l_2 ((0,L]) = \emptyset $,
        \item $f$ identifies $l_1 (t)$ with $l_2 (t)$ for every $t \in [0,L]$, and,
        \item $f$ restricted to $G - \bigcup_{i=1,2} Image(l_i)$ is an isometry onto its image in $H$.
    \end{enumerate}
\end{definition}

\vspace{4mm}

We shall call the two maps $l_1, l_2$ ``Supporting Parameterization Maps of $f$". \hypertarget{folding edges}{Furthermore,} $E_1, E_2$ shall be called the \hypertarget{folding edges and vertices}{``Folding Edges of $G$ rel $f$"}, and $v$ the ``Folding Vertex of $G$ rel $f$". Given a graph $G$, the phrase ``Folding $G$" shall refer to the act of creating another graph $H$ and a fold map $f$ from $G$ to $H$. \hypertarget{Splitting Vertex}{We} shall also give the name ``Splitting Vertex of $H$ rel $f$" to the unique point $w \in H$ determined by $w = f \circ l_1 (L) = f \circ l_2 (L) $. Note that, even in the case where $l_1(L) = l_2(L)$ is not a vertex of $G$, $w$ is a natural vertex of $H$. 

\vspace{4mm}

Note that in the above definition, the initial assumptions made while choosing $E_1$ and $E_2$ allows the possibility that they could be the same edge (directed the same way or the opposite). However, item 1. disallows the case that they are the same directed edge. While $E_1$ and $E_2$ being the same edge but directed the opposite way to each other is still allowed, item 1. makes sure that $f$ does not map a loop (a directed edge whose initial and terminal vertices are the same) into a single natural edge that is not a loop. 

\vspace{4mm}



\begin{remark}[Folds on Topological Graphs]
    \label{Folds on Topological Graphs}
    \hypertarget{Folds on Topological Graphs}{Even} though the above definition [\ref{def: fold}] of a ``Fold" was only laid out in the context of metric graphs, were $G$ just a graph without a metric, we can still define a ``Fold of $G$" by slightly altering our definition in the following way: Instead of $l_i$ being an isometric embedding we only demand that it be a homeomorphism on to its image for each $i = 1,2$. Instead of picking the interval $[0,L]$ as  the domain for those parameterizations, we may just choose the interval $[0,1]$ and demand that the image of $[0,1]$ under $l_i$ for each $i = 1,2$, be contained in $E_i$. Lastly, we will replace criterion (2.) with the requirement that $f$ restricted to $G - \bigcup_{i=1,2} Image(l_i)$ is a homeomorphism onto its image in $H$ (instead of an isometry).
\end{remark}

\vspace{4mm}

\begin{remark} [Fold Maps are Homotopy Equivalences]
    \label{Rmk: Fold Maps are Homotopy Equivalences}
    In the definition \ref{def: fold} of a fold, since we chose $L$ in such a way so that the two parameterizations $l_i$, map $L$ to two distinct points, $f$ is a homotopy equivalence. This ensures that for any $p \in G$, $f$ induces an isomorphism between the fundamental groups $\pi_1 (G,p)$ and $\pi_1 (H,f(p))$. 
    (Relevant properties of Fundamental Groups of Graphs can be found in \cite{Mosher2020TheTG} by L. Mosher.)
\end{remark}

\vspace{4mm}


\subsubsection*{Splits}
\label{sec: splits}


To make a distinction between a sequence of folds indexed forward (eg. $G_0 \xrightarrow[]{f_{0}} G_{1} \xrightarrow[]{f_{1}} G_{2} \xrightarrow[]{f_{2}} ...$) and a sequence indexed backwards (eg. $G_0 \xleftarrow[]{f_{-1}} G_{-1} \xleftarrow[]{f_{-2}} G_{-2} \xleftarrow[]{f_{-3}} ...$), we shall lay out the concept of ``Splitting a Graph" as the inverse action of ``Folding a Graph", and then we shall call a sequence of folds indexed backwards (i.e. an inverse sequence of folds) a ``Split Sequence" (to be precisely defined in the next Subsection). 

\vspace{4mm}

\begin{definition}[Split]
    \label{Split}
    \hypertarget{split}{Given} a graph $H$, ``Splitting $H$" is the act of creating another graph $G$ and a fold map $f$ from $G$ to $H$. Given graphs $H, G$, a``Split $s$ from $H$ to $G$" is the inverse relation of a fold $f$ from $G$ to $H$. 
    In this case, $f$ shall be called the ``Inverse Fold of  $s$" (denoted ``$\overline{s}$") and $s$ the ``Inverse Split of $f$" (denoted ``$\overline{f}$"). 
\end{definition}

\vspace{4mm}


\subsubsection{Split Sequences and Solenoids}
\label{sec: Split Sequences}

\vspace{4mm}

Here, we formally define ``Split Sequences", and introduce a criterion called ``Properness" to narrow down our discussion. We start with laying out some preliminary terminology.

\vspace{4mm}

\begin{definition}[Fold Compositions and Backtracking]
    \label{Fold Compositions and Backtracking}
    \hypertarget{Fold Compositions and Backtracking}{Let} $n \in \mathds{N}$, $i,j \in \n$ such that $i < j$, and let $G_j \xleftarrow[]{f_{j-1}} G_{j-1} \xleftarrow[]{f_{j-2}} G_{j-2} \xleftarrow[]{f_{j-3}} ...  \xleftarrow[]{f_i} G_{i}$ be a finite sequence such that,
    \begin{enumerate}
        \item for each $k \in \{ i , i+1 , ... , j-1, j \}$, $G_k$ is a finite graph of rank $n$
        \item for each $k \in \{ i , i+1 , ... , j-1 \}$, $f_{k} : G_k \longrightarrow G_{k+1}$ is a \hyperlink{fold}{fold}.
    \end{enumerate}
    We shall use the notation ``$f^i_j$" to indicate the composition of fold maps $f_{j-1} \circ f_{j-2} \circ ... \circ f_{i+1} \circ f_i $. We say that ``$f^i_j$ has No Backtracking" if $f^i_j$ restricted to $G_i - \{$natural vertices of $ G_i \} $ is locally injective. 
\end{definition}

\vspace{4mm}

\begin{definition}[Core Graphs of a Given Rank]
    \label{Core Graphs}
    \hypertarget{core graphs}{Let} $n \in \mathds{N}$. A graph $G$ is called a ``Core Graph of Rank $n$" or a ``Rank $n$ Core Graph" if $G$ has no valence 1 vertices, and the fundamental group of $G$ is the free group of rank $n$. 
\end{definition}
    
\vspace{4mm}

\begin{definition}[Split Sequences and Topological Split Rays of Rank $n$]
    \label{Split Sequence}
    \hypertarget{Split Sequence}{Let} $n \in \mathds{N}$. 
    A ``Split Sequence of Rank $n$" (resp. a ``Topological Split Ray of Rank $n$") is a sequence of folds between metric graphs (resp. topological graphs) $\zeta : G_0 \xleftarrow[]{f_{-1}} G_{-1} \xleftarrow[]{f_{-2}} G_{-2} \xleftarrow[]{f_{-3}} ...$ such that, 
    \begin{enumerate}
        \item for each $j \in \n$, $G_j$ is a core graph of rank $n$, and,
        \item for each $j \in -\mathds{N}$, the composition of folds $f^0_j$ has no backtracking.
    \end{enumerate}
    \hypertarget{level graph}{Furthermore,} for each $j \in \n$, $G_j$ is called the ``Level $j$ Graph of $\zeta$" and will be referred to as a ``Level Graph of $\zeta$" when specifying the level is unnecessary. 
\end{definition} 

\vspace{4mm}

Throughout the rest of this text, each time we introduce a split sequence denoted by $\zeta$, it shall be understood that it's level graphs and fold maps are denoted as in the definition above. 

\vspace{4mm}

\hypertarget{underlying split ray}{For} a given split sequence $\zeta$, the ``Underlying Topological Split Ray of $\zeta$" denoted ``$\zeta^{top}$" is the topological split ray obtained by stripping each level graph of $\zeta$ of its metric. 

\vspace{4mm}

\begin{definition}[Proper Split Sequences of Rank $n$]
    \label{Proper Split Sequences}
    \hypertarget{Proper Split Sequences}{We} say that a split sequence $\zeta$ of rank $n$, is a ``Proper Split Sequence of Rank $n$" if, 
    the set $\{ f^j_0 (v) \in G_0 : j \in - \mathds{N}, \ v$ is a \hyperlink{natural vertex}{natural vertex} of $G_j \}$ is finite.
\end{definition}

\vspace{4mm}

\begin{definition}[Strongly Proper Split Sequences of Rank $n$]
    \label{Strongly Proper Split Sequences}
    \hypertarget{Strongly Proper Split Sequences}{We} say that a split sequence $\zeta$ of rank $n$, is a ``Strongly Proper Split Sequence of Rank $n$" if, each fold in $\zeta$ is a graph map (or equivalently, for each $j \in \n$, the image of each natural vertex in $G_{j-1}$ under $f_{j-1}$ is a natural vertex in $G_j$).
\end{definition}

\vspace{4mm}




\subsubsection*{Solenoids as Inverse Limits}
\label{sec: Solenoids as Inverse Limits}

\vspace{4mm}

\begin{definition}[Solenoids and Proper Solenoids]
    \label{def: Solenoids and Proper Solenoids}
    Given a split sequence $ \zeta : G_0 \xleftarrow[]{f_{-1}} G_{-1} \xleftarrow[]{f_{-2}} G_{-2} \xleftarrow[]{f_{-3}} ... $, the ``Solenoid Induced by $ \zeta $" (denoted ``$ X ( \zeta ) $" or simply ``$X$" when there's no room for ambiguity) is the inverse limit of $ \zeta $ in the category of topological spaces. A solenoid induced by a proper (resp. strongly proper) split sequence shall be called a ``Proper (resp. Strongly Proper) Solenoid".
\end{definition}

\vspace{4mm}

From this point forward, we will limit our discussion to proper solenoids. And whenever we mention a solenoid $X(\zeta)$ or $X$, it will represent a proper solenoid.

\vspace{4mm}

Given the above setting, it's natural to think of $X$ as a purely topological object. However, we shall see that $X$  inherits a 1-dimensional singular-foliation structure from the level graphs of $\zeta$. We will explore this notion in more detail in Chapters \ref{Ch: Solenoids as Foliated Spaces} and \ref{Ch: Leaves and Transversals}. We shall lay out the groundwork necessary to aid this exploration, in the rest of this chapter. 


\vspace{4mm}

\subsection{The Topology of a Proper Solenoid}
\label{sec: The Topology of a Proper Solenoid}

\vspace{4mm}

Throughout this section, $\zeta : G_0 \xleftarrow[]{f_{-1}} G_{-1} \xleftarrow[]{f_{-2}} G_{-2} \xleftarrow[]{f_{-3}} ...$ will denote a given proper split sequence and $X = X(\zeta)$ its induced solenoid.

\vspace{4mm}

In Sub-section \ref{sec: The Standard Basis for a Proper Solenoid}, we will establish a convenient basis for the topology of $X$ called the ``Standard Basis". We will spend the next two subsections exploring key features of $X$, such as ``\hyperlink{fibers}{Fibers}" [Sub-section \ref{sec: Fibers}], ``\hyperlink{plaque}{Plaques}", ``\hyperlink{Singularity}{Leaf Interior Points}" and ``\hyperlink{Singularity}{Singularities}" [Sub-section \ref{sec: Plaques, Leaf Interior Points and Singularities}]. 

\vspace{4mm}

The following brief exposition lays out notations that we will frequently use in this text to refer to subsets and points of the solenoid $X$. It also contains, an exploration of $X$'s topology inherited from being an inverse limit object. The readers who are already familiar with the matter may skip to Sub-section \ref{sec: The Standard Basis for a Proper Solenoid} and look up the notations when relevant. 

\vspace{4mm}

\begin{remark}[$X$ as a Set and as a Topological Space]
    $X$ as a set can be expressed as $\{ (x_0 , x_{-1} , x_{-2} , ... ) \in \Pi_{j \in - \mathds{N}^*} G_j : $ for each $j \in -\mathds{N} , f_j (x_j) = x_{j+1} \}$. (We shall frequently use tuple-notation while referring to a point of the solenoid.) The Topology of $X$ is the subspace topology inherited from $\Pi_{j \in - \mathds{N}^*} G_j $. 
\end{remark}

\vspace{4mm}

\begin{notation}[Sub Inverse Sequences]
    \label{Sub Inverse Sequences}
    \hypertarget{Sub Inverse Sequences}{A} ``Sub Inverse Sequence of $\zeta$" is a sequence $\mathcal{U} :=  U_0 \xleftarrow[]{f_{-1} |_{U_{-1}}} U_{-1} \xleftarrow[]{f_{-2} |_{U_{-2}}} U_{-2} \xleftarrow[]{f_{-3} |_{U_{-1}}} ...$ such that for each $j \in \n$, $U_j \subseteq G_j$ with $f_{j-1} (U_{j-1}) = U_j$. (When expressing such a sequence, for notational convenience, for each $j \in -\mathds{N}$, we may write $f_{j}$ instead of $f_{j} |_{U_{j}}$, and it will be understood that we are referring to the relevant restriction of the fold map.)
\end{notation}

\vspace{4mm}

 We call $\mathcal{U}$ a ``Point Sub Inverse Sequence of $\zeta$" when for each $j \in \n$, $U_j$ is a single point of $G_j$. Note that a point sub inverse sequence of $\zeta$, uniquely determines a point in $X(\zeta)$.

\vspace{4mm}

\hypertarget{sub inverse limits notation}{In} the above setting, we will denote the inverse limit of $\mathcal{U}$ in the category of topological spaces ``$\varprojlim^{\zeta}_{j \in \n} U_j$", or simply ``$\varprojlim^{\zeta} \mathcal{U}$" when the context is clear.

\vspace{4mm}


\vspace{4mm}

Assume the notational set-up in Notation \ref{Sub Inverse Sequences}. Then $X(\mathcal{U}) := \varprojlim^{\zeta} \mathcal{U} \subseteq X(\zeta)$ and 
$X(\mathcal{U})$ inherits the subspace topology from $X(\zeta)$ (and in a broader sense from $ \Pi_{j \in - \mathds{N}^*} G_j $).

\vspace{4mm}

\begin{remark}[The Coordinate Projection Maps and the Topology of $X$]
    \label{The Coordinate Projection Maps and the Topology of X}
    \hypertarget{The Coordinate Projection Maps}{For} each $k \in \n$, let ``$\overline{\pi}_k$" denote the coordinate projection function from $\Pi_{j \in - \mathds{N}^*} G_j$ to  $G_k$ that takes each point of the form $(x_0 , x_{-1} , x_{-2} , ... ) \in \Pi_{j \in - \mathds{N}^*} G_j$ to $x_k \in G_k$. And we shall denote $\overline{\pi}_k |_{X(\zeta)}$, ``$\pi_k$". The topology of $\Pi_{j \in - \mathds{N}^*} G_j$  is the coarsest topology that makes the coordinate projection maps continuous. %

    \vspace{4mm}
    
    Therefore, the topology of $X$ is generated by the collection 
    $ \{ \pi^{-1}_j (U_j) \subseteq X \ | \ j \in \n, \ U_j$ is an open neighborhood of $G_j \}$.
\end{remark}

\vspace{4mm}

\begin{center}
\setlength{\unitlength}{1cm}
\begin{picture}(4,6) 
\label{Coordinate Projection Maps}
   \put(3,5.9){\vector(-1,0){2.26}}
        \put(3,5.8){\ X}
        \put(1.5,6){\ $\pi_0$}
        \put(0.1,5.8){\ $G_0$}
    \put(3,5.8){\vector(-1,-1){2.28}} 
    \put(0.37,3.67){\vector(0,1){2}} 
        \put(1.5,4.8){\ $\pi_{-1}$} 
        \put(0.1,3.3){\ $G_{-1}$} 
        \put(0.4,4.8){\ $f_{-1}$} 
    \put(3,5.7){\vector(-1,-2){2.32}} 
    \put(0.37,1.17){\vector(0,1){2}} 
        \put(1.5,2.8){\ $\pi_{-2}$} 
        \put(0.1,0.8){\ $G_{-2}$} 
        \put(0.4,2.5){\ $f_{-2}$} 
    \put(0.37,0.4){\vector(0,1){0.3}} 
    \put(0.175,0.3){\ .}
    \put(0.175,0.2){\ .}
    \put(0.175,0.1){\ .}
\end{picture}
\end{center}

\vspace{4mm}

Note that we can further refine the basis that generates the topology of $X$ to $\{ \pi^{-1}_j (U_j) \subseteq X \ | \ j \in \n, \ U_j \in \mathcal{B}_j \}$ where for each $j \in \n$, $\mathcal{B}_j$ is a selected convenient basis for the topology of $G_j$. We will carry out the process of constructing such a basis (called the ``Standard Basis of $X$") in subsection \ref{sec: The Standard Basis for a Proper Solenoid}.

\vspace{4mm}

\subsubsection{The Standard Basis for a Proper Solenoid}
\label{sec: The Standard Basis for a Proper Solenoid}

\vspace{4mm}

In this Sub-section, we will establish a basis for the topology of $X$ called the ``\hyperlink{standard basis}{Standard Basis}". 
Later, we will use this basis to investigate some key properties of the solenoid. 
But first we shall lay out some preliminary definitions and terminology. 

\vspace{4mm}

\begin{definition}[n-pronged Star Sets and Star Neighbourhoods]
\label{Star Set}
\hypertarget{star set}{For a} positive integer $ n > 1$, the ``Model $n-$Pronged Star Set" is $ \frac{\{1,...,n\} \times [0,1)}{(i,0) \sim  (j,0)} $ and is denoted by $Star(n)$. Given a topological space $Y$, an ``$n-$Pronged Star Set of $Y$" is a subset of $Y$ homeomorphic to $Star(n)$. An ``$n-$Pronged Star Neighborhood of $Y$" is an open subset of $Y$ homeomorphic to $Star(n)$. 
\end{definition}

\vspace{4mm}

\begin{definition}[The Center, Prongs and Turns of a Star Set]
    \label{center, prong, turn}
    \hypertarget{center, prong, turn}{Assume} the set-up of the above definition. The ``Center" of $Star(n)$ is the unique point created by identifying points $(j,0) \in \{1,...,n\} \times [0,1)$. For each $i \in \{1,...,n\}$, the subset of $Star(n)$ corresponding to $\{i\} \times (0,1)$ is called a ``Prong of $Star(n)$" (sometimes called an ``Open Prong" for clarity). A set that is the union of two distinct prongs and the center is called a ``Turn of $Star(n)$". When a star set $S$ is recognized by being homeomorphic to $Star(n)$, via a homeomorphism $h: Star(n) \longrightarrow S$, the center of $S$ is the unique point that the center of $Star(n)$ maps to, under $h$. A prong (resp. turn) of $S$ is a subset of $S$ that is the image of a prong (resp. turn) of $Star(n)$, under $h$. 
\end{definition}

\vspace{4mm}

\begin{remark}[Star Neighborhoods of a Graph]
    Let $G$ be a graph. Note that, in $G$, a $2-$pronged star neighborhood is simply an embedded open interval contained in a natural edge-interior of $G$. And in that case, the notion of ``Center" in that $2-$pronged star neighborhood is an artificial designation we give to one of its interior points. On the other hand, when $n>2$, an $n-$pronged star neighborhood of a graph is simply a neighborhood of a natural vertex with valence $n$. Given an $n-$pronged star set with $n>2$, the ``Center" of the star set is simply the unique natural vertex contained in the vertex.
\end{remark}

\vspace{4mm}

Note that the topology of any graph is generated by the collection of its star neighborhoods.

\vspace{4mm}



While we can use the collection $\{ \pi_j^{-1} (V) : j \in \n, V$ is a star neighborhood of $G_j \}$ as a basis for the topology of $X(\zeta)$ [Remark \ref{The Coordinate Projection Maps and the Topology of X}], it will prove to be more convenient later if we choose only the star neighborhoods of level graphs, when their forward images (under corresponding fold compositions in $\zeta$) are also star neighborhoods (in the corresponding level graphs). 

\vspace{4mm}

Since we plan to utilize the properness of our split sequences throughout this text, laying out some useful terminology is in order. Let $\zeta$ be a proper split sequence. Given a subset $U$ of $\zeta$, we introduce the term ``Shadows" of $U$ as a way of referring to both forward images and pre-images under fold compositions. Then in Definition \ref{Special and Regular Points}, given a point $p$ of a level graph of $\zeta$, we'll categorize $p$ as either a ``Special Point" or a ``Regular Point" based on a criterion depending on $p$'s images and pre-images. This terminology will be directly used in introducing the elements of ``The Standard Basis" of $X$ [Definitions \ref{Standard Star Neighbourhoods} and \ref{def: standard basis}].

\vspace{4mm}

\begin{definition}[Shadows]
    \hypertarget{shadows}{Given a proper split sequence} $\zeta$, two integers $K, j \in -\mathds{N}^*$ and a subset $U \subseteq G_K$, the ``Shadow of $U$ in Level $j$" is given by\\
     \begin{equation*}
    Shadow^{\zeta}_j (U) := \left \{
  \begin{array}{lr} 
      f^K_j (U) & \text{if} \ j>K \\
       U & \text{if} \ j = K \\
       (f^j_K)^{-1} (U) & \text{if} \ j<K
      \end{array}\\
      \right.
    \end{equation*}
    Often we will simplify the notation $Shadow^{\zeta}_j (U)$ to ``$Shadow_j (U)$" when the context is clear.
    When $U = \{p\}$ for some point $p$ in $G_K$, we will simply denote $Shadow^{\zeta}_j (U)$ by ``$Shadow^{\zeta}_j (p)$". 
\end{definition}

\vspace{4mm}



\begin{definition}[Special and Regular Points]
    \label{Special and Regular Points}
    \hypertarget{special points}{Let} $K \in \n$ and $p \in G_K$. We say that ``$p$ is a Special Point of $G_K$ rel $\zeta$" if there exists $j \in \n$ so that $Shadow_j (p)$ is a natural vertex of $G_j$. We call $p$ a ``Regular Point of $G_K$ rel $\zeta$" if $p$ is not a special point of $G_K$ rel $\zeta$. 



\end{definition}

\vspace{4mm}

\begin{definition}[Standard Star Sets, Standard Star Neighborhoods, Standard Neighborhoods]
    \label{Standard Star Neighbourhoods}
    \hypertarget{Standard Star Neighbourhoods}{Given} a proper split sequence $\zeta$, $j \in -\mathds{N}^*$, and a star set $V \subset G_j$ centered at $p$, we call $V$ a ``Standard Star Set of $\zeta$" if $V- \{p\}$ does not contain any \hyperlink{special points}{special points rel $\zeta$}. If a standard star set $V \subset G_j$ is open in $G_j$, we shall call $V$ a ``Standard Star Neighborhood of $G_j$ rel $\zeta$". 
    We call a neighborhood $O \subseteq X(\zeta)$ a ``Standard Neighborhood of $X(\zeta)$" if there exists $j \in \n$ and a standard star neighborhood $V$ in $G_j$ such that $O = \pi^{-1}_j (V)$. 
\end{definition}

\vspace{4mm}

\begin{definition}[Turns in $\zeta$]
    \label{turn in zeta}
    \hypertarget{turn in zeta}{Given} a proper split sequence $\zeta$ and $j \in \n$, a ``Level $j$ Standard Turn in $\zeta$" is a turn of a \hyperlink{Standard Star Neighbourhoods}{standard star neighborhood} of a \hyperlink{level graph}{level graph} $G_j$ of $\zeta$. 
\end{definition}

\vspace{4mm}

\begin{remark}[The Standard Basis for the Topology of $X(\zeta)$.]
    \label{def: standard basis}
    \hypertarget{standard basis}{Recall} from Remark \ref{The Coordinate Projection Maps and the Topology of X} that for each $j \in \n$, $\pi_j :  X \longrightarrow G_j$ is the projection map taking $(x_0 , x_{-1} , x_{-2} , ...)$ to $x_j$. Note that for each $j \in \n$, since each  \hyperlink{Standard Star Neighbourhoods}{standard star neighborhood} of $G_j$ is open in $G_j$, and since each star neighborhood of $G_j$ can be expressed as a union of \hyperlink{Standard Star Neighbourhoods}{standard star neighborhoods} of $G_j$, the collection of standard star neighborhoods of $G_j$, generates the topology of $G_j$. Thus, the collection of \hyperlink{Standard Star Neighbourhoods}{standard neighborhoods} of $X(\zeta)$ (which we shall name the ``Standard Basis of $X(\zeta)$") generates the topology of $X(\zeta)$.

\end{remark}

\vspace{4mm}





The next few sub-sections will explore various topological features of the solenoid. 

\vspace{4mm}


\subsubsection{Shadows and Fibers}
\label{sec: Fibers}

\vspace{4mm}

Here, we investigate the \hyperlink{shadows}{shadows} of given standard star neighborhoods of $\zeta$'s level graphs, and use that perceptive to explore a feature of the solenoid, that we shall call ``Fibers". 


\begin{figure}[htp]
    \centering
    \includegraphics[width=10cm]{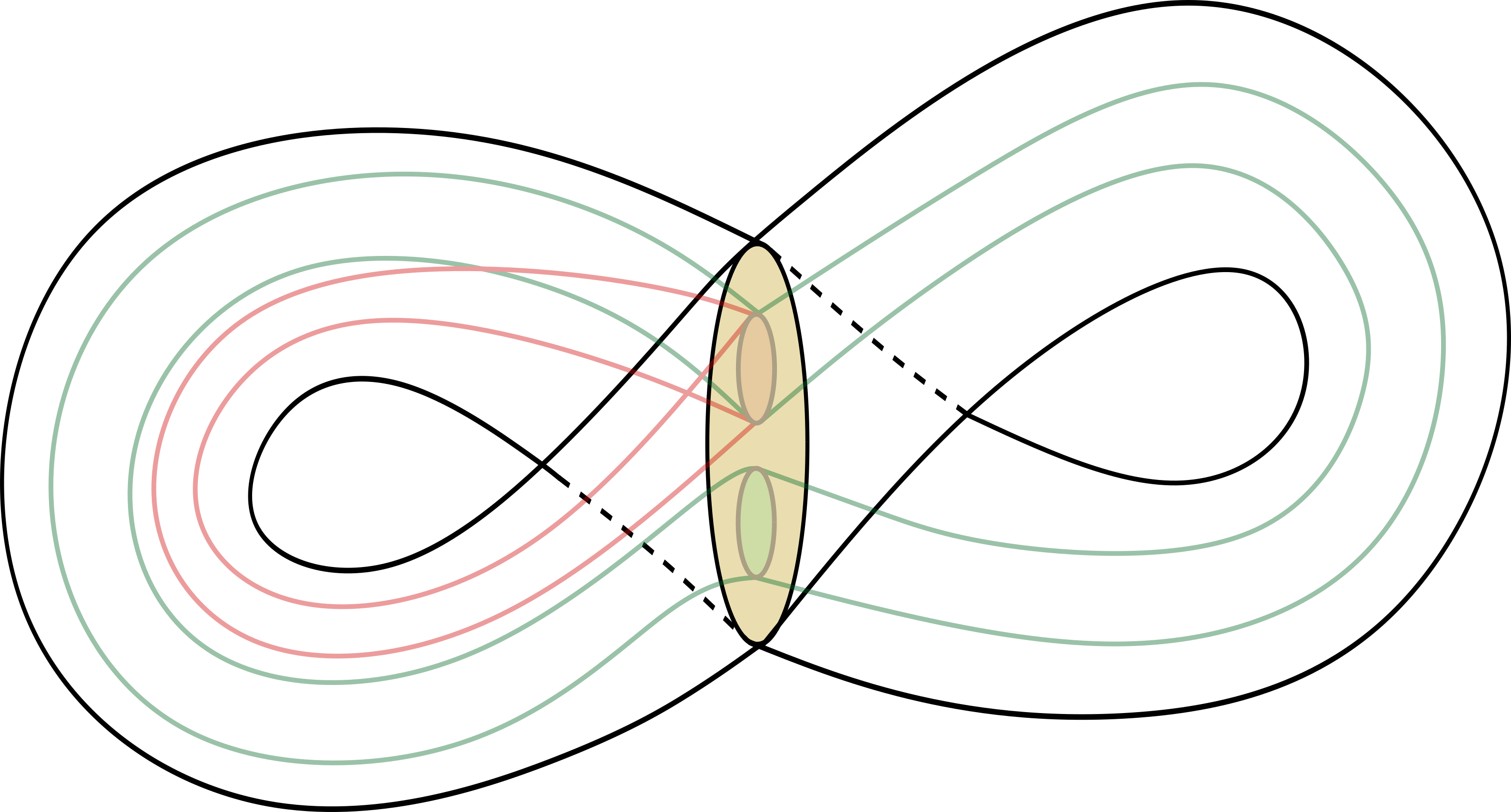}
    \caption{A Fiber Inside a Solenoid}
    \label{fig:A Fiber Inside a Solenoid}
\end{figure}

\vspace{4mm}


We shall borrow our next piece of terminology from the theory of fibrations. It will become apparent, after chapter 3, as to why this terminology is appropriate.


\vspace{4mm}

\begin{definition}[Fibers of $X$]
    \label{def: fibers}
    \hypertarget{fibers}{Let} $j \in \n$ and $q \in G_j$. Then the ``Fiber at $q$ of $X$" denoted ``$F_q$", is $\pi^{-1}_j (q) \subset X$.
\end{definition}

\vspace{4mm}


The following two remarks lay the intuitive groundwork to define a partition system of a fiber (introduced in Remark \ref{The Topology of a Fiber}), whose members, taken collectively, generates the topology of the fiber. 

\vspace{4mm}

\begin{remark}[Sequences of Shadows, and Subsets of $X$]
    Given $K \in -\mathds{N}^*$ and $U \subseteq G_K$, note that the inverse limit of the system: $\hyperlink{shadows}{Shadow}^{\zeta}_K (U) \xleftarrow[]{f_{K-1}} Shadow^{\zeta}_{K-1} (U) \xleftarrow[]{f_{K-2}} Shadow^{\zeta}_{K-2} (U) \xleftarrow[]{f_{K-3}} ...$ is precisely the subset $\pi^{-1}_K (U)$ in $X(\zeta)$. (i.e. $\pi^{-1}_K (U) = \hyperlink{sub inverse limits notation}{\varprojlim^{\zeta}_{j \in -\mathds{N}^* + K} Shadow^{\zeta}_{j} (U)}$.)  
\end{remark}

\vspace{4mm}




\begin{figure}[htbp]
    \centering

    \begin{tikzpicture}
    \tikzmath{\x1 = 0 ; \y1 = 2 ; \ya = -1 ;
    \x2 = 3 ; \y2 = 2 ; \yb = -1 ;
    \x3 = 6 ; \y3 = 2 ; \yc = -1 ;
    \x4 = 9 ; \y4 = 2 ; \yd = -1 ;
    \v = 1 ; }

    \node at (\x1, \ya - 1.5) {Case a.};
    \draw[thick] (\x1 -1 , \y1)--(\x1 +1 , \y1) ;
    \draw[thick] (\x1 , \y1)--(\x1 +1 , \y1 + 0.5) ;
    \draw[thick] (\x1 , \y1)--(\x1 +1 , \y1 - 0.5) ;
    \draw[thick] (\x1 -1 , \ya)--(\x1 +1 , \ya) ;
    \draw[thick] (\x1 , \ya)--(\x1 +1 , \ya + 0.5) ;
    \draw[thick] (\x1 , \ya)--(\x1 +1 , \ya - 0.5) ;
    \draw[->] (\x1, \ya + 0.75) to[out=70, in=-70] (\x1,\y1 -0.75);
    \draw[->][blue] (\x1 - 0.2,\y1 -0.75) to[out=-70, in=70] (\x1 -0.2, \ya + 0.75);
    \filldraw (\x1 + 0.01, \y1 ) circle (\v pt);
    \filldraw (\x1 + 0.01 , \ya) circle (\v pt);
    \node[above] at (\x1, \y1) {$p$};
    \node[below] at (\x1, \ya) {$q$};

    
    \node at (\x2, \yb - 1.5) {Case b.};
    \draw[pattern=vertical lines, pattern color=black!15, draw=none] (\x2-1,\yb) rectangle (\x2 , \yb + 0.3);
    \draw[thick] (\x2 , \y2)--(\x2 +1 , \y2) ;
    \draw[thick][teal] (\x2 -1 , \y2)--(\x2, \y2) ;
    \draw[thick][orange] (\x2 , \y2)--(\x2 +1 , \y2 + 0.5) ;
    \draw[thick] (\x2 , \y2)--(\x2 +1 , \y2 - 0.5) ;
    \draw[thick] (\x2 , \yb)--(\x2 +1 , \yb) ;
    \draw[thick][teal] (\x2 -1 , \yb)--(\x2, \yb) ;
    \draw[thick][orange] (\x2 , \yb + 0.3) to[out=0, in=-120] (\x2 +1 , \yb + 0.7) ;
    \draw[thick][orange] (\x2 , \yb + 0.3) -- (\x2-1 , \yb + 0.3);
    \draw[thick] (\x2 , \yb)--(\x2 +1 , \yb - 0.5) ;
    \draw[->] (\x2, \yb + 0.75) to[out=70, in=-70] (\x2,\y2 -0.75);
    \draw[->][blue] (\x2 - 0.2,\y2 -0.75) to[out=-70, in=70] (\x2 -0.2, \yb +0.75);
    \filldraw (\x2 + 0.01, \y2 ) circle (\v pt);
    \filldraw (\x2  + 0.01 , \yb) circle (\v pt);
    \node[above] at (\x2, \y2) {$p$};
    \node[below] at (\x2, \yb) {$q$};

    \node at (\x3, \yc - 1.5) {Case c.};
    \draw[pattern=vertical lines, pattern color=black!15, draw=none] (\x3-1,\yc) rectangle (\x3 , \yc + 0.3);
    \draw[pattern=north east lines, pattern color=black!15, draw=none] (\x3 , \yc) -- (\x3 + 0.8 , \yc - 0.8) -- (\x3 +1 , \yc - 0.5) -- cycle;
    \draw[thick][teal] (\x3 -1 , \y3)--(\x3, \y3) ;
    \draw[thick][orange] (\x3 , \y3)--(\x3 +1 , \y3 + 0.5) ;
    \draw[thick][teal] (\x3 , \y3)--(\x3 +1 , \y3 - 0.5) ;
    \draw[thick] (\x3, \y3)--(\x3 +1 , \y3) ;
    \draw[thick][teal] (\x3 -1 , \yc)--(\x3 , \yc) ;
    \draw[thick][orange] (\x3 , \yc + 0.3) to[out=0, in=-120] (\x3 +1 , \yc + 0.7) ;
    \draw[thick][orange] (\x3 , \yc + 0.3) -- (\x3-1 , \yc + 0.3); 
    \draw[thick][teal] (\x3 , \yc)--(\x3 +1 , \yc - 0.5) ; 
    \draw[thick] (\x3 , \yc)--(\x3 +1 , \yc) ; 
    \draw[->] (\x3, \yc + 0.75) to[out=70, in=-70] (\x3,\y3 -0.75);
    \draw[->][blue] (\x3 - 0.2,\y3 -0.75) to[out=-70, in=70] (\x3 -0.2, \yc +0.75);
    \draw[thick][orange] (\x3 , \yc) -- (\x3 + 0.76 , \yc - 0.76); 
    \filldraw (\x3 + 0.01, \y3 ) circle (\v pt);
    \filldraw (\x3 , \yc) circle (\v pt);
    \node[above] at (\x3, \y3) {$p$};
    \node[below] at (\x3, \yc) {$q$};
    
    \node at (\x4, \yd - 1.5) {Case d.};
    \draw[pattern=north west lines, pattern color=black!15, draw=none] (\x4 , \yd) -- (\x4 +1 , \yd + 0.5) -- (\x4 + 0.8 , \yd + 0.8) -- cycle;
    \draw[thick] (\x4 -1 , \y4)--(\x4 +1 , \y4) ;
    \draw[thick][teal] (\x4 , \y4)--(\x4 +1 , \y4 + 0.5) ; 
    \draw[thick] (\x4 , \y4)--(\x4 +1 , \y4 - 0.5) ;
    \draw[thick] (\x4 -1 , \yd)--(\x4 +1 , \yd) ;
    \draw[thick][teal] (\x4 , \yd)--(\x4 +1 , \yd + 0.5) ; 
    \draw[thick][orange] (\x4 , \yd)--(\x4 + 0.77 , \yd + 0.77); 
    \draw[thick] (\x4 , \yd)--(\x4 +1 , \yd - 0.5) ;
    \draw[->] (\x4, \yd + 0.75) to[out=70, in=-70] (\x4,\y4 -0.75);
    \draw[->][blue] (\x4 - 0.2,\y4 -0.75) to[out=-70, in=70] (\x4 -0.2, \yd + 0.75);
    \filldraw (\x4 + 0.01, \y4 ) circle (\v pt);
    \filldraw (\x4  + 0.01 , \yd) circle (\v pt);
    \node[above] at (\x4, \y4) {$p$};
    \node[below] at (\x4, \yd) {$q$};
\end{tikzpicture}
    
    \caption{Splits on Standard Star Neighborhoods}
    \label{fig:Folds on Standard Star Neighbourhoods}
\end{figure}
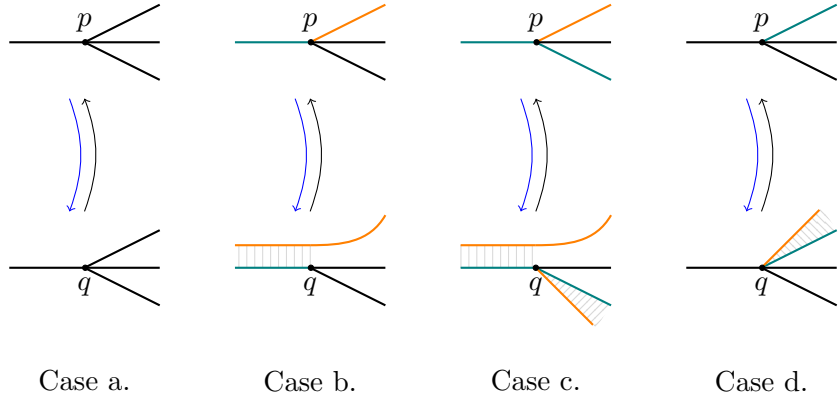

\vspace{4mm}

In Figure \ref{fig:Folds on Standard Star Neighbourhoods}, the split direction is indicated in blue, and the fold direction in black. In each case, let $V$ denote the standard star neighborhood depicted in the upper half of the figure, let $p$ be the center of $V$, and let $f$ denote the depicted fold.
\begin{enumerate}
    \item[a.] represents the case where the split does not affect $V$.  
    \item[b.] represents the case where $p$ is the \hyperlink{Splitting Vertex}{splitting vertex}, and the \hyperlink{folding edges and vertices}{folding vertex} $q$ is not in $f^{-1} (p)$. While we only depict a special sub-case of a. here, note that there is a sub-case, where there could be extra prongs attached $p$ where some pre-images of those prongs under $f$ will be attached to the orange edge. 
    \item[c.] represents the case where $p$ is the splitting vertex, and $p$ has a pre-image that is the folding vertex. Similar to Case b., here, we only depict a special sub-case while there is a sub-case, where there are extra prongs attached to $p$ so that some pre-images of those prongs under $f$ will be attached to the orange edge. 
    \item[d.] represents the case where $p$ has a unique pre-image under $f$, and that pre-image is the folding vertex.
\end{enumerate}
There is one more case that we do not explicitly depict here. That is the instance where $V$ only has two prongs, in which case, either the split $\overline{f}$ does not affect $V$ (as in Case a.), or the pre-image of $V$ under $f$ is two separate $2-$pronged star neighborhoods, or the pre-image of $V$ is a single $3-$pronged star neighborhood.
In each case, note that the pre-image of $V$ under $f$, has either one or two connected components in the domain of $f$.

\vspace{4mm}

\begin{remark}[Shadow Components, Shadows of Standard Star Neighborhoods]
    \label{Shadow Components}
    \hypertarget{Shadow Component}{Let} $K \in \n$ and let $V$ be a \hyperlink{Standard Star Neighbourhoods}{standard star neighborhood} of $G_K$. Then for each integer $j < K$, we shall call a connected component of \hyperlink{shadows}{$Shadow_j (V)$} in $G_j$, a ``Shadow Component of $V$ in Level $j$". \hypertarget{SC notation}{In} this context, for a given point $p \in Shadow_j (V)$ (resp. a path connected set $U \subseteq Shadow_j (V)$), the unique shadow component of $V$ in level $j$ containing $p$, shall be denoted by ``$SC_j (V,p)$" (resp. ``$SC_j (V,U)$").

    \vspace{4mm}
    
    Note that, due to the kind of splits that are involved, in Level $K-1$, $V$ has at most two shadow components [See Figure \ref{fig:Folds on Standard Star Neighbourhoods}]. Therefore, in level $j$, $V$ has at most $2^{K-j}$ shadow components. Furthermore, each shadow component of $V$ is a standard star neighborhood of the graph it's contained in. i.e. Each shadow of a standard star neighborhood is a disjoint union of standard star neighborhoods. 
\end{remark}

\vspace{4mm}







\begin{figure}[htbp]
    \centering
    \begin{tikzpicture}[xscale=0.8, yscale=0.8]
    \draw[thick, fill=yellow!10] (0,0) ellipse [x radius=1.9cm, y radius=4.275cm];

    \draw[thick, fill=orange!20] (0,2) ellipse [x radius=0.8cm, y radius=1.8cm];
    \draw[thick, fill=green!20] (0,-2) ellipse [x radius=0.8cm, y radius=1.8cm];

    \draw[thick, fill=blue!20] (0,-1.2) ellipse [x radius=0.3cm, y radius=0.675cm];
    \draw[thick, fill=cyan!20] (0,-2.8) ellipse [x radius=0.3cm, y radius=0.675cm];

    \node at (2.4,0) {$F_{p}$};
    \node at (1.2,2) {$F_{q_1}$};
    \node at (1.2,-2) {$F_{q_2}$};


    \draw[thick] (5,4) -- (8,4); 
    \filldraw[yellow] (6.5,4) circle (2pt); 
    \node[above] at (6.5,4.2) {$p$}; 
    \node at (9.5,4) {$\subset G_K$};
    
    \draw[thick] (5,1.5) -- (8,1.5);
    \filldraw[orange] (6.5,1.5) circle (2pt);
    \node[above] at (6.5,1.7) {$q_1$};
    \draw[thick] (5,1) -- (8,1);
    \filldraw[green] (6.5,1) circle (2pt);
    \node at (6.5,0.5) {$q_2$};
    \node at (9.5,1.25) {$\subset G_{K-1}$};
    \draw[pattern=vertical lines, pattern color=black!15, draw=none] (8,1) rectangle (5,1.5);

    \draw[thick] (5,-2.1) -- (8,-2.1);
    \filldraw[orange] (6.5,-2.1) circle (2pt);
    \node[above] at (6.5,-1.9) {$q_{11}$};

    \draw[thick] (5,-3.3) -- (8,-3.3);
    \filldraw[violet] (6.5,-3.3) circle (2pt);
    \node[above] at (6.5,-3.1) {$q_{21}$};
    \draw[thick] (5,-3.8) -- (8,-3.8);
    \filldraw[cyan] (6.5,-3.8) circle (2pt);
    \node at (6.5,-4.4) {$q_{21}$};
    \draw[pattern=vertical lines, pattern color=black!15, draw=none] (5,-3.3) rectangle (8,-3.8);
    \node at (9.5,-2.9) {$\subset G_{K-2}$};
\end{tikzpicture}
\vspace{4mm}
\caption{The Standard Partition System of a Fiber}
\label{fig:Partitioning a fiber}
\end{figure}
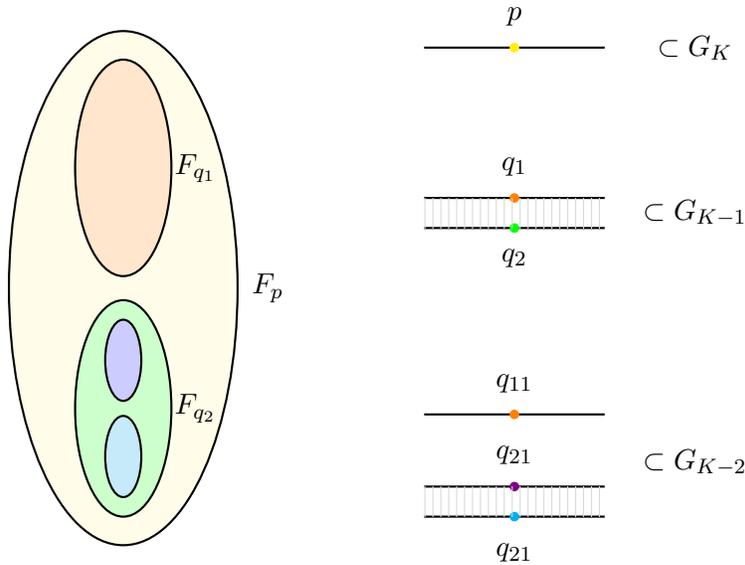

\vspace{4mm}


\begin{remark}[Sub-Fibers, and The Topology of a Fiber]
    \label{The Topology of a Fiber}
    \hypertarget{The Topology of a Fiber}{Let} $K \in \n$, and $p \in G_K$. For each integer $j \leq K$ and for each level $j$ \hyperlink{shadows}{shadow} of $p$, $F_q$ is called a ``Level $j$ Sub-Fiber of $F_p$". 
    Furthermore, for each integer $j \leq K$, the collection of all level $j$ sub-fibers of $F_p$ given by $\mathcal{P}_{j,p} := \{ F_q \}_{q \in (f^j_K)^{-1} (p)}$ (which is a partition of $F_p$), is called the ``Level $j$ Standard Partition of $F_p$". 
    \hypertarget{Sub-Fiber Basis}{The} collection $\{ F \subseteq F_p : F$ is a sub-fiber of $F_p \}$ is a basis for the topology of $F_p$, and thus shall be called the ``Sub-Fiber Basis of $F_p$". 
\end{remark}

\vspace{4mm}

Though the following notation was established earlier, we will formalize it here.

\vspace{4mm}

\begin{notation}[Projections of a Point]
    \label{Projections of a Point}
    \hypertarget{Projections of a Point}{Let} $x \in X$, and $j \in \n$. Then $``x_j " := \hyperlink{The Coordinate Projection Maps}{\pi_j} (x) \in G_j$. The point $x$ shall be often expressed in the format $(x_0 , x_{-1} , x_{-2} , ...)$ $ \in X \subset \Pi_{j \in \n} G_j$.
\end{notation}

\vspace{4mm}

Note that, for each $x = (x_0, x_{-1} , x_{-2} , ...) \in X$, $x$ is the unique intersection point of the nested sequence of fibers $F_{x_0} \supset F_{x_{-1}} \supset F_{x_{-2}} \supset ...$.

\vspace{4mm}


\vspace{4mm}

The basis mentioned in Remark \ref{The Topology of a Fiber} falls in to a larger class of bases called ``Tree Bases" [Definition \ref{Tree Bases}] which we will define and explore in Chapter \ref{Ch: Leaves and Transversals}, and utilize in Chapter \ref{Ch: TM(X)}.

\vspace{4mm}





\vspace{4mm}

We will end this sub-section with two important properties of fibers.

\vspace{4mm}

\begin{property}
    \label{Fibers are totally disconnected}
    Fibers are totally disconnected.
\end{property}
    
\begin{proof}
    Given a fiber $F$ of $X(\zeta)$, and any two distinct points $x = (x_0 , x_{-1} , x_{-2} , ...)$ and $y = (y_0 , y_{-1} , y_{-2} , ...)$ in $F$, there must be a $k \in \n$ such that $x_j \neq y_j $ for each $j<k$. Pick just one of those $j$'s and it follows that we're able to choose two small enough standard star neighborhoods $U,V$ in $G_j$ such that $x_j \in U$, $y_j \in V$, and $U \cap V = \emptyset$.
    Thus, the open neighborhoods $\pi^{-1}_j (U)$ and $\pi^{-1}_j (V)$, respectively containing $x$ and $y$, do not intersect in $X$. 
\end{proof}

\vspace{4mm}

\begin{property}
    \label{Fibers are compact}
    Fibers are compact.
\end{property}

\begin{proof}
    Consider a fiber $F$ of $X(\zeta)$. Note that for each $j \in -\mathds{N}^*$, $\pi_j (F)$ is a finite subset of $G_j$ and is therefore compact in $G_j$. Then from Tychonoff's theorem, we have that  $P_F := \prod_{j \in -\mathds{N}^*} \pi_j (F)$ is compact. 
    We intend to show that $F$ is a closed subspace of the compact space $F_P$.
    Recall from Remark \ref{The Coordinate Projection Maps and the Topology of X}, that for each $k \in \n$, $\overline{\pi}_k$ denotes the coordinate projection map from $\prod_{j \in -\mathds{N}^*} G_j$ to $G_k$.

    \vspace{4mm}
    
    Now let $y = (y_0 , y_{-1} , y_{-2} , ... ) \in P_F - F$. Then there exists $j \in -\mathds{N}^*$ such that $f_{j-1} (y_{j-1}) \neq y_j$. Therefore, the open subset of $P_F$ containing $y$ given by $O :=  \overline{\pi}^{-1}_j (y_j) \bigcap \overline{\pi}^{-1}_{j-1} (y_{j-1}) $ is entirely contained in $P_F -F$. Thus $F$ is a closed in $P_F$.
\end{proof}

\vspace{4mm}


\subsubsection{Plaques, Leaf Interior Points and Singularities}
\label{sec: Plaques, Leaf Interior Points and Singularities}

\vspace{4mm}

Let $\zeta$ be a proper stabilized split sequence and $X = X(\zeta)$ its induced solenoid. In this sub-section, we define and study some key features of $X(\zeta)$, called ``Plaques", ``Leaf Interior Points" and ``Singularities".

\vspace{4mm}


To help visualize the immediate surrounding of a point in $X$, we 
identify a specific type of \hyperlink{Sub Inverse Sequences}{sub inverse sequence} of $\zeta$, 
called a ``\hyperlink{star chain}{Star Chain}" [Definition \ref{star chain}, Lemma \ref{every point has a maximal star chain}] where, the inverse limit of a star chain is a star set in $X$. Then we use star chains to classify points in $X$ into two categories called ``\hyperlink{Singularity}{Leaf Interior Points}" and ``\hyperlink{Singularity}{Singularities}" [Definition \ref{Singularity}, Corollary \ref{point classification}].

\vspace{4mm}

We define a ``\hyperlink{plaque}{Plaque}" [Definition \ref{def: plaque}] of a \hyperlink{Standard Star Neighbourhoods}{Standard Neighborhood} $O$, as a path component of $O$, and show that each plaque of a standard neighborhood is a star set [Lemmas \ref{shapse of plaques 1}, and \ref{shapse of plaques 1}].

\vspace{4mm}

After giving an upper-bound to the number of natural vertices in a graph with a fixed Euler number [Remark \ref{upperbound to natural vertices}], we show that $X$ has only finitely many singularities [Proposition \ref{finitely many singularities}]. 

\vspace{4mm}

We start with establishing terminology.

\vspace{4mm}

\begin{definition}[Standard Plaques in $X$]
    \label{def: plaque}
    \hypertarget{plaque}{Let} $U$ be a topological space, $x \in U$, and $V$ a path connected subspace of $U$. Each path component of $U$ is called a ``Plaque of $U$". Furthermore,
    we denote by ``$Plaque(U,x)$" (resp. ``$Plaque(U,V)$") the unique path component of $U$ containing $x$ (resp. $V$). In the context of the solenoid $X$, a ``Standard Plaque in $X$" shall refer to a plaque of a \hyperlink{Standard Star Neighbourhoods}{standard neighborhood} in $X$.
\end{definition}

\vspace{4mm}

In the discussion that follows, we will be particularly interested in ``$Plaque(O,x)$" where $O$ is a standard neighborhood of $X$ and $x \in O$.

\vspace{4mm}



\vspace{4mm}

\begin{definition}[Leaf Interior Points and Singularities of $X$]
    \label{Singularity}
    \hypertarget{Singularity}{A} ``Leaf Interior Point of $X$" is a point $x \in X$ that has a standard neighborhood $O \ni x$ such that \hyperlink{plaque}{$Plaque(O,x)$} is homeomorphic to an open interval.
    A ``Singularity of $X$" is a point $x \in X$ such that, for each standard neighborhood $O \ni x$, $Plaque(O,x)$ is an \hyperlink{star set}{$n-$pronged star set} for some $n > 2$.
\end{definition}

\vspace{4mm}

Later in this Sub-section we will show  that, 
each standard plaque in $X$ is a star set [Lemma \ref{shapse of plaques 2}], and furthermore, that
for each point $x \in X$, there exists a unique integer $n_x \geq 2$ such that each standard plaque in $X$ centered at $x$, is an $n_x-$pronged star set [Prop. \ref{shapse of plaques 3}]. Therefore, it will follow that each point of $X$ is either a leaf interior point or a singularity [Cor. \ref{point classification}].

\vspace{4mm}




We shall next define a feature called a ``\hyperlink{star chain}{Star Chain}" that'll help us identify singularities in the wild.
We'll lay groundwork for this by singling out, for $x \in X$, a relevant collection of star sets containing each $x_j$.

\vspace{4mm}

Recall from Remark \ref{Shadow Components}, that for each $K \in \n$, each standard star neighborhood $V$ in $G_K$, and for each integer $j < K$, \hyperlink{shadows}{$Shadow_j (V)$} is a disjoint union of standard star neighborhoods in $G_j$. For each point $x = (x_0 , x_{-1} , x_{-2} , ... ) \in X$, we may choose a standard star neighborhood $\Tilde{V}_0$ in $G_0$ with center $x_0$, then for each $j \in -\mathds{N}$, consider the standard star neighborhood $\Tilde{V}_j := \hyperlink{SC notation}{SC_j (\Tilde{V}_0,x_j)}$. This arrangement of standard star neighborhoods, will be refined further in the definition below. 

\vspace{4mm}

\begin{definition}[$n-$pronged Star Chains and Maximal Star Chains]
    \label{star chain}
    \hypertarget{star chain}{Let} $ x = (x_0,x_{-1},x_{-2},...) \in X(\zeta) $. If there exist $k \in - \mathds{N}^* $, $n \in \mathds{N}$, and a sequence of $n-$pronged standard star sets $U_{j} \subset G_{j}$ for $j \in -\mathds{N}^* + k$, 
    \begin{equation*}
        U_{k} \xleftarrow[]{f_{k-1}} U_{k-1} \xleftarrow[]{f_{k-2}} U_{k-2} \xleftarrow[]{f_{k-3}}...
    \end{equation*}
    such that, for each $j \in \n + k$,
    \begin{enumerate}
        \item $U_j$ is centered at $x_j$ (i.e. if $U_j$ has a point with valence $> 2$, then that point is $x_j$), and,
        \item $f_{j-1}$ restricted to $U_{j-1}$ is a homeomorphism on to $U_j$,
    \end{enumerate}
    then we say that this sequence is an ``$n-$Pronged Star Chain for $x$ Starting at Level $k$ in $\zeta$". Furthermore we call such a sequence an ``$n-$Pronged Maximal Star Chain for $x$ Starting at Level $k$ in $\zeta$" if in addition to the above criteria, it also satisfies the following.
    \begin{enumerate}
        \item[3.] There does not exist an integer $K \leq k$ and an $(n+1)-$pronged star chain for $x$,
        \begin{equation*}
        V_{K} \xleftarrow[]{f_{K-1}} V_{K-1} \xleftarrow[]{f_{K-2}} V_{K-2} \xleftarrow[]{f_{K-3}}...
    \end{equation*} 
    with $U_j \subset V_j$ for each $j \in - \mathds{N}^* + K$.
    \end{enumerate}  
\end{definition}

\vspace{4mm}

We aim to treat maximal star chains of a point as an identifier as to whether the point is a leaf interior point or a singularity (and to further signify the nature of the singularity if it is one). 

\vspace{4mm}

\begin{remark}[The Inverse Limit of an $n-$pronged Star Chain]
    \label{The Inverse Limit of an n-pronged Star Chain}
    Let $x \in X$ have an $n-$pronged star chain
    \begin{equation*}
        \mathcal{U} : U_{k} \xleftarrow[]{f_{k-1}} U_{k-1} \xleftarrow[]{f_{k-2}} U_{k-2} \xleftarrow[]{f_{k-3}}...
    \end{equation*} 
    Let $U := \hyperlink{sub inverse limits notation}{\varprojlim^{\zeta} \mathcal{U}} \subset X$. \hypertarget{star chain representing a set in X}{Then} we call $\mathcal{U}$ a ``Star Chain Representing $U$ that starts with $U_k$" and that ``$U_k$ is the Starting Set of $\mathcal{U}$". 
    Since, for each $j \in \n + k$, $f_{j-1}$ restricted to $U_{j-1}$ is a homeomorphism onto $U_{j}$, it follows that $U \subset X(\zeta)$ is homeomorphic to $U_k$ and thus is a star set of $X$ with $x$ as its center.
    Note that whenever we're given a subset $V \subset X$, and are told that $\mathcal{V}$ is a star chain representing $V$, specifying the starting set of $\mathcal{V}$ uniquely determines the entire star chain $\mathcal{V}$. 
\end{remark}

\vspace{4mm}

We will use the above result in Lemmas \ref{shapse of plaques 1} and \ref{shapse of plaques 2}, to show that each plaque of a standard neighborhood is a star set. 

\vspace{4mm}

In the discussion of solenoids we will frequently go back and forth between star sets and star neighborhoods. We shall define the following term as a connecting tissue among the two types of sets and we shall use the terminology in the upcoming proofs.

\vspace{4mm}

\begin{definition}[Neighborhood Completion of a Star Set]
    \label{Neighbourhood Completion}
    \hypertarget{neighbourhood_completion}{Given} a star set $V$ centered at a vertex $p$ in a graph $G$, a ``Neighborhood Completion of $V$ in $G$" is a star neighborhood $\Tilde{V}$ of $G$ containing $V$ 
    such that each \hyperlink{center, prong, turn}{prong} of $V$ is also a prong of $\Tilde{V}$.
    Now let $j \in \n$, $G_j$ be a level graph of $\zeta$, and $V$ a standard star set of $G_j$ rel $\zeta$. Then we say that ``$\Tilde{V}$ is a Standard Neighborhood Completion of $V$ rel $\zeta$" if $\Tilde{V}$ is a standard star neighborhood of $G_j$ rel $\zeta$ that is also a neighborhood completion of $V$ in $G_j$.
\end{definition}

\vspace{4mm}

The rest of this sub-section is dedicated to proving three results. 
\begin{enumerate}
    \item Each $x \in X$ has a maximal star chain [Lemma \ref{every point has a maximal star chain}].
    \item Each plaque of a standard neighborhood of $X$ is a star set [Lemmas \ref{shapse of plaques 1}, and \ref{shapse of plaques 2}]. More specifically, for each $V \subset X$, $V$ is a standard plaque in $X$ iff $V$ is an inverse limit of a maximal star chain in $\zeta$ [Proposition \ref{shapse of plaques 3}].
    \item $X$ contains only finitely many singularities [Proposition \ref{finitely many singularities}]. 
\end{enumerate}

\vspace{4mm}

\subsubsection*{Each Point of the Solenoid has a Maximal Star Chain}

\vspace{4mm}

Here, we will show the titular result (which will later be used to define ``$n-$Pronged Singularities"). 

\vspace{4mm}



\begin{lemma}
    \label{every point has a maximal star chain}
    \hypertarget{maximal_star_chain}{Every $x \in X(\zeta)$ has a maximal star chain.}
\end{lemma}

\begin{proof}
        Let $ x = (x_0,x_{-1},x_{-2},...) \in X(\zeta) $. Due to the types of folds that are permitted in $\zeta$, for each $j \in -\mathds{N}^*$ and for each star neighborhood $U_j$ of $x_j$, $U_j$ has at least one turn $I$ such that there exists a turn $\Tilde{I}$ in $\hyperlink{Shadow Component}{SC_{j-1} (U_j , x_{j-1})}$ that maps homeomorphically onto $I$ via $f_{j-1}$. 
        Thus, at the very least $x$ has a $2-$pronged star chain. Furthermore, since all the level graphs of $\zeta$ has the same Euler characteristic, there is an upper-bound on $\{$the number of natural edges in $G_j : j \in \n \}$ and therefore on $\{$valence of $v$ in $G_j : j \in \n$ and $v \in G_j \}$. (The aforementioned upperbound is explicitly given in Lemma \ref{upperbound to natural vertices}.) Thus the non-empty set $S_x := \{ n \in \mathds{N} : x$ has an $n-$pronged star chain$\}$ has a maximum. 
\end{proof}

\vspace{4mm}

From the definition [\ref{star chain}] of a maximal star chain, it is clear that, if $x \in X$ has an $n-$pronged maximal star chain where $n \in \mathds{N}$, then each maximal star chain of $x$ is $n-$pronged. Thus the definition of an ``\hyperlink{n-pronged sing}{$n-$Pronged Singularity}" (to be laid out in the next section) shall be well defined.

\vspace{4mm}

\subsubsection*{Each Plaque of a Standard Neighborhood is a Star Set}

\vspace{4mm}



Lemmas \ref{shapse of plaques 1} and \ref{shapse of plaques 2} will show the titular result. The aforementioned two lemmas will also show that, $V \subset X$ is a \hyperlink{plaque}{standard plaque in $X$} if and only if $V$ is the inverse limit of a \hyperlink{star chain}{maximal star chain in $\zeta$} [Prop. \ref{shapse of plaques 3}]. This will allow us to refine the definitions of plaques [Def. \ref{Singular Plaques}] and singularities [Def. \ref{n-pronged sing}], then discuss how different standard plaques intersect each other [Remark \ref{intersecting plaques}].
Remark \ref{Singular Plaques and Star Chains} will lay out several equivalent criteria that determines a singularity of $X$. Then Corollary \ref{point classification} will show that each point $x \in X$ is either a singularity or a leaf interior point.

\vspace{4mm}

\begin{lemma}
    \label{shapse of plaques 1}
        \hypertarget{shapes_of_plaques}{Let} $x \in X$, $K \in \n$, $n \in \mathds{N}$, and let
    \begin{equation*}
        \mathcal{V} : V_{K} \xleftarrow[]{f_{K-1}} V_{K-1} \xleftarrow[]{f_{K-2}} V_{K-2} \xleftarrow[]{f_{K-3}}...
    \end{equation*}
    be a maximal $n-$pronged star chain for $x$.
    Furthermore, let $\Tilde{V}_K$ be a standard neighborhood completion for $V_K$, and let $O := \pi^{-1}_K (\Tilde{V}_K)$. Then, $\hyperlink{plaque}{Plaque (O , x)} = \varprojlim^{\zeta} \mathcal{V} $  (which is homeomorphic to $V_K$, and therefore $Plaque (O , x)$ is an $n-$ pronged star set with center $x$).
\end{lemma}

\vspace{4mm}


\begin{proof}
    Assume the set-up and the notations given in the lemma. Let $V := \hyperlink{sub inverse limits notation}{\varprojlim^{\zeta} \mathcal{V}}$. Since $V$ is a star set, it is path connected, and therefore is contained in $Plaque (O , x)$. We seek to show that $Plaque (O , x) = V$ by establishing that each $y \in O - V$ and $V$ are contained in different connected components of $O$. 
    Note that if $y$ and $V$ are contained in separate connected components of $O$, then they must be contained in separate path components of $O$. This, combined with the fact that $V$ is path connected, will imply that $V$ is itself a path component of $O$, and thus should be equal to $Plaque (O , x)$. 

    \vspace{4mm}

    Now let $y \in O - V$. We aim to show that there exist two disjoint open sets $O_V, O_y$ in $O$, such that $V \subseteq O_V , y \in O_y$, and $O_V \bigsqcup O_y = O$.

    \vspace{4mm}

    Since $y \notin V$, there must be an integer $k \leq K$ such that $y_k \notin V_k$. Furthermore, in order to not contradict the maximality of the star chain $\mathcal{V}$, there must exist an integer $J < K$, such that
    for each $j \leq J$, $y_j$ cannot belong to a prong of the shadow component $\hyperlink{SC notation}{SC_{J} (\Tilde{V}_K , V_{J})}$.

    \vspace{4mm}
    
    Recall from Remark \ref{Shadow Components}, that  each connected component of $\hyperlink{shadows}{Shadow_J (\Tilde{V}_K)}$ is a standard star neighborhood of $G_J$. Now let $\Tilde{V}_{J} := \hyperlink{SC notation}{SC_{J} (\Tilde{V}_K , V_{J})}$ and let $\Tilde{U}_J$ be the disjoint union of all level $J$ \hyperlink{Shadow Component}{shadow components} of $\Tilde{V}_K$, except $\Tilde{V}_{J}$. Note that $y_J \in \Tilde{U}_J$, and that both $\Tilde{V}_J$ and $\Tilde{U}_J$ are open subsets of $Shadow_J (\Tilde{V}_K)$. Furthermore, note that $Shadow_J (\Tilde{V}_K) = \Tilde{V}_{J} \bigsqcup \Tilde{U}_{J}$. Let $O_V := \pi^{-1}_J (\Tilde{V}_{J})$ and $O_y := \pi^{-1}_J (\Tilde{U}_{J})$. 
\end{proof}

\vspace{4mm}

\begin{lemma}
    \label{shapse of plaques 2}
    Let $O$ be a standard neighborhood of $X$. Then each plaque $V$ of $O$ is the inverse limit of a maximal star chain for some $x \in V$, and $V$ is therefore a star set. 
\end{lemma}

\begin{proof}
    Let $K \in \n$, and let $\Tilde{V}_K \subset G_K$ be a standard star neighborhood such that $O = \pi^{-1}_K (\Tilde{V}_K)$. 
    Let $p$ be the center of $\Tilde{V}_K$. Note that each plaque of $O$ must have a point that projects to $p$ (since $\zeta$ permits no backtracking). Thus $O = \bigcup_{x \in F_p} Plaque (O,x)$.
    Furthermore, for each $x \in F_p$, there exists a level $J := J_x \leq K$ and a maximal star chain for $x$ that starts at $J$, $\mathcal{V} : V_{J} \xleftarrow[]{f_{J-1}} V_{J-1} \xleftarrow[]{f_{J-2}} V_{J-2} \xleftarrow[]{f_{J-3}}...$ [from Lemma \ref{every point has a maximal star chain}]. Now let $\Tilde{V}_J := SC_J (\Tilde{V}_K , V_J)$. Note that $\Tilde{U}_J := Shadow_J (\Tilde{V}_K) - \Tilde{V}_J$ is an open neighborhood of $Shadow_J (\Tilde{V}_K)$, and thus $x \in \pi^{-1}_J (\Tilde{V}_J)$ and $\pi^{-1}_J (\Tilde{U}_J)$ are disconnected in $O$.
    Thus, $Plaque (O, x) = Plaque (\pi^{-1}_J (\Tilde{V}_J) , x) = \varprojlim^{\zeta} \mathcal{V}$ is a star set [from Lemma \ref{shapse of plaques 1}].
\end{proof}

\vspace{4mm}

\begin{proposition}
    \label{shapse of plaques 3}
    Let $V \subset X$. $V$ is a \hyperlink{plaque}{standard plaque in $X$} if and only if $V$ is the inverse limit of a \hyperlink{star chain}{maximal star chain in $\zeta$}.
\end{proposition}

\begin{proof}
    Follows from Lemmas \ref{shapse of plaques 1} and \ref{shapse of plaques 2}.
\end{proof}

\vspace{4mm}

The above result allows the usage of more specific terminology for plaques and singularities. 

\vspace{4mm}

\begin{definition}[Singular Plaques, and Non-Singular Plaques]
    \label{Singular Plaques}
    \hypertarget{Singular Plaques}{Let} $V$ be a plaque of a standard neighborhood $O$ of $X$ (i.e. a standard plaque in $X$). Then $V$ is called a ``Singular Plaque of $O$ in $X$", or simply a ``Singular Plaque in $X$", if $V$ has a singularity as its center. If $V$ is not a singular plaque of $O$ (equivalently, if $V$ is a plaque of $O$ that is a $2-$pronged star set), then $V$ is  called a ``Non-Singular Plaque of $O$ in $X$", or simply a ``Non-Singular Plaque in $X$".
\end{definition}

\vspace{4mm}

Note that, for each $n-$pronged maximal star chain $\mathcal{V}$ in $\zeta$, where $n \in \mathds{N}$,
\begin{itemize}
    \item \hyperlink{sub inverse limits notation}{$\varprojlim^{\zeta} \mathcal{V}$} is a singular plaque in $X(\zeta)$ iff $n > 2$, and,
    \item $\varprojlim^{\zeta} \mathcal{V}$ is a non-singular plaque in $X(\zeta)$ iff $n = 2$.
\end{itemize}

\vspace{4mm}


Recall that for a point $x \in X$ to qualify as a singularity, it is required that each standard plaque containing $x$ is a singular plaque [Def. \ref{Singularity}]. The next remark shall show that, if $x \in X$ is the center of one singular plaque, then each standard plaque that contains $x$ is a singular plaque.

\vspace{4mm}

\begin{remark}[Intersecting Singular Plaques]
    \label{Intersecting Singular Plaques}
    \hypertarget{Intersecting Singular Plaques}{Let} $V , V'$ be two singular plaques in $X$ respectively centered at $x,x'$. Then by projecting $V$ and $V'$ to each level graph, and obtaining two \hyperlink{star chain}{maximal star chains} for $x$ and $x'$ (possibly starting at two different levels $K,K' \in \n$ respectively), we can make the following conclusion.
    Exactly one of the following is true. 
    \begin{enumerate}
        \item $V \cap V' = \emptyset$.
        \item $V \cap V' \neq \emptyset$ and $x = x'$, which implies that $V \cap V'$ is an $n-$pronged star set for some $n \in \mathds{N}$ where both $V$ and $V'$ are also $n-$pronged.
        \item $V \cap V' \neq \emptyset$ and $x \neq x'$, which implies that $V \cap V' = $ a disjoint union of finitely many \hyperlink{Singular Plaques}{non-singular plaques in $X$}.
    \end{enumerate}
\end{remark}

\vspace{4mm}

The following remark lays out several equivalent criteria that determines a singularity of $X$. While items 3. and 4. of the remark below, follow directly from the remark above, items 5. and 6. follow from the fact that, $V \subset X$ is a singular plaque in $X$ if and only if $V$ is the inverse limit of an $n-$pronged maximal star chain in $\zeta$ with $n > 2$ [Proposition \ref{shapse of plaques 3}].

\vspace{4mm}

\begin{remark}[Singularities Recognized Using Singular Plaques and Maximal Star Chains]
    \label{Singular Plaques and Star Chains}
    \hypertarget{Singular Plaques and Star Chains}{Let} $x \in X$. Then, the following are equivalent.
    \begin{enumerate}
        \item $x$ is a singularity of $X$.
        \item Each \hyperlink{Singular Plaques}{standard plaque} containing $x$ is a \hyperlink{Singular Plaques}{singular plaque}. [This is the definition of a \hyperlink{Singularity}{singularity}.]
        \item There exists an integer $n > 2$ such that each standard plaque containing $x$ is an $n-$pronged singular plaque. [Equivalent to item 2 via Remark \ref{Intersecting Singular Plaques}.]
        \item There exists a singular plaque that has $x$ as its center. [Equivalent to item 2 via Remark \ref{Intersecting Singular Plaques}.]
        \item There exists an integer $n > 2$ such that $x$ has an $n-$pronged maximal star chain. [Equivalent to item 4 via Prop. \ref{shapse of plaques 3}.]
        \item There exists an integer $n > 2$ such that each \hyperlink{star chain}{maximal star chain} for $x$ is $n-$pronged. [Equivalent to item 5 via the definition of a maximal star chain.]
    \end{enumerate}
\end{remark}

\vspace{4mm}

It follows from the above remark that the following terminology is well-defined. 

\vspace{4mm}

\begin{definition}[$n-$Pronged Singularity]
    \label{n-pronged sing}
    \hypertarget{n-pronged sing}{If} $x \in X(\zeta)$ has a maximal $n-$pronged star chain where $n > 2$, then we say that $x$ is an ``$n-$Pronged Singularity" of $X(\zeta)$.
\end{definition}

\vspace{4mm}

\begin{corollary}
    \label{point classification}
    Let $\zeta$ be a proper split sequence. Each $x \in X(\zeta)$ is either a leaf interior point or an $n-$pronged singularity for some $n > 2$.
\end{corollary}

\begin{proof}
    Each $x \in X(\zeta)$ has an $n-$pronged maximal star chain where either $n=2$ or $n>2$. The rest follows from Proposition \ref{shapse of plaques 3} and item 6. of Remark \ref{Singular Plaques and Star Chains}.
\end{proof}

\vspace{4mm}


\subsubsection*{The Solenoid has Finitely Many Singularities}

\vspace{4mm}

Here, to show the titular result, we first establish upper-bounds for the number of natural edges and vertices for a core graph of fixed rank [Lemma \ref{upperbound to natural vertices}]. The following definition and remark (to be used in the proof of the aforementioned lemma) are an exploration of how to use a core graph of a given rank, to obtain another core graph (of the same rank), that has a higher number of natural vertices (and edges) than the original graph.

\vspace{4mm}

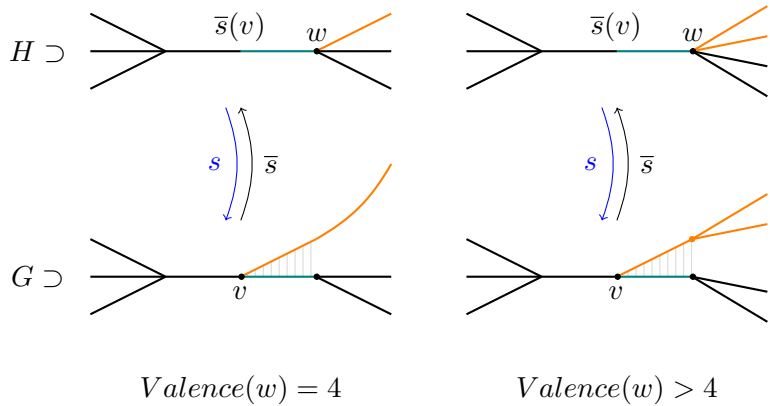
\begin{figure}[htbp]
    \centering
    \vspace{4mm}
    \begin{tikzpicture}
    \tikzmath{\x1 = 0 ; \y1 = 2 ; \ya = -1 ;
    \x2 = 5 ; \y2 = 2 ; \yb = -1 ;
    \v = 1 ; }

    \node at (\x1 -3.7 , \y1) {$H \supset$};
    \node at (\x1 -3.7 , \ya) {$G \supset$};
    \node at (\x1-1, \ya - 1.5) {$Valence (w) = 4$};
    \draw[pattern=vertical lines, pattern color=black!15, draw=none] (\x1 -1 , \ya) -- (\x1 , \ya + 0.5) -- (\x1 , \ya) -- cycle;
    \draw[thick] (\x1-3  , \y1)--(\x1 -1 , \y1) ;
    \draw[thick] (\x1  , \y1)--(\x1 +1 , \y1) ;
    \draw[thick][teal] (\x1 -1 , \y1)--(\x1 , \y1) ;
    \draw[thick][orange] (\x1 , \y1)--(\x1 +1 , \y1 + 0.5) ; 
    \draw[thick] (\x1 , \y1)--(\x1 +1 , \y1 - 0.5) ; 
    \draw[thick] (\x1-2 , \y1)--(\x1 -3 , \y1 + 0.5) ; 
    \draw[thick] (\x1-2 , \y1)--(\x1 -3 , \y1 - 0.5) ; 
    \draw[thick] (\x1-3  , \ya)--(\x1 -1 , \ya) ;
    \draw[thick] (\x1 , \ya)--(\x1 +1 , \ya) ;
    \draw[thick][teal] (\x1 -1 , \ya)--(\x1 , \ya) ;
    \draw[thick][orange] (\x1 -1 , \ya)--(\x1 , \ya + 0.5) ; 
    \draw[thick][orange] (\x1 , \ya +0.5) to[out=30, in=-120] (\x1 +1 , \ya + 1.5) ; 
    \draw[thick] (\x1 , \ya)--(\x1 +1 , \ya - 0.5) ; 
    \draw[thick] (\x1-2 , \ya)--(\x1 -3 , \ya + 0.5) ; 
    \draw[thick] (\x1-2 , \ya)--(\x1 -3 , \ya - 0.5) ; 
    \draw[->] (\x1 -1, \ya + 0.75) to[out=70, in=-70] (\x1 -1,\y1 -0.75);
    \draw[->][blue] (\x1 -1 - 0.2,\y1 -0.75) to[out=-70, in=70] (\x1 -1 -0.2, \ya + 0.75);
    \filldraw (\x1 + 0.01, \y1 ) circle (\v pt);
    \filldraw (\x1 + 0.01 , \ya) circle (\v pt);
    \filldraw (\x1 -1 + 0.01 , \ya) circle (\v pt);
    \node[above] at (\x1 + 0.01, \y1 ) {$w$};
    \node[below] at (\x1 - 1 -0.01 , \ya) {$v$};
    \node[above] at (\x1 - 1 -0.01 , \y1) {$\overline{s} (v)$};
    \node at (\x1 -1 -0.35, \y1 - 1.5) {\textcolor{blue}{$s$}};
    \node at (\x1 -1 +0.4, \y1 - 1.5) {$\overline{s}$};

    \node at (\x2-1, \yb - 1.5) {$Valence (w) > 4$};
    \draw[pattern=vertical lines, pattern color=black!15, draw=none] (\x2-1 , \yb) -- (\x2 , \yb +0.5) -- (\x2 , \yb) -- cycle;
    \draw[thick] (\x2-3  , \y2)--(\x2 -1 , \y2) ;
    \draw[thick][teal] (\x2 -1 , \y2)--(\x2 , \y2) ;
    \draw[thick][orange] (\x2 , \y2)--(\x2 +1 , \y2 + 0.6) ; 
    \draw[thick][orange] (\x2 , \y2)--(\x2 +1 , \y2 + 0.2) ; 
    \draw[thick] (\x2 , \y2)--(\x2 +1 , \y2 - 0.2) ; 
    \draw[thick] (\x2 , \y2)--(\x2 +1 , \y2 - 0.6) ; 
    \draw[thick] (\x2-2 , \y2)--(\x2 -3 , \y2 + 0.5) ; 
    \draw[thick] (\x2-2 , \y2)--(\x2 -3 , \y2 - 0.5) ; 
    \draw[thick] (\x2-3  , \yb)--(\x2 -1 , \yb) ;
    \draw[thick][teal] (\x2 -1 , \yb)--(\x2 , \yb) ;
    \draw[thick][orange] (\x2 , \yb +0.5)--(\x2 +1 , \yb + 0.6 +0.5) ; 
    \draw[thick][orange] (\x2 , \yb +0.5)--(\x2 +1 , \yb + 0.2 +0.5) ; 
    \draw[thick][orange] (\x2-1 , \yb) -- (\x2 , \yb +0.5); 
    \draw[thick] (\x2 , \yb)--(\x2 +1 , \yb - 0.2) ; 
    \draw[thick] (\x2 , \yb)--(\x2 +1 , \yb - 0.6) ; 
    \draw[thick] (\x2-2 , \yb)--(\x2 -3 , \yb + 0.5) ; 
    \draw[thick] (\x2-2 , \yb)--(\x2 -3 , \yb - 0.5) ; 
    \draw[->] (\x2 -1, \yb + 0.75) to[out=70, in=-70] (\x2 -1,\y2 -0.75);
    \draw[->][blue] (\x2 -1 - 0.2,\y2 -0.75) to[out=-70, in=70] (\x2 -1 -0.2, \yb + 0.75);
    \filldraw (\x2 + 0.01, \y2 ) circle (\v pt);
    \filldraw (\x2 + 0.01 , \yb) circle (\v pt);
    \filldraw (\x2 -1 + 0.01 , \yb) circle (\v pt);
    \filldraw[orange] (\x2 , \yb +0.5) circle (\v pt);
    \node[above] at (\x2 + 0.01, \y2 ) {$w$};
    \node[below] at (\x2 - 1 -0.01 , \yb) {$v$};
    \node[above] at (\x2 - 1 -0.01 , \y2) {$\overline{s} (v)$};
    \node at (\x2 -1 -0.35, \y2 - 1.5) {\textcolor{blue}{$s$}};
    \node at (\x2 -1 +0.4, \y2 - 1.5) {$\overline{s}$};
\end{tikzpicture}

    \caption{Tri-valence Creating Splits}
    \label{fig:Tri-valence Creating Splits}
\end{figure}

\vspace{4mm}

\begin{definition}[Tri-valence Creating Splits]
    \label{Tri-valence Creating Splits}
    \hypertarget{Tri-valence Creating Splits}{Let} $s: H \longrightarrow G$ be a split. Furthermore, let 
    $v \in G$ be the \hyperlink{folding edges and vertices}{folding vertex} of $G$ rel $\overline{s}$, and $w \in H$ the \hyperlink{folding edges and vertices}{splitting vertex} of $H$ rel $\overline{s}$.
    We call $s$ a ``Tri-valence Creating Split from $H$ to $G$" if,
    \begin{enumerate}
        \item the folding vertex $v \in G$ has valence $3$ and $\overline{s} (v) \in H$ has valence $2$, and,
        \item the valence of the splitting vertex $w \in H$ is strictly higher than $3$.
    \end{enumerate}
\end{definition}

\vspace{4mm}

\begin{notation}[The Number of Natural Edges and Vertices of a Graph]
    \label{E(G)}
    \hypertarget{E(G)}{Let} $G$ be a finite graph. ``$E(G)$" denotes the number of natural edges of $G$, and ``$V(G)$" denotes the number of natural vertices of $G$.
\end{notation}

\vspace{4mm}

\begin{remark}[Tri-valence Creating Splits Produce Graphs with Higher Numbers of Vertices and Edges]
    \label{Tri-valence Creating Splits Increase Things}
    \hypertarget{Tri-valence Creating Splits Increase Things}{Let} $n \in \mathds{N}$. Suppose that $H$ is a rank $n$ core graph with at least one vertex $w$ whose valence is $> 3$. (Note that any core graph of finite rank is a finite graph.) Then there exists a \hyperlink{Tri-valence Creating Splits}{tri-valence creating split} $s: H \longrightarrow G$ where $G$ is a rank $n$ core graph such that $E(H) < E(G)$ and $V(H) < V(G)$ [See Figure \ref{fig:Tri-valence Creating Splits}].
\end{remark}

\vspace{4mm}

\begin{lemma}
    \label{upperbound to natural vertices}
    Let $n \in \mathds{N}$, and let $G$ be a core graph of rank $n$. Then the number of natural vertices of $G$ is bounded above by $2(n-1)$, and the number of natural edges of $G$ is bounded above by $3(n-1)$. i.e. $\hyperlink{E(G)}{V(G)} \leq 2(n-1)$ and $\hyperlink{E(G)}{E(G)} \leq 3(n-1)$.
\end{lemma}

\begin{proof}
    Let $n \in \mathds{N}$, and let $G$ be a core graph of rank $n$. Then either (case a.) all natural vertices of $G$ has valence $3$, or (case b.) $G$ has at least one natural vertex whose valence is strictly higher than $3$.

    \vspace{4mm}

    (Case a.) Suppose all natural vertices of $G$ has valence $3$. Then, $G$ has $3/2$ natural edges per each natural vertex. i.e. $E(G) = \frac{3}{2} V(G)$. Then since the Euler Characteristic of $G$ is $\mathcal{X} (G) = 1 - n$, we have, $V(G) - E(G) = 1 - n \Longrightarrow V(G) = 2(n-1)$, and $E(G) = 3(n-1)$. Furthermore, since we only made one assumption about $G$ other than its rank, we can conclude that each rank $n$ core graph $H$ such that all its natural vertices have valence $3$, satisfy $V(H) = 2(n-1)$ and $E(H) = 3(n-1)$.

    \vspace{4mm}

    (Case b.) Now suppose $G$ has at least one natural vertex $w$ whose valence is strictly higher than $3$. Then recall from Remark \ref{Tri-valence Creating Splits Increase Things} and Figure \ref{fig:Tri-valence Creating Splits} that there exists a tri-valence creating split $s_1 : G \longrightarrow H_1$ such that $V(G) < V(H_1)$ and $E(G) < E(H_1)$. Since $G$ is a finite graph, $G$ has finitely many vertices whose valence is higher than $3$, and each such vertex is of a finite valence. Thus, there exists $m \in \mathds{N}$, and a finite sequence of splits $G \xrightarrow[]{s_1} H_1 \xrightarrow[]{s_1} H_2 \xrightarrow[]{s_3} ...  \xrightarrow[]{s_m} H_m $ such that, each vertex of $H_m$ has valence $3$, and for each $j \in \{1,...,m\}$, $s_j$ is a tri-valence creating split. Then, it follows from the conclusion of case a. and from Remark \ref{Tri-valence Creating Splits Increase Things}, that $V(G) < V(H_1) < V(H_2) < ... < V(H_m) = 2(n-1)$, and that $E(G) < E(H_1) < E(H_2) < ... < E(H_m) = 3(n-1)$.
\end{proof}

\vspace{4mm}

\vspace{4mm}

Note that for a rank $n$ proper split sequence $\zeta$, the statements in the above lemma would be true for each level graph. 

\vspace{4mm}

Next, we define some useful terminology that will be utelized in the upcoming proof. 

\vspace{4mm}

\begin{definition}[Pre-Singularities and Pseudo-Singularities]
    \label{Pre-Singularities}
    \hypertarget{Pre-Singularities}{Let} $\zeta$ be a proper split sequence, $X = X(\zeta)$ it's induced solenoid. Furthermore, let $x = (x_0, x_{-1}, x_{-2} , ... ) \in X$. We say that $x$ is a ``Pre-Singularity of $X$", if there exists a $K \in \n$ such that, for each integer $j \leq K$, $x_j$ is a natural vertex of $G_j$. We call $x$ a ``Pseudo-Singularity of $X$" if $x$ is a pre-singularity, but is not a singularity.
\end{definition}

\vspace{4mm}

\begin{notation}[The Set of Singularities of a Solenoid]
    \label{Sing(X)}
    \hypertarget{Sing(X)}{Let} $X$ be a proper solenoid. We denote by ``$Sing(X)$" the set of singularities of $X$. We denote by ``$PSing(X)$" the set of pre-singularities of $X$.
\end{notation}



\vspace{4mm}

\begin{proposition}
    \label{finitely many singularities}
    \hypertarget{finitely_many_singularities}{Let} $n \in \mathds{N}$, and let $\zeta : G_0 \xleftarrow[]{f_{-1}} G_{-1} \xleftarrow[]{f_{-2}} G_{-2} \xleftarrow[]{f_{-3}} ...$ be a proper split sequence of rank $n$. Then, the number of \hyperlink{Pre-Singularities}{pre-singularities} of $X(\zeta)$ is bounded above by $2(n-1)$. Therefore, the number of singularities of $X(\zeta)$ is also bounded above by $2(n-1)$.
\end{proposition}

\begin{proof}
    Assume the set-up of the lemma.
    For each pre-singularity $s \in X(\zeta)$, by definition, there exists $K_s \in \n$ such that, for each integer $j \leq K_s$, $x_j$ is a natural vertex of $G_j$. 
    
    \vspace{4mm}

    Now, assume that a subset $S_m \subseteq PSing(X)$ has $m := 2(n-1) + 1$ distinct elements. Then there must exist some integer $J \leq Min \{ K_s : s \in S_m \}$ such that, in $G_{J}$, the set of projections $\{ s_J : s \in S_m \}$ consists of $m$ distinct natural vertices. This contradicts the upper-bound established in Lemma \ref{upperbound to natural vertices} for natural vertices of a rank $n$ core graph such as $G_J$.
\end{proof}

\vspace{4mm}

In the next section, we will establish ``Stabilizing Hypotheses" for a proper split sequence, so that when $\zeta$ follows them, the usage of all the tools we established to study $X(\zeta)$ will become much more straight forward. For instance, when $\zeta$ is ``Stabilized", each singularity in $X(\zeta)$ will have a maximal star chain that starts at level $0$. What follows is a careful exploration of this stabilization process. 

\vspace{4mm}

\subsection{Stabilizing Split Sequences}
\label{sec: Stabilizing Split Sequences}

\vspace{4mm}

Given a proper split sequence $\zeta$, the process of studying its induced solenoid $X(\zeta)$ becomes easier, if $\zeta$ satisfies a collection of criteria that we shall call the ``Stabilizing Hypotheses" [Remark \ref{The Stabilizing Hypotheses}]. 

\vspace{4mm}

\begin{notation}[Tail of a Split Sequence]
    \label{Tail of a Split Sequence}
    \hypertarget{Tail of a Split Sequence}{Given} a proper split sequence $\zeta : G_0 \xleftarrow[]{f_{-1}} G_{-1} \xleftarrow[]{f_{-2}} G_{-2} \xleftarrow[]{f_{-3}} ...$, and $K \in -\mathds{N}$, we call the split sequence, $G_K \xleftarrow[]{f_{K-1}} G_{K-1} \xleftarrow[]{f_{K-2}} G_{K-2} \xleftarrow[]{f_{K-3}} ...$, the ``$K-$Tail of $\zeta$", and denote it ``$\zeta_{< K}$".
\end{notation}

\vspace{4mm}

We will lay out the ``Stabilizing Hypotheses" in Remark \ref{The Stabilizing Hypotheses}, and we will show in Lemmas \ref{Vertex Stabilizing Criteria} , \ref{Turn Stabilizing Criteria} that, for each proper split sequence $\zeta$ it is possible to find a level $K \in \n$ such that \hyperlink{Tail of a Split Sequence}{$\zeta_{<K}$} satisfies the necessary stabilizing hypotheses. Since $\zeta$ and $\zeta_{<K}$ both induce the same solenoid, to study all proper solenoids, it is enough to study all proper split sequences that satisfy the stabilization hypotheses.  

\vspace{4mm}

In the discussion in this section, assume that $\zeta$ is a given proper solenoid, and $X = X(\zeta)$ its induced solenoid.

\vspace{4mm}

\subsubsection{Preliminary Concepts and the Stabilization Hypotheses}
\label{sec: Preliminary Concepts and the Stabilization Hypotheses}

\vspace{4mm}

Here, we will establish some preliminary concepts needed to understanding the ``Stabilizing Hypotheses". 

\vspace{4mm}

We start with defining ``Leaf Segments, Pre-Leaf Segments and Singular Segments" [Definition \ref{leaf segments}],  ``Taken Turns and Sustained Turns" [Definition \ref{Taken Turns and Surviving Turns}], ``Turn Tunnels and Extended Turn Tunnels" [Definition \ref{Turn Tunnels and Extended Turn Tunnels}], before laying out the intended consequences of the Stabilizing Hypotheses in Remark \ref{The Intended Consequences of Stabilizing Hypotheses}.

\vspace{4mm}

Recall [from Lemma \ref{shapse of plaques 2}] that each \hyperlink{plaque}{standard plaque in $X$} is a \hyperlink{star set}{star set}.

\vspace{4mm}

\begin{definition}[Leaf Segments, Pre-Leaf Segments and Singular Segments]
    \label{leaf segments}
    \hypertarget{leaf segments}{Let} $L \subset X$. Then $L$ is called a ``Pre-Leaf Segment of $X$", if $L$ is a turn of a standard plaque in $X$. Furthermore, a pre-leaf segment $L$ is called a
    \begin{itemize}
        \item ``Leaf Segment of $X$" if, $L$ is a non-singular plaque in $X$.
        \item ``Singular Segment of $X$" if, $L$ is a turn of a singular plaque in $X$.
    \end{itemize}
\end{definition}



\vspace{4mm}

Since the collection of standard neighborhoods of $X$ covers $X$, the collection of all pre-leaf segments of $X$ covers $X$. In Chapter \ref{Ch: Solenoids as Foliated Spaces}, we will see that using certain sub-collections of pre-leaf segments will be useful in the process of visualizing patches of $X$. The following two definitions lay out the sub-collections of pre-leaf segments that will prove to be useful in later discussions.

\vspace{4mm}

\begin{definition}[Taken Turns and Sustained Turns]
    \label{Taken Turns and Surviving Turns}
    \hypertarget{Taken Turns}{Let} 
    $K \in - \mathds{N}^*$,
    and let $I$ be a \hyperlink{turn in zeta}{level $K$ standard turn in $\zeta$}.
    Furthermore, let $L \subset X$ be a \hyperlink{leaf segments}{pre-leaf segment} (resp. a \hyperlink{leaf segments}{leaf segment}) of $X$.
    We say that ``$L$ Sustains the Turn $I$" (resp.``$L$ Takes the Turn $I$") or that the ``Turn $I$ is Sustained by $L$" (resp. ``Turn $I$ is Taken by $L$") if there exists a homeomorphism $ l : I \longrightarrow L$ such that $\pi_K \circ l = $ the identity function on $I$.

    \vspace{4mm}

    We say that ``$I$ is a Sustained Turn in $\zeta$" (resp. ``$I$ is a Taken Turn in $\zeta$") and that ``$I$ is Sustained by $X$" (resp. ``$I$ is Taken by $X$") if there exists a pre-leaf segment (resp. leaf segment) in $X$ that sustains (resp. takes) $I$. 
\end{definition}

\vspace{4mm}

Note that since $\zeta$ permits no \hyperlink{Fold Compositions and Backtracking}{backtracking}, each standard turn in $\zeta$ is sustained by $X$. But not each standard turn in $\zeta$ is necessarily taken by $X$.




\vspace{4mm}

\begin{definition}[Turn Tunnels and Extended Turn Tunnels]
    \label{Turn Tunnels and Extended Turn Tunnels}
    \hypertarget{Def: Tun and ExTun}{Assume} the set-up of the above definition. The ``Turn Tunnel of $I$" (denoted ``$Tun_I$") is the union of all \hyperlink{leaf segments}{leaf segments} in $X$ that \hyperlink{Taken Turns}{take} the turn $I$. ``$SingSeg_I$" shall denote the union of all the \hyperlink{leaf segments}{singular segments} that sustain $I$. The ``Extended Turn Tunnel of $I$ (denoted ``$ExTun_I$") is the union of all \hyperlink{leaf segments}{pre-leaf segments} that sustain the turn $I$ (i.e. the union of $Tun_I $ and $ SingSeg_I$).
    
\end{definition}

\vspace{4mm}

We will explore turn tunnels and extended turn tunnels further in Section \ref{sec: Tunnels}, and we will use their shapes to visualize proper solenoids throughout Chapter \ref{Ch: Solenoids as Foliated Spaces}. 

\vspace{4mm}

\begin{remark}[Inverse Limits of $2-$Pronged Star Chains]
    Let $K \in \n$, and let $I$ be a level $K$ standard turn in $\zeta$.
    \begin{itemize}
        \item Each $2-$pronged star chain in $\zeta$ that \hyperlink{star chain representing a set in X}{starts with} $I$, uniquely determines a pre-leaf segment sustained by $I$ (given by $L := \varprojlim^{\zeta}  \mathcal{I}$ ), and,
        \item each pre-leaf segment $L$ that sustains $I$, uniquely determines a $2-$pronged \hyperlink{star chain representing a set in X}{star chain representing $L$ that starts with $I$} (given by $\pi_{K} (L)\xleftarrow[]{f_{K-1}} \pi_{K-1} (L) \xleftarrow[]{f_{K-2}} \pi_{K-2} (L) \xleftarrow[]{f_{K-3}}...$).
    \end{itemize}
    We will use this notion to parameterize $Tun_I$ and $ExTun_I$ in Sub-section \ref{Sec: Canonically Given Tunnel Parameterizations}.
\end{remark}


\vspace{4mm}

Having established all the necessary terminology, we will now lay out the consequences of the stabilizing hypothesis, that will be helpful in the upcoming sections. 

\vspace{4mm}


\begin{remark}[The Intended Consequences of Stabilizing Hypotheses]
    \label{The Intended Consequences of Stabilizing Hypotheses}
    The ``Stabilizing Hypotheses" [defined in the next remark] are designed so that when $\zeta$ satisfies those hypotheses, the \hyperlink{Standard Star Neighbourhoods}{standard neighborhoods}, \hyperlink{Def: Tun and ExTun}{turn tunnels} and \hyperlink{Def: Tun and ExTun}{extended turn tunnels} of $X(\zeta)$ satisfy the following properties.
\begin{enumerate}
    \item Each standard neighborhood of $X(\zeta)$ contains at most one pre-singularity [Directly follows from items 1 and 2 of Remark \ref{The Stabilizing Hypotheses}]. 
    \item For each \hyperlink{turn in zeta}{standard turn $I$ in $\zeta$}, $SingSeg_I = ExTun_I - Tun_I$ is either empty or is a single singular segment [Follows from item 1 above].
    \item Each singularity of $X(\zeta)$ has a maximal star chain that starts at level $0$ [Follows from item 3 of Remark \ref{The Stabilizing Hypotheses}]. Furthermore, each extended turn tunnel is homeomorphic to a product space $C \times I$ where $C$ is a totally disconnected set and $I$ is an open interval [Proposition \ref{ExTun are Tunnel Sets}].
    \item For each \hyperlink{turn in zeta}{standard turn $I$ in $\zeta$} such that
    $SingSeg_I , Tun_I \neq \emptyset$,
    the unique singular segment $L$ contained in $ExTun_I$, entirely consists of limit points of $Tun_I$ [Follows from item 4 of Remark \ref{The Stabilizing Hypotheses}, as shown later in Lemma \ref{limit segments consist of limit points}]. 
\end{enumerate}
\end{remark}
 
\vspace{4mm}



\begin{remark}[The Stabilizing Hypotheses]
    \label{The Stabilizing Hypotheses}
    \hypertarget{Stabilizing Hypotheses}{Given} a proper split sequence  $ \zeta : G_0 \xleftarrow[]{f_{-1}} G_{-1} \xleftarrow[]{f_{-2}} G_{-2} \xleftarrow[]{f_{-3}} ... $, we say that ``$\zeta$ Satisfies the Stabilizing Hypotheses" if, it satisfies the two ``Vertex Stabilizing Criteria" that are laid out in items 1, 2 below, and the two ``Turn Stabilizing Criteria" that are laid out in items 3, 4 below.

    \vspace{4mm}
    
    \hypertarget{VSC}{Vertex Stabilizing Criteria}: For each $j \in \n$ and for each $q \in G_j$,
    \begin{enumerate}
        \item If $q$ is a natural vertex of $G_j$, then $F_q$ contains at most one pre-singularity,
        \item If $q$ is not a natural vertex of $G_j$, then $F_q$ contains no pre-singularities,
    \end{enumerate}
    \hypertarget{TSC}{Turn Stabilizing Criteria}: For each $K \in \n$ and for each turn $I$ of a standard star neighborhood in $G_K$,
    \begin{enumerate}
        \item[3.] There is no $3-$pronged star set $U$ in $X$ such that $\pi_K (U) = I$,
        \item[4.] \hypertarget{4th hype}{If} there exists a singular segment $L$ that \hyperlink{Taken Turns}{sustains the turn} $I$, then exactly one of the following is true: Either,
        \begin{enumerate}
            \item for each integer 
            $j \in \n$, 
            there exists a leaf segment that \hyperlink{Taken Turns}{takes the turn} $\pi_j (L)$, or,
            \item No image of 
            $I$ is \hyperlink{Taken Turns}{taken by $X$}.
            (i.e. no leaf segment takes the turn $\pi_0 (L) \subset G_0$.)
        \end{enumerate}
    \end{enumerate}
\end{remark}

\vspace{4mm}

\begin{definition}[Stabilized Split Sequences]
    \label{Stabilized Split Sequence}
    \hypertarget{Stabilized Split Sequence}{A proper} split sequence $\zeta$ shall be called a ``Stabilized Proper Split Sequence" if $\zeta$ satisfies the \hyperlink{Stabilizing Hypotheses}{stabilizing hypotheses}.
\end{definition}

\vspace{4mm}

\subsubsection{Each Proper Solenoid Can be Induced by a Stabilized Split Sequence}
\label{sec: Each Proper Solenoid Can be Induced by a Stabilized Split Sequence}

\vspace{4mm}

We will show the titular result in two parts. Lemma \ref{Vertex Stabilizing Criteria} will focus on the vertex stabilizing criteria. Then we will establish some necessary terminology that are used in the proof of Lemma \ref{Turn Stabilizing Criteria} which will focus on the turn stabilizing criteria. 

\vspace{4mm}

\begin{lemma}
    \label{Vertex Stabilizing Criteria}
    Let $\zeta$ be a proper split sequence. Then there exists a $K \in \n$ such that \hyperlink{Tail of a Split Sequence}{$\zeta_{<K}$} satisfies the vertex stabilizing criteria.
\end{lemma}

\begin{proof}
    Let $\zeta$ be a proper split sequence. Recall from Proposition \ref{finitely many singularities}, that there are only finitely many pre-singularities in $X(\zeta)$. 
    It is true by definition (of a pre-singularity) that, for each $x \in PSing(X)$, there exists $K_x \in \n$ such that, for each integer $j < K_x$, $\pi_j (x) = x_j$ is a natural vertex of $G_j$.
    Since $PSing(X)$ is a finite collection, it is also possible to choose the aforementioned level $K_x$ for each $x$, such that there exists a maximal star chain for $x$ that starts at level $K_x$. (Note that we do not use maximal star chains in the current proof, although choosing the levels in this manner will be relevant to the proof of the next lemma, in which we shall utilize the level $K$ we'll construct in the current proof.)
    Thus, for each integer $j \leq Min\{ K_x : x \in PSing(X) \}$, each pre-singularity projects to a natural vertex of $G_j$.
    Furthermore, since $PSing(X)$ is a finite collection, and since each distinct pair of pre-singularities $x,x'$ must have some level $J_{\{x,x'\}} \in \n$ such that for each integer $j \leq J_{\{x,x'\}} $, $\pi_j (x) \neq \pi_j (x')$, 
    we can find an integer $K \leq Min\{ K_x : x \in PSing(X) \}$ such that, for each integer $j \leq K$, $\pi_j |_{PSing(X)}$ is one to one. Thus, for each integer $j \leq K$, and for each $p \in G_j$, if $p \in V(G_j)$, there can be at most one pre-singularity projecting to $p$, and if $p \notin V(G_j)$, there can be no pre-singularities projecting to $p$.
\end{proof}

\vspace{4mm}

The proof of Lemma \ref{Turn Stabilizing Criteria} will be a continuation of the above proof, and it will show that there exists an integer $K' \leq$ the above-defined $K$, such that $\zeta_{<K'}$ satisfies the turn stabilizing criteria. But first, we shall define some terminology that will be used in the following proof and in the upcoming chapters.

\vspace{4mm}

We start with identifying a special type of \hyperlink{standard basis}{standard star neighbourhoods} that we will utilize later. One of our intentions here, is to single out a special kind of turn called an ``Optimal Turn" such that, given a proper split sequence $\zeta$, and a level $j \in \n$,
\begin{enumerate}
    \item[i.] $G_j$ has only finitely many optimal turns, and,
    \item[ii.] each \hyperlink{turn in zeta}{level $j$ standard turn in $\zeta$}, is contained in an optimal turn,
\end{enumerate}
so that, by studying this finite collection of optimal turns at each level, we can make conclusions about all turns in $\zeta$. These optimal neighborhoods and turns, defined below, will be heavily utilized in Section \ref{Sec: An Atlas for the Solenoid} to provide a singular foliated atlas for a given proper solenoid.

\vspace{4mm}

\begin{definition}[Optimal Star Neighborhoods and Optimal Turns]
    \label{Optimal Star Neighbourhoods}
    \hypertarget{Optimal Star Neighborhoods}{Let} $\zeta$ be a proper split sequence that satisfies the \hyperlink{VSC}{vertex stabilization criteria}, $j \in \n$ and $G_j$ the level $j$ graph of $\zeta$.
    We call a directed natural edge of a graph a ``Loop" if its initial vertex is the same as its terminal vertex. We call a point $p \in G_j$ an ``Optimal Point in $G_j$" if $p$ is either 
    a \hyperlink{special points}{special point}
    of $G_j$ or the mid-point of a loop in $G_j$. We call a \hyperlink{standard basis}{standard star neighborhood} $V$ of $G_j$ a ``Pre-Optimal Star Neighborhoods in  $G_j$ rel $\zeta$" if,
    \begin{enumerate}
        \item $V$ is centered at an optimal point in $G_j$, and,
        \item $V$ with its center removed does not contain any optimal points in $G_j$.
    \end{enumerate}
    Furthermore, we call $V$ a ``Level $j$ Optimal Star Neighborhood of $\zeta$" if $V$ isn't properly contained in any other pre-optimal star neighborhoods in $G_j$ with the same center as $V$. 

    \vspace{4mm}

    \hypertarget{optimal turns}{A turn} $I$ in $G_j$ is called a ``Level $j$ Optimal Turn in $\zeta$" or an ``Optimal Turn of $G_j$" if $I$ is a turn of an optimal star neighborhood in $G_j$. 
    \hypertarget{Optimal Neighborhood}{For} each level $j$ optimal star neighborhood $V$, we call $\pi_{j}^{-1} (V)$ a ``Level $j$ Optimal Neighborhood in $X$".
\end{definition}


\vspace{4mm}

\begin{remark}[Properties of Optimal Star Neighborhoods]
    \label{Properties of Optimal Star Neighbourhoods}
    Let $j \in \n$.
    \begin{enumerate}
        \item Since $G_j$ only contains finitely many special points, and finitely many loops,
        we have that $G_j$ contains finitely many optimal star neighborhoods (therefore finitely many optimal turns), and the collection of all optimal star neighborhoods in $G_j$ cover $G_j$. And thus, the collection of all \hyperlink{Optimal Neighborhood}{level $j$ optimal neighborhoods} cover $X$. 
        \item Since each \hyperlink{Standard Star Neighbourhoods}{standard star neighborhood} in $G_j$, is contained in a level $j$ optimal star neighborhood, 
        any \hyperlink{turn in zeta}{level $j$ standard turn} is contained in a level $j$ optimal turn. Thus any pre-leaf segment of $X(\zeta)$ is contained in a pre-leaf segment that sustains an optimal turn.
    \end{enumerate}
\end{remark}

\vspace{4mm}


\begin{lemma}
    \label{Turn Stabilizing Criteria}
    Let $\zeta$ be a proper split sequence. Then there exists a $K' \in \n$ such that $\zeta_{<K'}$ satisfies the turn stabilizing criteria.
\end{lemma}

\begin{proof}
    Assume all the notations that were established in the proof of Lemma \ref{Vertex Stabilizing Criteria}. 
    Let $K \in \n$ be what was established at the end of the earlier proof. Then for each singularity $s$ of $X$, we have a maximal star chain $\mathcal{V}^s :  V^s_{K} \xleftarrow[]{f_{K-1}} V^s_{K-1} \xleftarrow[]{f_{K-2}} V^s_{K-2} \xleftarrow[]{f_{K-3}}...$ that starts at $K$, and therefore $\pi_K$ restricted to the \hyperlink{Singular Plaques}{singular plaque} $V^s := \varprojlim^{\zeta} \mathcal{V}^s$, is a homeomorphism onto its image $V^s_K$ in $G_K$.

    \vspace{4mm}
    
    We aim to show that $\zeta_{<K}$ satisfies criterion 3 from Remark \ref{The Stabilizing Hypotheses} (and later in the proof, we will show the existence of an integer $K' \leq K$ such that $\zeta_{<K'}$ satisfies criterion 4). Note that if there exists a $3-$pronged star set $U$ in $X$ that projects to a turn in any level graph of $\zeta_{<K}$, then $U$ would project to a turn in $G_K$.
    Now suppose that there is a turn $I$ in $G_K$ and a $3-$pronged star set $U \subset X(\zeta)$ such that $\pi_K (U) = I$. Then the center $x$ of $U$ is a singularity of $X(\zeta)$. Consider the maximal star chain $\mathcal{V}^x$ for $x$ that was established earlier. Then $U' := U \cap V^x $ must be a $3-$pronged star set that projects to the interval $I \cap V^x_K$, implying that $\pi_K |_{U'}$ is not a homeomorphism. This contradicts the earlier conclusion that $\pi_K |_{V^x}$ is a homeomorphism. Thus, for each $3-$pronged star set $\Tilde{U}$ of $X$, $\Tilde{U}$ projects to no turn of $G_K$, and therefore, $\Tilde{U}$ projects to no turn of $G_j$ where $j \in \n +K$. 
    
    \vspace{4mm}

    We aim to find a $K' \in \n + K$, such that $\zeta_{<K'}$ satisfies criterion 4 from Remark \ref{The Stabilizing Hypotheses}. We will accomplish this by choosing, for each optimal turn $I$ of $G_K$, an appropriate level ``$K^I$" (as described in the next paragraph), and then letting $K' := Min \{ K^I : I$ is an optimal turn in $G_K \}$, which is well defined since $G_K$ only has finitely many optimal turns. 

    \vspace{4mm}
    
    Now, let $I$ be a \hyperlink{optimal turns}{level $K$ optimal turn of $\zeta$}. If $I$ is not sustained by any \hyperlink{leaf segments}{singular segment}, let $K^I := K$. Now suppose $I$ is sustained by a singular segment $L$ in $X$. (It follows from vertex stabilization that $L$ is unique.)
    Let $\mathcal{I}^L :  I_{K} \xleftarrow[]{f_{K-1}} I_{K-1} \xleftarrow[]{f_{K-2}} I_{K-2} \xleftarrow[]{f_{K-3}}...$ be the unique $2-$pronged \hyperlink{star chain representing a set in X}{star chain 
    representing $L$ that starts with $I$} (i.e $I_K = I$ and $\varprojlim^{\zeta} \mathcal{I}^L = L$).
    Note that, if there exists an integer $J < K$ such that $I_J$ is taken by a leaf segment, then for each integer $j \in [J , K]$, $I_j$ is taken by the same leaf segment. 
    Thus, exactly one of the following is true. Either,
    \begin{enumerate}
        \item[(a.)] for each integer $j \leq K$, there exists a leaf segment that takes $I_j$, or,
        \item[(b.)] there exists an integer $K_{L} \leq K$ such that,
        \begin{enumerate}
            \item[i.] for each integer $j \leq K_{L} $, no leaf segment takes $I_j$, and,
            \item[ii.] for each integer $j > K_L$ in $\n +K$,
        there exists a leaf segment that takes $I_j$.
        \end{enumerate}
    \end{enumerate}
    If case (a.) is true, then let $K^I := K$. If case (b.) is true, then let $K^I := K_L$. Then, let $K' := Min \{ K^I : I$ is an optimal turn in $G_K \}$.

\end{proof}

\vspace{4mm}

\begin{proposition}
    \label{Every proper split sequence can be stabilized}
    Let $\zeta$ be a proper split sequence. Then there exists $K \in \n$ such that $\zeta_{<K}$ is a stabilized proper split sequence.
\end{proposition}

\begin{proof}
    Follows from Lemmas \ref{Vertex Stabilizing Criteria} and \ref{Turn Stabilizing Criteria}.
\end{proof}

\vspace{4mm}

The immediate consequences of Stabilization Hypotheses, were laid out in Remark \ref{The Stabilizing Hypotheses}. 

\vspace{4mm}

\begin{remark}[More Consequences of Stabilization]
    \label{More Consequences of Stabilization}
    Let $\zeta$ be a proper stabilized split sequence. Then,
    \begin{enumerate}
        \item Each maximal star chain in $\zeta$ starts at level $0$.
        \item For each $j \in \n$ and for each \hyperlink{plaque}{standard plaque} $V$ in $X$, the map $\pi_j |_{V}$ is a homeomorphism onto its image in $G_j$.
    \end{enumerate}
\end{remark}

\vspace{4mm}

In Chapter \ref{Ch: Solenoids as Foliated Spaces} and beyond, we will always assume that the split sequences being discussed are proper and stabilized.

\vspace{4mm}

\begin{center}
    \section{Solenoids as Foliated Spaces} \label{Ch: Solenoids as Foliated Spaces}
\end{center}


\vspace{4mm}

\h Throughout the discussion in this chapter, we assume that $\zeta$ is a given \hyperlink{Stabilized Split Sequence}{stabilized} proper split sequence and $X = X(\zeta)$ its induced solenoid.

\vspace{4mm}

Note that the collection of turn tunnels of $X$ cover $X - Sing(X)$, and the collection of extended turn tunnels of $X$ cover $X$. Thus understanding the shapes of these tunnel sets can shed light on understanding the shape of $X$. The purpose of this chapter is to show that we can build a naturally arising singular foliated atlas on $X$ (resp. a foliated atlas on $X - Sing(X)$) akin to singular foliated atlases (resp. foliated atlases) on surfaces. 




\vspace{4mm}

We start this discussion with some context and inspiration from foliations and singular foliations on surfaces, and a description of how atlases are used to endow the said structures on surfaces [Section \ref{Foliated Atlases}]. 

\vspace{4mm}

In Section \ref{sec: Tunnels}, we will build naturally arising parameterizations for turn tunnels and extended turn tunnels of $X$ [\ref{Sec: Canonically Given Tunnel Parameterizations}], then use those parameterizations to show properties of the aforementioned classes of tunnels [\ref{Sec: Properties of Tun}]. We will show that each turn tunnel is open in $X$ [Lem. \ref{turn tunnels are open}], and that each extended turn tunnel (and each turn tunnel) is homeomorphic to the product of a totally disconnected space and an open interval [Prop. \ref{ExTun are Tunnel Sets}, Rem. \ref{An Extended Turn Tunnel with a Few Pre-Leaf Segments Removed}]. It will also be apparent that given an extended turn tunnel (resp. turn tunnel) $T$, each pre leaf segment (resp. leaf segment) contained in $T$, is a path component of $T$ [Rem. \ref{pre-leaf Segments and Path components}]. Then, in subsection \ref{Sec: Tun in ExTun}, we will investigate how exactly turn tunnels are contained in extended turn tunnels.

\vspace{4mm}

In Section \ref{Sec: Shape of Standard Nbhds}, we shall investigate the shape of standard neighborhoods, by observing how standard neighborhoods can be constructed with extended turn tunnels. We will define ``Star Tunnel Sets" [Def. \ref{star tunnel sets}] as certain quotients of finitely many tunnels and a star set, then show that each singular standard neighborhood is the disjoint union of a star tunnels set and finitely many turn tunnels [Prop. \ref{sing stnd nbhds vs star tunnel nbhds}].

\vspace{4mm}

In Section \ref{Sec: An Atlas for the Solenoid}, we will show that, $X - Sing(X)$ has a finite open cover $\mathcal{TC}(X)$, where each element of $\mathcal{TC}(X)$ is a turn tunnel [\ref{sec: The Tunnel Atlas}], and furthermore, $X$ has a finite open cover $\mathcal{STC}(X)$ where each element of $\mathcal{STC}(X)$ is either a star tunnel neighborhood or a turn tunnel [\ref{sec: The Star Tunnel Atlas}]. We will then show that, for each distinct pair of elements $T,T' \in \mathcal{TC}(X)$ (resp. $T,T' \in \mathcal{STC}(X)$) that non-trivially intersect, $T \cap T'$ is a finite disjoint unions of \hyperlink{abstract tunnels}{tunnels} [Lem. \ref{how turn atlas charts intersect}, \ref{intersecting elements of STC(X)}].

\vspace{4mm}

\subsection{Singular Foliated Atlases on Surfaces}
\label{Foliated Atlases}

\vspace{4mm}

A foliation can arise naturally on a manifold as a family of solutions to a system of differential equations \cite{bams/1183535509}. However, in this text we aim to adopt the concept of a foliation into a non-manifold context. 
\vspace{4mm}

This is an expositional section on foliations and singular foliations on surfaces (meant as a precursor to the exploration of the naturally arising foliated-structure of solenoids), containing some definitions borrowed from the wider literature, and other definitions constructed to be equivalent to the classical definitions so that foliations are looked at from a purely topological point of view. Readers who are familiar with the aforementioned preliminary material may skip to Section \ref{sec: Tunnels}.

\vspace{4mm}


\subsubsection*{Foliations on Surfaces}

\vspace{4mm}

In the context of differentiable manifolds, a ``Coordinate Chart" is a way of endowing the notion of ``Euclidean Coordinates" on a neighborhood of a manifold, and an ``Atlas" (which is a collection of coordinate charts whose target neighborhoods cover a given manifold) is a way of formalizing the notion that `a manifold is locally Euclidean'. 

\vspace{4mm}

\begin{definition}[An Atlas on a Manifold]
    Let $d \in \mathds{N}$, and let $M$ be a manifold of dimension $d$.
    A ``$d-$Dimensional Coordinate Chart on $M$" is a homeomorphism $\phi : U \longrightarrow V$ where $U$ is a 
    neighborhood of $\mathds{R}^d$ and $V$ is a neighborhood of $M$.
    Let $\mathcal{A}$ be a set and let $Atl(M) := \{\phi_{\alpha} : U_{\alpha} \longrightarrow V_{\alpha} \}_{\alpha \in \mathcal{A}}$ be a collection of $d-$dimensional coordinate charts on $M$.
    For each distinct $\alpha , \beta \in \mathcal{A}$, if $V := V_{\alpha} \cap V_{\beta} \neq \emptyset$, we call the map $\phi_{\beta}^{-1} \circ \phi_{\alpha} |_{\phi_{\alpha}^{-1} (V)} :  \phi_{\alpha}^{-1} (V) \longrightarrow \phi^{-1}_{\beta} (V)$ a ``Transition Map of $Atl(M)$".
    We call the collection $Atl(M)$ an ``Atlas on $M$ with Charts of Dimension $d$" if,
    \begin{enumerate}
        \item the collection $\{ V_{\alpha} \}_{\alpha \in \mathcal{A}}$ is an open cover for $M$, and,
        \item each transition map of $Alt(M)$ is a homeomorphism.
    \end{enumerate}
    Item 2. above is called the ``Topological Consistency Criterion for $Atl(M)$". In addition to the above, if
    \begin{enumerate}
        \item[3.] each transition map of $Atl(M)$ is differentiable (resp. smooth),
    \end{enumerate}
    then we call $Atl(M)$ a ``Differentiable Atlas on $M$" (resp. a ``Smooth Atlas on $M$"). We call item 3. the ``Differentiable (resp. Smooth) Consistency Criterion for $Atl(M)$".
\end{definition}

\vspace{4mm}

Given a topological manifold $M$, defining a differentiable (resp. smooth) atlas on $M$ is a prominent way of endowing on $M$ a differentiable (resp. smooth) structure. By changing the nature of the consistency criterion given in item 3. above, to a specific kind of a structure-preserving map, we may endow on $M$ a structure of our choosing. The next few definitions will explore this notion where the desired structure is a foliation. 


\vspace{4mm}

\begin{definition}[Foliations and Singular Foliations on a Surface]
    \label{Foliations on a Surfaces}
    Let $S$ be a surface.
    We call $L \subset S$ an ``Injectively Immersed $1-$Manifold in $S$" (resp. Injectively Immersed Branched $1-$Manifold in $S$") if, $L$ is the image of an injective immersion from a $1-$manifold (resp. a branched $1-$manifold) to $S$. 

    \vspace{4mm}

    A ``One-Dimensional Foliation on $S$" (resp. ``One-Dimensional Singular Foliation on $S$") denoted ``$\mathcal{F}$" is an equivalence relation on $S$ such that each equivalence class of $\mathcal{F}$ is an injectively immersed $1-$manifold (resp. an injectively immersed $1-$manifold or an injectively immersed branched $1-$manifold) in $S$. The pair $(S, \mathcal{F})$ is called a ``Foliated Surface" (resp. a ``Singular Foliated Surface"), and each equivalence class of $\mathcal{F}$ is called a ``Leaf of $(S, \mathcal{F})$". 

    \vspace{4mm}

    Furthermore, if $(S, \mathcal{F})$ is a singular foliated surface, then each leaf of $(S, \mathcal{F})$ that is strictly an immersed branched $1-$manifold, is called a ``Singular Leaf of $(S, \mathcal{F})$".
\end{definition}

\vspace{4mm}

Note that in the above context, $S = \bigsqcup_{L \text{ is a leaf of } (S, \mathcal{F})} L$. 



\vspace{4mm}

The next few definitions shall lay out the concept of a foliated atlas on a surface [Def. \ref{Foliated Atlases on Surfaces}], and how we obtain a foliation out of such an atlas [Def. \ref{surface leaves}]. 

\vspace{4mm}

\begin{definition}[Foliated Atlases on Surfaces]
    \label{Foliated Atlases on Surfaces}
    Let $S$ be a surface, and let $Atlas(S) = \{\phi_{\alpha} : U_{\alpha} \longrightarrow V_{\alpha} \}_{\alpha \in \mathcal{A}}$, be an atlas on $M$ with charts of dimension $2$. 
    We call $Atlas(S)$ a ``Foliated Atlas of $S$" if,
    \begin{enumerate}
        \item for each $\alpha \in \mathcal{A}$, there exist two open intervals $C_{\alpha}$ and $I_{\alpha}$ such that $U_{\alpha} = C_{\alpha} \times I_{\alpha}$. (For each $c \in C_{\alpha}$ the set $\phi_{\alpha} (\{c\} \times I_{\alpha}) \subset V_{\alpha}$ is \hypertarget{atlas leaf segment}{called} a ``Leaf Segment of $V_{\alpha}$ Given by $Atlas(S)$".)
        \item When $\alpha \in \mathcal{A}$, and, $C , I$ are respectively open sub-intervals of $C_{\alpha}, I_{\alpha}$, \hypertarget{atlas sub-leaf segment}{we} call $\phi_{\alpha} (C \times U)$ a ``Sub-Product-Space of $V_{\alpha}$ Given by $Atlas(S)$", and for each $c \in C$, we call $\phi_{\alpha} (\{c\} \times I)$ a ``Sub-Leaf Segment of $V_{\alpha}$ Given by $Atlas(S)$".
        $Atlas(S)$ satisfies the secondary consistency criterion: 
        For each $\alpha , \beta \in \mathcal{A}$, and for each connected component $V$ of $V_{\alpha} \cap V_{\beta}$,
        \begin{enumerate}
            \item $V$ is a sub-product-space of $V_{\alpha}$ and of $V_{\beta}$ in $Atlas(S)$, and,
            \item for each $L \subseteq V$, $L$ is a sub-leaf segment of $V_{\alpha}$ if and only if $L$ is a sub-leaf segment of $V_{\beta}$. 
        \end{enumerate}
    \end{enumerate}   
\end{definition}

\vspace{4mm}
 

The following definition will precisely explain how to obtain a foliation (in terms of its leaves) out of a given foliated atlas on a surface. 

\vspace{4mm}

\begin{definition}[The Foliation Given by an Atlas, and Leaves of a Foliation]
    \label{surface leaves}
    \hypertarget{surface leaves}{Let} $S$ be a surface, and $Atlas(S) = \{\phi_{\alpha} : U_{\alpha} \longrightarrow V_{\alpha} \}_{\alpha \in \mathcal{A}}$ a foliated atlas of $S$. We call $L \in S$ a ``Concatenation of Leaf Segments in $S$ Given by $Atlas(S)$" if either $L$ is a \hyperlink{atlas leaf segment}{leaf segment given by $Atlas(S)$}, or there exist $m \in \mathds{N}+1$, and a collection of leaf segments $\{ L_i \}_{i=1}^m$ given by $Atlas(S)$ such that,
    \begin{enumerate}
        \item for each $i \in \{1,...,m-1\}$, $L_i \cap L_{i+1}$ is a \hyperlink{atlas sub-leaf segment}{sub-leaf segment given by $Atlas(S)$},
         and,
        \item $L = \bigcup_{i=1}^m L_i$.
    \end{enumerate}

    We denote by ``$\mathcal{F}_{Atlas(S)}$" the equivalence class determining the ``Foliation Given by $Atlas(S)$", defined as follows: Let $x, y \in S$. 
    We say that ``$x$ and $y$ are $\mathcal{F}_{Atlas(S)}$-Equivalent to Each Other" if, there exists a concatenation of leaf segments $L$ in $S$ given by $Atlas(S)$ such that, $x,y \in L$. 
    Then each ``Leaf Given by $Atlas(S)$" is an equivalence class of $\mathcal{F}_{Atlas(S)}$.
\end{definition}

\subsubsection*{Singular Foliations on Surfaces}

\vspace{4mm}

Here, we shall lay out the concepts necessary to define a singular foliated atlas (which can endow on a surface, a one-dimensional singular foliation). Then in later sections, we will adopt the features of singular foliations to the context of solenoids. 

\vspace{4mm}

It should be noted that, in the following discussion, we will impose the restriction that the underlying surface is compact and hyperbolic, similar to the surfaces discussed in \cite{Thurston2012}. It assures that $0-$pronged or $1-$pronged singularities do not occur. This is the surface context that is closest to the context of proper solenoids, since in a proper solenoid each singularity will have a higher number of prongs than $2$.

\vspace{4mm}

The purpose behind the following definition is to establish a model set for the ``Singular Neighborhoods" of a surface (of the aforementioned type) with a singular foliation. 


\vspace{4mm}

\begin{definition}[Model Singular Neighborhood, Singular Segment]
    \label{Model Singular Neighbourhood}
    Let $m \in \mathds{N} + 2$ and $\mathds{Z}_m$ be the additive group $\mathds{Z} / m \mathds{Z}$. Let $\{ B_i \}_{i \in \mathds{Z}_m}$ be a collection of spaces each homeomorphic to $[0,1) \times (-1,1)$ along with a collection of homeomorphisms $\{ h_i : [0,1) \times (-1,1) \longrightarrow B_i\}_{i \in \mathds{Z}_m}$. 
    Given a quotient map $q : (\bigsqcup_{i \in \mathds{Z}_m} B_i ) \longrightarrow B$ such that,
    \begin{enumerate}
        \item for each $i \in \mathds{Z}_m$ and for each $t \in [0,1)$, $q$ satisfies $q \circ h_i ((0,t)) = q \circ h_{i+1} ((0,-t))$ (gluing together $h_i ( \{0\} \times [0,1))$ and  $h_{i+1} ( \{0\} \times (-1,0])$), and,
        \item for each  $i \in \mathds{Z}_m$ and for each $(s,t) \in [0,1) \times (-1,1)$,
        \begin{equation*}
        | q^{-1} (q \circ h_i ((s,t))) | = \begin{cases}
			1, & \text{if} \ (s,t) \in (0,1) \times (-1,1)\\
            2, & \text{if} \ (s,t) \in \{0\} \times ((-1,1) -\{0\}) \\
            m, & \text{if} \ (s,t) = (0,0)
		 \end{cases}
    \end{equation*}
    \end{enumerate}
    we call $q$ a ``Star Quotient Map". An image of a star quotient map (along with the underlying structure given by $q$) is called a ``Model Singular Neighbourhood". 
\end{definition}

\vspace{4mm}

In the above context, $L \subset B$ is called a ``Leaf Segment of $(B,q)$" if $L = q \circ h_i (\{t\} \times (-1,1))$ for some fixed $t \in (0,1)$ and fixed $i \in \mathds{Z}_m$. $L \subset B$ is called the ``Singular Segment of $(B,q)$" if $L =  \bigcup_{i \in \mathds{Z}_m} ( q \circ h_i (\{0\} \times (-1,1)))$.

\vspace{4mm}

\begin{definition}[A Singular Foliated Atlas on a Surface]
    \label{singular foliation}
    \hypertarget{singular foliation}{Let $S$ be} a compact hyperbolic surface, and let $Atlas(S) = \{\phi_{\alpha} : U_{\alpha} \longrightarrow V_{\alpha} \}_{\alpha \in \mathcal{A}}$, an atlas where for each $\alpha \in \mathcal{A}$, $U_{\alpha}$ is a neighborhood of $\mathds{R}^2$ and the consistency criterion is that the transition maps are homeomorphisms. We call $Atlas(S)$ a ``Singular Foliated Atlas on $S$" if, there exists a discrete subspace $D$ of $S$, called the ``Set of Singularities Given by $Atlas(S)$", such that,
    \begin{enumerate}
        \item for each $x \in D$, there exists a unique $\alpha_x \in \mathcal{A}$ such that $V_{\alpha_x} \cap (D -\{x\}) = \emptyset$ and $U_{\alpha_x}$ is homeomorphic to a model singular neighborhood,
        \item $Atlas(S) - \{\phi_{\alpha_x} : U_{\alpha_x} \longrightarrow V_{\alpha_x} \}_{x \in D}$ is a foliated atlas on $S - D$, and,
        \item For each $x \in D$, and for each leaf segment (resp. singular segment) $L$ of $U_{\alpha_x}$, we call $\phi_{\alpha_x} (L) \subset V_{\alpha_x}$ a ``Leaf Segment (resp. Singular Segment) of $V_{\alpha_x}$ Given by $Atlas(S)$". 
        $Atlas(S)$ satisfies the additional consistency criterion: For each $\alpha , \beta \in \mathcal{A}$, for each connected component $V$ of $V_{\alpha} \cap V_{\beta}$, and for each leaf segment or singular segment $L$ of $V_{\alpha}$ that non-trivially intersects $V$, 
        \begin{enumerate}
            \item $L \cap V$ is an embedded open interval in $S - D$,
            \item $L \cap V$ is contained in a leaf segment or a singular segment of $V_{\alpha}$, and, of $V_{\beta}$.
        \end{enumerate}
    \end{enumerate}
\end{definition}

\vspace{4mm}

The following definition shows how to obtain the singular foliation determined by a given singular foliated atlas on a surface, by constructing its leaves.

\vspace{4mm}


\begin{definition}[Leaves of a Singular Foliation]
    \label{Leaves of a Singular Foliation}
    Let $S$ be a compact hyperbolic surface with a singular foliated atlas $Atlas(S) = \{\phi_{\alpha} : U_{\alpha} \longrightarrow V_{\alpha} \}_{\alpha \in \mathcal{A}}$. Let $D$ denote the set of singularities given by $Atlas(S)$.
    We call $L \subset S$ a ``Segment Given by $Atlas(S)$" if either $L$ is a leaf segment given by $Atlas(S)$ or a singular segment given by $Atlas(S)$.
    We call $L \subseteq S$ a ``Concatenation of Segments in $S$ Given by $Atlas(S)$" if either $L$ is a segment given by $Atlas(S)$, or there exist $m \in \mathds{N}+1$, and a collection of segments $\{ L_i \}_{i=1}^m$ given by $Atlas(S)$ such that,
    \begin{enumerate}
        \item for each $i \in \{1,...,m-1\}$, $L_i \cap L_{i+1}$ is an embedded open interval in $S - D$,
         and,
        \item $L = \bigcup_{i=1}^m L_i$.
    \end{enumerate}
    We denote by ``$\mathcal{F}_{Atlas(S)}$" the equivalence class determining the ``Foliation Given by $Atlas(S)$", defined as follows: Let $x, y \in S$. 
    We say that ``$x$ and $y$ are $\mathcal{F}_{Atlas(S)}$-Equivalent to Each Other" if, there exists a concatenation of segments $L$ in $S$ given by $Atlas(S)$ such that, $x,y \in L$. 
    Then each ``Leaf Given by $Atlas(S)$" is an equivalence class of $\mathcal{F}_{Atlas(S)}$. A Leaf $L$ of $(S, \mathcal{F}_{Atlas(S)})$ is called a ``Singular Leaf of $(S, \mathcal{F}_{Atlas(S)})$" if $L \cap D \neq \emptyset$.
\end{definition}

\vspace{4mm}

In Section \ref{sec: leaves}, we shall use the above definition as a blueprint in defining ``Leaves" of the solenoid, that spring up not as a result of a structure endowed upon the solenoid by an atlas, but as naturally occurring features of the solenoid due to it's topology. 
In the next few sections we will lay the groundwork to build a singular foliated atlas for a solenoid.

\vspace{4mm}

In Section \ref{sec: Star Tunnel Sets}, we will define ``Star Tunnel Sets" [Definition \ref{star tunnel sets}] that will model a singular neighborhood of a solenoid, analogous to the Model Singular Set [Definition \ref{Model Singular Neighbourhood}] of a singular foliated surface. But first we will explore the building blocks (called Tunnels) out of which the Star Tunnel Sets will be built.

\vspace{4mm}


\subsection{Tunnels}
\label{sec: Tunnels}

\vspace{4mm}

For the rest of this chapter, we will assume that $X = X(\zeta)$ is the solenoid induced by a stabilized proper split sequence [Remark \ref{The Stabilizing Hypotheses}, Definition \ref{Stabilized Split Sequence}] $\zeta  : G_0 \xleftarrow[]{f_{-1}} G_{-1} \xleftarrow[]{f_{-2}} G_{-2} \xleftarrow[]{f_{-3}} ...$.

\vspace{4mm}

Before we establish a singular foliated atlas for $X$, as an intermediate step, we will establish a foliated atlas for the non-singular potion of the solenoid, $X -  Sing(X)$, called the ``\hyperlink{Def: The Tunnel Atlas}{Tunnel Atlas}". All this will be laid out in Section \ref{Sec: An Atlas for the Solenoid}. In this section, we focus on understanding the building blocks out of which a typical coordinate chart of these atlases are built.

\vspace{4mm}

Recall turn tunnels and extended turn tunnels [Definition \ref{Turn Tunnels and Extended Turn Tunnels}] from Section \ref{sec: Stabilizing Split Sequences}. We will start with defining a wider class of spaces called ``Tunnels" [Definition \ref{Tunnel Sets}] that both turn tunnels and extended turn tunnels fall under. Then in Sub-section \ref{Sec: Canonically Given Tunnel Parameterizations}, we will lay out a way of recognizing tunnels in $X$ called ``Canonically Given Tunnel Parameterizations". 

\vspace{4mm}

In Sub-section \ref{Sec: Properties of Tun}, we will show that, for each \hyperlink{turn in zeta}{turn $I$ in $\zeta$}, $Tun_I$ and $ExTun_I$ are each homeomorphic to a product space of the form $C \times I$ where $C$ is a totally disconnected space [Proposition \ref{ExTun are Tunnel Sets}, Remark \ref{An Extended Turn Tunnel with a Few Pre-Leaf Segments Removed}]. 

\vspace{4mm}

In Sub-section \ref{Sec: Tun in ExTun}, we will explore how turn tunnels are embedded in extended turn tunnels. 
More specifically, we will classify singular segments in to ``\hyperlink{limit vs isolated segments}{Limit Segments}" (a singular segment that entirely consists of limits points of a turn tunnels) and ``\hyperlink{limit vs isolated segments}{Isolated Segments}" (a singular segment that is not a limit segment). And we will show that, 
for a turn $I$ of a level $j$ standard star neighborhood $V$, 
such that $Tun_I \neq \emptyset$, $ExTun_I$ is the closure of $Tun_I$ in $\pi^{-1}_j (V)$ [Proposition \ref{ExTun is closed in O}, Corollary \ref{ExTun_is_Cl(Tun)}].

\vspace{4mm}

We will start with laying out some preliminary definitions.


\vspace{4mm}

\begin{definition}[Tunnel Parameterization]
    \label{Tunnel Parameterization}
    Given a proper split sequence $\zeta$, $j \in \n$, a \hyperlink{Standard Star Neighbourhoods}{standard star neighborhood} $\Tilde{V} \subset G_j$ centered at $p$, a turn $I$ of $\Tilde{V}$, a totally disconnected set $C$ and a map $h : C \times I \longrightarrow X(\zeta)$, we call $h$ a ``Tunnel Parameterization of $X(\zeta)$ rel $I$" if, $h$ is a homeomorphism on to a set $T \subseteq \pi^{-1}_j (I) \subset X(\zeta)$ such that,
    for each $q \in I$, $ (h (C \times \{q\})) = \pi_j^{-1} (q) \cap T  $.
\end{definition}

\vspace{4mm}

\hypertarget{cross-section}{For} each $q \in I$, we shall call $h (C \times \{q\})$ a ``Cross-Section of $T$".
Furthermore, we shall call $h (C \times \{p\})$ the ``Central Cross-Section of $T$ rel $\Tilde{V}$". 


\vspace{4mm}

\begin{definition}[Tunnels and Tunnel Neighborhoods]
    \label{Tunnel Sets}
    \hypertarget{Tunnel Sets}{Given} a proper solenoid $X(\zeta)$, a subset $T$ of it is called a ``Tunnel of $X(\zeta)$" if $T$ is the image of a tunnel parameterization. A tunnel of $X(\zeta)$ is called a ``Tunnel Neighborhood of $X(\zeta)$" if it's open in $X(\zeta)$. 
\end{definition}

\hypertarget{abstract tunnels}{To be} used in a more general context, we shall define any product space of the format $a \ totaly \ disconnected \ space \times an \ open \ interval$ an ``Abstract Tunnel".

\vspace{4mm}

In Remark \ref{Canonical Tunnel Parameterizations}, we will lay out a specific example of tunnel parameterizations that arise as a byproduct of $X's$ topology. 



\vspace{4mm}


\subsubsection{Canonically Given Tunnel Parameterizations}
\label{Sec: Canonically Given Tunnel Parameterizations}

\vspace{4mm}

This section is dedicated to establishing the concept of ``Canonically Given Tunnel Parameterizations" [Remark \ref{Canonical Tunnel Parameterizations}], which will help identify \hyperlink{Def: Tun and ExTun}{turn tunnels} and \hyperlink{Def: Tun and ExTun}{extended turn tunnels} of the solenoid. 



\vspace{4mm}

The first part of the discussion [Def. \ref{pre-turns}, Lemma \ref{pre-turn_shadows_are_star_set_unions}, Rem. \ref{Fold Compositions Restricted to Pre-Turns}] will establish terminology necessary to identifying pre-leaf segments that belong to the same extended turn tunnel.
Then we will introduce ``Parameterizations Synchronous to $\zeta$" for pre-leaf segments [Definition \ref{Synchronous Turn Parameterizations}] as a precursor to parameterizing extended turn tunnels. 

\vspace{4mm}

We then lay out the following two straight forward results.
\begin{itemize}
    \item Let $I$ be a standard turn in $\zeta$, and $L, L'$ two distinct pre-leaf segments that sustain $I$. Then, $L \cap L' = \emptyset$ [Lemma \ref{Intersecting Pre Leaf Segments}].
    \item For each pre-leaf segment $L$ in $X$, and for each fiber $F$ of $X$ that non-trivially intersects $L$, $F \cap L$ is a single point [Lemma \ref{Intersectig Fibers and Pre Leaf Segments}].
\end{itemize}
The aforementioned two lemmas make sure that ``Canonically Given Tunnel Parameterizations" [Remark \ref{Canonical Tunnel Parameterizations}] are bijective.


\vspace{4mm}

\begin{definition}[Pre-Turns and Pre-Turn Shadows]
    \label{pre-turns} 
    \hypertarget{pre-turns}{Let $I$} be a \hyperlink{turn in zeta}{level $K$ standard turn in $\zeta$} for some $K \in \n$. Then for $j \in - \mathds{N} + K$, a level $j$ turn $\Tilde{I}$ in $\zeta$, is called a ``Level $j$ Pre-Turn of $I$ in $\zeta$" if $f^j_K$ maps $\Tilde{I}$ homeomorphically on to $I$. We call $I$ the ``Trivial Pre-Turn of $I$".
    For each $j \in \n +K$, the ``Pre-Turn Shadow of $I$ in Level $j$" is the subset of \hyperlink{shadows}{$Shadow_j (I)$} given by $PTS_j (I) :=$ the union of all level $j$ pre-turns of $I$ in $\zeta$.
\end{definition}

\vspace{4mm}

\begin{lemma}
    \label{pre-turn_shadows_are_star_set_unions}
    \hypertarget{pre-turn_shadows_are_star_set_unions}{Let $I$ be a turn of a standard star neighborhood $\Tilde{V}_K$ in $G_K$ and let $j < K$. Then $PTS_j (I)$ is a disjoint union of finitely many standard star sets in $G_j$.}
\end{lemma}

\begin{proof}
    Let $I$ be a turn of a standard star neighborhood $\Tilde{V}_K$ in $G_K$ and let $j < K$. Recall from Remark \ref{Shadow Components} that, $Shadow_j (\Tilde{V}_K)$ is a disjoint union of finitely many standard star neighborhoods of $G_j$. And $PTS_j (I)$  is a union of open interval segments $\subseteq Shadow_j (\Tilde{V}_K)$.
\end{proof}

\vspace{4mm}


\begin{remark}[Fold Compositions Restricted to Pre-Turns]
    \label{Fold Compositions Restricted to Pre-Turns}
    Let $j,K \in \n$ such that $j < K$, let $I$ be a \hyperlink{turn in zeta}{level $K$ standard turn in $\zeta$}, and $\Tilde{I}$ a \hyperlink{pre-turns}{level $j$ pre-turn} of $I$. Note that $I$ and $\Tilde{I}$ each contains at most one special point. Therefore, $f^j_K |_{\Tilde{I}}$ is a homeomorphism onto $I$. 
\end{remark}


\vspace{4mm}

\begin{definition}[Parameterizations Synchronous to $\zeta$]
    \label{Synchronous Turn Parameterizations}
    \hypertarget{Synchronous Turn Parameterizations}{Let} $K \in \n$, and let $I$ be a \hyperlink{turn in zeta}{level $K$ standard turn in $\zeta$}. Then for an integer $j < K$, and for a pre-turn $\Tilde{I}$ of $I$, the homeomorphism $l^K_{j, \Tilde{I}} : I \longrightarrow \Tilde{I}$ defined by $l^K_{j, \Tilde{I}} := (f^j_K |_{\Tilde{I}})^{-1}$ is called the ``Parameterization of $\Tilde{I}$ Synchronous to $\zeta$ Based at $I$". (It follows from Remark \ref{Fold Compositions Restricted to Pre-Turns} that $l^K_{j, \Tilde{I}}$ is well defined.)

    \vspace{4mm}

    Now, let $L$ be a pre-leaf segment that sustains the turn $I$. Recall that there exists a $2-$pronged star chain $\mathcal{I}_L : I_{K} \xleftarrow[]{f_{K-1}} I_{K-1} \xleftarrow[]{f_{K-2}} I_{K-2} \xleftarrow[]{f_{K-3}}...$, \hyperlink{star chain representing a set in X}{representing} $L$, where,
    \begin{enumerate}
        \item $I_K = I$, and,
        \item $\varprojlim^{\zeta}_{j \in K \n} I_j = L$.
    \end{enumerate}
   Since for each $j \in \n +K$, parameterization of $I_j$ synchronous to $\zeta$ based at $I$, is uniquely determined, we have a unique parameterization $l^K_{L} : I \longrightarrow L$ given by $l^K_{L} (t) = (t , l^K_{K-1 , I_{K-1}} (t) , l^K_{K-2 , I_{K-2}} (t) , ... )$, 
   such that $\pi_K \circ l^K_{L} =$ the identity map on $I$. We shall call this parameterization the ``Parameterization of $L$ Synchronous to $\zeta$ Based at $I$".

   \vspace{4mm}

   The collection of parameterizations synchronous to $\zeta$ based at $I$ of,
   \begin{itemize}
       \item all pre-turns of $I$, and,
       \item all pre-leaf segments that sustains the turn $I$,
   \end{itemize}
   shall be called the ``System of $\zeta-$Synchronous Parameterizations Based at $I$".
\end{definition}

\vspace{4mm}

We will use the aforementioned $\zeta-$synchronous parameterizations to construct the tituler object in Remark \ref{Canonical Tunnel Parameterizations}.
But first we lay out the following two lemmas meant to ensure that, the canonically given tunnel parameterizations are bijective.

\vspace{4mm}

\begin{lemma}
    \label{Intersecting Pre Leaf Segments}
    Let $I$ be a \hyperlink{turn in zeta}{level $K$ standard turn in $\zeta$} for some $K \in \n$. Let $L , L'$ be two distinct pre-leaf segments that take $I$. Then $L \cap L = \emptyset$.
\end{lemma}

\begin{proof}
    Assume the set up of the lemma. Let $\Tilde{V}_K$ be the level $K$ standard star neighborhood that $I$ is a turn of, and let $O := \pi^{-1}_K (\Tilde{V}_K) \subset X$. Then $L, L'$ each is a turn of a plaque of $O$. (Recall from Lemma \ref{shapse of plaques 2} that each plaque of $O$ is a star set.) Assume that $L \cap L' \neq \emptyset$. Then $L$ and $L'$ must be turns of the same plaque $V$ of $O$. Since $L \neq L' $, $L \cup L'$ can only be either $3-$pronged or $4-$pronged star set that projects to $I$. This contradicts stabilization hypotheses 3. [Remark \ref{The Stabilizing Hypotheses}].
\end{proof}

\vspace{4mm}

\begin{lemma}
    \label{Intersectig Fibers and Pre Leaf Segments}
    Let $K \in \n$, let $I$ be a level $K$ standard turn in $\zeta$, and let $L$ be a pre-leaf segment that sustains $I$. Then for each $q \in I$, $F_q$ intersects $L$ at exactly one point.
\end{lemma}

\begin{proof}
    Assume the set-up of the lemma.
    It follows from Remark \ref{Fold Compositions Restricted to Pre-Turns} that the Parameterization $l^K_L : I \longrightarrow L$ constructed in Definition \ref{Synchronous Turn Parameterizations} is a well defined bijection. So, there can be one point in $L$ that can project to $p$.
\end{proof}

\vspace{4mm}

\begin{proposition}
    \label{Point in Cross Section uniqely given by pre leaf segment}
    Let $K \in \n$, let $I$ be a level $K$ standard turn in $\zeta$. Then for each $p \in I$, and for each pre-leaf segment $L$ that sustains $I$, there is a point $y^L (p) \in ExTun_I$ uniquely determined by $p$ and $L$.
\end{proposition}

\begin{proof}
    Follows from lemmas \ref{Intersecting Pre Leaf Segments} and \ref{Intersectig Fibers and Pre Leaf Segments}.
\end{proof}

\vspace{4mm}

\begin{remark}[The Canonically Given Tunnel Parameterizations]
    \label{Canonical Tunnel Parameterizations}
    \hypertarget{Canonical Tunnel Parameterizations}{Let} $I$ be a turn of a \hyperlink{Standard Star Neighbourhoods}{standard star neighborhood} $V$ centered at $p$ in some level graph of $\zeta$. Let $T_I \in \{ Tun_I , ExTun_I \}$. Furthermore, let $C_I := F_p \cap T_I$. Then for each \hyperlink{leaf segments}{pre-leaf segment} $L$ that takes $I$, there is a unique point $y \in F_p$ such that $y \in L$ [Follows from Prop. \ref{Point in Cross Section uniqely given by pre leaf segment}]. And for each $y \in T_I$, there exists a unique pre-leaf segment $L^y \subseteq T_I$ that contains $y$.

    \vspace{4mm}

    For each $y \in C_I$, 
    we will use the notation ``$l^I_y$" to refer to the unique \hyperlink{Synchronous Turn Parameterizations}{parameterization of $L$ synchronous to $\zeta$ based at $I$} (i.e. $l^K_{L} : I \longrightarrow L$ as laid out in Remark \ref{Synchronous Turn Parameterizations}) where $L$ is the unique pre-leaf segment $L$ in $ExTun_I$ that contains $y$. 
    Then the ``Canonically Given Tunnel Parameterization of $T_I$ rel $C_I$" is the map $h_I : C_I \times I \longrightarrow T_I$ defined by the criterion: for each $(y,t) \in C_I \times I$, $h_I (y,t) := l^I_y (t)$.

    Since,
    \begin{enumerate}
        \item for each $c \in C_I$, $h$ restricted to $\{ c \} \times I$ is a homeomorphism,
        \item the images of all such restrictions cover $ExTun_I$, and,
        \item no pair of distinct pre-leaf segments in $ExTun_I$ intersect each other [Lemma \ref{Intersecting Pre Leaf Segments}],
    \end{enumerate}
    $h_I$ is a bijection. We will show in Proposition \ref{ExTun are Tunnel Sets}, Remark \ref{An Extended Turn Tunnel with a Few Pre-Leaf Segments Removed} respectively, that $h$ is a homeomorphism when $T_I = ExTun_I , Tun_I$ resp. 
\end{remark}

\vspace{4mm}

\hypertarget{Central Cross-Section}{In} the above context, we shall call $C_I$ the ``Central Cross-Section of $T_I$ rel $p$". Note that when $V$ is not $2-$pronged, the notion of the ``Center" is canonically determined, and in that case, we shall unambiguously call $C_I$ the ``Central Cross-Section of $T_I$". 


\vspace{4mm}

Finally, we define a notion of sub tunnels that will prove to be useful later in the text.

\vspace{4mm}

\begin{definition}[Sub Tunnels]
    \label{Sub Tunnels}
    \hypertarget{Sub Tunnels}{Let} $I$ be a \hyperlink{turn in zeta}{level $j$ standard turn in $\zeta$} for some $j \in \n$ and let $T_I \in \{ Tun_I , ExTun_I \}$. Let $h_I : C_I \times I \longrightarrow T_I$ be the canonically given tunnel parameterization for $T_I$ as defined above. Let $T'_I \subset T_I$. 
    \begin{itemize}
        \item If $T'_I = h_I (C'_I \times I)$ for some $C'_I \subset C_I$, we call $T'_I$ the ``Sub-Fiber Tunnel of $T_I$ rel $C'_I$".
        \item If $T'_I = h_I (C_I \times I')$ where $I'$ is a proper open sub-interval of $I$, we call $T'_I$ the ``Sub-Turn Tunnel of $T_I$ rel $I'$".
        \item Generally, we shall call $T'_I$ a ``Sub Tunnel of $T_I$" if $T'_I = h_I (C'_{I'} \times I')$ where $C'_I \subset C_I$ and $I'$ is a proper open sub-interval of $I$.
    \end{itemize}
\end{definition}

\vspace{4mm}

A more special type of sub tunnels (called base tunnels) will be introduced in Definition \ref{Base Tunnels}, that will be utilized in observing the topology of an extended turn tunnel.


\vspace{4mm}


\subsubsection{Properties of Tunnels}
\label{Sec: Properties of Tun}



\vspace{4mm}

In this sub-section we shall show that each turn tunnel of $X$ is open in $X$ [Lemma \ref{turn tunnels are open}] and each extended turn tunnel of $X$ is a tunnel of $X$ [Proposition \ref{ExTun are Tunnel Sets}] which will directly follow by showing that each path component of an extended turn tunnel is a pre-leaf segment [Remark \ref{pre-leaf Segments and Path components}] and each turn tunnel of $X$ is a tunnel of $X$ [Remark \ref{An Extended Turn Tunnel with a Few Pre-Leaf Segments Removed}].

\vspace{4mm}

\begin{lemma}
    \label{turn tunnels are open}
    \hypertarget{tun_is_open}{Each turn tunnel of $X(\zeta)$ is open in $X(\zeta)$.}
\end{lemma}

\begin{proof}
    Let $K \in \n$, and
    let $\Tilde{V}_K$ be a standard star neighborhood in $G_K$ centered at a point $p$. Furthermore, let $I$ be a turn of $\Tilde{V}_K$ that is taken by at least one leaf segment of $X$. Let $x \in Tun_I$. We seek to show that $x$ has an open neighborhood contained in $Tun_I$.
    
    \vspace{4mm}

    If $\Tilde{V}_K$ is a $2-$pronged star neighborhood, then $I = \Tilde{V}_K$, and thus $\pi^{-1}_K (\Tilde{V}_K) = ExTun_I = Tun_I$ (since $ExTun_I$ having a singular segment would imply that there is a $3-$pronged star set in $X$ that projects to $I$, which violates $\zeta$'s stability).
    Now suppose that  $\Tilde{V}_K$ is not $2-$pronged. 

    \vspace{4mm}
    
    Note that $x$ is part of a leaf segment $L$ that projects homeomorphically to $I$ via $\pi_K$, and therefore $x$ has a $2-$pronged maximal star chain $\mathcal{I}^x : I_{K} \xleftarrow[]{f_{K-1}} I_{K-1} \xleftarrow[]{f_{K-2}} I_{K-2} \xleftarrow[]{f_{K-3}}...$ such that $I_K = I$. Then either (Case a.) $L$ does not contain any \hyperlink{Pre-Singularities}{pseudo singularities}, or (Case b.) $L$ contains a pseudo singularity.

    \vspace{4mm}

    (Case a.) Suppose $L$ does not contain any pseudo singularities. Then there exists an integer $j < K$ such that $I_j$ does not contain any natural vertices. i.e. $I_j$ is a $2-$pronged star neighborhood in $G_j$. Thus we have $x \in \pi_j^{-1} (I_j) \subseteq Tun_I$. 

    \vspace{4mm}

    (Case b.) Suppose $L$ contains a pseudo singularity. Then for each integer $j \leq K$, $I_j$ contains a natural vertex of $G_j$.
    For each integer $j < K$, let the connected component of $(f^j_K)^{-1} (\Tilde{V}_K)$ containing $I_j$ be denoted by $\Tilde{I}_j$. (Note that in this context, $\Tilde{I}_j$ is a \hyperlink{neighbourhood_completion}{standard neighborhood completion} of $I_j$ centered at $x_j$.) We seek to find an integer $J<K$ so that $\Tilde{I}_J$ does not intersect $(f^J_K)^{-1} (\Tilde{V}_K - I)$.

    \vspace{4mm}
    
    Let $B := \partial \Tilde{V}_K - \partial I$. Now, let $q \in B$. $(p,q)$ represents an open prong of $\Tilde{V}_K$ that does not intersect $I$.  Since 
    $L$ is a \hyperlink{Singular Plaques}{non-singular plaque} in $X$, and since $\mathcal{I}^x$ is a $2-$pronged maximal star chain \hyperlink{star chain representing a set in X}{representing $L$},
    there must be an integer $J_q < K$ such that 
    $\Tilde{I}_{J_q}$ does not intersect $(f^{J_q}_K )^{-1} ((p,q))$.
    Furthermore, for each $j<J_q$, $\Tilde{I}_{j}$ does not intersect $(f^{j}_K )^{-1} ((p,q))$. Since $B$ is finite, we can define $J := min \{ J_q : q \in B \}$. 
    
    \vspace{4mm}

    Now, we have the standard star neighborhood $\Tilde{I}_{J}$ such that $f^{J}_K (\Tilde{I}_{J}) = I$. Therefore $x \in \pi^{-1}_{J} (\Tilde{I}_{J}) \subseteq ExTun_I$. 
    If there is a singular segment in $ExTun_I$ that  projects to $\Tilde{I}_{J}$, then there are two pre-singularities projecting to the unique special point in $\Tilde{I}_{J}$, which violates $\zeta$'s stability.
    Thus, we have $x \in \pi^{-1}_{J} (\Tilde{I}_{J}) \subseteq Tun_I$.

\end{proof}

\vspace{4mm}


The next three results should be perceived as one continuous discussion intended to show that \hyperlink{Extended Turn Tunnels}{Extended Turn Tunnels} are \hyperlink{Tunnel Sets}{Tunnel}, and that canonically given tunnel parameterizations are homeomorphisms.

\vspace{4mm}

Given a \hyperlink{turn in zeta}{standard turn $I$ in $\zeta$}, consider $ExTun_I$. 

\begin{itemize}
    \item After defining ``Base Tunnels of $ExTun_I$" in Definition \ref{Base Tunnels}, in Lemma \ref{Base Tunnel Open}, we show that each base tunnel of $ExTun_I$ is open in $ExTun_I$.
    \item In Lemma \ref{Base Tunnel Basis for ExTun}, we show that each open subset of $ExTun_I$ is a union of base tunnels of $ExTun_I$, and therefore the collection of base tunnels of $ExTun_I$ generates the topology of $ExTun_I$.
    \item In Proposition \ref{ExTun are Tunnel Sets}, we use the aforementioned basis to show that the \hyperlink{Canonical Tunnel Parameterizations}{canonically given tunnel parameterization} $h_I : C_I \times I \longrightarrow ExTun_I$ for $ExTun_I$ is a homeomorphism, and thus $ExTun_I$ is a \hyperlink{Tunnel Sets}{tunnel} of $X$.
\end{itemize}

\vspace{4mm}


Given the goal of showing $h_I$ is a homeomorphism, the collection of \hyperlink{Base Tunnels}{base tunnels} of $ExTun_I$ is meant to be the collection of images (under $h_I$) of product sets in $C_I \times I$ of the format $\Tilde{C} \times \Tilde{I}$, where $\Tilde{I}$ is a sub-interval of $I$, and $\Tilde{C} = C_I \bigcap$ (a \hyperlink{fibers}{fiber} of $X$ that intersects $C_I$ non-trivially).
\vspace{4mm}




\begin{definition}[Base Tunnels]
\label{Base Tunnels}
    \hypertarget{Base Tunnels}{Let} $I$ be a \hyperlink{turn in zeta}{standard turn in $\zeta$} and consider $ ExTun_I$. For each sub-interval $\Tilde{I}$ of $I$, $j \in K  \n$, and connected component $U$ of \hyperlink{pre-turns}{$PTS_j (I)$}, the ``Base Tunnel of $ExTun_I$ rel $(j ,  \Tilde{I}, U)$" is given by $BT_I (j , \Tilde{I}, U) := \bigcup_{J \in \{Pre-Turns \ of \ \Tilde{I} \ contained \ in \ U\}} ExTun_J$. 
\end{definition}

\vspace{4mm}

\begin{figure}
    \centering
    \begin{tikzpicture}
    \tikzmath{\x1 = 0 ; \y1 = 0 ; \ya = -3 ;
    \x2 = 4 ; \y2 = 0 ; \yb = -3 ;
    \v = 1 ; }

    \draw (\x1 -1 , \y1) -- (\x1 +1 , \y1);
    \draw[dashed] (\x1 , \y1) -- (\x1 +1 , \y1 + 0.5);
    \draw[dashed] (\x1 , \y1) -- (\x1 +1 , \y1 - 0.5);
    \draw[line width=3pt, opacity=0.5][cyan] (\x1 - 0.7 , \y1 ) -- (\x1 - 0.3 , \y1);
    \draw[line width=3pt, opacity=0.5][cyan] (\x1 - 0.7 , \ya ) -- (\x1 - 0.3 , \ya);
    \draw[line width=3pt, opacity=0.5][cyan] (\x1 - 0.7 , \ya + 0.35) -- (\x1 - 0.3 , \ya + 0.15);

    \draw (\x1 -1 , \ya) -- (\x1 +1 , \ya);
    \draw[dashed] (\x1 , \ya) -- (\x1 +1 , \ya + 0.5);
    \draw (\x1 -1 , \ya + 0.5) -- (\x1 , \ya);
    
    \draw (\x1 -1 , \ya - 0.5) -- (\x1 +1 , \ya - 0.5);
    \draw[dashed] (\x1 , \ya - 0.5) -- (\x1 +1 , \ya - 1);

    \draw (\x2 -1 , \y2) -- (\x2 +1 , \y2);
    \draw[dashed] (\x2 , \y2) -- (\x2 +1 , \y2 + 0.5);
    \draw[dashed] (\x2 , \y2) -- (\x2 +1 , \y2 - 0.5);

    \draw (\x2 -1 , \yb) -- (\x2 +1 , \yb);
    \draw[dashed] (\x2 , \yb) -- (\x2 +1 , \yb + 0.5);
    \draw (\x2 -1 , \yb + 0.5) -- (\x2 , \yb);
    
    \draw (\x2 -1 , \yb - 0.5) -- (\x2 +1 , \yb - 0.5);
    \draw[dashed] (\x2 , \yb - 0.5) -- (\x2 +1 , \yb - 1);

    \draw[line width=3pt, opacity=0.5][cyan] (\x2 - 0.4 , \y2 ) -- (\x2 + 0.4 , \y2);
    \draw[line width=3pt, opacity=0.5][cyan] (\x2 - 0.4 , \yb ) -- (\x2 + 0.4 , \yb);
    \draw[line width=3pt, opacity=0.5][cyan] (\x2 - 0.4 , \yb + 0.2) -- (\x2 , \yb );

    \node at (\x1 -2, \y1 ) {Level $K$ :};
    \node at (\x1 -2, \ya ) {Level $j$ :};
    \node at (\x1 + 2 , \y1) {$\Tilde{I} \subseteq \Tilde{V}_K$};
    \node at (\x1 + 2, \ya + 0.2) {$\Tilde{I}_j \subseteq U$};
    \node at (\x1 , \ya - 2) {Case a.};
    \node at (\x2 , \yb - 2) {Case b.};
\end{tikzpicture}
    \caption{Projections of Base Tunnels}
    \label{fig:Projections of Base Tunnels}
\end{figure}
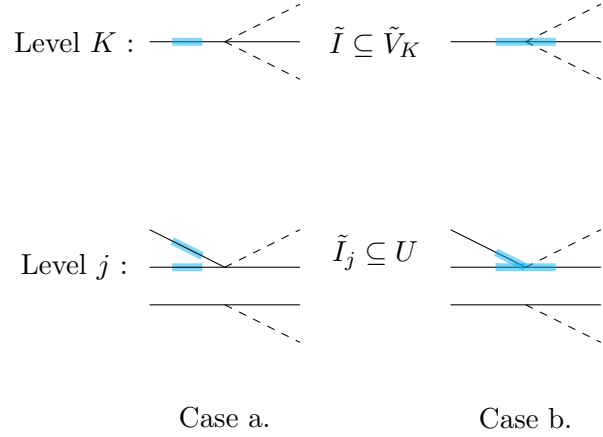

\vspace{4mm}

\begin{lemma}
    \label{Base Tunnel Open}
    Given a \hyperlink{turn in zeta}{standard turn $I$ in $\zeta$}, each \hyperlink{Base Tunnels}{base tunnel} of $ExTun_I$ is open in $ExTun_I$.
\end{lemma}

\begin{proof}
    Let $K \in \n$, and let $I$ be a level $K$ standard turn.
    Furthermore, let $\Tilde{I}$ be a sub-interval of $I$, $\ j \in \n +K$, and $U$ a connected component of $PTS_j (I)$. Let $T :=$ \hyperlink{Sub-Tunnels}{$BT_I (j ,  \Tilde{I}, U)$} $ = \bigcup_{J \in \{Pre-Turns \ of \ \Tilde{I} \ contained \ in \ U\}} ExTun_J$. Let $x \in T$. We must show that there exists an open set $O$ in $X$ such that $x \in O \cap ExTun_I \subseteq T$. Note that $\hyperlink{Projections of a Point}{x_j} \in U$. Let $\Tilde{I}_j$ denote the union of all pre-turns of $\Tilde{I}$ in $U$ (See Figure \ref{fig:Projections of Base Tunnels}). 

    \vspace{4mm}
    
    (Case a.) Suppose that $x_j$ is contained in the interior of a natural edge $E$ of $G_j$.
    Then there exists an open interval $I'$ embedded in $Interior(E) \subset G_j$, such that $x_j \in I' \subseteq \Tilde{I}_j$. Thus we have $x \in Tun_{I'} \subseteq BT_I (j ,  \Tilde{I}, U)$. (Recall from Lemma \ref{turn tunnels are open} that $Tun_{I'}$ is open.)

    \vspace{4mm}

    (Case b.) Suppose $x_j$ is a natural vertex of $G_j$.
    Let $\Tilde{U}$ be the unique \hyperlink{neighbourhood_completion}{standard neighborhood completion} of $U$ in $G_j$ that is the unique connected component of $(f^j_K)^{-1} (\Tilde{V}_K)$ containing $U$. (i.e. $\Tilde{U} := \hyperlink{Shadow Component}{SC_j (\Tilde{V}_K , U)}$.) Then $x_j$ must be the center of $\Tilde{U}$.
    Recall that $\Tilde{U}$ is a standard star neighborhood rel $\zeta$ [Remark \ref{Shadow Components}].
    Then let $\Tilde{V} := (\Tilde{U} - U) \cup \Tilde{I}_j$. Note that $x_j \in \Tilde{I}_j \subseteq \Tilde{V}$. Since $\Tilde{V}$ is a subset of $\Tilde{U}$ that is (as a star set) centered at the same point as $\Tilde{U}$, and has the same number of prongs as $\Tilde{U}$, $\Tilde{V}$ is also a star neighbourhood of $x_j$ in $G_j$.
    Thus we have, $x \in \pi^{-1}_j (\Tilde{V}) \cap ExTun_I = \pi^{-1}_j (\Tilde{I}_j) \cap ExTun_I \subseteq T$.
\end{proof}

\vspace{4mm}

To assist with the proof of the next lemma, we introduce the following terminology.

\vspace{4mm}

\begin{notation}[Pre-Turn Components]
    \label{Pre-Turn Components}
    \hypertarget{Pre-Turn Components}{Let} $K \in \n$, and let $I$ be a level $K$ standard turn in $\zeta$. Let $j \in \n +K$ and let $q \in \hyperlink{pre-turns}{PTS_j (I)}$. Then 
    the ``Level $j$ Pre-Turn Component of $I$ Containing $q$" denoted
    ``$PTC_j (I ,q)$" is the unique connected component of $PTS_j (I)$ in $G_j$ that contains $q$.
\end{notation}

\vspace{4mm}


\begin{lemma}
    \label{Base Tunnel Basis for ExTun}
    Given a \hyperlink{turn in zeta}{standard turn $I$ in $\zeta$}, the collection of \hyperlink{Base Tunnels}{base tunnels} of $ExTun_I$ generates the topology of $ExTun_I$.
\end{lemma}

\begin{proof}
    Let $K \in \n$, $\Tilde{V}_K$ a standard star neighborhood in $G_K$, and 
    $I$ a turn of $\Tilde{V}_K$.
    Since we've already shown that each base tunnel of $ExTun_I$ is open, all that's left to show is the claim: each open set of $ExTun_I$ is a union of base tunnels of $ExTun_I$.
    Note that the topology of $X(\zeta)$ can be generated by the collection $\{ \pi^{-1}_j (\Tilde{V}_j) \subset X \ : \  j \in \n +K , V_j$ is a standard star neighborhood of $G_j \}$. 

    \vspace{4mm}
    
    Now let $j \in \n +K $. Furthermore, let $\Tilde{V}_j$ be a standard star neighborhood of $G_j$ where $O := \pi^{-1}_j (\Tilde{V}_j)$ and $O \cap ExTun_I \neq \emptyset$. Let $x \in O \cap ExTun_I$. We must show that there exists a base tunnel $T$ of $ExTun_I$ such that $x \in T \subseteq O \cap ExTun_I$.

    \vspace{4mm}

    Recall that a base tunnel depends on a level, a sub-interval of $I$, and a connected component of $PTS_j (I)$. We will find each of the aforementioned three in steps. Let $U :=$ the connected component of \hyperlink{pre-turns}{$PTS_j (I)$} that contains $x_j$ (i.e. $U := \hyperlink{Pre-Turn Components}{PTC_j (I, x_j)}$). Note that $U = \bigcup_{\alpha = 1}^n I^{\alpha}$ for some $n \in \mathds{N}$ where, for each $\alpha= 1,...,n$, $I^{\alpha}$ is a pre-turn of $I$ at level $j$.

    \vspace{4mm}

    Let $L$ be the unique pre-leaf segment in $ExTun_I$ that contains $x$. And let $\beta$ be the unique index in $\{1,...,n \}$ such that $\pi_j (L) = I^{\beta}$. Then, from the stability of $\zeta$ (more specifically, since no $3-$pronged star set of $X$ projects to $I$), there exists an integer $J \leq j$ such that $PTC_J (I , x_J)$ does not intersect $PTS_J (I^{\alpha})$ for any $\alpha \in \{1,...,n\} - \{\beta\}$.

    \vspace{4mm}

    Note that, since $x_j \in I^{\beta} \bigcap \Tilde{V}_j$, there exists an embedded open interval $I_j$ in $G_j$ such that $x_j \in I_j \subseteq I^{\beta} \bigcap \Tilde{V}_j$. Now let $\Tilde{I} := f^j_K (I_j) \subseteq I \subset G_K$. Furthermore, let $U_J := PTC_J (I , x_J)$.
    Consider $T := BT_I (J , \Tilde{I}, U_J)$. 

    \vspace{4mm}
    
    For each pre-turn $I_J$ of $\Tilde{I}$ contained in $U_J$, $f^J_j (I_J) = I_j \subset \Tilde{V}_j$.
    Therefore, $x_J \in $ (the union of level $J$ pre-turns of $\Tilde{I}$ contained in $U_J ) \subseteq PTS_J (I) \bigcap ((f^J_j)^{-1} (\Tilde{V}_j))$. Thus, $x \in \bigcup_{I' \ is \ a \ pre-turn \ of \ \Tilde{I} \ contained \ in \ U_J} ExTun_{I'} \subseteq ExTun_I \bigcap \pi^{-1} (\Tilde{V}_j)$.
\end{proof}

\vspace{4mm}

\vspace{4mm}

\begin{proposition}
    \label{ExTun are Tunnel Sets}
    \hypertarget{stable_turns_induce_tunnel_Sets}{Let} $I$ be a \hyperlink{turn in zeta}{standard turn in $\zeta$}. $ExTun_I$ is a tunnel set.
\end{proposition}

\begin{proof}
    Let $\Tilde{V}_K$ be a standard star neighbourhood centered at $p \in G_K$ and let $I$ be a turn of $\Tilde{V}_K$. Let $C_I := ExTun_I \bigcap \hyperlink{fibers}{F_p}$. Recall from Remark \ref{Canonical Tunnel Parameterizations} that \hyperlink{Canonical Tunnel Parameterizations}{the canonically given tunnel parameterization} of $ExTun_I$, $h_I : C_I \times I \longrightarrow ExTun_I$ is defined by $h_I ((y,t)) := l^I_y (t)$ for each $(y,t) \in C_I \times I$, and is a bijection. For convenience of notation, let $h := h_I$ and for each $y \in C_I$, let $l_y := l^I_y$.

    \vspace{4mm}
    
    To show that $h$ is a homeomorphism, we will use the basis for the topology of $ExTun_I$ that we established in lemma \ref{Base Tunnel Basis for ExTun}.

    \vspace{4mm}
    
    Note that the topology of $C_I \times I$ is generated by the collection $\{ \Tilde{C} \times \Tilde{I} \subseteq C_I \times I : \Tilde{I} $ is an open sub-interval of $ I , \Tilde{C} = \pi^{-1}_j (q) \bigcap C_I $ for some integer $ j < K $ and $ q \in (f^j_K)^{-1} (p)\}$. 

    \vspace{4mm}
    
    Now let $\Tilde{I}$ be an open sub-interval of $I$, $j$ be an integer less than $K$ and $q \in (f^j_K)^{-1} (p)$. Furthermore let $\Tilde{C} := \pi^{-1}_j (q) \bigcap C_I$. 
    Let $U := PTC_j (I , q)$.
    Note that $h (\Tilde{C} \times \Tilde{I}) = \{ l_y (t) : y \in \Tilde{C} , t \in \Tilde{I} \} = \bigcup_{J \in \{pre-turns \ of \ \Tilde{I} \  contained \ in \ U \}} ExTun_J$. 
 \end{proof}

 \vspace{4mm}

 The above result leads to some usefull observations.

 \vspace{4mm}

 \begin{remark}[Pre-Leaf Segments are Path Components of Tunnels]
     \label{pre-leaf Segments and Path components}
    Let $I$ be a \hyperlink{turn in zeta}{standard turn in $\zeta$}, and consider $ExTun_I$. 
    Since according to the above proposition $h_I : C_I \times I \longrightarrow ExTun_I$ (where $C_I$ is totally disconnected), is a homeomorphism, it follows that,
    for each $c \in C_I$, $h_I$ takes the path component $\{c\} \times I$ of $C_I \times I$ to a unique pre-leaf segment in $ExTun_I$. Since homeomorphisms preserve path components, each pre-leaf segment in $ExTun_I$ is a path component of $ExTun_I$. 
 \end{remark}

  \vspace{4mm}

  \begin{remark}[An Extended Turn Tunnel with a Few Pre-Leaf Segments Removed]
  \label{An Extended Turn Tunnel with a Few Pre-Leaf Segments Removed}
      It follows from the above remark that, an extended turn tunnel with a collection of it's pre-leaf segments removed, is still a \hyperlink{Tunnel Sets}{tunnel of $X$}. Most significantly, this implies that for a \hyperlink{turn in zeta}{standard turn $I$ in $\zeta$}, $Tun_I$ is a tunnel of $X$. 
  \end{remark}

\begin{corollary}
    Each turn tunnel of $X$ is a tunnel neighbourhood of $X$.
\end{corollary}

\begin{proof}
    Follows from Proposition \ref{ExTun are Tunnel Sets}, Lemma \ref{turn tunnels are open} and Remark \ref{An Extended Turn Tunnel with a Few Pre-Leaf Segments Removed}.
\end{proof}
 

\vspace{4mm}




\subsubsection{Turn Tunnels Inside Extended Turn Tunnels}
\label{Sec: Tun in ExTun}

\vspace{4mm}

We will dedicate this sub-section to explore the relationship between turn tunnels and extended turn tunnels. In Definition \ref{limit vs isolated segments}, we will categorize singular segments of $X$ into ``Limit Segments" and ``Isolated Segments". We will show in Lemma \ref{limit segments consist of limit points} that, for a given \hyperlink{turn in zeta}{standard turn $I$ in $\zeta$} that is taken by $X$, and a limit segment $L'$ that sustains $I$, $L'$ entirely consists of limit points of $Tun_I$.

\vspace{4mm}

Given $K \in \n$, a standard star neighborhood $V$ in $G_K$, a turn $I$ of $V$ and $O := \pi^{-1}_K (V)$, Proposition \ref{ExTun is closed in O} will show that $ExTun_I$ is closed in $O$. Then we will show in Corollary \ref{ExTun_is_Cl(Tun)} that, when $Tun_I \neq \emptyset$, $ExTun_I$ is the closure of $Tun_I$ in $O$.



\vspace{4mm}

Recall that $\zeta$ is a \hyperlink{Stabilized Split Sequence}{stabilized} proper split sequence and $X = X(\zeta)$ its induced solenoid. 

\vspace{4mm}

\begin{remark}[Each Extended Turn Tunnel Contains at Most One Singular Segment]
    \label{ExTun_I - Tun_I is a singular shadpw segment}
    Let $V$ be a \hyperlink{Standard Star Neighbourhoods}{standard star neighborhood of $\zeta$} and $I$ a turn of $V$. Recall that it follows from the \hyperlink{Stabilizing Hypotheses}{stabilization hypotheses} [Remark \ref{The Stabilizing Hypotheses}], that the \hyperlink{Standard Star Neighbourhoods}{standard neighborhood} in $X$ that projects to $V$, contains at most one singularity. Thus $SingSeg_I = ExTun_I - Tun_I$ is either empty or equal to a single singular segment that sustains $I$. 
\end{remark}

\vspace{4mm}

The following definition refines the notion of a singular segment in $X$.


\vspace{4mm}

\begin{definition}[Limit Segments and Isolated Segments]
    \label{limit vs isolated segments}
    \hypertarget{limit vs isolated segments}{Let} $K \in \n$ and let $I$ be a \hyperlink{turn in zeta}{level $K$ standard turn in $\zeta$}. Suppose that there exists a singular segment $L$ that takes $I$. 
    Since $\zeta$ satisfies the \hyperlink{4th hype}{4th stabilization hypothesis}, exactly one of the following is true: Either,
        \begin{enumerate}
            \item for each $j \in \n$, there exists a leaf segment that takes $\pi_j (L)$, or,
            \item no leaf segment takes the turn $\pi_0 (L) \subset G_0$ (thus no leaf segment takes the turn $\pi_j (L)$ for any $j \in \n$).
        \end{enumerate}
    We shall call $L$ a ``Limit Segment in $X$ that Sustains $I$" if case 1. is true. And if case 2. is true, we shall call $L$ an ``Isolated Segment in $X$ that Sustains $I$". 
\end{definition}

\vspace{4mm}

Given that the terms $limit$ and $isolated$ come with preconceived notions and implications, without delay, we shall show with the next lemma that a limit segment that sustains a turn $I$ does in fact consist of limit points of $Tun_I$ [Lemma \ref{limit segments consist of limit points}]. Note that, if $L$ is an isolated segment that sustains a turn $I$, 
then by definition, $L$ is the only pre-leaf segment in $X$ that sustains $I$ (i.e. $ExTun_I = L$ and $Tun_I = \emptyset$). 


\vspace{4mm}

\begin{lemma}
    \label{limit segments consist of limit points}
    Let $K \in \n$, $I$ a \hyperlink{turn in zeta}{level $K$ standard turn in $\zeta$} and $L'$ a \hyperlink{limit vs isolated segments}{limit segment in $X$ that sustains $I$}. Then $L'$ is entirely made up of limit points of $Tun_I$.
\end{lemma}

 \begin{proof}

    
    
    Assume the set-up of the lemma. For each $j \in \n + K$, let $I_j := \pi_j (L')$. Note that $L'$ is the inverse limit of $I_K \xleftarrow[]{f_{K-1}} I_{K-1} \xleftarrow[]{f_{K-2}} I_{K-2} \xleftarrow[]{f_{K-3}}...$ in $X(\zeta)$.
    

    \vspace{4mm}
    
    Let $x \in L'$. We shall show that $x$ is a limit point of $Tun_I$ in $X(\zeta)$.
    Note that the topology of $X(\zeta)$ can be generated by the collection $\mathcal{B} := \{ \pi^{-1}_j (V) : j \in \n +K , V$ is a \hyperlink{Standard Star Neighbourhoods}{standard star neighborhood} in $G_j \}$. 
    Now let $J \in \n +K$, $\Tilde{U}_J$ be a standard star neighborhood in $G_J$ and let $O := \pi^{-1}_J (\Tilde{U}_J)$.
    We aim to show that $(O - \{x\}) \cap Tun_I \neq \emptyset$. 

    \vspace{4mm}
    
    Consider $I_J$. Note that $x_J \in I_J \cap \Tilde{U}_J$. Since $I_J$ is an embedded open interval of $G_J$ and $\Tilde{U}_J$ a star neighborhood, there must exist another (perhaps smaller) embedded open interval $\Tilde{I}$ in $G_J$ such that $x_J \in \Tilde{I} \subseteq I_J \cap \Tilde{U}_J$.

    \vspace{4mm}

    Since $L'$ is a limit segment, by definition, for each $j \in \n + K$ there exists a leaf segment that takes $I_j$. Let $L$ be a leaf segment that takes $I_J$. Then since $\Tilde{I} \subseteq I_J$, there exists a leaf segment $\Tilde{L} \subseteq L$ that takes $\Tilde{I}$, and satisfies $\Tilde{L} \subseteq \pi^{-1}_J (\Tilde{U}_J) = O$. Therefore $\Tilde{L} \subseteq Tun_{I_J} \cap \pi^{-1} (\Tilde{U}_J) \subseteq Tun_I \cap O$. It follows from $L \cap L' = \emptyset$ [Lemma \ref{Intersecting Pre Leaf Segments}], and $x \in L'$, that $x \notin \Tilde{L} \subseteq L$. Thus we have
    $\emptyset \neq \Tilde{L} \subseteq (O - \{x\}) \cap Tun_I$. 
\end{proof}

\vspace{4mm}

Consider $K \in \n$, a standard star neighborhood $V$ in $G_K$, a turn $I$ of $V$ and $O := \pi^{-1}_K (V)$.
The next two results will show that, 
\begin{itemize}
    \item $ExTun_I$ is closed in $O$ [Proposition \ref{ExTun is closed in O}], and therefore,
    \item when $Tun_I \neq \emptyset$, $ExTun_I$ is the closure of $Tun_I$ in $O$ [Corollary \ref{ExTun_is_Cl(Tun)}].
\end{itemize}

\vspace{4mm}

\begin{proposition}
    \label{ExTun is closed in O}
    Let $K \in \n$, and $I$ a turn of a standard star neighborhood $\Tilde{V}_K$ in $G_K$. Then $ExTun_I$ is closed in $O := \pi^{-1}_K (\Tilde{V}_K)$.
\end{proposition}

\begin{proof}
    Assume the set-up of the lemma. Suppose $O$ does not contain any singularities. Then clearly $ExTun_I = Tun_I$. Furthermore, $O$ is the union of finitely many turn tunnels $\{ Tun_{I'} : I'$ is a turn of $\Tilde{V}_K \}$ each of which is open in $O$ [Lemma \ref{turn tunnels are open}] and has no intersection with each other. Thus $O - Tun_I =$ a union of turn tunnels, and is therefore open, implying $ExTun_I = Tun_I$ is closed in O. 

    \vspace{4mm}
    
    Now assume $O$ contains a singularity $x$. Let $\mathcal{V} := V_{K} \xleftarrow[]{f_{K-1}} V_{K-1} \xleftarrow[]{f_{K-2}} V_{K-2} \xleftarrow[]{f_{K-3}}...$ be the unique maximal star chain for $x$ such that $\Tilde{V}_K$ is a \hyperlink{neighbourhood_completion}{standard neighborhood completion} of $V_K$.
    Either (Case i.) $I \nsubseteq V_K$ or (Case ii.) $I \subseteq V_K$.

    \vspace{4mm}
    
    (Case i.): Suppose $I \nsubseteq V_K$. Then since the unique singularity in $O$ is not contained in $ExTun_I$, we have $ExTun_I = Tun_I$. Once again, we seek to show that $O - Tun_I$ is open. We will find for each $y \in O - Tun_I$, an open set $U$ in $O - Tun_I$ such that $y \in U \subseteq O - Tun_I$. Now let $y \in O - Tun_I$. Note that $y$ either belongs to an interval-shaped plaque of $O$ or it belongs to the unique singular plaque of $O$ [Lemma \ref{shapse of plaques 2}]. We will handle the two sub-cases separately.
    
    \vspace{4mm}
    
    (Case i-a): Suppose $Plaque(O,y)$ is homeomorphic to an interval. Then since $y \notin Tun_I$, $y \in Tun_{\Tilde{I}}$ for some turn $\Tilde{I}$ of $\Tilde{V}_K$ that is different from $I$. Then we have $y \in Tun_{\Tilde{I}} \subseteq O - Tun_I$.

    \vspace{4mm}
    
    (Case i-b): Now suppose $y \in \varprojlim^{\zeta}_{j \in -\mathds{N}^* + K} V_{j}$. Since the star chain $\mathcal{V} = V_{K} \xleftarrow[]{f_{K-1}} V_{K-1} \xleftarrow[]{f_{K-2}} V_{K-2} \xleftarrow[]{f_{K-3}}...$ is maximal, and since $I \nsubseteq V_K$, 
    there must exist an integer $J<K$ such that no level $J$ pre-turn of $I$ is contained in $SC_J (\Tilde{V}_K , V_J)$. 
    i.e. $\hyperlink{pre-turns}{PTS_J (I)} \bigcap \Tilde{V}_J \neq \emptyset$ where $\Tilde{V}_J := \hyperlink{Shadow Component}{SC_J (\Tilde{V}_K , V_J)} =$ the connected component of $(f^J_K)^{-1}(\Tilde{V}_K)$ that contains $V_J$.
    So, we have $y \in \pi^{-1}_J (\Tilde{V}_J) \subseteq O - \pi^{-1}_J (PTS_J (I)) \subseteq O - Tun_I$. Thus $Tun_I$ is closed in $O$.

    \vspace{4mm}
    
    (Case ii.): Now suppose $I \subseteq V_K$. 
    Let $y \in O - ExTun_I$.
    (Case ii-a): $y$ belongs to an interval-shaped plaque of $O$ and therefore, $y \in Tun_{\Tilde{I}}$ for some turn $\Tilde{I}$ of $\Tilde{V}_K$ that is different from $I$. Then we have $y \in Tun_{\Tilde{I}} \subseteq O - ExTun_I$.
    (Case ii-b): $y \in \varprojlim^{\zeta}_{j \in -\mathds{N}^* + K} V_{j} =: V$.  Let $L' \subset V$ be the unique singular segment that sustains $I$.
    Since $y \notin ExTun_I$, $y \notin L'$. 
    Since $\pi_K |_V $ is a homeomorphism onto $V_K$, and since $\pi_K |_V $ takes $L'$ to $I$, $y \notin L'$ implies that $\pi_K (y)$ belongs to an open prong $\Tilde{P}$ of $V_K$ that does not intersect $I$.
    Thus we have $y \in \pi^{-1}_K (\Tilde{P}) \subseteq O - \pi^{-1}_K (I) \subseteq O - ExTun_I$. 
\end{proof}

\vspace{4mm}

\begin{corollary}
    \label{ExTun_is_Cl(Tun)}
    \hypertarget{ExTun_is_Cl(Tun)}{Let} $K \in \n$ and $I$ a turn in a standard star neighborhood $\Tilde{V}_K$ in $G_K$. Suppose $Tun_I \neq \emptyset$.
    Then $ExTun_I$ is the closure of $Tun_I$ in $O := \pi^{-1}_K (\Tilde{V}_K)$.
\end{corollary}

\begin{proof}
    Assume the set up of the corollary. If $ExTun_I = Tun_I$, then from Proposition \ref{ExTun is closed in O}, we have that $ExTun_I = Tun_I$ is it's own closure in $O$. Now suppose that $ExTun_I - Tun_I \neq \emptyset$. 
    i.e. $ExTun_I - Tun_I = L'$ is a singular segment. 
    Now let $C := $ closure of $Tun_I$ in $O$. From Proposition \ref{ExTun is closed in O}, we have that $ExTun_I$ is a closed subset of $O$ that contains $Tun_I$, and thus $C \subseteq ExTun_I$. 
    Since $\zeta$ is \hyperlink{Stabilized Split Sequence}{stabilized}, $Tun_I \neq \emptyset$ implies that $L'$ is a \hyperlink{limit vs isolated segments}{limit segment}. 
    It follows from Lemma \ref{Intersecting Pre Leaf Segments} that $ExTun_I = Tun_I \bigsqcup L'$.
    Since $L'$ entirely consists of limit points of $Tun_I$ [Lemma \ref{limit segments consist of limit points}] we have that, $ExTun_I \subseteq C$. 
\end{proof}

\vspace{4mm}

We will show in Lemma \ref{ExTun-CS are compact} that each cross section of an extended turn tunnel, is compact in $X$. 

\vspace{4mm}

The rest of this chapter is dedicated to establishing a foliated atlas for $X - Sing(X)$ [Sub-section \ref{sec: The Tunnel Atlas}], and then a singular foliated atlas for $X$ [Sub-section \ref{sec: The Star Tunnel Atlas}]. In the next section we will investigate the shape of a standard neighborhood, paying special attention to studying the case when a standard neighborhood contains a singularity.

\vspace{4mm}


\subsection{The Shape of Standard Neighborhoods}
\label{Sec: Shape of Standard Nbhds}

\subsubsection{Star Tunnel Sets}
\label{sec: Star Tunnel Sets}


Here we introduce the class of sets that will function as a model for a typical neighborhood of a singularity in $X$, called a ``Star Tunnel Set". In Sub-section \ref{sec: Star Tunnel Sets in St Nbhds}, we will show that 
each singularity in $X$ is contained in a neighborhood that is a \hyperlink{star tunnel sets}{star tunnel set}. 

\vspace{4mm}

To create a star tunnel set, we start with a \hyperlink{star set}{star set}, and finitely many \hyperlink{abstract tunnels}{abstract tunnels}, then select a particular path component from each abstract tunnel that gets glued to a particular turn in the star set, in the following way. 

\vspace{4mm}

\begin{figure}[htbp]
    \centering
    \begin{minipage}{0.45\textwidth}
        \centering
        \begin{tikzpicture}
            \draw[top color=gray!20, bottom color=black!20, middle color=white!20] (0 , 0) -- (4 , 0) to[out=70, in=-70] (4 , 1) -- (0,1) to[out=-70, in=70] cycle;
            \draw (0 , 0) to[out=110, in=-110] (0 , 1); 
            \draw[dashed] (2 , 0) to[out=110, in=-110] (2 , 1); 
            \draw (2 , 0) to[out=70, in=-70] (2 , 1); 
            \draw[dashed] (4 , 0) to[out=110, in=-110] (4 , 1); 
        \end{tikzpicture}
    \end{minipage}
    \hfill 
    \begin{minipage}{0.45\textwidth}
        \centering
        \begin{tikzpicture}
            \draw[top color=gray!20, bottom color=black!20, middle color=white!20] (0 , 0) -- (4 , 0) to[out=70, in=-70] (4 , 1) -- (0,1) to[out=-70, in=70] cycle;
            \draw (0 , 0) to[out=110, in=-110] (0 , 1); 
            \draw[dashed] (2 , 0) to[out=110, in=-110] (2 , 1); 
            \draw (2 , 0) to[out=70, in=-70] (2 , 1); 
            \draw[dashed] (4 , 0) to[out=110, in=-110] (4 , 1); 

            \draw[top color=blue!10, bottom color=blue!20, middle color=white!20, opacity=0.5] (0 , 0) -- (2,0) to[out=0 , in=-105] (3.6 , 2) to[out=170 , in=-40] (2.75 , 2.4) to[out=-110 , in=0] (2,1) -- (0,1) to[out=-70, in=70] cycle;
            \draw[opacity=0.5] (3.6 , 2) to[out=120 , in=0] (2.75 , 2.4); 

            \draw[red!50] (0,0.5) -- (4,0.5);
            \draw[red!50] (2,0.5) to[out=5 , in=-110]  (3.2, 2.2);

            \draw[draw=none] (0,-1) rectangle (2,-1.4);
        \end{tikzpicture}
    \end{minipage}
    \caption{A Tunnel Set and a Star Tunnel Set}
    \label{fig: A Tunnel Set and a Star Tunnel Set}
\end{figure}
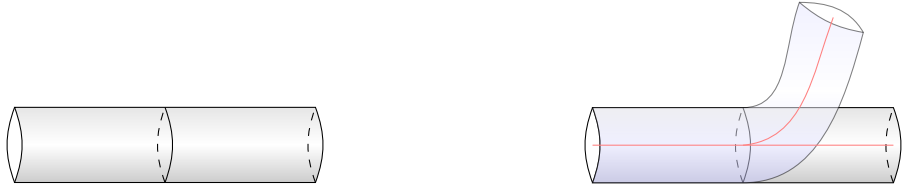

\vspace{4mm}

\begin{definition}[Star Tunnel Sets]
    \label{star tunnel sets}
    \hypertarget{star tunnel sets}{Let} $n \in \mathds{N}$, and let $V$ be an \hyperlink{star set}{$n-$pronged star set} centered at a point $p$, whose set of \hyperlink{Turn and Prong}{turns} is $\Tilde{\mathcal{T}}$. Let $\{ T_I := C_I \times I \}_{I \in \Tilde{\mathcal{T}}}$ be a collection of $\binom{n}{2}$  \hyperlink{abstract tunnels}{abstract tunnels}, where for each  $I \in \Tilde{\mathcal{T}}$, $C_I$ is a totally disconnected set. Furthermore, let $A := (\bigsqcup_{I \in \Tilde{\mathcal{T}}} T_I ) \bigsqcup V $. Let $q : A \longrightarrow B$ be a quotient map. We call $q$ an ``$n-$Pronged Star Tunnel Quotient Map" if, for each $I \in \Tilde{\mathcal{T}}$ there exists a unique $c_I \in C_I$ such that, 

\begin{enumerate}
    \item for each $t \in I$, $q |_{T_I} (( c_I , t)) = q |_V (t)$ 
    (i.e. $q$ identifies $\{ c_I \} \times I \subseteq T_I$ with $I \subseteq V$),  
    
    \item for each $x \in T_I$,  
     \begin{equation*}
        | q^{-1} (q (x)) | = \begin{cases}
			1, & \text{if} \ x \in (C_I - \{ c_I \} ) \times I\\
            n, & \text{if} \ x \in \{ c_I \} \times (I -\{p\}) \\
            \binom{n}{2} + 1, & \text{if} \ x = (c_I,p)
		 \end{cases}
    \end{equation*}
    
    \item for each $x \in V$, 
    \begin{equation*}
        | q^{-1} (q (x)) | = \begin{cases}
            n, & \text{if} \ x \in V - \{p\}\\
            \binom{n}{2} + 1, & \text{if} \ x = p
		 \end{cases}
    \end{equation*} 
    and, 
    
    \item if $C_I \neq \{ c_I \}$, then $q (\{c_I\} \times I)$ entirely consists of limit points of $q ((C_I - \{ c_I \} ) \times I)$.\\
\end{enumerate}
    
The image of an $n-$pronged star tunnel quotient map, is called an ``$n-$Pronged Star Tunnel Set". If a star tunnel set $S$ contained in a space $Y$, is open in $Y$, then $S$ is called a ``Star Tunnel Neighborhood of $Y$".  
\end{definition}

Note that for some $I \in \mathcal{T}$, $C_I$ is permitted to be a single point, making the tunnel $T_I$ just an open interval, in which case, the criterion 4 above is vacuously true.

 \vspace{4mm}


\subsubsection{Star Tunnel Sets and Standard Neighborhoods}
\label{sec: Star Tunnel Sets in St Nbhds}

\vspace{4mm}

\begin{definition}[Singular Standard Neighborhoods]
    \label{Singular Standard Neighborhoods}
    \hypertarget{Singular Standard Neighborhoods}{A} \hyperlink{Standard Star Neighbourhoods}{standard neighborhood} $O$ in $X$ shall be called a ``Singular Standard Neighborhood in $X$" if $O$ contains a singularity.
\end{definition}

\vspace{4mm}


Here, we will show that each \hyperlink{Singular Standard Neighborhoods}{singular standard neighborhood} $O$ of $X$, is the disjoint union of a \hyperlink{star tunnel sets}{star tunnel neighborhood} and finitely many \hyperlink{Def: Tun and ExTun}{turn tunnels} [Proposition \ref{sing stnd nbhds vs star tunnel nbhds}].

\vspace{4mm}

Note that a \hyperlink{Singular Plaques}{singular plaque in $X$} does not necessarily always project to a star neighborhood in $G_j$ for each level $j \in \n$. Some such projections may have a non-trivial neighborhood completion inside the corresponding level graph. We will now establish terminology that lets us separate the standard turns in $\zeta$ that are \hyperlink{Taken Turns}{sustained by} some \hyperlink{leaf segments}{singular segment}, from the standard turns that aren't.


\vspace{4mm}

\begin{notation}[Singular Supporting Turns, and Leftover Turns]
    \label{sing vs left}
    \hypertarget{sing vs left}{Let} $K \in \n$, $\Tilde{V}_K$ a standard star neighborhood in $G_K$ such that, $O := \pi^{-1}_K (\Tilde{V}_K)$ contains a singularity. We denote the unique singular plaque of $O$, ``$V_{\text{sing}} (O)$" and its projection $\pi_K (V_{\text{sing}} (O))$ in $G_K$, ``$V^O_K$".

    \vspace{4mm}
    
    A turn $I$ of $\Tilde{V}_K$ shall be called a ``Singular Supporting Turn of $\Tilde{V}_K$ rel $O$" if $I$ is a turn of $V^O_K$ (equivalently $I$ is sustained by a singular segment in $X$), and $I$ is called a ``Leftover Turn of $\Tilde{V}_K$ rel $O$" if $I$ is not a turn of $V^O_K$ (equivalently $I$ is not sustained by a singular segment in $X$). We shall denote by ``$\mathcal{T}_{\text{sing}} (\Tilde{V}_K)$" the collection of all singular supporting turns of $\Tilde{V}_K$, and ``$\mathcal{T}_{\text{left}} (\Tilde{V}_K)$" shall denote the collection of all the leftover turns of $\Tilde{V}_K$. 
\end{notation}

\vspace{4mm}

\begin{definition}[Star Tunnel Component of a Standard Neighborhood]
    \label{ST(O)}
    \hypertarget{ST(O)}{Let} $K \in \n$, $\Tilde{V}_K$ a standard star neighborhood in $G_K$ such that $O := \pi^{-1}_K (\Tilde{V}_K)$ contains a singularity (i.e. $O$ is a singular standard neighborhood of $X$). Then
    ``The Star Tunnel Component of $O$ in $X$" denoted ``$ST(O)$" is the subset of $O$ given by $\bigcup_{I \in \hyperlink{sing vs left}{\mathcal{T}_{\text{sing}} (\Tilde{V}_K)}} ExTun_I$.
\end{definition}

\vspace{4mm}

We will now use the notion of canonically given tunnel parameterizations [Remark \ref{Canonical Tunnel Parameterizations}] to construct a parameterization for each star tunnel component in $X$. In Lemma \ref{canonical star tun paras are star tun quotients}, we shall show that the parameterization defined below is a \hyperlink{star tunnel sets}{star tunnel quotient map}. 

\vspace{4mm}

\begin{definition}[Canonically Given Star Tunnel Parameterizations]
    \label{Canonically Given Star Tunnel Parameterization}
    \hypertarget{Canonically Given Star Tunnel Parameterization}{Assume} the set-up of the above definition. Let $V := \hyperlink{sing vs left}{V_{\text{sing}} (O)}$. (Recall from Remark \ref{More Consequences of Stabilization}, that $\pi_K |_V$ is a homeomorphism onto $\hyperlink{sing vs left}{V^O_K} \subseteq \Tilde{V}_K$.)
    For each $I \in \mathcal{T}_{\text{sing}} (\Tilde{V}_K)$, let $h_I : C_I \times I \longrightarrow ExTun_I$ be the canonically given tunnel parameterizations defined in Remark \ref{Canonical Tunnel Parameterizations}. Let $A := (\bigsqcup_{I \in \mathcal{T}_{\text{sing}} (\Tilde{V}_K)} (C_I \times I )) \bigsqcup V^O_K$ . Then define $h_O : A \longrightarrow ST(O)$ as follows.
    \begin{enumerate}
        \item $h_O |_{V^O_K} := (\pi_K |_V )^{-1} $ onto $V$, and,
        \item for each $I \in \mathcal{T}_{\text{sing}} (\Tilde{V}_K)$, $h_O |_{C_I \times I} := h_I$ onto $ExTun_I$.
    \end{enumerate}
    The above-defined $h_O$ is called the ``Canonically Given Star Tunnel Parameterization of $ST(O)$ in $X$".
\end{definition}


\vspace{4mm}

The following lemma shows that, for each singular standard neighborhood $O$, $ST(O)$ is a \hyperlink{star tunnel sets}{star tunnel set}.

\vspace{4mm}

 \begin{lemma}
    \label{canonical star tun paras are star tun quotients}
    Each canonically given star tunnel parameterization in $X$ is a \hyperlink{star tunnel sets}{star tunnel quotient map}. 
 \end{lemma}

 \begin{proof}
     Assume the set-up of Definition \ref{ST(O)}. Suppose that the unique singular plaque $V := \hyperlink{sing vs left}{V_{\text{sing}} (O)}$ of $O$ is $n-$pronged for some $n \in \mathds{N}$. We must show that $h_O$ satisfies the 4 criteria listed in Definition \ref{star tunnel sets}. For each $I \in \mathcal{T}_{\text{sing}} (\Tilde{V}_K)$, let $L(I) := (\pi_K |_V )^{-1} (I)$ and choose $c_I \in C_I$ to be the unique singularity contained in the \hyperlink{Central Cross-Section}{central cross-section} $C_I$.
     Then, we have that for each $I \in \mathcal{T}_{\text{sing}} (\Tilde{V}_K)$, $h_O$ identifies $I \subset \hyperlink{sing vs left}{V^O_K}$ with $\{c_I\} \times I \subseteq C_I \times I$ (Criterion 1.). Criteria 2. and 3. follow from the five facts listed below (that are immediate consequences of Lemma \ref{Intersecting Pre Leaf Segments} and  the fact that $\pi_K |_V$ is a homeomorphism). 
     \begin{enumerate}
         \item[i.] For each $I \in \mathcal{T}_{\text{sing}} (\Tilde{V}_K)$, $V \cap Tun_I = \emptyset$.
         \item[ii.] For each $I \in \mathcal{T}_{\text{sing}} (\Tilde{V}_K)$, 
         $L(I)$ is equal to $V \cap ExTun_I$.
         \item[iii.] For each distinct pair of turns $I, I' \in \mathcal{T}_{\text{sing}} (\Tilde{V}_K)$, $ExTun_I \cap Tun_{I'} = \emptyset$, and $ExTun_I \cap ExTun_{I'} = ( L(I) \cap L(I') ) \subset V$. 
         \item[iv.] For each $I \in \mathcal{T}_{\text{sing}} (\Tilde{V}_K)$, and for each $y \in L(I) - the \ singularity$, there are exactly $(n-1)$ many turns $I' \in \mathcal{T}_{\text{sing}} (\Tilde{V}_K)$ such that $y \in ExTun_{I'}$.
         \item[v.] For each of the $\binom{n}{2}$ turns $I$ of $V^O_K$,  $ExTun_I$ contains the singularity of $O$. 
     \end{enumerate}
     For each turn $L$ of $V$, $L$ is either an \hyperlink{limit vs isolated segments}{isolated segment} (in which case $Tun_{\pi_K(L)} = \emptyset$ from \hyperlink{Stabilizing Hypotheses}{stabilization hypothesis 4}) or $L$ is a \hyperlink{limit vs isolated segments}{limit segment}, in which case $L$ entirely consists of limit points of $Tun_{\pi_K(L)}$ [Lemma \ref{limit segments consist of limit points}]. Thus criterion 4. is satisfied. 
 \end{proof}

\vspace{4mm}

 \vspace{4mm}

\begin{proposition}
    \label{sing stnd nbhds vs star tunnel nbhds}
    Each star tunnel component of a singular standard neighborhood in $X$ is open in $X$.
    Furthermore, each \hyperlink{Singular Standard Neighborhoods}{singular standard neighborhood in $X$} is the disjoint union of a \hyperlink{star tunnel sets}{star tunnel neighborhood} and finitely many turn tunnels. 
\end{proposition}

\begin{proof}
    Let $K \in \n$, $\Tilde{V}_K$ a standard star neighborhood in $G_K$ such that $O := \pi^{-1}_K (\Tilde{V}_K)$ is a singular standard neighborhood in $X$. Note that for each $I \in \mathcal{T}_{\text{left}} (\Tilde{V}_K)$, $Tun_I$ does not contain any singular segments and thus $ExTun_I = Tun_I$. Thus $O = \bigcup_{I \ \text{is a turn of} \ \Tilde{V}_K} ExTun_I = (\bigsqcup_{I \in \mathcal{T}_{\text{left}} (\Tilde{V}_K)} Tun_I ) \bigsqcup (ST(O))$.
    Since for each $I \in \mathcal{T}_{\text{left}}$, $ExTun_I = Tun_I$ is closed in $O$,
    $ST(O)$ is open in $X$. Then it follows from Lemma \ref{canonical star tun paras are star tun quotients} that $ST(O)$ is a star tunnel neighborhood. 
\end{proof}



\vspace{4mm}

In the next section, we will complete everything necessary to start viewing each proper solenoid as a foliated object. 

\vspace{4mm}


\subsection{An Atlas for the Solenoid}
\label{Sec: An Atlas for the Solenoid}

\vspace{4mm}

Since solenoids are far from being manifolds, the pursuit of a $foliated \ atlas$ for a solenoid will have significant differences from the manifold-case. It's worth pointing out that we have already defined the notion of \hyperlink{leaf segments}{leaf segments} for solenoids that do not depend on any atlas. The way we define a foliated atlas for $X$ will necessarily be restricted by the need that this very natural notion of leaf segments risen as a result of $X$'s topology, must coincide with the notion of leaf segments given by the atlas. 

\vspace{4mm}

We will find a foliated atlas for $X -  Sing(X)$ in Sub-section \ref{sec: The Tunnel Atlas}, and then a singular foliated atlas for $X$ in Sub-section \ref{sec: The Star Tunnel Atlas}. 

\vspace{4mm}

\subsubsection{The Tunnel Atlas}
\label{sec: The Tunnel Atlas}

\vspace{4mm}

Here, we will build a foliated atlas for the non-singular potion of the solenoid. Throughout this section, we assume that the solenoid $X(\zeta)$ is induced by a proper \hyperlink{Stabilized Split Sequence}{stabilized} split sequence $\zeta$. We will start with laying out the particular open cover of $X - Sing(X)$ associated with the atlas that will be introduced in Definition \ref{Def: The Tunnel Atlas}.

\vspace{4mm}

Recall Optimal Star Neighbourhoods of $X$ [Definition \ref{Optimal Star Neighbourhoods}]. 

\vspace{4mm}

\begin{definition}[The Tunnel Cover of $X - Sing(X)$]
    \label{Def: Tunnel Cover}
    ``The Tunnel Cover of $X - Sing(X)$" is given by 
    $\mathcal{TC} (X) := \{ Tun_I \subset X : I$ is a \hyperlink{optimal turns}{level $0$ optimal turn} of $\zeta \}$. 
\end{definition}

\vspace{4mm}

Recall that the collection of all (finitely many) \hyperlink{optimal turns}{level $0$ optimal turns of $X$} covers $G_0$ [Remark \ref{Properties of Optimal Star Neighbourhoods}], and therefore $\mathcal{TC}(X)$ is a finite open cover for $X -Sing(X)$ [Lemma \ref{turn tunnels are open}].

\vspace{4mm}


\begin{definition}[The Tunnel Atlas of $X - Sing(X)$]
    \label{Def: The Tunnel Atlas}
    \hypertarget{Def: The Tunnel Atlas}{``The Tunnel Atlas of $X - Sing(X)$"} is given by $\mathcal{TA}(X) := \{ h_I : h_I$ is the \hyperlink{Canonical Tunnel Parameterizations}{canonically given tunnel parameterization} of $Tun_I$ where $I$ is \hyperlink{optimal turns}{level $0$ optimal turn} of $\zeta \}$. 
\end{definition}

\vspace{4mm}

Following the terminology widely used in discussing atlases, by ``Coordinate Chart" we refer to an element of the atlas. In the case of foliated surfaces [Definition \ref{Foliated Atlases on Surfaces}], the requirement that different coordinate charts must preserve the concept of leaf segments is a crucial part of being a foliated atlas. In the context of solenoids, the tunnel atlas naturally satisfies an analogous requirement. The rest of this sub-section is a careful treatment of expressing this fact.

\vspace{4mm}

The next lemma [Lemma \ref{how turn atlas charts intersect}] shows that when two elements of \hyperlink{Def: Tunnel Cover}{the tunnel cover} intersect, the intersection itself is a disjoint union of \hyperlink{Tunnel Sets}{tunnels} (where in each of these tunnels, each path component is a leaf segment of the solenoid, as implied by Remark \ref{pre-leaf Segments and Path components}). 

\vspace{4mm}

The following remark will show that two ``meaningfully" intersecting level $0$ optimal turns, intersect at a disjoint union of finitely many turns. By a ``meaningful" intersection we mean that the two corresponding turn tunnels have a non-empty intersection.
This result will be utilized in the proof of Lemma \ref{how turn atlas charts intersect}. 

\vspace{4mm}


\begin{remark}[Two Intersecting Optimal Turns]
    \label{Two Intersecting Turns in a Level Graph}
    Let $I, I'$ be two distinct \hyperlink{optimal turns}{level $0$ optimal turns} such that $Tun_I \cap Tun_{I'} \neq \emptyset$. This implies that $I \cap I' \neq \emptyset$. Let $I$ and $I'$ respectively be turns of the level $0$ optimal star neighborhoods $V$ and $V'$.
    The case where $V = V'$ implies that $Tun_I \cap Tun_{I'} =  \emptyset$. Thus $V \neq V'$. Two distinct level $0$ optimal star neighborhoods intersect at a disjoint union of finitely many turns. Thus $I \cap I'$ is the disjoint union of up to two turns. 
\end{remark}

\vspace{4mm}

\begin{definition}[Interval Components]
    \label{Interval Components}
    \hypertarget{Interval Components}{Assume} the set-up of the above remark. (i.e. $I, I'$ are two distinct \hyperlink{optimal turns}{level $0$ optimal turns} such that $Tun_I \cap Tun_{I'} \neq \emptyset$.) Then a connected component $U$ of $I \cap I'$ shall be called an ``Interval Component of $I \cap I'$" if $U$ is an open interval.
\end{definition}

\vspace{4mm}



The following lemma describes the intersection $Tun_I \cap Tun_{I'}$ for the aforementioned types of turns.

\vspace{4mm}

\begin{lemma}
    \label{how turn atlas charts intersect}
    Let $I, I'$ be two distinct \hyperlink{optimal turns}{level $0$ optimal turns} such that $Tun_I \cap Tun_{I'} \neq \emptyset$. Let $\mathcal{C} = \mathcal{C}_{I,I'}$ be the collection of \hyperlink{Interval Components}{interval components of $I \cap I'$}. (Note that each element of $\mathcal{C}$ is a \hyperlink{Turns and Prongs}{level 0 standard turn in $\zeta$} and that $\mathcal{C}$ contains at most $2$ elements.) Then $Tun_I \cap Tun_{I'} = \bigsqcup_{J \in \mathcal{C}} T_J$ where, for each $J \in \mathcal{C}$, $T_J$ is a \hyperlink{Sub Tunnels}{sub-fiber tunnel} of $Tun_J$.
\end{lemma}

\begin{proof}
    Assume the set-up of the above lemma. Let $J \in C_{I,I'}$.  
    Choose $q_J \in J \subseteq I \cap I'$. 
    Since $J$ is a turn that cannot contain any \hyperlink{special points}{special points}, we may designate $q_J$ as the ``center" of the $2-$pronged star neighborhood $J$, and denote $C_J := Tun_J \cap F_q$ the \hyperlink{Central Cross-Section}{central cross-section of $Tun_J$ rel $q$}. Furthermore, let $h_J : C_J \times J \longrightarrow Tun_J$ be the canonically given tunnel parameterization for $Tun_J$ rel $C_J$. 
    Now, let $\Tilde{C}_J := $
    $ \{ y \in F_{q_J} : y \in $ a leaf segment that takes $I$ and $y \in $ a leaf segment that takes $I' \}$. 
    Let $T_J := h_J (\Tilde{C}_J \times J)$.

    \vspace{4mm}

    From the way we defined $T_J$ for each $J \in \mathcal{C}$, it follows that $\bigsqcup_{J \in \mathcal{C}} T_J \subseteq Tun_I \cap Tun_{I'}$.
    To show that $Tun_I \cap Tun_{I'} \subseteq \bigsqcup_{J \in \mathcal{C}} T_J$, suppose $L , L'$ is a pair of leaf segments such that, $L$ takes $I$, $L'$ takes $I'$, and $L \cap L' \neq \emptyset$. Let $\mathcal{I}$ (resp. $\mathcal{I}'$) denote the unique $2-$pronged star chain \hyperlink{star chain representing a set in X}{representing} $L$ (resp. $L'$) \hyperlink{star chain representing a set in X}{starting at} $I$ (resp. $I'$).
    Then by observing $\mathcal{I}$ and $\mathcal{I}'$, we can conclude that $L \cap L'$ must be a leaf segment that takes $J$ for some $J \in \mathcal{C}_{I,I'}$. 
\end{proof}

\vspace{4mm}

\begin{definition}[Turn Tunnel Components]
    Let $I, I'$ be two distinct \hyperlink{optimal turns}{level $0$ optimal turns} such that $Tun_I \cap Tun_{I'} \neq \emptyset$. And let $\mathcal{C} = \mathcal{C}_{I,I'}$ be the collection of \hyperlink{Interval Components}{interval components of $I \cap I'$}. For each $J \in \mathcal{C}$, we call the \hyperlink{Sub Tunnels}{sub-fiber tunnel} of $Tun_J$ given by $T_J := Tun_J \cap (Tun_I \cap Tun_{I'})$ a ``Turn Tunnel Component of $Tun_I \cap Tun_{I'}$". 
\end{definition}

\vspace{4mm}

It follows from Lemma \ref{how turn atlas charts intersect} that, $Tun_I \cap Tun_I'$ is the disjoint union of its turn tunnel components. 

\vspace{4mm}

In the next sub-section we mimic the layout of this sub-section to define a singular foliated atlas for $X$.

\vspace{4mm}

\subsubsection{The Star Tunnel Atlas}
\label{sec: The Star Tunnel Atlas}


Once again, we assume that the solenoid $X(\zeta)$ is induced by a proper \hyperlink{Stabilized Split Sequence}{stabilized} split sequence $\zeta$.

\vspace{4mm}

Here, we will build a singular foliated atlas for $X$. We will start with laying out the particular open cover of $X$ that will be associated with the atlas we will lay out in Definition \ref{Def: Star Tunnel Atlas}.

\vspace{4mm}

Recall Optimal Star Neighborhoods of $X$ [Definition \ref{Optimal Star Neighbourhoods}], and the Tunnel Cover of $X- Sing(X)$ denoted $\mathcal{TC}(X)$ [Definition \ref{Def: Tunnel Cover}].

\vspace{4mm}



\begin{definition}[The Star Tunnel Cover of $X$]
    \label{Def: Star Tunnel Cover}
    \hypertarget{Star Tunnels of X}{The} ``Collection of Optimal Star Tunnels of $X$" is given by $\mathcal{ST}(X) := \{ ST(O) \subset X : O$ is a level $0$ optimal neighborhood of $X$ that is singular$\}$.
    \hypertarget{Def: Star Tunnel Cover}{The} ``Star Tunnel Cover of $X$" is given by $\mathcal{STC}(X) := (\mathcal{ST}(X)) \bigcup (\mathcal{TC}(X))$. 
\end{definition}

\vspace{4mm}

\begin{remark}[The Star Tunnel Cover is a Finite Open Cover of $X$]
    \label{STC is fin a open cover}
    Note that $\mathcal{ST}(X)$ is a finite collection [Rem. \ref{Properties of Optimal Star Neighbourhoods}], and that each element of it is an open neighborhood of a singularity in $X$ [Prop. \ref{sing stnd nbhds vs star tunnel nbhds}].
    Since $\mathcal{TC}(X)$ is a finite open cover of $X - Sing(X)$, and since $Sing(X)$ is contained in the union of elements of $\mathcal{ST}(X)$,  $\mathcal{STC}(X)$ is a finite open cover of $X$.
\end{remark}


\vspace{4mm}

Recall canonically given star tunnel parameterizations from Definition \ref{Canonically Given Star Tunnel Parameterization}.

\vspace{4mm}

\begin{definition}[The Star Tunnel Atlas of $X$]
    \label{Def: Star Tunnel Atlas}
    \hypertarget{Def: Star Tunnel Atlas}{``The Star Tunnel Atlas of $X$"} is given by $\mathcal{STA}(X) := \hyperlink{Def: The Tunnel Atlas}{\mathcal{TA}(X)} \ \bigcup \ \{ h_O \ | \ ST(O) \in \mathcal{ST} (X)$, and $h_O$ is the
    \hyperlink{Canonically Given Star Tunnel Parameterization}{canonically given star tunnel parameterization of $ST(O)$}$\}$.
\end{definition}



\vspace{4mm}

Next, we will generalize Lemma \ref{how turn atlas charts intersect} in the context of the star tunnel cover.

\vspace{4mm}

\begin{lemma}
    \label{intersecting elements of STC(X)}
    Let $T, T'$ be two distinct elements of $\mathcal{STC}(X)$ such that $T \cap T' \neq \emptyset$. Then $T \cap T' = \bigsqcup_{i=1}^m T_i$ for some $m \in \mathds{N}$ where for each $i = 1,...,m$, there exists a turn $I_i$ in $G_0$ such that,
    \begin{enumerate}
        \item $I_i$ is a connected component of $\pi_0 (T) \cap \pi_0 (T')$, and,
        \item $T_i$ is a \hyperlink{Sub Tunnels}{sub-fiber tunnel} of $Tun_{I_i}$.
    \end{enumerate}
\end{lemma}

\begin{proof}
    Assume the set-up of the lemma and let $\mathcal{C}$ be the collection of connected components of $T \cap T'$. Then either (case a.) both $T$ and $T'$ are turn tunnels, or (case b.) one of $T, T'$ is a turn tunnel and the other is a star tunnel neighborhood of a singularity,
    or (case c.) both $T$ and $T'$ are star tunnel neighborhoods of singularities. If case a. is true, then the conclusion follows from Lemma \ref{how turn atlas charts intersect}.

    \vspace{4mm}

    (Case b.) Suppose $T$ is a turn tunnel and $T'$ a star tunnel neighborhood. Let $I$ be the level $0$ optimal turn such that $T = Tun_I$ and let $V$ be the level $0$ optimal star neighborhood such that $O := \pi^{-1}_0 (V)$ and $T' = ST(O)$. (Case b-i.) Suppose $I$ is a turn of $\pi_0 (T')$. Then $T \cap T' = T = Tun_I$. 
    (Case b-ii.) Suppose $I$ is a turn of $V$ but not of $\pi_0 (T')$. i.e. $I$ is a \hyperlink{sing vs left}{leftover turn} of $V$ rel $O$.
    Then it follows from the proof of Prop. \ref{sing stnd nbhds vs star tunnel nbhds} that, $Tun_I \cap ST(O) = \emptyset$. Thus this is not a valid case.
    (Case b-iii.) Suppose $I$ is a turn of a level $0$ optimal star neighborhood $\Tilde{V} \neq V$. Then $V$ and $\Tilde{V}$ have different centers. Thus $V \cap I = \bigsqcup_{i=1}^m I_i$ where $m \in \{ 1,2 \}$ such that, for each $i=1,...,m$, $I_i$ is a level $0$ turn. Then $\pi^{-1}_0(V) \supset Tun_{I_i}$ for each $i=1,...,m$. 

    \vspace{4mm}

    (Case c.) Suppose both $T$ and $T'$ are star tunnel neighborhoods. Let $V , V'$ respectively be the level $0$ optimal star neighborhoods such that $O = \pi^{-1}_0 (V) , O' = \pi^{-1}_0 (V')$ where $T = ST(O)$ and $T' = ST(O')$. Since $T$ and $T'$ are distinct, so are $V$ and $V'$, which means $V$ and $V'$ have different centers. Since $T \cap T' \neq \emptyset$, $\pi_0 (T) \cap \pi_0 (T')$ is a non-trivial disjoint union of embedded open intervals in $G_0$. Now let $m$ be the number of connected components of $\pi_0 (T) \cap \pi_0 (T')$. Then we can express $\pi_0 (T) \cap \pi_0 (T')$ as $\bigsqcup_{i=1}^m I_i$ where, for each $i = 1,...,m$, $I_i$ is a level $0$ turn. Let $i = 1,...,m$. Since $I_i \subset \pi_0 (T) \cap \pi_0 (T')$, we have $Tun_{I_i} \subset T \cap T'$.
\end{proof}

\vspace{4mm}

From now on, we will use the \hyperlink{Def: Star Tunnel Atlas}{star tunnel atlas} to envision $X$ as a singular foliated object, and the \hyperlink{Def: The Tunnel Atlas}{tunnel atlas} to envision $X - Sing(X)$ as a foliated object. In the next chapter, we will use these atlases along with the concepts of \hyperlink{leaf segments}{leaf segments} and \hyperlink{plaque}{plaques of standard neighborhoods} to introduce ``Leaves" [Section \ref{sec: leaves}] and then ``Transversals" [Section \ref{sec: Transversals}] of a proper solenoid. 

\vspace{4mm}



\begin{center}
    \section{Leaves and Transversals} \label{Ch: Leaves and Transversals}
\end{center}

\vspace{4mm}

\h Throughout the discussion in this chapter, we assume that $\zeta$ is a given \hyperlink{Stabilized Split Sequence}{stabilized} proper split sequence and $X = X(\zeta)$ its induced solenoid. We will heavily utilize 
the star tunnel cover $\mathcal{STC}(X)$ established in Definition \ref{Def: Star Tunnel Cover}.

\vspace{4mm}








\subsection{Leaves}
\label{sec: leaves}

In this section, we shall inspect the features of the proper solenoid $X$ that are analogues to leaves of a foliated surface [Definition \ref{surface leaves}] and of a singular foliated surface [Definition \ref{Leaves of a Singular Foliation}]. Just like in the case of singular foliated surfaces, in solenoids, leaves will turn out to be immersed graphs [Lemma \ref{leaves are immersed graphs}]. However, unlike the surface case, leaves of a proper solenoid $X$ will turn out to be path components of $X$ [Lemma \ref{leaves vs path components}].

\vspace{4mm}

We begin with some preliminary definitions.

\vspace{4mm}

\subsubsection{Leaves and Partial Leaves}
\label{Leaves and Partial Leaves}

\vspace{4mm}

This sub-section is dedicated to defining ``Leaves" and ``Partial Leaves" of the solenoid [Definition \ref{Leaves, Partial Leaves}], by first laying out certain equivalence relations [Definition \ref{Leaf Equivalence}]. To build these definitions we use the terminology laid out earlier of optimal neighborhoods and optimal turns [Definition \ref{Optimal Star Neighbourhoods}]. 



\vspace{4mm}

\begin{definition}[Optimal Plaques, and Optimal Leaf Segments]
    \label{Optimal Plaques and Leaf Segments}
    \hypertarget{Optimal Plaques and Leaf Segments}{A} subset $V \subset X$ is called an ``Optimal Plaque in $X$" if $V$ is a  plaque of a level $0$ optimal neighborhood of $X$. 
    A subset $L \subset X$ is called an ``Optimal Leaf Segment in $X$" if $L$ is an optimal plaque that is \hyperlink{Singular Plaques}{non-singular} (equivalently, $L$ is a leaf segment that takes a \hyperlink{turn in zeta}{level $0$ optimal turn in $\zeta$}).
\end{definition}

\vspace{4mm}


The collection of all optimal plaques (resp. optimal leaf segments) cover $X$ (resp. $X - Sing(X)$) 
[Remark \ref{Properties of Optimal Star Neighbourhoods}]. 

\vspace{4mm}

We plan to use certain unions of optimal plaques (resp. optimal leaf segments) to define ``\hyperlink{Leaves, Partial Leaves}{Leaves}" (resp. ``\hyperlink{Leaves, Partial Leaves}{Partial Leaves}") of $X$. Note that since each standard plaque in $X$ is contained in an optimal plaque in $X$ (resp. since each leaf segment in $X$ is contained in an optimal leaf segment in $X$), when we build the discussion of leaves (resp. partial leaves), we can limit our focus to optimal plaques (resp. optimal leaf segments) instead of considering the broader definitions. 

\vspace{4mm}

\begin{remark}[Intersecting Optimal Plaques and Intersecting Optimal Leaf Segments]
    \label{intersecting plaques}
    Let $V$ and $V'$ be two distinct optimal plaques that intersect non-trivially. Let $O$ (resp. $O'$) be the optimal neighborhood that $V$ (resp. $V'$) is a plaque of. If $O = O'$ then $V \cap V' = \emptyset$, thus $O$ and $O'$ are distinct. Since any two distinct optimal star neighborhoods (in $G_0$) that intersect non-trivially, intersect at a disjoint finite union of turns, and since $\pi_0 |_V , \pi_0 |_{V'}$ both are homeomorphisms onto star sets in $G_0$ with \hyperlink{Neighbourhood Completion}{standard neighborhood completions} that are optimal,
    $V \cap V'$ is a disjoint union of finitely many leaf segments. 

    \vspace{4mm}

    Let $L$ and $L'$ be two distinct optimal leaf segments that intersect non-trivially. Since they intersect, $\pi_0 (L)$ and $\pi_0 (L')$ are turns of two distinct optimal star neighborhoods in $G_0$. Thus $L \cap L'$ is either one leaf segment or the disjoint union of two leaf segments. 
\end{remark}

\vspace{4mm}

\begin{definition}[Concatenations of Plaques/Leaf Segments]
    \hypertarget{Concatenations}{Let} $\mathcal{V} := \bigcup_{i=1}^m V_i$ be a union of optimal plaques (resp. optimal leaf segments) of $X$ for some $m \in \mathds{N}$. We call $\mathcal{V}$ a ``Concatenation of Plaques in $X$" (resp. a ``Concatenation of Leaf Segments in $X$") if, when $m \neq 1$, each $i \in \{1,...,m-1\}$, $V_i \cap V_{i+1}$ contains a leaf segment. 
\end{definition}


\vspace{4mm}

\begin{definition}[Plaque Equivalence and Leaf Equivalence]
    \label{Leaf Equivalence}
    \hypertarget{Leaf Equivalence}{For} each $x,y \in X$, we say that ``$x$ and $y$ are Plaque Equivalent" if there exists a \hyperlink{Concatenations}{concatenation of plaques} in $X$ that contains both $x$ and $y$.
    For each $x,y \in X - Sing(X)$, we say that ``$x$ and $y$ are Leaf Equivalent" if there exists a \hyperlink{Concatenations}{concatenation of leaf segments} in $X$ that contains both $x$ and $y$.
\end{definition}

\vspace{4mm}

\begin{lemma}
    Plaque equivalence is an equivalence relation on $X$.
\end{lemma}

\begin{proof}
    The only criterion for being an equivalence relation that doesn't directly follow from the definition of plaque equivalence is transitivity. Let $x,y \in X$ (resp. $y,z \in X$) be contained in a finite concatenation of plaques $\bigcup_{i=1}^m V_i$ for some $m \in \mathds{N}$ (resp. $\bigcup_{i=1}^{m'} V'_i$ for some $m' \in \mathds{N}$). Without loss of generality we may assume that $x \in V_1$, $y \in V_m$, $y \in V'_1$ and $z \in V'_{m'}$. 

    \vspace{4mm}

    (Case a.) Suppose $y$ is a singularity. Then there is a unique \hyperlink{Optimal Neighborhood}{level $0$ optimal neighborhood} $O$ containing $y$, and a unique plaque $V$ of $O$ containing $y$. So we have $V = V_m = V'_1$.

    \vspace{4mm}

    (Case b.) Now suppose $y$ is not a singularity. Then since intersecting optimal plaques in $X$, intersect at a disjoint union of leaf segments [Remark \ref{intersecting plaques}], there exists a leaf segment $L$ in $X$ such that $y \in L \subseteq  V_m \cap V'_1$. 
\end{proof}

\vspace{4mm}

\begin{lemma}
    Leaf equivalence is an equivalence relation on $X - Sing(X)$.
\end{lemma}

\begin{proof}
    The only criterion for being an equivalence relation that doesn't directly follow from the definition of leaf equivalence is transitivity. Let $x,y \in X - Sing(X)$ (resp. $y,z \in X - Sing(X)$) be contained in a finite concatenation of leaf segments $\bigcup_{i=1}^m L_i$ for some $m \in \mathds{N}$ (resp. $\bigcup_{i=1}^{m'} L'_i$ for some $m' \in \mathds{N}$). Without loss of generality we may assume that $x \in L_1$, $y \in L_m$, $y \in L'_1$ and $z \in L'_{m'}$.

    \vspace{4mm}

    Since any two level $0$ optimal leaf segments that intersect, intersect at a disjoint union of finitely many leaf segments [Remark \ref{Two Intersecting Turns in a Level Graph}], there exists a leaf segment $L$ in $X$ such that $y \in L \subseteq L_m \cap L'_1$. 
\end{proof}

\vspace{4mm}

\begin{definition}[Leaves, and Partial Leaves]
    \label{Leaves, Partial Leaves}
    \hypertarget{Leaves, Partial Leaves}{$L \subseteq X$} is called a ``Leaf of $X$" if $L$ is a plaque equivalence class of $X$. $L \subseteq X - Sing(X)$ is called a ``Partial Leaf of $X$" if $L$ is a leaf equivalence class of $X - Sing(X)$.
\end{definition}

\vspace{4mm}


The next definition gives a name to the particular unions of concatenations of plaques (resp. leaf segments) that form leaves (resp. partial leaves) of the solenoid. \hypertarget{maximal}{Recall} that given a set $Y$, a property $P$ that can be potentially satisfied by some subsets of $Y$, and $S \subseteq Y$ that satisfies $P$, $S$ is said to be ``Maximal in $Y$ with respect to $P$" if there are no subsets of $Y$ satisfying $P$ that properly contain $S$. 

\vspace{4mm}

\begin{definition}[Webs of Leaf Segments and Webs of Plaques]
    \label{def: Webs}
    \hypertarget{def: Webs}{Let} $L \subseteq X$. $L$ is called a ``Web of Plaques in $X$" (resp. a ``Web of Leaf Segments in $X$") if,
    \begin{enumerate}
        \item for each $x,y \in L$, there exists a concatenation of plaques (resp. leaf segments) $V$ such that $x,y \in V \subseteq L$, and,
        \item $L$ is \hyperlink{maximal}{maximal} in $X$ with respect to the property stated in item 1.
    \end{enumerate}
\end{definition}

\vspace{4mm}

Note that the collection of webs of plaques in $X$ partition $X$. Meanwhile the collection of webs of leaf segments in $X$ partition $X- Sing(X)$. 

\vspace{4mm}

\begin{remark}
    The two facts below follow immediately from the above definition.
    \begin{enumerate}
        \item Let $x, y \in X$. $x$ and $y$ are in the same leaf of $X$ $\Longleftrightarrow$ $x$ and $y$ are in the same web of plaques in $X$.
        \item Let $x, y  \in X - Sing(X) $. $x$ and $y$ are in the same partial leaf of $X$ $\Longleftrightarrow$ $x$ and $y$ are in the same web of leaf segments in $X$.
    \end{enumerate}
\end{remark}

\vspace{4mm}

In the next sub-section, we will explore further how leaves and partial leaves are embedded in the solenoid. 

\vspace{4mm}

\subsubsection{The Shapes of Leaves and Partial Leaves}
\label{sec: The Shapes of Leaves and Partial Leaves}

\vspace{4mm}

The two main results here [Lemma \ref{leaves vs path components} and Lemma \ref{leaves are immersed graphs}], will show that leaves (resp. partial leaves) are the path components of $X$ (resp. $X - Sing(X)$), and that leaves (resp. partial leaves) are immersed graphs (resp. $1-$manifolds) in $X$ (resp. $X - Sing(X)$).

\vspace{4mm}




\begin{lemma}
    \label{leaves vs path components}
    Let $L \subseteq X$. $L$ is a path component of $X$ (resp. of $X - Sing(X)$) if and only if $L$ is a leaf (resp. a partial leaf) of $X$.
\end{lemma}

\begin{proof}
    Recall that the collection of \hyperlink{def: Webs}{webs of plaques} (resp. \hyperlink{def: Webs}{webs of leaf segments}) in $X$ partition $X$ (resp. $X - Sing(X)$) and that each web of plaques (resp. leaf segments) is by design path connected in $X$ (resp. in $X -Sing(X)$). If two distinct webs of plaques (resp. leaf segments) $L$ and $L'$ had a path $\gamma$ connecting a point $x \in L$ and $y \in L'$, then $Image(\gamma)$ will necessarily have to lie in a concatenation of plaques (resp. leaf segments) $V \subseteq L \cap L'$, contradicting that $L$ and $L'$ are distinct. Thus we have that for each pair of leaves (resp. partial leaves) $L , L'$ that are contained in the same path component $C$ of $X$ (resp. of $X - Sing(X)$), $L = L' = C$. 
\end{proof}



\vspace{4mm}

In Definition \ref{Graph Immersion} we will establish, for each leaf $L$, a graph $\Gamma_L$ that maps onto $L$ under a bijective immersion.
But first, we lay out a method of identifying a topological graph in a way that'll prove to be convenient in this context. By a ``Topological Graph" what we mean is the underlying topological space of a graph.

\vspace{4mm}

\begin{remark}[The Topology of a Graph]
    \hypertarget{maximal star sets}{Recall} the definition of a \hyperlink{star set}{star set} from definition \ref{Star Set}. We say that a star set $V$ belonging to a topological space $Y$ is ``Prong-Maximal in $Y$" if there are no other star sets in $Y$ containing $V$, that have a higher number of \hyperlink{Turn and Prong}{prongs} than $V$. Now let $Y$ be a second countable Hausdorff space. Then, $Y$ is a topological graph if and only if the topology of $Y$ can be generated by the collection of prong-maximal star sets of $Y$. 
\end{remark}

\vspace{4mm}


\begin{definition}[The Graph-Immersion of $L$]
    \label{Graph Immersion}
    \hypertarget{Graph Immersion}{Let} $L$ be a leaf or a partial leaf of $X(\zeta)$. The ``The Graph-Immersion of $L$" denoted ``$\Gamma_L$" is the topological graph given by the set $L$ with the topology generated by the collection of $L$'s \hyperlink{maximal star sets}{prong-maximal star sets}. (In other words, to construct $\Gamma_L$, we first take $L$ merely as a set, but we use $L$'s subspace topology to determine the collection $\mathcal{B}$ of $L$'s prong-maximal star sets, then forget the subspace topology, and implement on $L$ the set, the topology generated by $\mathcal{B}$.) 
\end{definition}

\vspace{4mm}

\hypertarget{Topological Immersion}{Note} that a continuous map $f:Y \longrightarrow Z$ is called a ``Topological Immersion", if for each $y \in Y$, there exists a neighborhood $U$ of $y$ in $Y$ such that $f|_U : U \longrightarrow f(U)$ is a homeomorphism.

\vspace{4mm}

\begin{lemma}
    \label{leaves are immersed graphs}
    Let $L$ be a leaf of $X(\zeta)$. The identity map $i :$ \hyperlink{Graph Immersion}{$ \Gamma_L$} $ \longrightarrow L$ is a \hyperlink{Topological Immersion}{topological immersion}.
\end{lemma}

\begin{proof}
    For each $x \in \Gamma_L$, choose an \hyperlink{Optimal Neighborhood}{optimal neighborhood} $O_x$ of $i(x)$. Let $\Tilde{V}_x$ be the \hyperlink{Optimal Star Neighborhoods}{optimal star neighborhood} in $G_0$ such that $O_x = \pi^{-1}_0 (\Tilde{V}_x)$. Let $V := Plaque(O_x,x)$. $V$ is a prong-maximal star set in $X$ and thus is open in $\Gamma_L$. We aim to show that $i|_{V}$ from $V \subset \Gamma_L$ to $V \subset L$ is a homeomorphism. Note that $\mathcal{B} := \{ O \cap V : O$ is a standard neighborhood of $X\}$ generates the topology of $V \subset L$ and $\mathcal{B}_{\Gamma} := \{ U : U$ is a prong-maximal star set contained in $V\}$ generates the topology of $V \subset  \Gamma_L$. We aim to show that $\mathcal{B}_{\Gamma} = \mathcal{B}$.

    \vspace{4mm}

    Let $U$ be a prong-maximal star set contained in $V$ (i.e. $U \in \mathcal{B}_{\Gamma}$). If $\Tilde{V} := \pi_0 (U)$ does not contain a natural vertex, then $\Tilde{V}$ is a $2-$pronged star neighborhood such that $ U = \pi^{-1}_0 (\Tilde{V}) \cap V \in \mathcal{B}$. Now suppose  $\pi_0 (U)$ contains a natural vertex. Then choose a \hyperlink{neighbourhood_completion}{standard neighborhood completion} $\Tilde{V}$ of $\pi_0 (U)$ in $\Tilde{V}_x$. Then we have $ U = \pi^{-1}_0 (\Tilde{V}) \cap V \in \mathcal{B}$.

    \vspace{4mm}

    Now let $O$ be a standard neighborhood such that $W := O \cap V \neq \emptyset$ (i.e. $W \in \mathcal{B}$). Recall that each plaque of a standard neighborhood is the inverse limit of a maximal star chain [Prop. \ref{shapse of plaques 3}], and therefore is a prong-maximal star set of $X$. If $W$ does not contain a singularity, then $W$ is a leaf segment and is thus a prong-maximal star set in $V$. If $W$ contains a singularity $x$, both $W$ and $V$ share $x$ as their center, and each singular plaque of $x$ is a star set with the same number of prongs. Since $V$ is prong-maximal, so is $W$. Thus $W \in \mathcal{B}_{\Gamma}$.
\end{proof}

\vspace{4mm}

From the above result it quickly follows that each partial leaf of $X$ is a $1-$manifold immersed in $X - Sing(X)$. The following two remarks elaborate on this fact. 

\vspace{4mm}

\begin{remark}[The Topology of a 1-Manifold]
    Let $Y$ be a second countable Hausdorff space. $Y$ is a one dimensional manifold if and only if the topology of $Y$ is generated by its collection of embedded open intervals. 
\end{remark}

\vspace{4mm}



Let $\Tilde{L}$ be a leaf of $X$. Then natural vertices of $\Gamma_{\Tilde{L}}$ correspond exactly to the singularities of $X$ contained in $\Tilde{L}$. Let $L$ be a path component of $\Tilde{L} - Sing(X)$. Then $L$ is a partial leaf of $X$. 

\vspace{4mm}

\begin{remark}[Partial Leaves are Immersed 1-Manifolds in $X$]
    \label{Partial Leaves are Immersed 1-Manifolds in X}
    Note that, given a graph $\Gamma$, any path component of $\Gamma - \{$The natural vertices of $\Gamma \}$ is a one dimensional manifold. Let $L$ be a partial leaf. Let $\Tilde{L}$ be the unique leaf containing $L$. Then $\Gamma_L$ is the set $L$ along with the subspace topology given to it by $\Gamma_{\Tilde{L}}$. Since a partial leaf, by definition, contains no singularities, each prong-maximal star set contained in $L$ is $2-$pronged (i.e. an open interval). So, $\Gamma_L$ is a $1-$manifold.
    Furthermore, the identity map from $\Gamma_L$ to $L$ is a restriction of the identity map from $\Gamma_{\Tilde{L}}$ to $\Tilde{L}$, that was proven to be an immersion in Lemma \ref{leaves are immersed graphs}. Thus, $\Gamma_L$ is a one dimensional manifold that maps onto $L$ under a bijective immersion. 
\end{remark}

\vspace{4mm}




\vspace{4mm}

\begin{remark}[Types of Partial Leaves]
    \label{Types of Partial Leaves}
    \hypertarget{Types of Partial Leaves}{Let} $L$ be a partial leaf of $X$.
    It follows from Lemma \ref{how turn atlas charts intersect} and Remark \ref{Partial Leaves are Immersed 1-Manifolds in X}, that there exist,
    \begin{enumerate}
        \item[i.] 
        a countable index set $\mathds{J} \subseteq \mathds{Z}$ such that, $\mathds{J} \in \{ \mathds{Z} , \mathds{N} \} \bigcup \{ \{ 1, ... , m \} \}_{m \in \mathds{N}}$, and,
        \item[ii.] a collection of optimal leaf segments $\{ L^i \}_{ i \in \mathds{J}}$,
    \end{enumerate}
    such that, 
    \begin{enumerate}
        \item for each $i , i+1 \in \mathds{J}$, $L^i \cap L^{i+1}$ contains a leaf segment in $X$, and,
        \item $L = \bigcup_{ i \in \mathds{J}} L^i$.
    \end{enumerate}
    In the above context, if $\mathds{J} \in \{ \{ 1, ... , m \} \}_{m \in \mathds{N}}$, we call $L$ a ``Finite Partial Leaf with Level-$0$-length $m$". If $\mathds{J} \in \{ \mathds{Z} , \mathds{N} \}$, we call $L$ an ``Infinite Partial Leaf". More specifically, when $\mathds{J} = \mathds{N}$ we say that $L$ is a ``One-Sided Infinite Partial Leaf" and when $\mathds{J} = \mathds{Z}$ we say that $L$ is a ``Bi-Infinite Partial Leaf".
\end{remark}


\vspace{4mm}

This understanding of leaves, and specially partial leaves, will be heavily utilized in defining transversals (in the next section) and in discussing various concepts of minimality (in the upcoming chapters). 


\subsection{Transversals}
\label{sec: Transversals}

\vspace{4mm}


In the context of surface foliations, a transversal is defined as a path segment that is transverse to each leaf it intersects. A surfaces upon which a foliation is typically defined, has an underlying differentiable structure, and thus has a well defined concept of two path segments being transverse to each other. In the context of solenoids however, we work in a purely topological setting. And thus we will use the term ``transverse" with a broader meaning. 

\vspace{4mm}

Throughout this section, we consider $\zeta$ to be a \hyperlink{Stabilized Split Sequence}{stabilized} proper split sequence and $X = X(\zeta)$ its induced solenoid.

\vspace{4mm}

\subsubsection{Transversals and Transverse Measures}

\vspace{4mm}

Recall that $X - Sing(X)$ is covered by turn tunnels. And each turn tunnel, being a product space of a totally disconnected set and an open interval, has an inherent notion of transversality. Namely, in a tunnel set parameterized by $C \times I$ (where $C$ is a totally disconnected set and $I$ an open interval) we can consider any cross-section $C \times \{ t \}$ (for some $t \in I$) as being transverse to any interval shaped section $\{ c\} \times I$ (for some $c \in C$). It is worth noting that a cross-section is merely a one type of transversal where the general definition of a transversal should be much broader. 

\vspace{4mm}

Before we can define transversals in general, we lay out a restricted definition.

\vspace{4mm}

\begin{definition}[Turn Transversals]
    \label{Turn Transversals}
    \hypertarget{Turn Transversals}{Let} $I$ be a \hyperlink{turn in zeta}{standard turn in $\zeta$} and $q \in I$. Then the ``Turn Transversal rel $I,q$ of $X(\zeta)$" denoted ``$H_{I,q}$" is given by $h_I (C_I \times \{q\} )$ where $h_I$ is the \hyperlink{Canonical Tunnel Parameterizations}{the canonically given tunnel parameterization} of $Tun_I$, and $C_I$ is the central cross-section of $Tun_I$.
\end{definition}


\vspace{4mm}


\begin{definition}[Flow Map]
    \label{Flow Map}
    \hypertarget{Flow Map}{Let} $I$ be a \hyperlink{turn in zeta}{standard turn in $\zeta$}, $h_I : C_I \times I \longrightarrow Tun_I$ the canonically given tunnel parameterization of $Tun_I$, where $C_I$ is the central cross-section of $Tun_I$. Then a homeomorphism $\phi : C_I \longrightarrow Tun_I$ is called a ``Flow Map of $Tun_I$ in $X(\zeta)$" if there exists a continuous map $\sigma : C_I \longrightarrow I$ such that, for each $y \in C_I$, $\phi (y) = h_I (y, \sigma(y))$. 
\end{definition}


\vspace{4mm}

\begin{definition}[Transversals]
    \label{Transversals}
    \hypertarget{Transversals}{A} set $F \subset X(\zeta)$ is called a ``Transversal of $X(\zeta)$" if there exist,
    \begin{enumerate}
        \item a \hyperlink{turn in zeta}{standard turn $I$ in $\zeta$}, where $C_I$ denotes the central cross-section of $Tun_I$,
        \item a \hyperlink{Flow Map}{flow map} $\phi : C_I \longrightarrow Tun_I$, and,
        \item a borel subset $H$ of $C_I$, 
    \end{enumerate}  
    such that $\phi (H) = F$.
\end{definition}


\vspace{4mm}


\begin{definition}[Flow Equivalence]
    \label{Flow Equivalence}
    \hypertarget{Flow Equivalence}{Let} $F$ and $F'$ be two transversals of $X$. A homeomorphism $\phi : F \longrightarrow F'$ is called a ``Flow Map Composition in $X$ from $F$ to $F'$" if 
    there exist, 
    \begin{enumerate}
        \item $m \in \mathds{N}$,
        \item a collection of transversals $\{F_i\}_{i=0}^m$ with $F_0 = F$, $F_m = F'$, where, for each $i \in \{1,...,m-1\}$, $F_i$ is contained in a turn transversal,
        \item a collection of maps $\{ {\Phi}_i \}_{i=1}^m$ where, for each $i=1,...,m$, ${\Phi}_i$ is either a flow map in $X$ or the inverse of a flow map in $X$, with $\phi_i := \Phi_i |_{F_{i-1}}$ taking $F_{i-1}$ to $F_i$,
    \end{enumerate}
    such that $\phi = \phi_m \circ ... \circ \phi_1$.
    \hypertarget{flow equivalence}{If two transversals} $F, F'$ have a flow map composition taking one of them to the other, then we say that $F$ and $F'$ are ``Flow Equivalent to Each Other".
\end{definition}

\vspace{4mm}

Note that flow equivalence in fact is an equivalence relation among transversals of $X$. We shall use this equivalence relation to define a notion of a transverse measure on $X$ akin to transverse measures on singular foliated surfaces. 

\vspace{4mm}

\subsubsection*{Transverse Measures}

\vspace{4mm}

We will first define the notion of a ``Transverse Semi-Measure on $X$" [Definition \ref{Transverse Semi-Measures}], which comprises of a borel measure defined on each \hyperlink{Transversals}{transversal} so that those measures are preserved under \hyperlink{Flow Equivalence}{flow equivalence} (which will be made precise in item 1. of the definition), and an atomic mass given to each singularity, so that each $singular \ mass$ is compatible with any atomic masses carried by each point on the prongs attached to the corresponding singularity (which will be made precise in item 2. of the definition). Then in Definition \ref{Transverse Measures}, we will introduce one more criterion that should be satisfied by a transverse semi-measure in order to be called a ``Transverse Measure on $X$". 

\vspace{4mm}

\begin{definition}[Transverse Semi-Measures on X]
    \label{Transverse Semi-Measures}
    \hypertarget{support of a tm}{Let} $\mathcal{T} (X) := \{ \{x\} : x \in Sing(X) \} \ \bigcup \ $ the collection of all transversals of $X$. \hypertarget{sigma algebra of a transversal}{Furthermore}, for each transversal $H$ of $X$, let ``$\Tilde{\mathcal{B}}_H$" denote the borel sigma algebra of $H$.
    \hypertarget{transverse semi-measure}{A ``Transverse Semi-Measure"} on $X$
    is a function $\mu : \mathcal{T} (X) \longrightarrow [0, \infty)$, such that,
    \begin{enumerate}
        \item for each transversal $H$ of $X$, $\mu$ restricted to $\Tilde{\mathcal{B}}_H$, denoted ``$\mu_H$", is a non-negative finite borel measure on $H$,
        \item for each pair of transversals $H , H'$ of $X$ that are flow equivalent to each other, and for each flow map composition $\phi$ taking $H$ to $H'$, $\phi$ is measure preserving with respect to $\mu_H , \mu_{H'}$, and,
        \item for each singularity $s$ of $X$, 
        $\mu (s) := \mu (\{s\})$ is a non-negative real number, that satisfies the following: 
        for each singular plaque $V$ of $X$ centered at $s$, and for each subset $A \subset V - \{x\}$ such that, for each open prong $P$ of $V$, $A$ contains exactly one point from $P$, we have 
        $\mu (s) = ( \sum_{y \in A} \mu(\{y\}) ) / 2$.
    \end{enumerate}
\end{definition}

\vspace{4mm}

In order for a \hyperlink{transverse semi-measure}{transverse semi-measure} $\mu$ on $X$ to be called a \hyperlink{transverse measure}{transverse measure}, we require that $\mu$ be compatible with something called a ``Turn Weight System on $X$" [Definition \ref{Turn Weight System}], which is designed to capture more constraints from the structure of $X$ that a semi-measure would not necessarily adhere to. But first we need to lay out some preliminary terminology.




\vspace{4mm}


\begin{definition}[Prong-Germs]
    \label{Prong-Germs}
    \hypertarget{Prong-Germs}{Let} $x$ be a singularity of $X$. For any singular plaque $V$ of $X$ centered at $x$, if an open prong of $V$ is denoted $P$, then let the closure of $P$ in $V$ be denoted ``$Cl_V (P)$". 
    Given two singular plaques $V$ and $V'$ in $X$ centered at $x$, and two open prongs $P$ and $P'$ respectively of $V$ and $V'$, we say that ``$P$ and $P'$ are Prong-Germ Equivalent to Each Other via $x$" if, $Cl_V (P) \cap Cl_{V'} (P')$ is an embedded half-closed interval in $X$ that contains $x$. (When the context is clear, we will forego mentioning the singularity $x$.) Note that on the collection of all open prongs of all singular plaques centered at $x$, the prong-germ equivalence is an equivalence relation. Each equivalence class of the prong-germ equivalence relation shall be called a ``Prong-Germ in $X$ Attached to $x$".
\end{definition}

\vspace{4mm}

\begin{definition}[Turn Weight System]
    \label{Turn Weight System}
    \hypertarget{Turn Weight System}{For} each singularity $s$ of $X$, let the set of \hyperlink{Prong-Germs}{prong-germs} attached to $s$ be denoted ``$\mathcal{P}_s$". 
    Given a singularity $s$ of $X$, a ``Turn Weight Function on $s$" is a function $W^s : \mathcal{P}_s \times \mathcal{P}_s \longrightarrow [0, \infty)$, such that,
    \begin{enumerate}
        \item for each $P \in \mathcal{P}_s$, $W^s (P,P) = 0$, and,
        \item for each $P,P' \in \mathcal{P}_s$, $ W^s (P,P') = W^s (P',P)$
    \end{enumerate}
    A ``Turn Weight System on $X$" is a function $W$ on the collection $\bigsqcup_{x \in Sing(X)} (\mathcal{P}_x \times \mathcal{P}_x)$ such that, $W$ restricted to $\mathcal{P}_{s} \times \mathcal{P}_{s}$ for any $s \in Sing(X)$, is a turn weight function on $s$.
\end{definition}

\vspace{4mm}

\begin{definition}[Transverse Measures on $X$]
    \label{Transverse Measures}
    \hypertarget{transverse measure}{Given} a \hyperlink{transverse semi-measure}{transverse semi-measure} $\mu$ on $X$ and a \hyperlink{Turn Weight System}{turn weight system} $W$ on $X$, we say that ``$\mu$ is Compatible with $W$" if, for each singularity $x$, for each prong-germ $P'$ attached to $x$, and for each point $y$ in any prong in $P'$, we have $\mu (\{y\} ) =  \sum_{P \in \mathcal{P}_x} W (P' , P)$.
    We say that a transverse semi-measure $\mu$ on $X$ is a ``Transverse Measure on $X$" if there exists a turn weight system on $X$ that is compatible with $\mu$.
\end{definition}

\vspace{4mm}


Note that if we're in a context in which a transverse measure on a solenoid does not carry atomic masses, the turn weights mentioned in the above definition will all have to be zero, and the transverse measure is, in a sense, supported by only the non-singular potion of the solenoid. Thus for a transverse measure $\mu$ that doesn't carry atomic masses, $\mu$ is entirely determined by 
all borel measures $\mu$ endows on transversals.

\vspace{4mm}

\begin{definition}[Non Atomic Transverse Measures, and Non Singular Transverse Measures]
    \label{Non Atomic TM}
    \hypertarget{Non Atomic TM}{A} ``Non Atomic Transverse Measure on $X$" is a transverse measure $\mu$ on $X$ that does not support atomic masses (i.e. for each $y \in X$, $\mu (\{y\}) = 0$). A ``Non Singular Transverse Measure on $X$" 
    is a transverse measure $\mu$ on $X$ such that, for each $s \in Sing(X)$, $\mu(s) = 0$.
\end{definition}


\vspace{4mm}

Note that non atomic transverse measures are non singular. However, there can exist non singular solenoids that may support atomic masses on finite leaves that do not contain any singularities.

\vspace{4mm}


In the next chapter, we will introduce a type of solenoid called $expanding$. We will show that, if $X$ is expanding, then $X$ does not have any finite leaves, and therefore, each transverse measure on $X$ is non-atomic, and non-singular [Prop. \ref{Expanding solenoids do not support atomic masses}]. Thus, to study transverse measures on expanding solenoids, it's reasonable to focus our attention on the non-singular potion of the solenoid. 

\vspace{4mm}

\begin{notation}[The Space of Transverse Measures]
    \label{def: TM(X)}
    \hypertarget{def: TM(X)}{The} Space of all transverse measures on $X$ is denoted ``$TM(X)$".
\end{notation}

\vspace{4mm}

While it is clear what $TM(X)$ is as a set, we have to clarify what we mean precisely when we use the word ``Space" in this context. The elements of $TM(X)$ have an inherit notion of addition, and scalar multiplication by non-negative reals. We will show in Chapter \ref{Ch: TM(X)} that $TM(X)$ has the structure of a convex cone inside a vector space. 

\vspace{4mm}

The discussions in Chapters \ref{ch: Minimality and Mingling} and \ref{Ch: TM(X)}, will lay the groundwork for establishing criteria that ensures a solenoid's minimality (in a topological context in Chapter \ref{ch: Minimality and Mingling}, and then later in Chapter \ref{Ch: A Criterion for Unique Ergodcity}, in a measure theoretic context). 

\vspace{4mm}

Since the discussions of minimality conducted in future chapters rely on a through understanding of transversals, we shall spend the rest of this chapter, laying out some useful properties of transversals.




\vspace{4mm}


\subsubsection{Properties of Transversals}
\label{sec: Properties of Transversal}

\vspace{4mm}

Here, we explore properties of transversals, giving special attention to turn transversals. The following is a list of results included in this section.
\begin{itemize}
    \item Transversals of $X$ are totally disconnected [Property \ref{Transversals are totally disconnected}].
    \item Each cross-section of an extended turn tunnel is compact [Lemma \ref{ExTun-CS are compact}], and therefore, each turn transversal of $X$ is either compact, or has a one point compactification [Property \ref{H has a point compactification}].
    \item The topology of each turn transversal $H$ of $X$, is generated by the collection of its sub turn transversals [Property \ref{Topology of a Turn Transversal}] called the ``Sub-Turn Basis of $H$" [Definition \ref{Sub-Turn Basis}]. 
    \item Sub turn bases of transversals fall under a general class of bases that we shall call ``Tree Bases for Topological Spaces" [Definition \ref{Tree Bases}]. Remark \ref{Properties of Sub-Turn Bases} lays out a list of straight forward properties of sub turn bases that shall be utilized in later chapters.
\end{itemize}

\vspace{4mm}

\begin{property}
    \label{Transversals are totally disconnected}
    Transversals of $X$ are totally disconnected.
\end{property}

\begin{proof}
    Since each transversal of $X$ is homeomorphic (via a flow map composition) to a borel subset of a turn transversal, it is enough to show that each turn transversal of $X$ is totally disconnected. Each turn transversal is contained in a fiber and we have from Property \ref{Fibers are totally disconnected} that fibers are totally disconnected.
\end{proof}

\vspace{4mm}

In the study of transverse measures on $X$, a convenient restriction, is limiting our investigations to turn transversals. Thus, the rest of this sub-section is focused on properties of turn transversals. 


\vspace{4mm}

Property \ref{H has a point compactification} states that each turn transversal has a one point compactification in $X$, and the following lemma will lay out a supporting argument for this property. 

\vspace{4mm}

\begin{lemma}
    \label{ExTun-CS are compact}
    Let $I$ be a standard turn in $\zeta$. Each cross-section of $ExTun_I$ is compact.
\end{lemma}

\begin{proof}
     Let $K \in \n$, $\Tilde{V}_K$ a level $K$ standard star neighborhood in $\zeta$ centered at $p \in G_K$, and $I$ a turn of $\Tilde{V}_K$. Let $q \in I$ and consider the cross section of $ExTun_I$ given by $C_{I,q} := ExTun_I \bigcap F_q$.

     \vspace{4mm}

     (Case i.) Suppose $\Tilde{V}_K$ is $2-$pronged. Then $I=\Tilde{V}_K$ is contained in the interior or a natural edge of $G_K$. Then from Property \ref{Fibers are compact}, we have that $C_{I,q} = F_q$ is compact.

     \vspace{4mm}

     (Case ii.) Suppose that $\Tilde{V}_K$ has more than $2$ prongs. Then $p$ is a natural vertex of $G_K$. Furthermore, (Case ii-a.) Suppose each pre-leaf segment that sustains $I$ is a leaf segment. Then $ExTun_I = Tun_I$ where $O := \pi^{-1}_K (\Tilde{V}_K) = \hyperlink{ST(O)}{ST(O)} \bigsqcup Tun_I \bigsqcup T$, where $T$ is either empty or is a finite disjoint union of turn tunnels [Prop. \ref{sing stnd nbhds vs star tunnel nbhds}]. Recall that $ST(O)$ is open in $X$ [Prop. \ref{sing stnd nbhds vs star tunnel nbhds}], and that each turn tunnel is open in $X$ [Lemma \ref{turn tunnels are open}]. Then we have $ F_q = (F_q \cap Tun_I) \bigsqcup (F_q \cap ST(O)) \bigsqcup (F_q \cap T) = C_{I,q} \bigsqcup$ an open set in $F_q$. Thus, $C_{I,q}$ is closed in $F_q$. Now (Case ii-a.) suppose that, there exists a singular segment that sustains $I$. Let $m \in \mathds{N}^*$ be the number of turns of $\Tilde{V}_K$ that contains $q$, and let $\{ I_i \}_{i=1}^{m-1}$ be the collection of turns of $\Tilde{V}_K$ that contains $q$ that isn't $I$. Then $F_q = (C_{I,q} )\bigsqcup (\bigsqcup_{i=1}^{m-1} H_{I_i , q} )$ where $\bigsqcup_{i=1}^{m-1} H_{I_i , q} $ is a union of open sets in $F_q$. Thus $C_{I,q}$ is closed in the compact set $F_q$.


\end{proof}

\vspace{4mm}

\begin{property}
    \label{H has a point compactification}
    Let $H$ be a turn transversal of $X$. Then $H$ is either compact or has a one point compactification in $X$.
\end{property}

\begin{proof}
    Let $I$ be a standard turn in $\zeta$ and $H$ a cross section of $Tun_I$.
    The cross section $\overline{H}$ of $ExTun_I$ that contains $H$ is compact [from Lemma \ref{ExTun-CS are compact}], and since $\zeta$ is stabilized,  $\overline{H}$ contains at most one point that isn't already contained in $H$.
\end{proof}



\vspace{4mm}

\begin{property}
    \label{Topology of a Turn Transversal}
    Let $H$ be a turn transversal of $X$.
    Then the topology of $H$ is generated by the collection of turn transversals contained in $H$.
\end{property}

\begin{proof}
    Let $K \in \n$, $I$ a level $K$ \hyperlink{turn in zeta}{standard turn in $\zeta$} and $p \in I$. Consider the \hyperlink{Turn Transversals}{turn transversal rel $I,p$}, $H_{I,p} =: H$. Let $Q_p := \bigcup_{j \in -\mathds{N} + K} (f^j_K)^{-1} (p)$. From being a subspace of $X$, the topology of $H$ is generated by the collection $ \mathcal{B} := \{ \hspace{0.8mm} \hyperlink{fibers}{F_q} \hspace{0.8mm} \cap \hspace{0.8mm} H \hspace{0.8mm} \}_{q \in Q_p}$. Note that since each turn tunnel is open in $X$ [Lemma \ref{turn tunnels are open}], each turn transversal contained in $H$ is open in $H$. To show that the collection of turn transversals contained in $H$ generates its topology, all that remains to be shown is that each element of $\mathcal{B}$ can be written as a union of turn transversals. 

    \vspace{4mm}

    Now let $j \in -\mathds{N} + K$ and $q \in (f^j_K)^{-1} (p)$ 
    such that $F_q \cap H \neq \emptyset$. 
    The aforementioned intersection being non-empty implies that there exists at least one level $j$ \hyperlink{pre-turns}{pre-turn} of $I$ containing $q$. Let $\{ I_i \}_{i=1}^m$, $m \in \mathds{N}$ be the collection of all level $j$ pre-turns of $I$ that contains $q$. Then $H \cap F_q = \bigsqcup_{i=1}^m \hyperlink{Turn Transversals}{H_{I_i , q}}$.
\end{proof}

\vspace{4mm}

\begin{definition}[Sub-Turn Basis]
    \label{Sub-Turn Basis}
    \hypertarget{Sub-Turn Basis}{Let} $H$ be a turn transversal of $X$. Then the collection of turn transversals contained in $H$ shall be called the ``Sub-Turn Basis of $H$ in $X$" and denoted ``$\mathcal{B}_H$".
\end{definition}

\vspace{4mm}

We now define a general class of bases for a topological spaces, that captures the special properties possessed by a sub-turn basis. 

\vspace{4mm}

\begin{definition}[Tree Bases for Topological Spaces]
    \label{Tree Bases}
    \hypertarget{Tree Bases}{Let} $Y$ be a set. A ``Tree Partition System for $Y$" is a sequence of finite partitions of $Y$, $\{ \mathcal{P}_j \}_{j \in \mathds{N}^*} = \{ \{ P^1_j , ... , P^{m_j}_j \} \}_{j \in \mathds{N}^*}$ where $ \{ m_j \}_{j \in \mathds{N}^*}$ is a monotonically non-decreasing sequence of positive integers, and,
    \begin{enumerate}
        \item $\mathcal{P}_0 = \{ \hspace{0.8mm} Y \hspace{0.8mm} \}$, 
        \item for each $j \in \mathds{N}^*$, 
        \begin{enumerate}
            \item $\bigsqcup_{i=1}^{m_j} P^i_j = Y$, and,
            \item there exists a surjection $\sigma_j : \mathcal{P}_{j+1} \longrightarrow \mathcal{P}_j$ such that, for each $\Tilde{P} \in \mathcal{P}_j$, we have $\Tilde{P} = \bigsqcup_{P \in (\sigma_j^{-1} (\Tilde{P}))} P$.
        \end{enumerate}
    \end{enumerate}
    If in the above definition, for each $j \in \mathds{N}^*$,
    each $\Tilde{P} \in \mathcal{P}_j$
    has at most two pre-images under $\sigma_j$,
    then we call $\{ \mathcal{P}_j \}_{j \in \mathds{N}^*}$ a ``Binary Tree Partition System for $Y$". 
    Now suppose $Y$ is a topological space, and $\mathcal{B}$ is a basis for its topology. If there exists a tree partition system $\{ \mathcal{P}_j \}_{j \in \mathds{N}^*}$, such that $\mathcal{B} = \bigcup_{j \in \mathds{N}^*} \mathcal{P}_j$, we call $\mathcal{B}$ a ``Tree Basis for the Topology of $Y$" and we call $\{ \mathcal{P}_j \}_{j \in \mathds{N}^*}$ a ``Tree Partition Breakdown of $\mathcal{B}$". 
\end{definition}

\vspace{4mm}

Note that, for each \hyperlink{fibers}{fiber} $F$ of $X$ (resp. for each turn transversal $H$ of $X$), the \hyperlink{Sub-Fiber Basis}{sub-fiber basis} of $F$ (resp. the sub-turn basis of $H$), is a tree basis for $F$ (resp. $H$), which can be refined into a binary tree partition system.

\vspace{4mm}

\begin{remark}[Properties of Sub-Turn Bases]
    \label{Properties of Sub-Turn Bases}
    \hypertarget{Properties of Sub-Turn Bases}{Let} $H$ be a turn transversal of $X$. Then \hyperlink{Sub-Turn Basis}{$\mathcal{B}_H$} is a tree basis for the topology of $H$, and has a binary tree partition breakdown. Thus $\mathcal{B}_H$ has the following useful properties:
    \begin{enumerate}
        \item $H \in \mathcal{B}_H$,
        \item for each $A \in \mathcal{B}_H$, $H-A$ is a finite union of elements of $\mathcal{B}_H$,
        \item for each $A, B \in \mathcal{B}_H$, either $A \cap B = \emptyset$ or one of $A, B$ is contained in the other.
    \end{enumerate}
\end{remark}

\vspace{4mm}

In Chapter \ref{Ch: TM(X)}, we will use the above remark and Lemma \ref{Topology of a Turn Transversal}, to show that there exists a sub-class of turn transversals called ``Edge Transversals" that we can use to encode a given transverse measure. 

\vspace{4mm}

We will show later in this text that, given a transverse measure $\mu$, merely recording the total measures of all edge transversals is enough to recover $\mu$ [Proposition \ref{turn transversals in terms of edge transversals}]. We will use this fact to study the space of transverse measures, and in Chapter \ref{Ch: A Criterion for Unique Ergodcity}, to provide a criterion under which a given solenoid is measure theoretically minimal (i.e uniquely ergodic). But first, in the next chapter, we will provide a criterion under which a proper solenoid is topologically minimal.

\vspace{4mm}

\begin{center}
    \section{Minimality and Mingling} \label{ch: Minimality and Mingling}
\end{center}

\vspace{4mm}


\h Throughout this chapter, we will assume that $X = X(\zeta)$ is the solenoid induced by a \hyperlink{Stabilized Split Sequence}{stabilized} proper split sequence $\zeta  : G_0 \xleftarrow[]{f_{-1}} G_{-1} \xleftarrow[]{f_{-2}} G_{-2} \xleftarrow[]{f_{-3}} ...$.

\vspace{4mm}

In this chapter, we shall answer the questions: what criterion makes a proper solenoid topologically minimal [Section \ref{sec: A criterion for topological minimality}], and then we shall explore some features of solenoids that help study aspects of non-minimality [Section \ref{sec: Features, and Types of Solenoids}]. 

\vspace{4mm}

In Subsection \ref{sec: The Concept of Minimality}, we will provide context on the concept of minimality (in a topological sense, and in a measure theoretic sense) as it appears in the wider literature. In Subsection \ref{sec: Minimal Solenoids}, we will define ``Minimal Solenoids", then 
we will provide a criterion called ``Full Mingling" [Definition \ref{Full Mingling}] so that when $\zeta$ satisfies full mingling, $X(\zeta)$ is minimal [Proposition \ref{The Mingling Lemma}].

\vspace{4mm}

In Subsection \ref{sec: Stable Cores of a Solenoid}, we will defined a feature of $X$ called ``Stable Cores of $X$" then use it to explore different types of non-minimality that can occur in $X$. In Subsection \ref{sec: Expanding Solenoids}, we introduce a type of solenoids called ``Expanding Solenoids" that have the useful property of each partial leaf being infinite.


\vspace{4mm}

In Chapters \ref{Ch: TM(X)} and \ref{Ch: A Criterion for Unique Ergodcity} we will assume that the underlying solenoid in question is expanding, to narrow our focus when we study spaces of transverse measures. 

\vspace{4mm}

\subsection{A Criterion for Topological Minimality}
\label{sec: A criterion for topological minimality}

\vspace{4mm}

We start with an exposition on how the concept of minimality is laid out and studied in the wider literature. 

\vspace{4mm}

\subsubsection{The Concept of Minimality}
\label{sec: The Concept of Minimality}

\vspace{4mm}


In broad generality, consider studying an $object$ that consists of a $structure$, upon which there is a ``structure preserving" $self-map$. More specifically, this system of (object, structure, self-map) could be of the forms (space, topology, homeomorphism), (measurable space, measure, measure-preserving map), (foliated surface, collection of transverse arcs, the flowing of transversals along leaves). In each case, a system being ``minimal" refers to it being indecomposable in a certain sense (meaning that the underlying object cannot be broken into two parts such that the action of the structure preserving map restricted to each part will be contained in the said part).

\vspace{4mm}

We shall first lay out some definitions and facts about minimal systems borrowed from classic settings without proof (more context can be found in \cite{walters1982introduction}) and then in the next subsection, we will discuss what minimality means in the context of Solenoids.

\vspace{4mm}

A homeomorphism $T:Y \longrightarrow Y$ is said to be ``Minimal" if for each $y \in Y$, the set $ \{ T^n (y) : n \in \mathds{Z} \}$ is dense in Y. This sense of topological minimality is equivalent to the statement: For every non-empty open subset $U$ of $Y$, ${\bigcup}^{\infty}_{-\infty} T^n (U) = Y$.

\vspace{4mm}

Let Y be a measurable space, $m$ a finite measure on Y and $T: Y \longrightarrow Y$ a measure preserving transformation. The system $(Y,m,T)$ is called ``Ergodic" if each measurable subset $B$ of $Y$ satisfying $T^{-1} (B) = B$ have either 0 measure or full measure. This sense of measure theoretic minimality (ergodicity) of the system is equivalent to the following statement: For every measurable subset $B$ of $Y$ with non-zero measure, we have $m({\bigcup}^{\infty}_{n=1} T^{-n} (B)) = m(Y)$.

\vspace{4mm}

If we take the system in the above definition, and strip it of its measure, we can ask the question: How many ergodic measures can we put on $Y$ that is preserved by $T$? The answer to this question obviously only depends on $T$, $Y$, and its underlying collection of measurable sets. If it's clear from the context that the underlying $\sigma-$algebra of this measurable space is the borel sigma-algebra, then the possible ergodic measures only depend on the topological space and the transformation.

\vspace{4mm}

In many different types of systems consisting of an underlying space $Y$ with a structure and a structure preserving transformation $T$, a thing of interest to seek is the space of all possible 
finite measures (or something analogous to measures) that is preserved by $T$. We shall denote this space $M(Y,T)$. It is of further interest to find if $M(Y,T)$ has any natural structure. Since a finite linear combination of finite measures produce another finite measure, and since this resulting measure remains preserved by $T$, $M(Y,T)$ naturally has the form of a convex cone (and can be imagined as a subspace of a vector space). 

\vspace{4mm}

For a system that has the format: (measurable space, structure, structure-preserving transformation or flow), if there exists only one ergodic measure (up to scaling), then we call the system ``uniquely ergodic". We shall explore what this notion means in the context of Solenoids in Chapter \ref{Ch: A Criterion for Unique Ergodcity}. For now, we shall focus on what makes a solenoid topologically minimal.

\vspace{4mm}



\subsubsection{Minimal Solenoids, and the Mingling Lemma}
\label{sec: Minimal Solenoids}

\vspace{4mm}

Unlike in the case of: (Space, Topology, Homeomorphism), in our context we do not work with a homeomorphism that preserves the structure. Instead we have the relation of flow equivalence [Definition \ref{Flow Equivalence}] serving a similar purpose. In a sense our system takes the form of: (Solenoid, Collection of transversals, Flow equivalence). 

\vspace{4mm}

Before we explore the measure theoretic minimality of a solenoid, we will lay out the groundwork needed to elaborate more precisely what it means for a proper solenoid to be topologically minimal.

\vspace{4mm}

\begin{definition}[Long Tunnels]
    Given a transversal $F$ of a proper solenoid $X$, we call any transversal contained in $F$ a ``Sub-Transversal of $F$".
    The ``Long Tunnel Generated by $F$ in $X$" (denoted ``$LTun(F)$") is the union of all transversals flow equivalent to a sub-transversal of $F$. The ``Closed Long Tunnel Generated by $F$ in $X$" (denoted ``$CLTun(F)$") is the closure of ``$LTun(F)$" in $X$.
\end{definition}

\vspace{4mm}



\begin{lemma}
    Let $L$ be a partial leaf of $X$ and $x \in L$. Then $L = LT(x)$.
\end{lemma}

\begin{proof}
    Let $L$ and $x$ be as above. Clearly $x$ is a leaf interior point. Let $y \in X - Sing(X)$. $x$ and $y$ are leaf equivalent to each other $\Longleftrightarrow$ there exists $m \in \mathds{N}$ and a concatenation of level $0$ optimal leaf segments $\bigcup_{i=1}^m L_i$ such that $x \in L_1$ and $y \in L_m \Longleftrightarrow$ there exists a flow map composition taking $x$ to $y$. 
\end{proof}

\vspace{4mm}

Next, we lay out three equivalent criteria that we shall use to define the notion of a ``Minimal Solenoid" [Definition \ref{Def: Minimal Solenoids}]. 


\vspace{4mm}

\begin{lemma}
    \label{equivalent criteria for minimality}
    Let $\zeta$ be a proper stabilized split sequence, and $X = X(\zeta)$. The following are equivalent.
    \begin{enumerate}
        \item For each transversal $F$ of $X$, $CLTun(F) = X$.
        \item For each transversal $F$ of $X$, $LTun(F)$ is dense in $X$.
        \item Each partial leaf of $X$ is dense in $X$.
    \end{enumerate} 
\end{lemma}
    
\begin{proof}
    (1.) $\Longleftrightarrow$ (2.) follows from the simple fact that any set is dense in its closure. Suppose (2.) is true. Since each leaf interior point of $X$ is a transversal of $X$ and each partial leaf of $X$ takes the form of a long tunnel of a leaf interior point, we have $(2.) \Longrightarrow (3.)$. Now suppose (3.) is true. Let $F$ be a transversal of $X$. Each $x \in F$ in a leaf interior point. Choose $x \in F$. 
    $LTun(x) \subseteq LTun(F)$ coupled with (3.), imply that $X = CLTun(x) \subseteq CLTun(F)$. 
\end{proof}


\vspace{4mm}

\begin{definition}[Minimal Solenoids]
    \label{Def: Minimal Solenoids}
    \hypertarget{Def: Minimal Solenoids}{Let} $\zeta$ be a proper stabilized split sequence, and $X = X(\zeta)$. We say that $X$ is a ``Minimal Solenoid" if each partial leaf of $X$ is dense in $X$ (or equivalently if any of the itemized criteria in Lemma \ref{equivalent criteria for minimality} is satisfied).
\end{definition}

\vspace{4mm}




\subsubsection*{The Mingling Lemma}
\label{sec: The Mingling Lemma}

\vspace{4mm}


Here, we will introduce ``\hyperlink{Full Mingling}{Full Mingling}" as a property of certain split sequences, that ensures minimality [Proposition \ref{The Mingling Lemma}].

\vspace{4mm}

Recall that we consider $\zeta  : G_0 \xleftarrow[]{f_{-1}} G_{-1} \xleftarrow[]{f_{-2}} G_{-2} \xleftarrow[]{f_{-3}} ...$ to be a proper stabilized split sequence.

\vspace{4mm}

\begin{definition}[Mingling]
    \label{Mingling}
    \hypertarget{Mingling}{Given} $i,j \in \n$ with $j>i$, and natural edges $E_k$ of $G_k$ for $k \in \{i,j\}$, we say that ``$E_i$ mingles with $E_j$" if $f^i_j (Interior(E_i))$ and $Interior(E_j)$ has a non-empty intersection. Furthermore, for $m \in \mathds{N}^*$, we say that ``$E_i$ mingles with $E_j$ $m-$times" if, $(Interior(E_i)) \bigcap ((f^i_j)^{-1} (Interior(E_j)))$ has $m$ connected components.
\end{definition}


\vspace{4mm}

\begin{definition}[Full Mingling]
    \label{Full Mingling}
    \hypertarget{Full Mingling}{We} say that ``$\zeta$ Displays Full Mingling" (or that ``$\zeta$ is a Fully Mingling Split Sequence") if, for each $j \in \n$, there exists an integer $k<j$ such that each natural edge in $G_k$ mingles with each natural edge in $G_j$.
\end{definition}

\vspace{4mm}

\begin{proposition}[The Mingling Lemma]
    \label{The Mingling Lemma}
    \hypertarget{The Mingling Lemma}{Fully} mingling split sequences induce minimal solenoids.   
\end{proposition}


\begin{proof}
    Suppose $\zeta$ is fully mingling. Let $L$ be a leaf or a partial leaf of $X(\zeta)$. Let $O \subset X$ be open. (We must show that $L$ intersects $O$.) Then for some $j \in \n$, $O$ contains a set of the form $\pi^{-1}_j (V)$ where $V$ is a star neighborhood of $G_j$. Choose an embedded open interval $I \subset V$ that is contained in the interior of a natural edge of $G_j$. We have $\pi^{-1}_j (I) \subset O$. Due to full mingling, there exists an integer $i<j$ such that every natural edge in $G_i$ mingles with every natural edge in $G_j$. Consider the projection of $L$ into $G_i$. $\pi_i (L)$ intersects at least one natural edge $E_i$ of $G_i$, and since $f^i_j (E_i)$ intersects each natural edge interior of $G_j$, there is an edge-interior segment $I_i$ in $E_i$ that gets mapped homeomorphically to $I$ via $f^i_j$. And we have $\pi^{-1}_i (I_i) \subseteq \pi^{-1}_j (I)$ and therefore the segment of $L$ contained in $\pi^{-1}_i (I_i)$ is also contained in $\pi^{-1}_j (I)$ and thus in $O$ as well.
\end{proof}

\vspace{4mm}

In Chapter \ref{Ch: A Criterion for Unique Ergodcity}, we will lay out a criterion called ``Semi-Normality" under which a proper solenoid is measure theoretically minimal (i.e. uniquely ergodic).
After defining semi-normality, we will immediately observe that semi-normal split sequences satisfy full mingling, and thus semi-normality will assure topological minimality of a solenoid as well.


\vspace{4mm}

\subsection{Features, and Types of Solenoids}
\label{sec: Features, and Types of Solenoids}

\vspace{4mm}

Here, we explore a useful feature of a given solenoid $X$, called ``Minimal Cores of $X$" to recognize the types of non-minimality that could exist. We then introduce a type of solenoids called ``Expanding Solenoids" that do not allow the primitive form of non-minimality of containing finite core graphs.

\vspace{4mm}

\subsubsection{Stable Cores of a Solenoid}
\label{sec: Stable Cores of a Solenoid}

\vspace{4mm}

\begin{definition}[Stable Cores of $X$]
    \label{Stable Cores}
    \hypertarget{Stable Cores}{Let} $n' \in \mathds{N}$, $K \in \n$, and let $\zeta'  : G'_K \xleftarrow[]{f_{K-1}} G'_{K-1} \xleftarrow[]{f_{K-2}} G'_{K-2} \xleftarrow[]{f_{K-3}} ...$ be a \hyperlink{Sub Inverse Sequences}{sub inverse system} of $\zeta$ such that, for each integer $j \leq K$, $G'_j$ is a rank $n'$ core graph. Then we call the solenoid induced by $\zeta'$, $X(\zeta') \subseteq X(\zeta)$ a ``Rank $n'$ Stable Core of $X(\zeta)$". Additionally, we say that ``$\zeta'$ Stabilizes at Level $K$" or that ``$K$ is a Starting Level for $X(\zeta')"$.
    If either $K = 0$, or if $f_K (G_K) \subseteq G_{K+1}$ is not a rank $n'$ core graph, then we say that ``$K$ is the Maximal Starting Level for $X(\zeta')$".
\end{definition}

\vspace{4mm}

We consider both $X$ and $\emptyset \subset X$ to be trivial stable cores of $X$. Furthermore, we consider $\emptyset \subset X$ to be the unique rank $0$ stable core of $X$. Where $\zeta$ is a rank $n$ split sequence, $X$ itself is the unique rank $n$ stable core of $X$. 

\vspace{4mm}

\begin{remark}[Non-Trivial Stable Cores Violate Minimality]
    Suppose $X$ has a non-trivial stable core $X'$ where $K \in \n$ is a starting level for $X'$. Then since $\pi_K (X') \subset G_K$ is a core graph, and $G_K - \pi_K (X') \neq \emptyset$, we have that $G_K - \pi_K (X')$ contains the interior of a natural edge $E$ in $G_K$. Let $I := Interior (E)$. 
    Then there exists a leaf $L$ that is entirely contained in $X'$ and thus $L$ does not intersect the neighborhood $\pi^{-1}_K (I) \subset X$. Therefore $X$ cannot be minimal.
\end{remark}

\vspace{4mm}

Note that when $X$ isn't 
minimal, 
studying the stable cores of $X$ can, in a sense, tell us about the nature of $X$'s non-minimality. 

\vspace{4mm}

\begin{definition}[Minimal Cores of $X$]
    A stable core $X'$ of $X$ is called a ``Minimal Core of $X$" if $X'$ is a minimal solenoid.
\end{definition}

\vspace{4mm}

\begin{definition}[Consecutively Closed Subsets]
    \label{Consecutively Closed Subsets}
    \hypertarget{Consecutively Closed Subsets}{Given} $ A \in \{ \mathds{N}, \mathds{N}^* , - \mathds{N}, -\mathds{N}^* , \mathds{Z} \}$ a ``Consecutively Closed Subset $\mathds{J}$ of $A$", is a subset of $A$ such that,
    \begin{enumerate}
        \item for each distinct $i, j \in \mathds{J}$, $[i,j] \cap A$ is contained in $\mathds{J}$, and,
        \item if $A$ has a lower bound (resp. an upper bound) $B$, then $ B \in \mathds{J}$.
    \end{enumerate}
\end{definition}

\vspace{4mm}

\begin{remark}[Nested Sequences of Stable Cores, and Minimal Cores of $X$]
    Note that from the nature of minimality, a minimal core of $X(\zeta)$ cannot properly contain another minimal core of $X$. Let $\{ X_i \}_{i \in \mathds{J}}$ be a sequence of stable cores of $X$, such that, $\mathds{J}$ is a consecutively closed subset of $\mathds{N}^*$, $X_0 = X$, and for each $i, i+1 \in \mathds{J}$, $X_i \supset X_{i+1}$.

    \vspace{4mm}
    
    For each $i \in \mathds{N}^*$, let $n_i \in \mathds{N}^*$ be the rank of $X_i$. Let $n \in \mathds{N}$ denote the rank of $\zeta$. Then $n_0 = n$. Furthermore, for each $i \in \mathds{N}^*$ that satisfies $X_{i+1} \neq \emptyset$, there must exist a level $K_i \in \n$ such that in $G_{K_i}$,
    \begin{enumerate}
        \item[i.] $\pi_{K_i} (X_i)$ is a core graph of rank $n_i$,
        \item[ii.] $\pi_{K_i} (X_{i+1})$ is a core graph of rank $n_{i+1}$, and,
        \item[iii.] $\pi_{K_i} (X_i) \supset \pi_{K_i} (X_{i+1})$,
    \end{enumerate}
    implying that $n_i > n_{i+1}$. Thus, $\mathds{J}$ must have an upper bound $B$ such that $X_B = \emptyset$.

    \vspace{4mm}

    A stable core $X' \neq \emptyset$ of $X$, is called an ``Innermost Stable Core of $X$" if, $X'$ does not properly contain any non-trivial stable cores of $X$. Any Stable Core of $X$ that is a minimal solenoid, shall be called a ``Minimal Core of $X$". The above discussion makes it apparent that,
    \begin{enumerate}
        \item given $X' \subseteq X$, $X'$ is an innermost stable core of $X$ if and only if $X'$ is a minimal core of $X$,
        \item $X =$ its unique innermost stable core if and only if $X$ is minimal,
        \item $X$ has at least one minimal core.
    \end{enumerate}
\end{remark}

\vspace{4mm}


\subsubsection{Expanding Solenoids}
\label{sec: Expanding Solenoids}

\vspace{4mm}

Here, we carry out the process of identifying what seems to be the widest natural class of solenoids that do not support transverse measures with atoms. 

\vspace{4mm}

We start with a definition of the titular object [Def. \ref{def: Expanding Solenoids}]. We then lay out some preliminary terminology needed to show that, if $X$ is an expanding solenoid, $X$ does not contain any finite partial leaves, and therefore, each transverse measure on $X$ is \hyperlink{Non Atomic TM}{non-atomic} [Prop. \ref{Expanding solenoids do not support atomic masses}]. Then we shall show that solenoids induced by fully mingling split sequences, are also expanding [Rem. \ref{FM implies Expanding}].

\vspace{4mm}





\begin{definition}[Expanding Solenoids]
    \label{def: Expanding Solenoids}
    \hypertarget{Expanding Solenoids}{A} split sequence $\zeta :  G_{0} \xleftarrow[]{f_{-1}} G_{-1} \xleftarrow[]{f_{-2}} G_{-2} \xleftarrow[]{f_{-3}}...$ is called an ``Expanding Split Sequence", and $X(\zeta)$ an ``Expanding Solenoid" if, for each $j \in \n$ and for each $E_j \in E(G_j)$, there exists an integer $k < j$ such that for each $E_k$ that mingles with $E_j$, $f^k_j |_{E_k}$ is not a homeomorphism.
\end{definition}

\vspace{4mm}

We will proceed to show that each partial leaf of an expanding solenoid $X$ is infinite, and therefore $X$ does not support transverse measures with atomic masses [Prop. \ref{Expanding solenoids do not support atomic masses}]. To make the recognition of finite partial leaves easier we introduce the following terminology.




\vspace{4mm}

\begin{definition}[Mingling Edge Sequences, and Strongly Mingling Edge Sequences]
    \label{Mingling Edge Sequences}
    \hypertarget{Mingling Edge Sequences}{A} sequence of natural edges $\{ E_j \}_{j \in \n}$ where for each $j \in \n$, $E_j \in E(G_j)$, is called a ``Mingling Edge Sequence in $\zeta$" if for each $j \in \n$, $E_{j-1}$ \hyperlink{Mingling}{mingles} with $E_j$. A mingling edge sequence $\{ E_j \}_{j \in \n}$ in $\zeta$ is called a ``Strongly Mingling Edge Sequence in $\zeta$" if there exists $K \in \n$ (called the ``Starting Level of $\{ E_j \}_{j \in \n}$'s Strong Mingling") such that, for each integer $j < K$, $f_j |_{E_j}$ is a homeomorphism onto $E_{j+1}$.
\end{definition}

\vspace{4mm}

Note that if $\zeta$ is an expanding split sequence, then each mingling edge sequence in $\zeta$ cannot be strongly mingling. 

\vspace{4mm}

\begin{proposition}
    \label{Expanding solenoids do not support atomic masses}
    Let $X$ be an expanding solenoid. Then $X$ does not contain any finite partial leaves, and therefore, every transverse measure on $X$ is non-atomic.
\end{proposition}

\begin{proof}
    Let $\zeta$ be an expanding split sequence. If $X = X(\zeta)$ contains a finite partial leaf $L$, then there exists $K \in \n$, such that for each integer $j < K$, $\pi_j (L)$ is a finite edge path in $G_j$ where, $f_j |_{\pi_j (L)}$ is a homeomorphism onto $\pi_{j+1} (L) \subseteq G_{j+1}$, 
    which implies the existence of a strongly mingling edge sequence in $\zeta$. Thus $X(\zeta)$ cannot contain any finite leaves.

    \vspace{4mm}

    Now suppose $\mu$ is a transverse measure on $X$ such that $\mu (\{ x \}) = \epsilon > 0$ for some $x \in X$. 

    \vspace{4mm}
    
    (Case a.) Suppose $x \notin Sing(X)$. Then there exists a unique partial leaf $L$ of $X$ containing $x$. Since $L$ is not finite, there exists $q \in G_0$ such that $L \cap F_q$ is not a finite set. Since $q$ has only finitely many \hyperlink{optimal turns}{level $0$ optimal turns} containing it, there must exist at least one level $0$ optimal turn $I$ containing $q$, such that, $L \cap \hyperlink{Turn Transversals}{H_{I,q}} $ is not finite. Then since the \hyperlink{Turn Transversals}{turn transversal} $H_{I,q}$ contains infinitely many distinct points flow equivalent to $x$, we have $\mu (H_{I,q}) > \infty$. This contradicts the fact that $\mu$ restricted to each turn transversal of $X$ is a finite borel measure. 
    
    \vspace{4mm}
    
    (Case b.) Suppose $x \in Sing(X)$. Let $V$ be a singular plaque centered at $x$, and let $A \subset V - \{x\}$ be a set, such that for each open prong $P$ of $V$, $A$ contains exactly one point from $P$. Then from Definition \ref{Transverse Semi-Measures}, we have that $\mu (x) = \sum_{y \in A} \mu (\{y\})$. Then $\mu (\{y\}) > 0$ must be true for at least one $y \in A$, which leads to the contradiction at the end of case a.
\end{proof}

\vspace{4mm}

The next remark and corollary, lay out the consequences of the above result in the context of fully mingling split sequences.

\vspace{4mm}

\begin{remark}[Fully Mingling Split Sequences Induce Expanding Solenoids]
    \label{FM implies Expanding}
    Let $\zeta$ be a stabilized proper split sequence. Suppose $\zeta$ is not expanding. i.e. there exists, a strongly mingling edge sequence $\mathcal{E} := \{ E_j \}_{j \in \n}$ in $\zeta$ with $K \in \n$, where $K$ is the starting level of $\mathcal{E}$'s strong mingling.
    Then for each integer $j < K$, $E_j$ does not mingle with any natural edge in $E (G_K) - \{ E_K \}$. Thus $\zeta$ is not fully mingling.
\end{remark}


\vspace{4mm}

\begin{corollary}
    Let $\zeta$ be a proper stabilized split sequence. If $\zeta$ is fully mingling, then each transverse measure on $X(\zeta)$ is non-atomic.
\end{corollary}

\begin{proof}
    Follows from Prop. \ref{Expanding solenoids do not support atomic masses}, and Rem. \ref{FM implies Expanding}.
\end{proof}

\vspace{4mm}

In the next chapter, we build a mechanism that helps us envision the space of transverse measures of a given solenoid, using the collection of natural edges of level graphs, in a loose sense, as a coordinate space that encodes transverse measures. For this process to play out, one assumption we must make, is that all transverse measures involved are non-atomic. Thus, for the discussion in the next two chapters, we assume that the underlying $\zeta$ is an expanding split sequence. 

\vspace{4mm}




\begin{center}
    \section{The Space of Transverse Measures} \label{Ch: TM(X)}
\end{center}

\vspace{4mm}

\h For the rest of this text, we will assume that $X = X(\zeta)$ is the solenoid induced by a \hyperlink{Stabilized Split Sequence}{stabilized} split sequence,
\begin{equation*}
    \zeta  : G_0 \xleftarrow[]{f_{-1}} G_{-1} \xleftarrow[]{f_{-2}} G_{-2} \xleftarrow[]{f_{-3}} ....
\end{equation*}
that is \hyperlink{Strongly Proper Split Sequences}{strongly proper} and \hyperlink{Expanding Solenoids}{expanding}. Recall that $TM(X)$ denote the space of transverse measures of $X$.

\vspace{4mm}

The main goal of this chapter is proving that $TM(X)$ can be realized as the inverse limit of a certain sequence of convex cones [Theorem \ref{TM(X) is an inverse limit}]. Alternative proofs for this result can be found in \cite{namazi2014ergodicdecompositionsfoldingunfolding} and \cite{BedHilLus2020}, respectively carried out in the contexts of `\textit{supporting laminations of frequency currents}'  and `\textit{symbolic laminations of graph towers}'. We begin with setting up terminology for the specific linear maps and convex cones that will be used to encode transverse measures on $X$.

\vspace{4mm}

\begin{definition}[Weight Maps and Pseudo-Weight Cones]
    Let $j \in \n$, and let $E(G_j)$ denote the set of natural edges of $G_j$. The real vector space $\mathds{R}^{E(G_j)}$ spanned by $E(G_j)$ is called ``The Level $j$ Pseudo-Weight Space of $\zeta$" (or ``The Pseudo-Weight Space of $G_j$") and denoted ``$R_j (\zeta)$". The cone of non-negative vectors in $R_j (\zeta)$ is called ``The Level $j$ Pseudo-Weight Cone of $\zeta$" (or ``The Pseudo-Weight Cone of $G_j$") and denoted $\Lambda_j (\zeta)$. The linear map from $R_{j-1}$ to $R_j$ defined by the \hyperlink{transition matrix}{transition matrix} of $f_{j-1}$ is called the ``Level $j-1$ Weight Map of $\zeta$" and is denoted $T_{j-1}$.
\end{definition}

\vspace{4mm}

\begin{definition}[The Pseudo-Weight Cone Sequence of $\zeta$]
    ``The Weight Map Sequence of $\zeta$" is the sequence,
    \begin{equation*}
        WMS(\zeta)  : R_0 (\zeta) \xleftarrow[]{T_{-1}} R_{-1} (\zeta) \xleftarrow[]{T_{-2}} R_{-2} (\zeta) \xleftarrow[]{T_{-3}} ...
    \end{equation*}
    The inverse limit of $WMS(\zeta)$ (in the category of vector spaces) is called the ``Space of Pseudo-Weight Assignments on $\zeta$" and denoted ``$R(\zeta)$". \hypertarget{space of weight assignments}{ ``The Pseudo-Weight Cone Sequence of $\zeta$"} is the sequence,
    \begin{equation*}
        PCS(\zeta)  : \Lambda_0 (\zeta) \xleftarrow[]{T_{-1} |_{\Lambda_{-1} (\zeta)}} \Lambda_{-1} (\zeta) \xleftarrow[]{T_{-2} |_{\Lambda_{-2}}} \Lambda_{-2} (\zeta) \xleftarrow[]{T_{-3} |_{\Lambda_{-3}}} ...
    \end{equation*}
    Going forward, when we write out the above sequence, for each $j \in -\mathds{N}$, we will use the notation $T_j$ instead of $T_j |_{\Lambda_{j}(\zeta)}$, even though the appropriate restriction is what we mean to represent. The inverse limit of $PCS(\zeta)$ is called the ``Space of Weight Assignments on $\zeta$" and is denoted ``$\Lambda(\zeta)$". 
\end{definition}

\vspace{4mm}

The concept of individual ``Weight Assignments on $\zeta$" will be precisely defined in \ref{Weight Assignments}. More descriptive versions of all definitions contained in this summary, can be found scattered throughout this chapter.


\vspace{4mm}

Note that $R(\zeta)$ is a vector space and that $\Lambda(\zeta)$ has the structure of a convex cone as a subspace of $R(\zeta)$. We will elaborate more on this in sub-section \ref{sec: The Shape of TM(X)}. Furthermore $TM(X)$ has the structure of a convex cone, since it is the structure all spaces of measures naturally possess. Now we state our main result.

\vspace{4mm}

\begin{thm}
    \label{TM(X) is an inverse limit}
    The space of \hyperlink{transverse measure}{transverse measures of $X(\zeta)$} is isomorphic to the \hyperlink{space of weight assignments}{space of weight assignments of $\zeta$} as convex cones in vector spaces.
\end{thm}

\vspace{4mm}


In Section \ref{Sec: Edge Weights and Weight Assignments}, 
we will discuss the fundamental building blocks necessary to
prove the theorem, and then we will show that each transverse measure $\mu$ on $X$ uniquely induces a weight assignment $w_{\mu}$ on $\zeta$ [Lemma \ref{Transverse Measures to Weight Assignments}].

\vspace{4mm}

In Section \ref{sec: Edge Weights and Transverse Measures}, we will prove that each weight assignment on $\zeta$ uniquely determines a transverse measure on $X$ [Proposition \ref{weight assignment to transverse measure}], and thus, we can envision $TM(X)$ as the space of all possible weight assignments on $\zeta$. The main proposition [\ref{turn transversals in terms of edge transversals}] that contributes to this proof, states that each turn transversal $H$ of $X$ can be uniquely expressed as the disjoint union of countably many edge transversals [Def. \ref{Edge Tunnel}] and a finite set. 

\vspace{4mm}

In Section \ref{sec: The Shape of TM(X)}, we will show that the bijection we construct from $TM(X)$ to $\Lambda(\zeta)$ preserves their respective convex cone structures
[Proposition \ref{TM(X) = inv lim ACS}]. 

\vspace{4mm}

Later, we apply this machinery in Chapter \ref{Ch: A Criterion for Unique Ergodcity} to provide a unique ergodicity criterion for solenoids. 

\vspace{4mm}

We begin the chapter with preliminary definitions. Once all the necessary terminology is laid out in Section \ref{Sec: Edge Weights and Weight Assignments}, an outline of the main argument contained in Section \ref{sec: Edge Weights and Transverse Measures}, is provided in Remark \ref{Weight Assignments and Transverse Measures}. 

\vspace{4mm}

\subsection{Edge Weights and Weight Assignments}
\label{Sec: Edge Weights and Weight Assignments}

\vspace{4mm}

Here, we lay out some preliminary concepts and terminology. As we move through the discussion in this section, more detailed overviews of the chapter will be provided, as necessary terminology is established.


\vspace{4mm}

\subsubsection*{Edge Weights}
\label{Sec: Edge Weights}

\vspace{4mm}

\begin{definition}[Edge Tunnels and Edge Transversals]
    \label{Edge Tunnel}
    \hypertarget{Edge Tunnel}{Let} $j \in \n$, and let $I$ be a level $j$ \hyperlink{turn in zeta}{standard turn in $\zeta$}. If $I$ is contained in the interior of a natural edge of $G_j$, then $I$ is called a ``Level $j$ Edge Turn in $\zeta$", and $Tun_I$ is called a ``Level $j$ Edge Tunnel in $X$". A turn transversal that is a \hyperlink{cross-section}{cross-section} of an edge tunnel is called an ``Level $j$ Edge Transversal in $X$". More specifically, when $E$ is the natural edge in $G_j$ whose interior contains the turn $I$, each cross-section of $Tun_I$ shall be called an ``$E-$transversal in $X$".
\end{definition}





\vspace{4mm}

Next, we confirm the following rather unsurprising fact about edge transversals, that will ensure that the concept of an ``Edge Weight" [to be laid out in Definition \ref{Edge Weights}] is unambiguous. 

\vspace{4mm}

\begin{lemma}
    \label{edge transversals with consistent weights}
    Let $\mu$ be a transverse measure on $X(\zeta)$, and $E$ a natural edge of the level $j$ graph of $\zeta$ for some $j \in \n$. Let $H_1$ and $H_2$ be two distinct $E-$transversals in $X$. Then, $\mu (H_1) = \mu (H_2)$.
\end{lemma}

\begin{proof}
    Assume the set-up of the above lemma. Let $I := Interior(E)$. If $I$ is a standard turn in $\zeta$ (i.e. $I$ is a turn of a standard star neighborhood, and therefore contains at most one \hyperlink{special points}{special point}), then the conclusion of the lemma follows from the fact that, in the same turn tunnel every cross section is \hyperlink{Flow Equivalence}{flow equivalent} to every other cross section. Now suppose that $I$ contains more than one special point. Since 
    properness of $\zeta$ assures that $I$ can only contain finitely many special points (and since strong properness implies properness), 
    there exists a finite collection of level $j$ standard turns $\{I_i \}_{i=1}^m$ in $\zeta$ (all contained in $I$) for some $m \in \mathds{N}+1$, such that,
    \begin{enumerate}
        \item for each $i \in \{1,..., m-1\}$, $I_i \cap I_{i+1}$ is an embedded open interval in $I \subset G_j$,
        \item there exist $p,q \in I$ where $H_1 = \hyperlink{Turn Transversals}{H_{I_1 ,p}}$ and $H_2 = \hyperlink{Turn Transversals}{H_{I_m ,q}}$.
    \end{enumerate}
     Thus, we can construct an appropriate \hyperlink{Flow Equivalence}{flow map composition} that takes $H_1$ to $H_2$. 
\end{proof}



\vspace{4mm}

\begin{definition}[Edge Weights]
    \label{Edge Weights}
    \hypertarget{edge weights}{Let} $\mu$ be a \hyperlink{transverse measure}{transverse measure} on $X(\zeta)$, and $E$ a natural edge of a level graph of $\zeta$. Then the ``Edge Weight of $E$ rel $(X(\zeta),\mu)$" denoted ``$\mu (E)$" is the quantity given by $\mu (H)$ where $H$ is any \hyperlink{Edge Tunnel}{$E$-transversal of $X$}. 
\end{definition}

\vspace{4mm}

The main result in this chapter [that is laid out in Section \ref{sec: Edge Weights and Transverse Measures}] will establish that given a transverse measure $\mu$ on a strongly proper expanding solenoid $X$, by recording only the edge weights rel $(X,\mu)$, we can recover the transverse measure induced by $\mu$ on any transversal of $X$. 
Thus in a sense a space of possible edge weights on $X$ can be perceived as a coordinate space for transverse measures on $X$. The possible edge weights in this context will be a collection of numbers that satisfy a few desired criteria called the ``Weight Equations" [to be laid out in Definition \ref{Weight Assignments}].

\vspace{4mm}




\subsubsection*{Weight Assignments}
\label{sec: Weight Assignments}

\vspace{4mm}





In order to define the titular concept, we will first establish some preliminary terminology, and refine the idea of mingling established in Definition \ref{Mingling}. 

\vspace{4mm}

\begin{definition}[Edge-Bases, and Mingling Numbers]
    \label{Edge-Bases, and Mingling Numbers}
    \hypertarget{Edge-Bases, and Mingling Numbers}{Let} $\zeta$ be a split sequence. For each $j \in \n$, we call the set $E(G_j)$ of all the natural edges of $G_j$, the ``Level $j$ Edge-Base of $\zeta$". And we call the collection $\mathcal{E} (\zeta) = \bigsqcup_{j \in \n} E(G_j)$ the ``Edge Space of $\zeta$". 
    Now let $j,k \in \n$ such that $k < j$. Furthermore, let $E_j \in E(G_j)$ and $E_k \in E(G_k)$. We call the non-negative integer given by $``M^{\zeta}(E_k,E_j)":= (Interior(E_k)) \bigcap ((f^k_j)^{-1} (Interior(E_j)))$, the ``Mingling Number of $E_k$ with $E_j$ rel $\zeta$" (i.e. loosely speaking, $M^{\zeta}(E_k,E_j)$ is the number of times the image of $E_k$ goes over $E_j$). 
\end{definition}

\vspace{4mm}

\begin{definition}[Weight Assignments]
    \label{Weight Assignments} 
    \hypertarget{Weight Assignments}{Let} $\zeta$ be a strongly proper split sequence,
    and let $w : \mathcal{E} (\zeta) \longrightarrow [0, \infty)$ a function.
    For $j \in \n$, and $E_j \in E(G_{j})$ we call the equation 
    \begin{equation*}
        w (E_j) = \sum_{E \in E(G_{j-1})} ( \ M^{\zeta}(E,E_j) \ w(E) \ )
    \end{equation*}
    the ``Level $j$ Weight Equation of $E$ rel $(\zeta, w)$".
    We say that $w$ is a ``Weight Assignment on $\zeta$" if all the weight equations rel $(\zeta, w)$ are true.
\end{definition}





\vspace{4mm}

Note that in the above context, when $w$ is a weight assignment on $\zeta$, for any pair of levels $j,k \in \n$ such that $k < j$, and for any $E_j \in G_j$, $w$ also satisfies what we shall call ``The Extended Weight Equation of $E_j$ with level $k$ rel $(\zeta, w)$",  given by,
\begin{equation*}
    w (E_j) = \sum_{E \in E(G_k)} ( \hspace{0.8mm} M^{\zeta}(E,E_j)  \hspace{0.8mm} w(E)  \hspace{0.8mm} ).
\end{equation*}





\vspace{4mm}

\begin{lemma}
    \label{Transverse Measures to Weight Assignments}
    Each transverse measure on $X(\zeta)$ defines a weight assignment on $\zeta$.
\end{lemma}

\begin{proof}
    Let $X(\zeta)$ be a strongly proper expanding solenoid and $\mu$ a transverse measure on $X(\zeta)$. For each $j \in \n$, and for each natural edge $E$ of $G_j$, let $w_{\mu} (E) :=$ \hyperlink{edge weights}{the edge weight of $E$ rel $(X(\zeta),\mu)$}. Note that, since $\zeta$ is strongly proper, for each $j \in \n$, for each level $j$ edge transversal $H$ in $X$, and for each integer $k < j$, $H$ can be expressed as a disjoint union of finitely many level $k$ edge transversals in $X$. Then $w_{\mu}$ is a well-defined function from $\mathcal{E} (\zeta)$ to $[0, \infty)$ that satisfies the \hyperlink{Weight Assignments}{weight equations rel $\zeta$}.
\end{proof}

\vspace{4mm}

\begin{definition}[The Weight Assignment Induced by a Transverse Measure]
    \label{The Weight Assignment Induced by a Transverse Measure}
    \hypertarget{The Weight Assignment Induced by a Transverse Measure}{Let} $\mu$ be a transverse measure on $X$. The weight assignment constructed using $\mu$ in the above proof shall be called the ``Weight Assignment Induced by $\mu$" and denoted ``$w(\mu)$". 
\end{definition}

We will show in Section \ref{sec: Edge Weights and Transverse Measures} that each weight assignment on $\zeta$ defines a unique transverse measure on $X(\zeta)$.


\subsection{Each Weight Assignment Determines a Transverse Measure}
\label{sec: Edge Weights and Transverse Measures}

\vspace{4mm}

In this section we will show that, given a transverse measure $\mu$ on $X$, by just recording all the \hyperlink{edge weights}{edge weights rel $(X,\mu)$}, we can recover the measure on any transversal (i.e \hyperlink{The Weight Assignment Induced by a Transverse Measure}{$w(\mu)$} uniquely determines $\mu$). Furthermore, we will show that, given any weight assignment $w$ on $\zeta$, we can construct a unique transverse measure $\mu_w$ on $X$ such that $\hyperlink{The Weight Assignment Induced by a Transverse Measure}{w(\mu_w)} = w$ [Proposition \ref{weight assignment to transverse measure}].
Therefore, the collection of all weight assignments on $\zeta$ can be viewed as a coordinate space for $TM(X)$.

\vspace{4mm}


The above statements will follow from the five claims laid out in Remark \ref{Weight Assignments and Transverse Measures}. But first, to assist with the oncoming discussion, in subsection \ref{sec: Flow Compatible Pre Measures}, we will borrow some facts from measure theory \cite{alma991018098349703276}, then lay out a few preliminary definitions adopting those concepts to our setting. 

\vspace{4mm}

\subsubsection{Flow Compatible Pre Measures}
\label{sec: Flow Compatible Pre Measures}

\vspace{4mm}

The following two remarks [\ref{Extension Theorem}, \ref{Pre-Measures on Tree Bases}] lay the foundation to establishing a useful precursor to transverse measures on $X$, called ``Flow Compatible Pre Measures on $X$" [Def. \ref{Pre Measures on X}]. While the following result is borrowed from \cite{alma991018098349703276}, in order to not conflate terminology, where Royden used the phrase 
``A Measure on an Algebra" we shall use the phrase ``A Pre Measure on an Algebra". 


\vspace{4mm}

\begin{remark}[Pre Measures and Carathéodory's Extension Theorem]
    \label{Extension Theorem}
    Let $Y$ be a set. A collection $\mathcal{B}$ of subsets of $Y$ is called a ``Algebra on $Y$" if, 
    \begin{enumerate}
        \item $\emptyset \in \mathcal{B}$, 
        \item for each $A, C \in \mathcal{B}$, $A \cup C \in \mathcal{B}$, and $Y - A \in \mathcal{B}$. 
    \end{enumerate}
    The stated criteria ensure that $\mathcal{B}$ will be closed under finite unions, and finite intersections. (The pair $(Y, \mathcal{B})$ is called a ``Field of Sets".)

    \vspace{4mm}
    
    Given a set $Y$ and an algebra $\mathcal{B}$ on $Y$, a ``Pre Measure on $(Y , \mathcal{B})$" is a non-negative extended real valued function $\mu$ on $\mathcal{B}$, such that,
    \begin{enumerate}
        \item[i.] $\mu (\emptyset) = 0$, and,
        \item[ii.] for each countable index set $\mathds{J}$, and for each countable collection of disjoint sets $\{ A_i \}_{i \in \mathds{J}} \subseteq \mathcal{B}$ 
        such that $\bigsqcup_{i \in \mathds{J}} A_i  \subset \mathcal{B}$, 
        we have $\mu ( \bigsqcup_{i \in \mathds{J}} A_i ) = \sum_{i \in \mathds{J}} \mu (A_i) $.
    \end{enumerate}
\end{remark}

\vspace{4mm}

Carathéodory's Extension Theorem \cite{alma991018098349703276} states that, when $Y$ is a set, each finite pre measure on an algebra $\mathcal{B}$ of $Y$, can be uniquely extended to a finite measure on the $\sigma-$algebra generated by $\mathcal{B}$ in $Y$. 

\vspace{4mm}

Recall from Remark \ref{Properties of Sub-Turn Bases}, that for each turn transversal $H$ of $X$, the \hyperlink{Sub-Turn Basis}{sub-turn basis of $H$} (denoted $\mathcal{B}_H$) is a \hyperlink{Tree Bases}{tree basis}. The following remark will use Carathéodory's Extension Theorem to show that a ``Pre Measure on a Tree Basis" (defined in the remark) can be uniquely extended to a measure on the $\sigma-$algebra generated by $\mathcal{B}_H$. 

\vspace{4mm}

\begin{remark}[A Pre Measure on a Tree Basis Extends Uniquely to a Measure on the Generated $\sigma-$Algebra]
    \label{Pre-Measures on Tree Bases}
    Let $Y$ be a topological space generated by a \hyperlink{Tree Bases}{tree basis} $\mathcal{B}$. We call a non-negative real valued function $\mu$ on $\mathcal{B}$, \hypertarget{Pre-Measures on Tree Bases}{a} ``Finite Pre-Measure on $(Y,\mathcal{B})$" if, 
    for each countable index set $\mathds{J} \subseteq \mathds{N}$, and for each 
    collection of disjoint sets $\{ A_i \}_{i \in \mathds{J}} \subset \mathcal{B}$ such that $\bigsqcup_{i \in \mathds{J}} A_i \in \mathcal{B}$ , 
    $\mu ( \bigsqcup_{i \in \mathds{J}} A_i ) = \sum_{i \in \mathds{J}} \mu (A_i) $. 

    \vspace{4mm}
    
    Now let $\mathcal{B}'$ be the algebra generated by $\mathcal{B}$. It follows from Remark \ref{Properties of Sub-Turn Bases}, that $\mathcal{B}' = \{A \subseteq Y : A$ is a finite union of elements in $\mathcal{B} \} = \{A \subseteq Y : A$ is a finite disjoint union of elements in $\mathcal{B} \}$. 
    Then for each countable index set $\mathds{J} \subseteq \mathds{N}$, and for each countable disjoint collection $\{ A_i \}_{i \in \mathds{J}} \subseteq \mathcal{B}$ such that $\bigsqcup_{i \in \mathds{J}} A_i \in \mathcal{B}' - \mathcal{B}$, let $\mu ( \bigsqcup_{i \in \mathds{J}} A_i ) := \sum_{i \in \mathds{J}} \mu (A_i)$. Note that this sum is finite, since it's less than or equal to $\mu (Y)$. 
    Then, it follows from Remark \ref{Extension Theorem}, that each finite pre measure on a tree basis $\mathcal{B}$ of the underlying space $Y$, uniquely extends to a finite measure on the $\sigma-$algebra generated by $\mathcal{B}$ in $Y$.
\end{remark}

\vspace{4mm}




We naturally wish to break down transverse measures on $X$ to the smallest and simplest components out of which they can be assembled. We can capture the essence of a transverse measure $\mu$ by observing the measure it induces on each \hyperlink{Turn Transversals}{turn transversal} $H$ of $X$, each of which, can then be encoded in the underlying pre measure on its \hyperlink{Sub-Turn Basis}{sub-turn basis} $\mathcal{B}_H$ [Def. \ref{Sub-Turn Basis}]. The next definition and lemma [\ref{Pre Measures on X}, \ref{flow compatible pre measures to transverse measures}] will aid us in phrasing the aforementioned process in more precise language. 


\vspace{4mm}

\begin{definition}[Flow Compatible Pre Measures on $X$]
    \label{Pre Measures on X}
    Let ``$\mathcal{TT}(X)$" denote the collection of all turn transversals of $X$.
    \hypertarget{Pre Measures on X}{A} ``Pre Measure on $X$" is a non-negative real function $\mu : \mathcal{TT}(X) \longrightarrow [0, \infty)$ such that, for each $H \in \mathcal{TT}(X)$, $\mu$ restricted $\mathcal{B}_H$ (denoted ``$\mu_H$") is a pre measure on $(H, \mathcal{B}_H)$ [Remark \ref{Pre-Measures on Tree Bases}].
    \hypertarget{Flow Compatible Pre Measures}{We} call a pre measure $\mu$ on $X$ a ``Flow Compatible Pre Measure on $X$" if, for each $H, H' \in \mathcal{TT}(X)$ that are flow equivalent to each other, $\mu (H) = \mu (H')$.
\end{definition}

\vspace{4mm}

\begin{lemma}
    \label{flow compatible pre measures to transverse measures}
    Each \hyperlink{Flow Compatible Pre Measures}{flow compatible pre measure} on $X(\zeta)$ determines a unique transverse measure on $X(\zeta)$.
\end{lemma}

\begin{proof}
    Assume the set-up of the lemma. Let $\nu$ be a flow compatible pre measure on $X = X(\zeta)$. It follows from Prop. \ref{Expanding solenoids do not support atomic masses} that, since $\zeta$ is expanding, each transverse measure on $X$ is non-atomic [Def. \ref{Non Atomic TM}], and therefore each singularity of $X$ will have zero measure. It follows from Remark \ref{Pre-Measures on Tree Bases} that, for each turn transversal $H$ of $X$, the pre measure $\nu_H$ on $(H ,\mathcal{B}_H)$ (given by $\nu$ restricted to $\mathcal{B}_H$), uniquely extends to a measure $\mu_H$ on the $\sigma-$algebra generated by $\mathcal{B}_H$ in $H$. Since $\nu$ is flow compatible, when two turn transversals $H, H'$ are flow compatible, the measures $\mu_H , \mu_{H'}$ are preserved by any flow map composition taking $H$ to $H'$ or vice versa. 

    \vspace{4mm}
    
    Now define the transverse measure $\mu$ on $X$ as follows.
    \begin{enumerate}
        \item For each turn transversal $H$ of $X$, and for each borel subset $F$ of $H$, let $\mu(F) := \mu_H (F)$.
        \item For each transversal $\Tilde{H}$ of $X$ that is not a turn transversal, by definition, there exists a borel subset $F$ of a turn transversal $H$, and a flow map composition $\phi: F \longrightarrow \Tilde{H}$. Define $\mu_{\Tilde{H}}$ to be the push-forward measure induced on $\Tilde{H}$ by $\phi$ (where the measure on $F$ is $\mu_H$ restricted to the borel subsets of $F$).
        \item For each singularity $s$ of $X$, let $\mu (s) := 0$.
    \end{enumerate}
    Since flow equivalence is an equivalence relation on the set of transversals of $X$, and since $\nu$ is flow compatible, $\mu$ is well defined. 
\end{proof}

\vspace{4mm}

\subsubsection{Overview of Section}
\label{sec: Overview of Section 6.2}

\vspace{4mm}

In the rest of the section, we will show that each weight assignment on $X$ uniquely determines a transverse measure on $X$. The following remark organizes the proof of the aforementioned claim, and includes references to the supporting arguments contained in the rest of the section. 


\vspace{4mm}

\begin{remark}[Weight Assignments and Transverse Measures]
    \label{Weight Assignments and Transverse Measures}
     Using the concepts laid out above, now we will provide the steps to proving that each \hyperlink{Weight Assignments}{weight assignment} $w$ on $\zeta$ uniquely determines a \hyperlink{transverse measure}{transverse measure} $\mu_w$ on $X(\zeta)$. Suppose $w$ is a given weight assignment on $\zeta$.
     \begin{enumerate}
         \item $w$ allows us to uniquely define a non-negative real valued function $\overline{\mu}_w$ (which we shall call the ``Total Weight Function Induced by $w$"), on the collection of all \hyperlink{Edge Tunnel}{edge transversals} of $X$ [Def. \ref{Total Weight}].
         \item  Each \hyperlink{Turn Transversals}{turn transversal} of $X$ can be expressed as the disjoint union of countably many edge transversals and perhaps a pre-singularity [Prop. \ref{turn transversals in terms of edge transversals}]. Therefore, $\overline{\mu}_w$ uniquely extends to the collection of all turn transversals of $X$ [Lem. \ref{Total Weight Exteding to Turn Transversals}].
         \item Given a turn transversal $H$ of $X$, 
         since $w$ satisfies the \hyperlink{Weight Assignments}{weight equations rel $\zeta$}, 
         $\overline{\mu}_w$ restricted to $\mathcal{B}_H$ is a finite pre measure on $(H,\mathcal{B}_H)$ [Lemma \ref{Total Weight Exteding to Turn Transversals}]. Furthermore, $\overline{\mu}_w$ is a \hyperlink{Flow Compatible Pre Measures}{flow compatible pre measure} on $X$ [Lemma \ref{flow compatible pre measure given by the total weights}]. 
         \item From Carathéodory's Extension Theorem [Rem. \ref{Extension Theorem}] and its application to tree bases [Rem. \ref{Pre-Measures on Tree Bases}], we have that each flow compatible pre measure on $X$ uniquely extends to a transverse measure on $X$ [Lem. \ref{flow compatible pre measures to transverse measures}].
         \item Therefore, $\overline{\mu}_H$ uniquely extends to a transverse measure on $X$ [Prop. \ref{weight assignment to transverse measure}].
     \end{enumerate}
\end{remark}

\vspace{4mm}

\subsubsection{Turn Transversals in Terms of Edge Transversals}
\label{sec: Turn Transversals Expressed in Terms of Edge Transversals}

\vspace{4mm}

Here we will show that, each turn transversal can be written as a disjoint union of countably many edge transversals and a set $S$ where $S$ is either empty or consists of a single point [Prop. \ref{turn transversals in terms of edge transversals}]. 

\vspace{4mm}

We start by refining the concept of a leaf segment into two classes called ``Edge Leaf Segments" and ``Pre Singular Segments" [Def. \ref{Edge Leaf Segments and Pseudo Singular Segments}]. Then we show that, given a standard turn $I$ in $\zeta$, the union of all edge leaf segments that take $I$ (denoted ``\hyperlink{Edge-Tun}{$Tun_I^{Edge}$}"), is a countable disjoint union of \hyperlink{Edge Tunnel}{edge tunnels} [Lemma \ref{Edge Tun_I as a countable disjoint union}], which implies that $Tun_I$ is the disjoint union of countably many edge tunnels and up to one additional pre leaf segment [Prop. \ref{Tun_I as a countable disjoint union}].


\vspace{4mm}




Recall the concept of edge tunnels from Definition \ref{Edge Tunnel}.
To recognize when a given \hyperlink{leaf segments}{leaf segment} is contained in an edge tunnel and when it's not, we introduce the following terminology.

\vspace{4mm}

\begin{definition}[Edge Leaf Segments and Pseudo Singular Segments]
    \label{Edge Leaf Segments and Pseudo Singular Segments}
    \hypertarget{Edge Leaf Segments and Pseudo Singular Segments}{Let} $L$ be a leaf segment in $X$. If there exists a level $j \in \n$ such that $\pi_j (L) \subset G_j$ is contained in the interior of a natural edge in $G_j$, we call $L$ an ``Edge Leaf Segment of $X$". If there exists a $K \in \n$ such that for each integer $j \leq K$, $\pi_j (L)$ contains a natural vertex of $G_j$ (equivalently, $L$ contains a \hyperlink{Pre-Singularities}{pseudo-singularity}), we call $L$ a ``Pseudo Singular Segment of $X$". 
\end{definition}

\vspace{4mm}

Note that, since $\zeta$ is stabilized (specifically from \hyperlink{VSC}{the vertex stabilization criteria}), if $L$ is a pseudo singular segment, $\pi_j (L)$ contains a natural vertex for each $j \in \n$. Thus, a given leaf segment is either an edge leaf segment or a pseudo singular segment (never both). 


\vspace{4mm}

\begin{remark}[Turn Tunnels as Unions of Edge Tunnels and Pre Singular Segments]
    \label{Edge Tun PSing Tun Notation}
    \hypertarget{Edge Tun PSing Tun Notation}{Let} $I$ be a \hyperlink{turn in zeta}{standard turn in $\zeta$}. Then $Tun_I$ is the disjoint union of all edge leaf segments in $Tun_I$ \hypertarget{Edge-Tun}{(denoted ``$Tun_I^{Edge}$")} and all pre singular segments in $Tun_I$ (denoted ``$Tun_I^{PSing}$"). It follows from stabilization that $Tun_I^{PSing}$ is either empty or is a single pseudo singular segment.
\end{remark}

\vspace{4mm}

In Lemma \ref{Edge Tun_I as a countable disjoint union}, we will show that, $Tun_I^{Edge}$ can be expressed as a finite disjoint union of edge tunnels.

\vspace{4mm}

\begin{definition}[Maximal Edge Tunnels]
    \label{Maximal Edge Tunnels}
    \hypertarget{Maximal Edge Tunnels}{Let} $I$ be a \hyperlink{turn in zeta}{standard turn in $\zeta$}. Let $\Tilde{I}$ be a pre-turn of $I$ such that $Tun_{\Tilde{I}}$ is an edge tunnel. If there does not exist an edge tunnel $T$ such that $Tun_{\Tilde{I}} \subset T \subset Tun_I$, we call $Tun_{\Tilde{I}}$ a ``Maximal Edge Tunnel in $Tun_I$".
\end{definition}

\vspace{4mm}

Note that by definition, each pair of maximal edge tunnels in $Tun_I$ must not intersect. We will show in the next lemma that $Tun_I$ contains countably many maximal edge tunnels. 

\vspace{4mm}

\begin{lemma}
    \label{Edge Tun_I as a countable disjoint union}
    \hypertarget{Edge Tun_I as a countable disjoint union}{Let} $I$ be a \hyperlink{turn in zeta}{standard turn in $\zeta$}. $Tun_I^{Edge}$ can be expressed as a countable disjoint union of edge tunnels.
    More specifically, $Tun_I$ has countably many \hyperlink{Maximal Edge Tunnels}{maximal edge tunnels}, and $Tun^{Edge}_I$ is the union of its maximal edge tunnels. 
\end{lemma}

\begin{proof}
     Let $K \in \n$, and let $I$ be a level $K$ standard turn in $\zeta$. Furthermore, let $\mathcal{EL}(I)$ denote the set of all edge leaf segments in $X$ that take $I$. If $I$ is an \hyperlink{Edge Tunnel}{edge turn}, then the conclusion of the lemma is true. Now assume $I$ is not an edge turn.

    \vspace{4mm}

    Each edge leaf segment that takes $I$, belong to some maximal edge tunnel in $Tun_I$. For each $L \in \mathcal{EL}(I)$, denote by $I^L$ the unique pre-turn of $I$ such that,
    \begin{enumerate}
        \item $L$ takes $I^L$, and,
        \item $Tun_{I^L}$ is a maximal edge tunnel in $Tun_I$.
    \end{enumerate}
    (i.e. let $I^L := \pi_k (L)$ where $k := Max \ \{ \ j \in \n + K : \pi_j (L)$ is an edge turn$ \ \}$.)
    Then, $Tun^{Edge}_I = \bigcup_{L \in  \mathcal{EL}(I)} (Tun_{I^L})$. 
    
    \vspace{4mm}
    
    \hypertarget{top edge turns}{Now} let $ Top\text{-}Edge\text{-}Turns (I) := \{ I^L : L \in \mathcal{EL}(I) \}$. It follows from the definition of maximal edge tunnels, that for each pair of distinct elements $J , J' \in Top\text{-}Edge\text{-}Turns (I)$, $(Tun_J) \bigcap (Tun_{J'}) = \emptyset$. Furthermore, since $\zeta$ has countably many levels, and since for each level $j \in -\mathds{N} + K$, there are only finitely many level $j$ pre-turns of $I$ [Lemma \ref{pre-turn_shadows_are_star_set_unions}], $Top$-$Edge$-$Turns (I)$ is a countable set.
\end{proof}

\vspace{4mm}

\begin{proposition}
    \label{Tun_I as a countable disjoint union}
    Let $I$ be a \hyperlink{turn in zeta}{standard turn in $\zeta$}. Then, we have that $Tun_I$ is equal to $(Tun^{Edge}_I) $ $\bigsqcup$ $ (Tun^{PSing}_I)$ where,
    \begin{itemize}
        \item $Tun^{Edge}_I$ is a disjoint union of countably many edge tunnels that are maximal in $Tun_I$, and,
        \item $Tun^{PSing}_I$ is either empty or is equal to a single pseudo singular segment.
    \end{itemize}
\end{proposition}

\begin{proof}
    Follows from Remark \ref{Edge Tun PSing Tun Notation} and Lemma \ref{Edge Tun_I as a countable disjoint union}.
\end{proof}

\vspace{4mm}

\vspace{4mm}

\begin{definition}[Maximal Edge Transversals]
    \label{Maximal Edge Transversals}
    \hypertarget{Maximal Edge Transversals}{Let} $H$ be a turn transversal in $X$. A transversal $H_1 \subseteq H$ is called a ``Maximal Edge Transversal in $H$" if $H_1$ is an edge transversal such that there exist no edge transversal $H_2$ with $H_1 \subset H_2 \subset H$.
\end{definition}

\vspace{4mm}

\begin{proposition}
    \label{turn transversals in terms of edge transversals}
    Let $I$ be a standard turn in $\zeta$. Then for any $q \in I$, there exists a countable index set $\mathds{J}$ such that,
    the turn transversal $H_{I,q} = ( \bigsqcup_{i \in \mathds{J}} H_i )\bigsqcup S$ where,
    \begin{enumerate}
        \item for each $i \in \mathds{J}$, $H_i$ is a \hyperlink{Maximal Edge Transversals}{maximal edge transversal in $H_{I,q}$}, and,
        \item $S$ is either empty or contains a single point.
    \end{enumerate}
\end{proposition}

\begin{proof}
    Assume the set-up of the proposition. If $I$ is an edge turn then $H_{I,q}$ itself is an edge transversal. Now suppose that $I$ is not an edge turn. Then recall from Prop. \ref{Tun_I as a countable disjoint union}, that there exists a countable collection of pre-turns of $I$, $TET(I) := \hyperlink{top edge turns}{Top\text{-}Edge\text{-}Turns(I)}$ such that $\{ Tun_J : J \in TET(I) \}$ is the collection of maximal edge tunnels in $Tun_I$. For each $J \in TET(I)$, let $q^J :=$ the unique pre-image of $q$ contained in $J$.
    Then, 
    $H_{I,q} = F_q \ \cap \ Tun_I = F_q \ \cap \ ( (\bigsqcup_{J \in TET(I)} Tun_J ) \bigsqcup ( Tun^{PSing}_I )) = $ 
    $ (\bigsqcup_{J \in TET(I)} ( F_q \ \cap \ Tun_J )) \ \bigsqcup \ ( F_q \ \cap \ Tun^{PSing}_I )$ 
    $ = (\bigsqcup_{J \in TET(I)} ( H_{J , q^J} )) \ \bigsqcup \ S$, where $S := F_q \ \cap \ Tun^{PSing}_I $. Since $Tun^{PSing}_I$ is either empty or is a single leaf segment, and since each leaf segment of a turn tunnel intersects each cross-section exactly once [Lemma \ref{Intersectig Fibers and Pre Leaf Segments}], $S$ is either empty or consists of a single point.
\end{proof}



\vspace{4mm}

\begin{remark}[Transverse Measures Can Be Recorded Using the Collection of Edge Transversals]
    The above proposition shows that any transverse measure on $X$ can be, in a sense, encoded using only the edge transversals of $X$. More precisely stated, given a transverse measure $\mu$ on $X$ and a turn transversal $H$ of $X$ we can express the measure on $H$ as $\mu (H) = \sum_{i \in \mathds{N}} \mu (H_i) + \mu(S) $ where for each $i \in \mathds{N}$, $H_i$ is an edge transversal and $S$ is finite. 
    Since $X$ is an expanding solenoid,
    any finite set has zero measure [Prop. \ref{Expanding solenoids do not support atomic masses}], and thus we have $\mu (H) = \sum_{i \in \mathds{N}} \mu (H_i)$. Furthermore, since any transversal $\Tilde{H}$ in $X$ is flow equivalent to a borel subset of a turn transversal, and since the borel $\sigma-$algebra of any turn transversal is generated by the collection of turn transversals contained in it, we can recover the measure on $\Tilde{H}$ using the measure on only the turn transversals and thus using only the edge transversals. 
\end{remark}

\vspace{4mm}

The above remark only ensures that when a transverse measure on $X(\zeta)$ is given, we can rebuild the measure from the weight assignment it induces on $\zeta$. We will show in the next sub-section that given a weight assignment on $\zeta$ we can build the transverse measure that then induces that weight assignment.

\vspace{4mm}


\subsubsection{Weight Assignments and Transverse Measures}
\label{Sec: Each Weight Assignment Determines a Transverse Measure}

\vspace{4mm}

Here, we will show that given a weight assignment $w$ on $\zeta$, we can uniquely construct a transverse measure $\mu_w$ on $X$ using $w$. 

\vspace{4mm}

For the discussion in this section, suppose that $w$ is a given \hyperlink{Weight Assignments}{weight assignment} on $\zeta$. We will use $w$ to construct a real valued non-negative function $\overline{\mu}_w$  called the ``Total Weight Function Induced by $w$" on the collection of all \hyperlink{Edge Tunnel}{edge transversals} of $X$ [Def. \ref{Total Weight}], then we will show that $\overline{\mu}_w$ uniquely extends to the collection of all \hyperlink{Turn Transversals}{turn transversals} in $X$ [Lem. \ref{Total Weight Exteding to Turn Transversals}], and it determines a unique \hyperlink{Flow Compatible Pre Measures}{flow compatible pre measure} on $X$ [Lem. \ref{flow compatible pre measure given by the total weights}], which then determines a unique \hyperlink{transverse measure}{transverse measure} on $X$ [Lem. \ref{flow compatible pre measures to transverse measures}].

\vspace{4mm}

\begin{definition}[Total Weight Function Induced by a Weight Assignment]
    \label{Total Weight}
    \hypertarget{Total Weight}{Let} $w$ be a \hyperlink{Weight Assignments}{weight assignment} on $\zeta$. Then the ``Total Weight Function Induced by $w$" denoted by ``$\overline{\mu}_w$" is a real valued non-negative function on the set of \hyperlink{Edge Tunnel}{edge transversals} of $X$ defined as follows: For each natural edge $E$ of a level graph of $\zeta$, and each \hyperlink{Edge Tunnel}{$E-$transversal} $H$, $\overline{\mu}_w (H) := \hyperlink{Weight Assignments}{w (E)}$.
\end{definition}

\vspace{4mm}



\begin{lemma}
    \label{total weight function is flow compatible}
    Let $\overline{\mu}_w$ be the \hyperlink{Total Weight}{total weight function induced by the weight assignment $w$}, on the collection of edge transversals in $X$. If $H , H'$ are two edge transversals that are flow equivalent to each other, then $\overline{\mu}_w (H) = \overline{\mu}_w (H')$.
\end{lemma}

\begin{proof}
    Note that an edge transversal can only be flow equivalent to another edge transversal, as long as they are both cross sections of the same edge tunnel.
    The rest follows from Lemma \ref{edge transversals with consistent weights}.
\end{proof}

\vspace{4mm}

\begin{lemma}
    \label{Total Weight Exteding to Turn Transversals}
    Let $\overline{\mu}_w$ be the total weight function induced by the weight assignment $w$, on the collection of edge transversals in $X$. Then  $\overline{\mu}_w$ uniquely extends to a finite pre measure on $X$ (i.e. $\overline{\mu}_w$ uniquely extends to a countably additive non-negative real valued function on the set of turn transversals of $X$).
\end{lemma}

\begin{proof}
    Assume the set-up of the lemma. Let $H$ be a turn transversal of $X$. From Proposition \ref{turn transversals in terms of edge transversals}, we have that $H = (\bigsqcup_{i \in \mathds{J}} H^i) \bigsqcup S$ where, 
    \begin{enumerate}
        \item[i.] $\mathds{J}$ is a countable set,
        \item[ii.] for each $i \in \mathds{N}$, $H^i$ is a \hyperlink{maximal}{maximal} edge transversal in $H$, and,
        \item[iii.] $S$ is either empty or finite.
    \end{enumerate}
    Define $\overline{\mu}_w (H)$ to be $\sum_{i \in \mathds{J}} \overline{\mu}_w (H^i)$. The uniqueness of this sum follows from the maximality of the edge transversals involved. 
    We will now show that this sum is finite.
    Let $j \in \n$ and let $I$ be a level $j$ standard turn in $\zeta$, such that $H$ is a cross-section of $Tun_I$. 

    \vspace{4mm}
    
    If (Case a.) $I$ is contained in the interior of a natural edge $E$ of $G_j$, then $\overline{\mu}_w (H) = w(E) < \infty$. 

    \vspace{4mm}
    
    Now suppose that (Case b.) $I$ is not an edge turn. 
    Let $E$ be a natural edge of $G_j$ that contains one of $I$'s prongs. (Here, we assume that $E$ denotes the interior of the natural edge mentioned.) Now consider $Tun_E$.
    Since $I \cap E \neq \emptyset$, we can choose a point $q \in I \cap E$. Consider the cross section $\Tilde{H} := H_{E,q}$ of $Tun_E$. Since $\zeta$ is strongly proper, $H_{E,q}$ is equal to the fiber $F_q$, and thus $H_{E,q} \supseteq H_{I,q}$.
    Then it follows that, for each $i \in \mathds{J}$, $\Tilde{H}$ contains an edge transversal $\Tilde{H}^i$ that is flow equivalent to $H^i$. Now, for $i \in \mathds{J}$, let $\Tilde{E}_i$ denote a natural edge such that $\Tilde{H}^i$ is an $\Tilde{E}^i-$transversal.

    \vspace{4mm}
    
    Then, from the definition of $\overline{\mu}_w$ [Definition \ref{Total Weight}], we have that, for each $i \in \mathds{J}$, $\overline{\mu}_w (H^i) = \overline{\mu}_w (\Tilde{H}^i) = \hyperlink{Weight Assignments}{w(\Tilde{E}^i)}$. Since $\bigsqcup_{i \in \mathds{N}} \Tilde{H}^i$ is a disjoint union of edge-transversals contained in $\Tilde{H}$, we have $\sum_{i \in \mathds{N}} \overline{\mu}_w (H^i) = \sum_{i \in \mathds{N}} \overline{\mu}_w (\Tilde{H}^i) < \overline{\mu}_w (\Tilde{H}) = w(E) < \infty$.
\end{proof}

\vspace{4mm}





\vspace{4mm}

\begin{lemma}
    \label{flow compatible pre measure given by the total weights}
    Let $w$ be a weight assignment on $\zeta$. The pre-measure $\overline{\mu}_w$ on $X$ defined in Lemma \ref{Total Weight Exteding to Turn Transversals}, is a \hyperlink{Flow Compatible Pre Measures}{flow compatible pre measure} on $X$.
\end{lemma}

\begin{proof}

    Assume the set-up of the lemma.
    Let $H, H'$ be a pair of flow equivalent turn transversals of $X$, and let $\phi : H \longrightarrow H'$ be a flow map composition. 
    To show that $\overline{\mu}_w (H) = \overline{\mu}_w (\phi (H))$, it is enough to show that for each \hyperlink{turn in zeta}{standard turn $I$ in $\zeta$}, and for each $p,q \in I$, $ \overline{\mu}_w (\hyperlink{Turn Transversals}{H_{I,p}}) = \overline{\mu}_w (\hyperlink{Turn Transversals}{H_{I,q}}) $. Recall from Lemma \ref{Tun_I as a countable disjoint union} that, $Tun_I = Tun_I^{Edge} \bigsqcup Tun_I^{PSing}$ where $Tun_I^{PSing}$ is the union of finitely many leaf segments and $Tun_I^{Edge} = \bigsqcup_{i \in \mathds{N}} Tun_{I_i}$ where for each $i \in \mathds{N}$, $Tun_{I_i}$ is a \hyperlink{maximal}{maximal} edge tunnel in $Tun_I$. For each $i \in \mathds{N}$, we have $\overline{\mu}_w (H_{I_i , p_i}) = \overline{\mu}_w (H_{I_i , q_i})$ where $p_i$ (resp. $q_i$) is the unique pre-image of $p$ (resp. $q$) in $I_i$. And these maximal edge transversals uniquely determine the total weight function value of $H$ and $H'$ to be equal.

    
    
\end{proof}


\vspace{4mm}

\begin{proposition}
    \label{weight assignment to transverse measure}
    Each weight assignment $w$ on $\zeta$ determines a unique transverse measure $\mu_w$ on $X$ such that $\hyperlink{The Weight Assignment Induced by a Transverse Measure}{w(\mu_w)} = w$.
\end{proposition}

\begin{proof}
    Follows from Lemma \ref{flow compatible pre measure given by the total weights} and Lemma \ref{flow compatible pre measures to transverse measures}.
\end{proof}


    

\vspace{4mm}

This concludes the first half of the proof of Theorem \ref{TM(X) is an inverse limit}. Having established that there exists a bijection between the space of weight assignments and $TM(X)$, we will dedicate the next section to laying out descriptive definitions of pseudo weight cones and the space of weight assignments, and showing that the aforementioned bijection preserves all the necessary structures processed by the spaces involved.




\vspace{4mm}

\subsection{The Shape of $TM(X)$}
\label{sec: The Shape of TM(X)}

\vspace{4mm}

Recall that, for the given solenoid $X(\zeta)$, $TM(X)$ denotes the space of transverse measures on $X$ as established in Definition \ref{Transverse Measures}.

\vspace{4mm}

It is well understood that every space of measures has the structure of a convex cone contained in the vector space that is the corresponding space of $signed$ $measures$. This is elaborated upon more in Section \ref{sec: $TM(X)$ as a Convex Cone}. 
But first, we will construct $TM(X)$ using the collection of weight assignments  on $\zeta$ [Definition \ref{Weight Assignments}]. 
The perspective that we build in the rest of this chapter will be used directly in Section \ref{Sec: An Upperbound to the Dimension of $TM(X)$} in showing that $TM(X)$ is finite dimensional [Proposition \ref{TM(X) is finite dimensional}].

\vspace{4mm}

In order to investigate the structure of \hyperlink{transverse measure}{$TM(X)$}, we will realize weight assignments on $\zeta$ as sequences of vectors that follow the weight equations that were laid out in Definition \ref{Weight Assignments}.  

\vspace{4mm}

The mechanism we're about to build in this chapter, is very similar to the concept of vector towers developed in \cite{BedHilLus2020}.

\vspace{4mm}

\subsubsection{Inverse Limit of $PCS(\zeta)$}

\vspace{4mm}

Recall that we assume $\zeta : G_0 \xleftarrow[]{f_{-1}} G_{-1} \xleftarrow[]{f_{-2}} G_{-2} \xleftarrow[]{f_{-3}} ...$ to be a strongly proper stabilized split sequence.
Here, we define a sequence of linear maps between convex cones called the ``Pseudo-Weight Cone Sequence of $\zeta$" [Def. \ref{ACS}] whose inverse limit can be perceived as the space in which \hyperlink{Weight Assignments}{weight assignments of $\zeta$} live.



\vspace{4mm}


\begin{definition}[Pseudo-Weight Spaces and Pseudo-Weight Cones]
    \label{Pre-Weight Cones}
    \hypertarget{Pre-Weight Cones}{For each} $j \in \n$, 
    recall that the set of natural edges of $G_j$ is denoted by $E(G_j)$.
    Then the real vector space spanned by the set of labels $E(G_j)$, shall be called ``The Pseudo-Weight Space of $\zeta$ at Level $j$" and denoted ``$R_j (\zeta)$" or simply ``$R_j$" when there is no ambiguity. (i.e. $R_j (\zeta) := \mathds{R}^{E(G_j)}$.) Furthermore, for each $j \in \n$, ``The Pseudo-Weight Cone of $\zeta$ at level $j$" denoted ``$\Lambda_j (\zeta)$" (or ``$\Lambda_j$" when there is no ambiguity) is the non-negative cone of $R_j$. 
\end{definition}

\vspace{4mm}



\begin{definition}[Transition Matrices]
    \label{transition matrix}
    \hypertarget{transition matrix}{Given} two core graphs $G, H$, and a 
    composition of folds
    $f : G \longrightarrow H$, the ``Transition Matrix of $f$" is the  the matrix $M(f)$ that is determined by the following;
    \begin{enumerate}
        \item the columns of $M(f)$ are indexed by the un-directed natural edges of $G$,
        \item the rows of $M(f)$ are indexed by the un-directed natural edges of $H$, and,
        \item the entry of $M(f)$ corresponding to $i^{\text{th}}$ row and $k^{\text{th}}$ column is given by the number of connected components of $E_k \cap f^{-1} (E_i)$ where $E_k$ is the interior of the edge in $G$ indexed by $k$, and $E_i$ is the interior of the  edge in $H$ indexed by $i$. 
    \end{enumerate}
\end{definition}

\vspace{4mm}

\begin{definition}[Weight Maps and Pseudo-Weight Spaces]
    \label{Weight Maps and Pseudo Weight Spaces}
    For each $j \in -\mathds{N}$, the ``Weight Map of $\zeta$ at Level $j$" is the linear map denoted ``$T_j$" from $R_j(\zeta)$ to $R_{j+1}(\zeta)$ given by the transition matrix of the fold map $f_j$, where $R_j(\zeta)$ and $R_{j+1}(\zeta)$ are the \hyperlink{Pre-Weight Cones}{pseudo-weight spaces} at levels $j$ and $j+1$ respectively as established in Definition \ref{Pre-Weight Cones}.
\end{definition}

\vspace{4mm}

Next, we give a name to the sequence of linear maps we obtained from the definition above. Our goal here, in a sense, is expressing $TM(X)$ in terms of coordinates given by weight assignments.


\vspace{4mm}

\begin{definition}[The Pseudo-Weight Cone Sequence of $\zeta$]
    \label{ACS}
    \hypertarget{ACS}{Consider} the notations established in the above three definitions. Let,
    \begin{equation*}
        \zeta: G_0 \xleftarrow[]{f_{-1}} G_{-1} \xleftarrow[]{f_{-2}} G_{-2} \xleftarrow[]{f_{-3}} ...
    \end{equation*}
    be a strongly proper split sequence. Then the sequence,
    \begin{equation*}
        R_0 (\zeta)\xleftarrow[]{T_{-1}} R_{-1} (\zeta) \xleftarrow[]{T_{-2}} R_{-2} (\zeta) \xleftarrow[]{T_{-3}} ...
    \end{equation*}
    is called the ``Weight Map Sequence of $\zeta$" and is denoted ``$WMS(\zeta)$", whose inverse limit (in the category of modules) shall be called the ``Space of Pseudo-Weight Assignments on $\zeta$" denoted ``$R(\zeta)$".

    \vspace{4mm}
    
    Furthermore, the sequence $\Lambda_0 (\zeta)\xleftarrow[]{T_{-1}} \Lambda_{-1} (\zeta) \xleftarrow[]{T_{-2}} \Lambda_{-2} (\zeta) \xleftarrow[]{T_{-3}} ...$ is called the ``Pseudo-Weight Cone Sequence of $\zeta$" and is denoted ``$PCS(\zeta)$", whose inverse limit (in the category of topological spaces) shall be called the ``Space of Weight Assignments on $\zeta$" denoted ``$\Lambda(\zeta)$". 
\end{definition}

\vspace{4mm}


\begin{remark}[$\Lambda(\zeta)$ as a Subspace of a Vector Space]
    \label{Lambda(zeta) as a Subspace of a Vector Space}
    Note that each point in $R(\zeta)$ can be written in the form of $(w_0 , w_{-1} , w_{-2} , ...)$ where for each $j \in \n$, $w_j \in R_j (\zeta)$ and $T_{j-1} (w_{j-1}) = w_j$. Scalar multiplication and vector addition in $R(\zeta)$ is consistent with the coordinate-wise scalar multiplication and vector addition coming from $R_j$ for each $j \in \n$.

    \vspace{4mm}
    
    Furthermore, note that even though $\Lambda(\zeta)$ as an inverse limit, only manifests as a topological space, since it is a subspace of the vector space $R(\zeta)$, elements of $\Lambda(\zeta)$ still have a well defined notion of scalar multiplication and vector addition inherited from $R(\zeta)$. We will show in Lemma \ref{Lambda is a convex cone} that $\Lambda(\zeta)$ is a convex cone in $R(\zeta)$. 
\end{remark}


\vspace{4mm}



 \begin{figure}[h!]
  \centering
\center{\includegraphics[height=100mm]{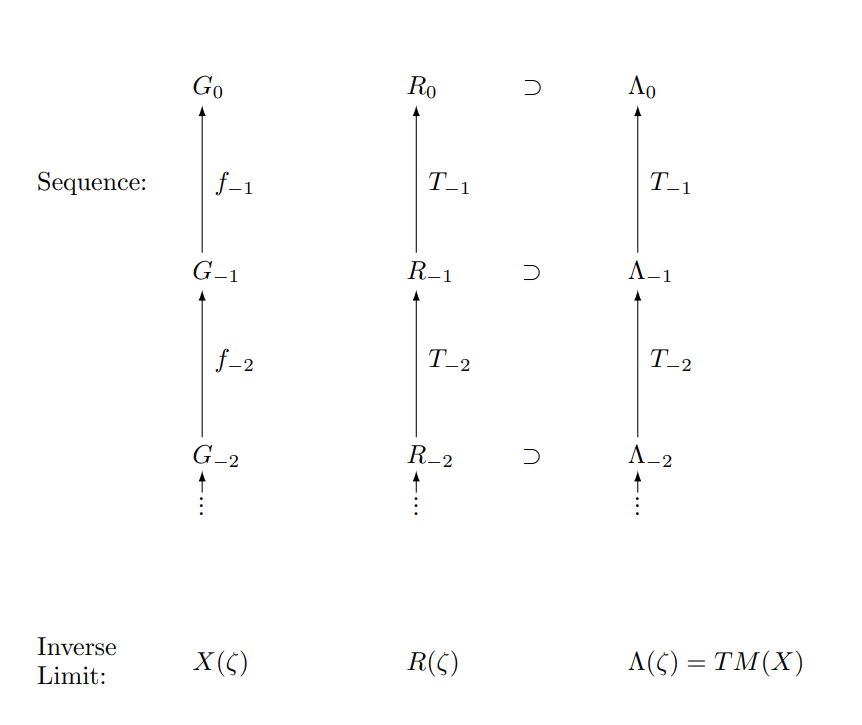}}
  \caption{The Inverse Limit of the Pseudo-Weight Cone Sequence}
\end{figure}


\vspace{4mm}

The following is merely a more descriptive version of Proposition \ref{weight assignment to transverse measure}.

\vspace{4mm}

\begin{proposition}
    \label{TM(X) as a set}
    Let $X(\zeta)$ be a strongly proper solenoid and let $\Lambda(\zeta)$ be the inverse limit of \hyperlink{ACS}{ the pseudo-weight cone sequence of $\zeta$}. $TM(X) = \Lambda(\zeta)$ as sets.
\end{proposition}

\begin{proof}
    $\Lambda(\zeta)$, by design, as a set is the collection of all weight assignments on $\zeta$. We showed in Lemma \ref{Transverse Measures to Weight Assignments} that each element of $TM(X)$ induces a weight assignment. It follows from Proposition \ref{weight assignment to transverse measure} that each weight assignment on $\zeta$ uniquely defines a transverse measure $\mu_w$ that induces $w$ in return (i.e. $w(\mu_w) = w$). 
\end{proof}

\vspace{4mm}

In the next section, we will show that the aforementioned bijection between $TM(X)$ and $\Lambda(\zeta)$, preserves the necessary structures those two spaces possess [Proposition \ref{TM(X) = inv lim ACS}]. 


\vspace{4mm}

\subsubsection{$TM(X)$ as a Convex Cone}
\label{sec: $TM(X)$ as a Convex Cone}

\vspace{4mm}


Since elements of $TM(X)$ are real measures that have an inherent notion of addition and an inherent notion of scalar multiplication by non-negative reals, even when we aren't looking at a wider real vector space that may contain $TM(X)$, we can still check if $TM(X)$ satisfies the criterion of being a convex cone of a real vector space.

\vspace{4mm}

Recall that a subspace $C$ of a real vector space is a convex cone if and only if, for each $v, w \in C$ and $\alpha , \beta \in [0, \infty)$, $\alpha v + \beta w \in C$. Just by the virtue of being a space of measures, $TM(X)$ satisfies this criterion. However, since we plan to use $\Lambda(\zeta)$ [Definition \ref{ACS}] in place of $TM(X)$ [Definition \ref{Transverse Measures}] later on, we will show that, in addition to being isomorphic as sets [Proposition \ref{TM(X) as a set}], $\Lambda(\zeta)$ and $TM(X)$ are isomorphic as convex cones as well [Lemma \ref{Lambda is a convex cone} , Proposition \ref{TM(X) = inv lim ACS}]. 

\vspace{4mm}

We will start by showing that $\Lambda(\zeta)$ is a convex cone as a subspace of \hyperlink{Pre-Weight Cones}{$R(\zeta)$}.

\vspace{4mm}

\begin{lemma}
    \label{Lambda is a convex cone}
    Given a strongly proper solenoid $X$, $\Lambda(\zeta)$ is a convex cone. 
\end{lemma}

\begin{proof}
    Let $\alpha , \beta \in [0, \infty)$ and $v, w \in \Lambda (\zeta)$ where $v = (v_0 , v_{-1} , v_{-2}, ...) $ and $ w = (w_0 , w_{-1} , w_{-2}, ...)$. Note that $\alpha v + \beta w = (\alpha v_0 + \beta w_0 , \alpha v_{-1} + \beta w_{-1} , \alpha v_{-2} + \beta w_{-2}, ...)$. For each $j \in \n$, since $\Lambda_j (\zeta)$ is a convex cone, we have $\alpha v_j + \beta w_j \in \Lambda_j (\zeta)$. Therefore $\alpha v + \beta w \in \Pi_{j \in \n} \Lambda_j (\zeta)$. Furthermore, for each $j \in \n$, $T_{j-1} (\alpha v_{j-1} + \beta w_{j-1}) = \alpha T_{j-1} (v_{j-1}) + \beta T_{j-1} (w_{j-1} ) = \alpha v_j + \beta w_j $. Thus $\alpha v + \beta w \in \Lambda(\zeta)$
\end{proof}

\vspace{4mm}

\begin{proposition}
    \label{TM(X) = inv lim ACS}
    Let $X(\zeta)$ be a strongly proper solenoid and consider $PCS(\zeta)$. There exists a bijection from $TM(X)$ to $\Lambda(\zeta) = \varprojlim PCS(\zeta)$ that preserves the operations of vector addition and scalar multiplication by non-negative reals.
\end{proposition}
    
\begin{proof}
    Let $\zeta$ be the split sequence $G_0 \xrightarrow[]{f_{0}} G_{1} \xrightarrow[]{f_{1}} G_{2} \xrightarrow[]{f_{2}} ...$. For each $j \in \n$, let $d_j \in \mathds{N}$ be the number of natural edges in $G_j$, and suppose we have indexed the set of natural edges of $G_j$ as follows: $E^j_1 , E^j_2 , ... , E^j_{d_j}$. For each transverse measure $\mu$ on $X(\zeta)$, for each $j \in \n$ and for each $i \in \{1,...,d_j\}$, let the edge weight of $E^j_i$ rel $(X(\zeta),\mu)$ [Definition \ref{Edge Weights}] be denoted by $\mu(E^j_i)$. Recall from Definitions \ref{Weight Maps and Pseudo Weight Spaces} and \ref{ACS}, that for each $j \in \n$, $R_j (\zeta)$ is the real vector space generated by the basis of labels $\{ E^j_i : i = 1, ... , d_j \}$ and $\Lambda_j (\zeta)$ is the non-negative cone of $R_j (\zeta)$. 

    \vspace{4mm}
    
    Now define $h : TM(X) \longrightarrow \Lambda(\zeta)$ as follows. For each $\mu \in TM(X)$, let $h(\mu) := (h_0 (\mu) , h_{-1} (\mu) , h_{-2} (\mu) , ... ) \in \Pi_{j \in \n} \Lambda_j (\zeta)$ such that, for each $j \in \n$, $h_j (\mu) := ( \mu (E^j_1) , \mu (E^j_2) , ... , \mu (E^j_{d_j}))$. Since edge-weights rel $\mu$ satisfy the weight equations [Lemma \ref{Transverse Measures to Weight Assignments}], $h(\mu) \in \Lambda (\zeta)$. 

    \vspace{4mm}

    From Lemma \ref{Transverse Measures to Weight Assignments} and Proposition \ref{weight assignment to transverse measure}, we have that $h$ is a bijection. To show that $h$ preserves scalar multiplication, choose $\alpha \in [0, \infty)$. (We do not have to consider negative multiples as a negative multiple of a measure is no longer a measure, and thus would no longer be an element of $TM(X)$.) Now consider the transverse measure $\alpha \mu \in TM(X)$. $h (\alpha \mu) = ( h_0 (\alpha \mu) , h_{-1} (\alpha \mu) , h_{-2} ( \alpha \mu) , ... )$  such that, for each $j \in \n$, $h_j ( \alpha \mu) := (\alpha \mu (E^j_1) , \alpha \mu (E^j_2) , ... , \alpha \mu (E^j_{d_j}))$. On the other hand, $\alpha h (\mu) = \alpha (h_0 (\mu) , h_{-1} (\mu) , h_{-2} (\mu) , ... ) = (\alpha h_0 (\mu) , \alpha h_{-1} (\mu) , \alpha h_{-2} (\mu) , ... ) $ according to the coordinate-wise scalar multiplication in $R(\zeta)$. 

    \vspace{4mm}

    Now to show that $h$ preserves addition, choose $m , m' \in TM(X)$. 
    Note that, 
    \begin{align*}
        h (\mu) + h (\mu') &= (h_0 (\mu) , h_{-1} (\mu) , h_{-2} (\mu) , ... )  + (h_0 (\mu') , h_{-1} (\mu') , h_{-2} (\mu') , ... ) \\
        &= (h_0 (\mu) + h_0 (\mu') , h_{-1} (\mu) + h_{-1} (\mu') , h_{-2} (\mu) + h_{-2} (\mu') , ... ),
    \end{align*}
        
    while $h (\mu + \mu') = (h_0 (\mu + \mu') , h_{-1} (\mu + \mu') , h_{-2} (\mu + \mu') , ... )$ where for each $j \in \n$ and for each $i \in \{1,...,d_j\}$, $(\mu + \mu') (E^j_i) = \mu (E^j_i) + \mu' (E^j_i) $.
    
\end{proof}

\vspace{4mm}


This concludes the proof of Theorem \ref{TM(X) is an inverse limit}, and thus from here on out, we shall use $TM(X)$ and $\Lambda(\zeta)$ interchangeably.
This perspective results in the following immediate consequence.  

\vspace{4mm}

\begin{corollary}
    \label{coarse ub for dim TM(X)}
    Let $n \in \mathds{N}$ and let $\zeta: G_0 \xleftarrow[]{f_{-1}} G_{-1} \xleftarrow[]{f_{-2}} G_{-2} \xleftarrow[]{f_{-3}} ...$ be a strongly proper expanding split sequence of rank $n$. Then, 
    \begin{equation*}
        dim(TM(X(\zeta))) \leq 3(n-1).
    \end{equation*}
\end{corollary}

\vspace{4mm}

\begin{proof}
    Assume the set-up of the corollary. For each $j \in \n$, $| E(G_j) |$ and therefore the dimension of $\mathds{R}^{E(G_j)}$ is bounded above by $3(n-1)$ [Lemma \ref{upperbound to natural vertices}]. The rest follows from the fact that $TM(X) = \Lambda(\zeta)$ is bounded above by the minimum recurring dimension in $\{ \mathds{R}^{E(G_j)} \}_{j \in \n}$.
\end{proof}

\vspace{4mm}

In the next chapter, we will provide a finer upperbound to the dimension of $TM(X)$ and establish a criterion under which a minimal solenoid is uniquely ergodic.

\vspace{4mm}


\vspace{4mm}




\vspace{4mm}

\begin{center}
    \section{A Criterion for Unique Ergodicity} \label{Ch: A Criterion for Unique Ergodcity}
\end{center}

\vspace{4mm}

\h Here, we will continue to assume that $X = X(\zeta)$ the solenoid induced by a \hyperlink{Stabilized Split Sequence}{stabilized}, \hyperlink{Strongly Proper Split Sequences}{strongly proper} split sequence 
\begin{equation*}
    \zeta  : G_0 \xleftarrow[]{f_{-1}} G_{-1} \xleftarrow[]{f_{-2}} G_{-2} \xleftarrow[]{f_{-3}} ...
\end{equation*}
and that, $X$ is  and \hyperlink{def: Expanding Solenoids}{expanding}.

\vspace{4mm}

This chapter is dedicated to proving the following theorem.

\vspace{4mm}

\begin{thm}
    Semi-Normal Solenoids are Uniquely Ergodic.
\end{thm}

\vspace{4mm}

We will define the criterion of \hyperlink{Semi-Noramal Solenoids}{semi-normality} in Subsection \ref{sec: The Criterion of Semi-Normality}.

\vspace{4mm}

In general, a system of the form (space, structure, a structure-preserving map or flow) is called uniquely ergodic if the system can be endowed with only one measure (or transverse measure when appropriate) up to scaling. Further details on ergodic systems, discussed in a general context, can be found in \cite{walters1982introduction}. Here, we will lay out what unique ergodicity means in the context of solenoids. 

\vspace{4mm}

\begin{definition}[Uniquely Ergodic Solenoids]
    \label{def: Uniquely Ergodic Solenoids}
    \hypertarget{def: Uniquely Ergodic Solenoids}{Given} a proper solenoid $X$, we say that ``$X$ is uniquely ergodic" if $X$ has only one \hyperlink{transverse measure}{transverse measure} up to scaling. This is equivalent to saying that \hyperlink{def: TM(X)}{$TM(X)$} is one dimensional.
\end{definition}

\vspace{4mm}

Section \ref{sec: Preliminary Concepts and a Summary of The Main Argument} lays out all the necessary terminology [Subsections \ref{sec: The Sequence of Weight Cones} and \ref{sec: The Criterion of Semi-Normality}] alongside an extended summary of the main proof [Subsection \ref{sec: A Summary of the Main Argument}]. Section \ref{sec: Supporting Arguments} provides the supporting arguments to the main result, and Section \ref{sec: Semi-Normal Solenoids are Uniquely Ergodic} provides the formal proof of the main result. 

\vspace{4mm}

\subsection{Preliminary Concepts and a Summary of The Main Argument}
\label{sec: Preliminary Concepts and a Summary of The Main Argument}

\vspace{4mm}

Recall that for any expanding solenoid $X$, $TM(X)$ can be thought of as the inverse limit of a sequence of linear maps on convex cones [Definition \ref{ACS}, Proposition \ref{TM(X) = inv lim ACS}].

\vspace{4mm}

Here, to show that semi-normal solenoids are uniquely ergodic [Theorem \ref{main thm}], we will adopt a similar methodology to how W. Veech gave a criterion for unique ergodicity of Interval Exchange Transformations in his 1978 paper \cite{Veech1978}. We use a lemma by W. Veech that appeared in the aforementioned paper, which we will lay out alongside a detailed proof, in Subsection \ref{sec: W. Veech's Lemma}.

\vspace{4mm}

Subsection \ref{sec: A Summary of the Main Argument} will provide an outline of the argument in this chapter. But first, we establish some preliminary concepts.

\vspace{4mm}


\subsubsection{The Sequence of Weight Cones}
\label{sec: The Sequence of Weight Cones}

\vspace{4mm}

We start with a necessary refinement of the \hyperlink{ACS}{pre-weight cone sequence} of $\zeta$.

\vspace{4mm}

\begin{notation}[Composition Maps]
    \label{Composition Maps}
    \hypertarget{Composition Maps}{For} each $j \in \n$, let $S_j$ denote a set and let $T_{j-1} : S_{j-1} \longrightarrow S_j$ denote a function. Consider the sequence $S_0 \xleftarrow[]{T_{-1}} S_{-1} \xleftarrow[]{T_{-2}} S_{-2} \xleftarrow[]{T_{-3}} ...$. In this context, for each $i,j \in \n$ such that $i < j$, the composition $T_{j-1} \circ ... \circ T_i : S_i \longrightarrow S_j$ shall be denoted ``$T^i_j$".
\end{notation}

\vspace{4mm}

 \begin{figure}[htpb]
  \centering
\center{\includegraphics[height=90mm]{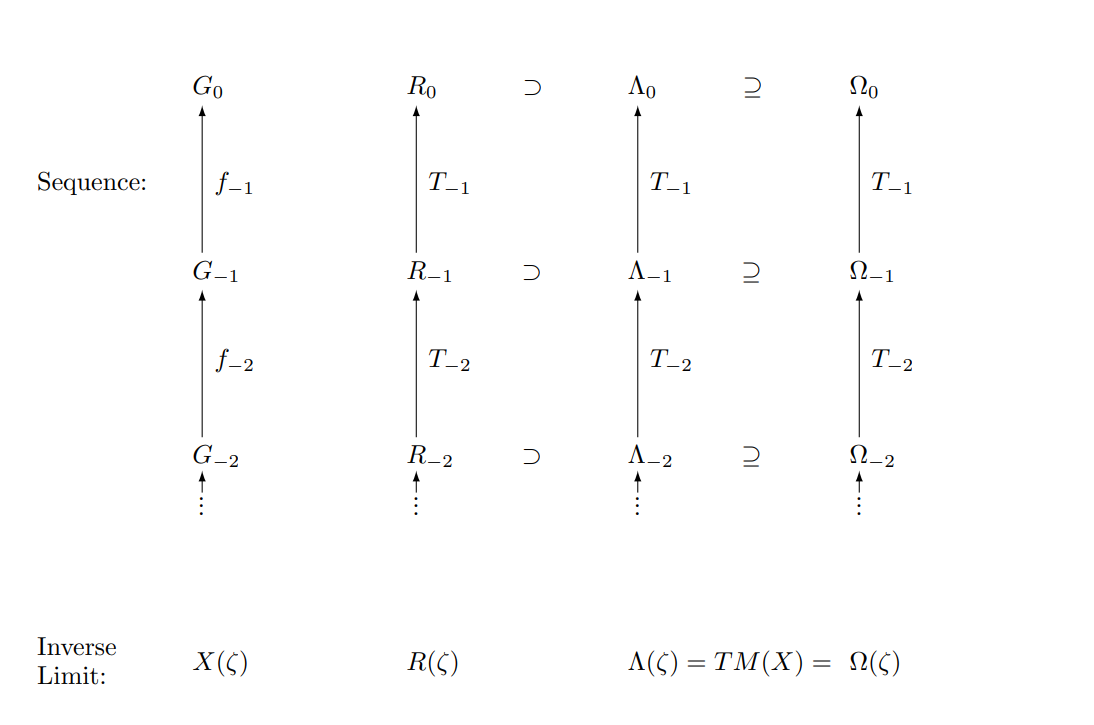}}
  \caption{The Inverse Limit of the Weight Cone Sequence}
  \label{The Inverse Limit of the Weight Cone Sequence}
\end{figure}

\vspace{4mm}

\begin{definition}[Weight Cones]
    \label{Weight Cones}
    \hypertarget{Weight Cones}{Let} $\zeta$ be a strongly proper split sequence and consider its $PCS(\zeta): \Lambda_0 (\zeta)\xleftarrow[]{T_{-1}} \Lambda_{-1} (\zeta) \xleftarrow[]{T_{-2}} \Lambda_{-2} (\zeta) \xleftarrow[]{T_{-3}} ...$ [Definition  \ref{ACS}]. 
    For each $K \in \n$ the ``Weight Cone of $\zeta$ at Level $K$" is given by,
    \begin{equation*}
        \Omega_K (\zeta) := \bigcap_{i \in -\mathds{N} + K} T^i_{K} (\Lambda_i (\zeta)).
    \end{equation*}
    We shall refer to $\Omega_K (\zeta)$ as ``$\Omega_K$" when there is no ambiguity. 
\end{definition}

\vspace{4mm}

\begin{remark}[$TM(X) = \Omega (\zeta)$]
    \label{TM(X) = Omega(zeta)}
    \hypertarget{TM(X) = Omega(zeta)}{Assume} the set-up of the above definition. Note that for each $j \in -\mathds{N}$, $T_j (\Omega_j) = \Omega_{j+1}$, and $\Omega_j = $ the projection of $\Lambda(\zeta)$ in $\Lambda_j (\zeta)$. Therefore, the inverse limit of the sequence of weight cones  $\Omega_0 (\zeta)\xleftarrow[]{T_{-1}} \Omega_{-1} (\zeta) \xleftarrow[]{T_{-2}} \Omega_{-2} (\zeta) \xleftarrow[]{T_{-3}} ...$, denoted ``$\Omega(\zeta)$" is equal to $\Lambda(\zeta) = TM(X)$. This sequence shall be referred to as the ``Weight Cone Sequence of $\zeta$" and shall be denoted ``$WCS(\zeta)$". 
\end{remark}

\vspace{4mm}

In Subsection \ref{Sec: An Upperbound to the Dimension of $TM(X)$}, we lay out an upperbound for the dimension of $TM(X)$ using the weight cone sequence defined above.

\vspace{4mm}

\subsubsection{The Criterion of Semi-Normality}
\label{sec: The Criterion of Semi-Normality}

\vspace{4mm}

Now, we shall lay out the criterion of semi-normality, first for a sequence of linear maps [Definition \ref{Semi-Normal Sequences of Linear Maps}], then for a strongly proper split sequence [Definition \ref{Semi-Noramal Solenoids}].

\vspace{4mm}

\begin{definition}[Semi-Normal Sequences of Linear Maps]
    \label{Semi-Normal Sequences of Linear Maps}
    \hypertarget{Semi-Normal Sequences of Linear Maps}{Let} $\{ d_j \}_{j \in \n}$ be a sequence of positive integers that has at least one constant subsequence.
    Furthermore, let $\mathcal{L} : \mathds{R}^{d_0} \xleftarrow[]{T_{-1}} \mathds{R}^{d_{-1}} \xleftarrow[]{T_{-2}} \mathds{R}^{d_{-2}} \xleftarrow[]{T_{-3}} ...$ be a sequence of linear maps. 
    Given $d \in \mathds{N}$, and a linear map $L : \mathds{R}^{d} \longrightarrow \mathds{R}^{d}$, we say that ``$L$ Recurs in $\mathcal{L}$" if,
    there exists a strictly decreasing sequence $\{J_k\}_{k \in \mathds{N}}$ of non-positive integers such that, for each $k \in \mathds{N}$,
    \begin{enumerate}
        \item $d_{J_k} = d$, and,
        \item $L = T^{J_{2k}}_{J_{2k-1}}$.
    \end{enumerate}
    Furthermore, \hypertarget{positive linear map}{we} say that a linear map is positive, if there exists a positive matrix (i.e. a matrix with all positive entries) representing the linear map. 
    We say that ``$\mathcal{L}$ is a Semi-Normal Sequence of Linear Maps" if there exists a recurring positive linear map in $\mathcal{L}$.
\end{definition}


\vspace{4mm}

It follows from Lemma \ref{upperbound to natural vertices}, that for each strongly proper split sequence $\zeta$, its weight map sequence $R_0 (\zeta)\xleftarrow[]{T_{-1}} R_{-1} (\zeta) \xleftarrow[]{T_{-2}} R_{-2} (\zeta) \xleftarrow[]{T_{-3}} ...$ has at least one \hyperlink{Lowest Recurring Dimension}{recurring dimension} (as laid out more precisely in Remark \ref{WMS(zeta) has a Smallest Recurring Dimension}).


\vspace{4mm}

\begin{definition}[Semi-Normal Split Sequences and Semi-Normal Solenoids]
    \label{Semi-Noramal Solenoids}
    \hypertarget{Semi-Noramal Solenoids}{Given} a strongly proper, expanding split sequence $\zeta$, we say that ``$\zeta$ is a Semi-Normal Split Sequence" and that ``$X(\zeta)$ is a Semi-Normal Solenoid" if $\zeta$'s weight map sequence is a semi-normal sequence of linear maps.
\end{definition}

\vspace{4mm}

\subsubsection{A Summary of the Main Argument}
\label{sec: A Summary of the Main Argument}

\vspace{4mm}

Having laid out all the necessary terminology, here, we provide an overview of the main argument that is formally laid out as the proof of Theorem \ref{main thm} in Section \ref{sec: Semi-Normal Solenoids are Uniquely Ergodic}. The supporting arguments linked here are provided in Section \ref{sec: Supporting Arguments}.

\vspace{4mm}

\hypertarget{A Summary of the Main Argument}{Suppose} that $\zeta$ is a semi-normal split sequence. We aim to prove that $X(\zeta)$ is uniquely ergodic, by showing that $TM(X) = \hyperlink{Weight Cones}{\Omega (\zeta)}$ is $1-$dimensional.
    \begin{enumerate}
    \item A Lemma by W. Veech (laid out in Subsection \ref{sec: W. Veech's Lemma}, and originally from \cite{Veech1978}), implies that, given $m \in \mathds{N}$, and a sequence of linear maps, $\mathcal{L} : \mathds{R}^m \xleftarrow[]{T_{-1}} \mathds{R}^m \xleftarrow[]{T_{-2}} \mathds{R}^m \xleftarrow[]{T_{-3}} ...$, where $\Lambda^m$ denotes the non-negative cone of $\mathds{R}^m$, if there exists a linear map that recurs in $\mathcal{L}$, then $\bigcap_{i \in -\mathds{N}} T^i_{0} (\Lambda^m) \subset \Lambda^m$ is one-dimensional. 
    \item We will show in Proposition \ref{TM(X) is finite dimensional}, that the smallest recurring dimension of the weight cones (laid out more precisely in Subsection \ref{Sec: An Upperbound to the Dimension of $TM(X)$}) is an upper-bound for the dimension of $TM(X)$.
    \item Since $\zeta$ is semi-normal, it follows from item 1, that there exists a strictly decreasing sequence $\{J_k \}_{k \in \mathds{N}}$ of non-positive numbers such that, for each $k \in \mathds{N}$, $\Omega_{J_k} (\zeta) \subset R_{J_k} (\zeta)$ is $1-$dimensional (as laid out more precisely in the proof of Theorem \ref{main thm}). This combined with item 2, imply that the dimension of $\Omega (\zeta) = TM(X)$ is bounded above by $1$. 
    \end{enumerate}


\vspace{4mm}

\subsection{Supporting Arguments}
\label{sec: Supporting Arguments}

\vspace{4mm}

\subsubsection{W. Veech's Lemma on One Dimensional Weight Cones}
\label{sec: W. Veech's Lemma}

\vspace{4mm}

Here, we lay out the lemma by Veech (that was mentioned in item 1. of the outline in section \ref{sec: A Summary of the Main Argument}), and provide a detailed proof which is a slightly extended version of Veech's original proof.
First, we establish some preliminary definitions.

\vspace{4mm}

\begin{definition}[$\delta$-boundedness]
    \label{delta boundedness}
    \hypertarget{delta boundedness}{Let} $\delta > 0$, $m \in \mathds{N}$, and let $M$ be an $m \times m $ matrix that is positive. For each $i,j \in \{ 1,...,m \}$, let $m_{ij}$ denote the entry in $M$ that's in $i^{th}$ row and $j^{th}$ column. We say that $M$ is ``$\delta-$bounded" if for each $i, j, k \in \{ 1,...,m \}$, we have $m_{ij} \leq \delta m_{ik}$ and $m_{ij} \leq \delta m_{kj}$.
\end{definition}

\vspace{4mm}



The main tool used in this subsection is the following inequality that is borrowed from \cite{Veech1978} alongside with Lemma \ref{Veech Lemma}.  

\vspace{4mm}

\begin{remark}[$\delta-$Boundedness Inequality]
    \label{delta-Boundedness Inequality}
    Let $m \in \mathds{N}$, consider the Euclidean space $\mathds{R}^m$, and let $\Lambda^m$ denote the non-negative cone of $\mathds{R}^m$.
    Let $d$ denote the Euclidean metric on $\mathds{R}^m$. Fore each $x,y \in \mathds{R}^m - \{0\}$, let 
    $D(x,y) := d( \frac{x}{|x|} , \frac{y}{|y|} )$ where $|.|$ denotes the Euclidean norm on $\mathds{R}^m$. \hypertarget{Projective Distance}{We} call $D(x,y)$ the ``Projective Distance between $x$ and $y$". Now let $B$ be an $m \times m$ matrix that is $\delta-$bounded for some $\delta > 0$. Let $u , v \in \Lambda^m$. Then the following inequality is true.
    \begin{equation*}
        D( Bu , Bv ) \leq D(u,v) + ln \left( \frac{1 + \delta e^{-D(u,v)}}{1 + \delta e^{D(u,v)}} \right)
    \end{equation*}
\end{remark}

\vspace{4mm}

A proof of the following lemma appears in \cite{Veech1978}, but for the reader's convenience, we shall include a slightly more descriptive version of it here.

\vspace{4mm}

\begin{lemma}
    \label{Veech Lemma}
    Let $\delta > 0$, $m \in \mathds{N}$, let $\Lambda^m$ be the non-negative cone in $\mathds{R}^m$ and let $B$ be an $m \times m$ matrix that is $\delta-$bounded. For each $i \in \mathds{N}$ let $U_i$ be a non-negative $m \times m$ matrix. Then $\bigcap_{i=1}^{\infty} B U_1 B U_2 B ... B U_i B (\Lambda^m) =: \Omega_0$ is one dimensional.
\end{lemma}

\begin{proof}
    Let $u , v \in \Omega_0$. We will show that $D(u,v) = 0$. Choose two sequences of pairs $\{ (u_i , v_i) \}_{i \in \mathds{N}} , \{ (w_i , z_i) \}_{i \in \mathds{N}} \subset \Lambda^m \times \Lambda^m$ as follows. Let $(u_1 , v_1) := (u,v)$ and let $(w_1 , z_1)$ be a point in $B^{-1} (u) \times B^{-1} (v)$. 
    We will choose the rest inductively. Let $k \in \mathds{N} + 1$, and suppose that $\{ (u_i , v_i) \}_{i =1}^{k-1} , \{ (w_i , z_i) \}_{i = 1}^{k-1} \subset \Lambda^m \times \Lambda^m$ have been already chosen. Then choose $(u_k , v_k)$ from $(U_{k-1})^{-1} (w_{k-1}) \times (U_{k-1})^{-1} ( z_{k-1})$, and then choose $(w_k , z_k)$ from $B^{-1} (u_{k}) \times B^{-1} (v_{k})$. 

    \vspace{4mm}

    Now consider the sequence $\{ t_i \}_{i \in \mathds{N}}$ of reals defined by the following: for each $k \in \mathds{N}$, $t_{2k -1} := D(u_k , v_k )$, and $t_{2k} := D(w_k , z_k )$. Note that, since $B$ is positive, and since for each $i \in \mathds{N}$, $U_i$ is non-negative, $\{ t_i \}_{i \in \mathds{N}}$ is a monotonically non-decreasing sequence. Furthermore, since the \hyperlink{Projective Distance}{projective distance} between any two non-negative vectors, is always contained in $[0,1]$, $t_i \in [0,1]$ for each $i \in \mathds{N}$. Thus, $\{ t_i \}_{i \in \mathds{N}}$ converges to some limit $\sigma \in [0,1]$. Since both $ \{ D(u_i , v_i ) \}_{i \in \mathds{N}} , \{ D(w_i , z_i ) \}_{i \in \mathds{N}}$ are subsequences of $\{ t_i \}_{i \in \mathds{N}}$, they must also converge to $\sigma$. 

    \vspace{4mm}

    Assume $D(u,v) := \alpha > 0$. Then since $\{ t_i \}_{i \in \mathds{N}}$ is monotonically non-decreasing, $\sigma > 0$. From the $\delta-$Boundedness Inequality [Remark \ref{delta-Boundedness Inequality}], we obtain that, for each $i \in \mathds{N}$,
    \begin{equation*}
        D( Bw_i , Bz_i ) \leq D(w_i,z_i) + ln \left( \frac{1 + \delta e^{-D(w_i,z_i)}}{1 + \delta e^{D(w_i,z_i)}} \right)
    \end{equation*}
    i.e.
    \begin{equation*}
        D( u_i , v_i ) \leq D(w_i,z_i) + ln \left( \frac{1 + \delta e^{-D(w_i,z_i)}}{1 + \delta e^{D(w_i,z_i)}} \right)
    \end{equation*}
    As $i$ approaches $\infty$, we obtain,
    \begin{equation*}
        \sigma \leq \sigma + ln \left( \frac{1 + \delta e^{-\sigma}}{1 + \delta e^{\sigma}} \right)
    \end{equation*}
    which means, when $\sigma$ is positive, we have,
    \begin{equation*}
        \sigma \leq \sigma - \epsilon
    \end{equation*}
    where $\epsilon > 0$. This is a contradiction. Thus $D(u,v) = 0$ which implies that $u ,v$ are part of the same ray in $\Lambda^m$
\end{proof}

\vspace{4mm}

While the above result uses a sequence of linear maps between Euclidean spaces of the same dimension, we plan to apply this result in the context of weight maps where the dimension of the vector spaces involved can vary. The result [Lemma \ref{TM(X) is finite dimensional}] in the following Subsection is meant to overcome this obstacle. 


\vspace{4mm}

\subsubsection{An Upperbound to the Dimension of $TM(X)$}
\label{Sec: An Upperbound to the Dimension of $TM(X)$}

\vspace{4mm}

\vspace{4mm}

In Subsection \ref{sec: Semi-Normal Solenoids are Uniquely Ergodic}, we will use Lemma \ref{Veech Lemma} coupled with the following proposition to recognize instances when $ \Omega (\zeta) = TM(X)$ is one dimensional.

\vspace{4mm}

\begin{definition}[Lowest Recurring Dimension]
    \label{Lowest Recurring Dimension}
    \hypertarget{Lowest Recurring Dimension}{For} each $j \in \n$, let $S_j$ denote either a real vector space or a convex cone of a real real vector space, and suppose that the dimension of $S_j$ is $d_j \in \mathds{N}$.  Furthermore, let $\mathcal{S} : S_0 \xleftarrow[]{T_{-1}} S_{-1} \xleftarrow[]{T_{-2}} S_{-2} \xleftarrow[]{T_{-3}} ...$ be a sequence of linear maps. Let $D_{\mathcal{S}} :=  \{ d \in \mathds{N} : \{ d_j \}_{j \in \n} $ has a constant subsequence where that constant is $d \}$ be called the ``Set of Recurring Dimensions of $\mathcal{S}$". If $D_{\mathcal{S}}$ is non-empty, we call the minimum of $D_{\mathcal{S}}$ the ``Smallest Recurring Dimension in $\mathcal{S}$".
\end{definition}

\vspace{4mm}

\begin{remark}[$WMS(\zeta)$ has a Smallest Recurring Dimension]
    \label{WMS(zeta) has a Smallest Recurring Dimension}
    Let $n \in \mathds{N}$. Let $\zeta$ be a proper split sequence of rank $n$. Consider the weight map sequence associated with $\zeta$, $WMS(\zeta) : R_0 (\zeta)\xleftarrow[]{T_{-1}} R_{-1} (\zeta) \xleftarrow[]{T_{-2}} R_{-2} (\zeta) \xleftarrow[]{T_{-3}} ...$. For each $j \in \n$, let $d_j$ be the dimension of $R_j$. Recall from that, for each $j \in \n$, $d_j =$ the number of natural edges in the level graph $G_j$ of $\zeta$. It follows from Lemma \ref{upperbound to natural vertices}, that the number of natural edges of a rank $n$ core graph is bounded below and above. Therefore, $WMS(\zeta)$ must have at least one recurring dimension. 
\end{remark}

\vspace{4mm}

\begin{proposition}
    \label{TM(X) is finite dimensional}
    Let $\zeta$ be a strongly proper split sequence. Consider the sequence of weight cones  $\Omega_0 (\zeta)\xleftarrow[]{T_{-1}} \Omega_{-1} (\zeta) \xleftarrow[]{T_{-2}} \Omega_{-2} (\zeta) \xleftarrow[]{T_{-3}} ...$ and its inverse limit $\Omega (\zeta)$. Let the smallest recurring dimension in $WCS(\zeta)$ be $d \in \mathds{N}$. Then $d$ is an upperbound to the dimension of $\Omega (\zeta) = TM(X)$.
\end{proposition}


\begin{proof}
    Assume the set-up of the proposition. Suppose $\Omega (\zeta)$ contains a set $\{w^i \}_{i=1}^{d+1}$ of $d+1$ linearly independent vectors (with respect to the vector space structure laid out in Remark \ref{Lambda(zeta) as a Subspace of a Vector Space} for $R(\zeta) \supset \Omega (\zeta)$). Then there must be a level $J \in \n$ such that their projections $\{w^i_J \}_{i=1}^{d+1}$ into $\Omega_J (\zeta) \subset R_J (\zeta)$ are linearly independent, and furthermore, for each integer $j < J$, $\{w^i_j \}_{i=1}^{d+1}$ is a collection of $d+1$ linearly independent vectors in $\Omega_j (\zeta) \subset R_j (\zeta)$. But since $d$ is a recurring dimension, it follows that there exists an integer $K < J$ such that $\Omega_K$ is of dimension $d$.
\end{proof}

\vspace{4mm}

The next section lays out a more detailed version of the summary in Subsection \ref{sec: A Summary of the Main Argument} in the form of a proof for the main theorem [\ref{main thm}]. 



\subsection{Semi-Normal Solenoids are Uniquely Ergodic}
\label{sec: Semi-Normal Solenoids are Uniquely Ergodic}

\vspace{4mm}

\vspace{4mm}

Here, we lay out the main theorem and its proof.

\vspace{4mm}

\subsubsection{The Main Theorem}
\label{sec: The Main Theorem}

\vspace{4mm}

\begin{thm}
    \label{main thm}
    \hypertarget{main thm}{Semi-normal} solenoids 
    are uniquely ergodic. 
\end{thm}

\begin{proof}
    Let $\zeta$ be a semi-normal split sequence [Definition \ref{Semi-Noramal Solenoids}] and $X = X(\zeta)$ the solenoid induced by $\zeta$. Semi-normality implies that $\zeta$ satisfies full mingling [Definition \ref{Full Mingling}], and therefore, from the mingling lemma [\ref{The Mingling Lemma}], we have that $X$ is a minimal solenoid. It follows from minimality that $\hyperlink{Transverse Measures}{TM(X)} = \hyperlink{ACS}{\Lambda (\zeta)} = \hyperlink{TM(X) = Omega(zeta)}{\Omega (\zeta)}$ [Proposition \ref{TM(X) = inv lim ACS}, Remark \ref{TM(X) = Omega(zeta)}]. 
    We aim to prove that $X(\zeta)$ is uniquely ergodic, by showing that $TM(X) = \hyperlink{Weight Cones}{\Omega (\zeta)}$ is $1-$dimensional.

    \vspace{4mm}

    Consider the weight cone sequence of $\zeta$, $WCS(\zeta) : \Omega_0 (\zeta)\xleftarrow[]{T_{-1}} \Omega_{-1} (\zeta) \xleftarrow[]{T_{-2}} \Omega_{-2} (\zeta)$ $\xleftarrow[]{T_{-3}} ...$ [Definition \ref{Weight Cones}], and its inverse limit $\Omega (\zeta)$.
    Recall from Proposition \ref{TM(X) is finite dimensional}, that the smallest recurring dimension in $WCS(\zeta)$ [Definition \ref{Lowest Recurring Dimension}] is an upper-bound for the dimension of $TM(X)$. We aim to show that the smallest recurring dimension in $WCS(\zeta)$ is one.

    \vspace{4mm}
    
    From semi-normality [Definition \ref{Semi-Normal Sequences of Linear Maps}], it follows that there exist, a recurring dimension $d \in \mathds{N}$ in the weight map sequence of $\zeta$, $ R_0 (\zeta)\xleftarrow[]{T_{-1}} R_{-1} (\zeta) \xleftarrow[]{T_{-2}} R_{-2} (\zeta) \xleftarrow[]{T_{-3}} ...$, a \hyperlink{positive linear map}{positive linear map} $L : \mathds{R}^{d} \longrightarrow \mathds{R}^{d}$, and
    a strictly decreasing sequence $\{J_k\}_{k \in \mathds{N}}$ of non-positive integers such that, for each $k \in \mathds{N}$,
    \begin{enumerate}
        \item $d_{J_k} = d$, and,
        \item $L = \hyperlink{Composition Maps}{T^{J_{2k}}_{J_{2k-1}}}$.
    \end{enumerate}
    Now consider the sequence of pseudo-weight cones given by,
    \begin{equation*}
        \Lambda_{J_1} (\zeta)\xleftarrow[]{T^{J_2}_{J_1}} \Lambda_{J_2} (\zeta) \xleftarrow[]{T^{J_3}_{J_2}} \Lambda_{J_3} (\zeta) \xleftarrow[]{T^{J_4}_{J_3}} ....
        \vspace{4mm}
    \end{equation*}
    Note that, for each $k \in \mathds{N}$, $\Lambda_{J_k} (\zeta)$ is the non-negative cone of $\mathds{R}^d$, $T^{J_{2k}}_{J_{2k-1}} = L$, and $T^{J_{2k+1}}_{J_{2k}}$ is a linear map represented by a non-negative $d \times d$ matrix. Note that the positive matrix corresponding to $L$ is $\delta-$bounded [Definition \ref{delta boundedness}] for some $\delta > 0$. Thus, it follows from Lemma \ref{Veech Lemma}, that for each $k \in \mathds{N}$, $\Omega_{J_{2k-1}} (\zeta)$ is one dimensional in $R_{J_{2k-1}} (\zeta)$.
\end{proof}

\vspace{4mm}

\vspace{4mm}

\vspace{4mm}

This leads to two naturally arising follow-up questions.
\begin{enumerate}
    \item Are there examples of minimal yet non uniquely ergodic solenoids?
    \item Can we construct a space of solenoids, define a natural measure on it, and claim that almost every solenoid is uniquely ergodic?
\end{enumerate}

\vspace{4mm}

If we replace the notion of ``solenoids" above with ``interval exchange maps" (resp. ``measured foliations on compact hyperbolic surfaces"), the first question was answered positively in \cite{Keynes1976} by Keynes and Newton (resp. a positive answer is inferred by \cite{Keynes1976} since each example of an IET can be used to construct a measured foliation with similar behavior), and the second was answered positively in \cite{413dc63a-6607-304b-a389-80bc0b0e5018}, \cite{bb0e1df0-cb17-3718-8ef6-9df709f6df42} independantly by Veech and Masur (resp. in \cite{bb0e1df0-cb17-3718-8ef6-9df709f6df42} by Masur). Furthermore, for two classes of objects (namely $\mathds{R}-$Trees and Currents on Free Groups) related to dynamics of $Out(F_n)$ (the outer automorphism group of the rank$-n$ free group), the first question was answered positively in \cite{Bestvina_2024} by Bestvina, Gupta and Tao. However, for many of the classes of dynamical systems that can be used to study dynamics of $Out(F_n)$, the second question still remains open.


\vspace{4mm}










\bibliography{references} 







\end{document}